\documentclass[reqno,10pt]{amsart}
\usepackage[utf8]{inputenc}
\usepackage{a4wide,amsmath,amsfonts,amssymb,amsthm}
\usepackage{mathtools,stmaryrd,bbm}
\usepackage{subfiles,appendix}
\usepackage{tikz}
\usepackage{hyperref} 

\title{Study of the collision phenomenon}

\newtheorem{lemm}{Lemma}[section]
\newtheorem{theo}[lemm]{Theorem}
\newtheorem{coro}[lemm]{Corollary}

\newtheorem{propo}[lemm]{Proposition}

\newtheorem{claim}[lemm]{Claim}
\newtheorem{assump}[lemm]{Assumption}
\newtheorem{rema}[lemm]{Remark}

\numberwithin{equation}{section}

\newcommand{\bj}{\mathbf{j}}
\newcommand{\bx}{\mathbf{x}}
\newcommand{\by}{\mathbf{y}}
\newcommand{\bomega}{\boldsymbol{\omega}}
\newcommand{\bz}{\mathbf{z}}

\newcommand{\n}{\overrightarrow{n}}
\newcommand{\mAi}{\overrightarrow{m_{i}}}
\newcommand{\mAone}{\overrightarrow{m_{1}}}
\newcommand{\mAtwo}{\overrightarrow{m_{2}}}

\newcommand{\pAi}{\overrightarrow{p_{i}}}
\newcommand{\pAone}{\overrightarrow{p_{1}}}
\newcommand{\pAtwo}{\overrightarrow{p_{2}}}

\newcommand{\C}{\mathcal{C}}
\newcommand{\E}{\mathcal{E}}

\newcommand{\MQ}{\overrightarrow{M Q}}
\newcommand{\MRi}{\overrightarrow{M R_i}}
\newcommand{\MRone}{\overrightarrow{M R_1}}
\newcommand{\MRtwo}{\overrightarrow{M R_2}}
\newcommand{\MRti}{\overrightarrow{M \Rti}}

\newcommand{\NQ}{\overrightarrow{N Q}}
\newcommand{\NRti}{\overrightarrow{N \Rti}}
\newcommand{\NRi}{\overrightarrow{N R_i}}
\newcommand{\NRtone}{\overrightarrow{N \Rtone}}
\newcommand{\NRttwo}{\overrightarrow{N \Rttwo}}
\newcommand{\Y}{\mathcal{Y}}
\newcommand{\Rtone}{\tilde{R}_1}
\newcommand{\Rttwo}{\tilde{R}_2}
\newcommand{\Rti}{\tilde{R}_i}

\newcommand{\red}[1]{\color{red}{#1}\color{black}}
\newcommand{\blue}[1]{\color{blue}{#1}\color{black}}

\usepackage{comment}
\excludecomment{toexclude}
\begin{toexclude}
\end{toexclude}

\usetikzlibrary{math} 
\tikzmath{\Num=200;}
\tikzmath{\Numbis=200;}


\begin{document}

\title[Collision of two nearly equal solitary waves for the ZK equation]{Dynamics of the collision of two nearly equal solitary waves for the Zakharov-Kuznetsov equation}

\author[D. Pilod]{Didier Pilod}
\address{Department of Mathematics, University of Bergen, Postbox 7800, 5020 Bergen, Norway}
\email{Didier.Pilod@uib.no}

\author[F. Valet]{Fr\'ed\'eric Valet}
\address{CY Cergy Paris Universit\'e. Laboratoire de recherche Analyse, G\'eom\'etrie, Mod\'elisation (UMR CNRS 8088), 2 avenue Adolphe Chauvin, 95302 Cergy-Pontoise Cedex, France}
\email{Frederic.Valet@cyu.fr}

\date{\today}

\maketitle

\begin{abstract}
		We study the dynamics of the collision of two solitary waves for the Zakharov-Kuznetsov equation in dimension $2$ and $3$. We describe the evolution of the solution behaving as a sum of $2$-solitary waves of nearly equal speeds at time $t=-\infty$ up to time $t=+\infty$. We show that this solution behaves as the sum of two modulated solitary waves and an error term which is small in $H^1$ for all time $t \in \mathbb R$. Finally, we also prove the stability of this solution for large times around the collision. 

  The proofs are a non-trivial extension of the ones of Martel and Merle for the quartic generalized Korteweg-de Vries equation to higher dimensions. First, despite the non-explicit nature of the solitary wave, we construct an approximate solution in an intrinsic way by canceling the error to the equation only in the natural directions of scaling and translation. Then, to control the difference between a solution and the approximate solution, we use a modified energy functional and a refined modulation estimate in the transverse variable. Moreover, we rely on the hamiltonian structure of the ODE governing the distance between the waves, which cannot be approximated by explicit solutions, to close the bootstrap estimates on the parameters. We hope that the techniques introduced here are robust and will prove useful in studying the collision phenomena for other focusing non-linear dispersive equations with non-explicit solitary waves.
	\end{abstract}


\section{Introduction}

\subsection{The Zakharov-Kuznetsov equation}

The Zakharov-Kuznetsov (ZK) equation 
\begin{equation} \label{ZK}
\partial_tu+\partial_x \left( \Delta u+u^2 \right)=0 .
\end{equation}
where $u=u(t,\bx)$ is a real-valued function, $\bx=(x,\by) \in \mathbb R \times \mathbb R^{d-1}$, $d \ge 2$, $t \in \mathbb R$ and $\Delta=\partial_x^2+\Delta_{\by}^2$ denotes the Laplacian in $\mathbb R^d$, is a natural high dimensional generalisation of the Korteweg-de Vries (KdV) equation, which corresponds to the case $d=1$. 
In dimensions $2$ and $3$, this equation has been introduced by Zakharov and Kuznetsov \cite{ZK74} (see also \cite{KRZ86,Bellan}) to describe the propagation of ionic-acoustic waves in a uniformly magnetized cold plasma. The physical derivation of ZK from the Euler-Poisson system was performed rigorously by Lannes, Linares and Saut \cite{LaLiSa13} (see also \cite{Pu13}). The ZK equation was also rigorously derived by Han-Kwan \cite{HanKwan13} from the Vlasov-Poisson system in the presence of an external magnetic field. 

Unlike the Korteweg-de Vries equation (KdV) in dimension $1$ or the Kadomtsev-Petviashvili equation (KP) in dimension $2$, ZK is not completely integrable in dimensions $2$ or $3$.
Nevertheless, it has a Hamiltonian structure and possesses the following conserved quantities:
\begin{align}\label{conserved_quantities}
    \int_{\mathbb{R}^2} u(t, \bx) d\bx, \quad M(u(t))=\int_{\mathbb{R}^2} u(t, \bx)^2 d\bx \quad \text{and} \quad E(u(t)) = \int_{\mathbb{R}^2} \frac{\vert \nabla u(t,\bx) \vert^2}{2} - \frac{u^3(t,\bx)}{3} d\bx.
\end{align}
Moreover, the solutions of ZK are invariant under scaling and translation; more precisely, if $u$ is a solution of \eqref{ZK}, then, for any $c>0$, $\bz_0 \in \mathbb R^d$, 
\begin{equation} \label{scaling}
u_{c,\bz_0}(t,\bx)=cu(c^{\frac32}t,c^{\frac12}(\bx-\bz_0))
\end{equation}
is also a solution of \eqref{ZK}. Here, we will focus on the two and three dimensional cases which are $L^2$-subcritical. 

The initial value problem (IVP) associated to the ZK equation has been shown to be globally well-posed in the energy space $H^1(\mathbb R^d)$, in dimension $d=2$ by Faminskii \cite{Fam95} and in dimension $d=3$ by Herr and Kinoshita \cite{HeKin23}. The result in dimension $d=2$ was improved in \cite{LiPa09,GruHe14,MoPi15} culminating with the well-posedness in $H^s(\mathbb R^2)$, $s>-\frac14$, in \cite{Kino21}. The local well-posedness result of Herr and Kinoshita \cite{HeKin23} holds in $H^s(\mathbb R^3)$, for $s>-\frac12$, which is optimal from a scaling point of view. We also refer to \cite{LiSa09,RiVe12,MoPi15} for former well-posedness results in the $3$-dimensional case. 

Below, we summarize the global well-posedness results of \cite{Fam95,HeKin23} in the energy space $H^1(\mathbb R^d)$, $d=2,3$: for any $u_0 \in H^1(\mathbb R^d)$, there exists a unique (in some sense) solution to \eqref{ZK}
in $C(\mathbb R : H^1(\mathbb R^d))$. Moreover, $M(u)(t)=M(u_0)$ and $E(u)(t)=E(u_0)$, for all $t \in \mathbb R$.

\subsection{Solitary waves} 
The ZK equation admits a family of special solutions propagating at constant speed in the first direction. These solutions, which play an important role in the ZK dynamics, are called \emph{solitary wave} solutions. They are of the form 
\begin{equation} \label{def:solwave}
u(t,\bx)=u(t,x,\by)=Q_c(x-ct,\by), \quad \text{with} \quad Q_c(\bx) \underset{|\bx| \to \infty}{\longrightarrow} 0 ,
\end{equation}
where $c>0$, $Q_c(\bx) = c Q(\sqrt{c} \bx)$ and $Q$ is the \emph{ground state} of
\begin{align}\label{eq:Q}
    -\Delta Q + Q -Q^2=0.
\end{align}
Indeed, the elliptic equation \eqref{eq:Q} has a unique smooth positive radial solution, called \emph{ground state}. We refer for example to \cite{Str77,BLP81} for the existence and to \cite{Kwo89} for the uniqueness. Even though there is no explicit formula for $Q$ as for the generalised KdV equation (gKdV) in dimension $1$, it is well-known that $Q$ decays exponentially at infinity. These classical results are gathered in the proposition below. 

\begin{propo}\label{propo:Q} The following assertions hold true. 
\begin{itemize}
    \item[(i)] There exists a smooth, positive, radial solution of \eqref{eq:Q}, which is exponentially decreasing at infinity. This solution, denoted by $Q$, is called a ground state of \eqref{eq:Q}.
    \item[(ii)] The ground-state $Q$ of \eqref{eq:Q} is unique.
    \item[(iii)] The first order asymptotic expansion of $Q$ at infinity is given by
    \begin{align}\label{eq:bound_Q_first_order}
        \exists \mu>0, \quad \lim_{r \rightarrow +\infty} r^{\frac{d-1}{2}}e^{r}Q(r) =\mu.
    \end{align}
\end{itemize}
\end{propo}

\begin{rema}
Whereas  first investigated in \cite{Str77}, the asymptotic expansion at infinity of $Q$ given in (iii) is a result from \cite{GNN81}. Its proof is based on the maximum principle and Hopf Lemma. We give in Appendix \ref{app:asymptotic_Q} an alternative proof in dimension $d=2$ relying on ODE arguments. 
\end{rema}

From now on, we will work with the ground state $Q$ associated to \eqref{eq:Q}. By abuse of notations, we sometimes consider $Q$ as a function defined on $\mathbb{R}_+$, with the identification $Q(\bx)=Q(\vert \bx \vert)$. In polar coordinates and in dimension $d=2$, $Q$ satisfies the equation
\begin{align*}
    -Q''(r) -\frac{d-1}{r}  Q'(r) + Q(r) - Q^2(r)=0.
\end{align*}
In this work, it will be important to have a sharper asymptotic of $Q$ at infinity as the one given in \eqref{eq:bound_Q_first_order}. The next proposition, whose proof is given in Appendix \ref{app:asymptotic_Q}, shows that in dimension $2$, the modified Bessel function of second kind $K_0$ is a good approximation of $Q$ at infinity. We refer for example to \cite{AS64,AS61} for the definition and the properties of $K_0$, (see also Appendix \ref{app:asymptotic_Q}). This result seems to be new and may be useful in the study of other nonlinear dispersive equations such as the nonlinear Schr\"odinger equation. 
\begin{propo}\label{propo:Q_K_0}
In dimension $d=2$, there exists $\kappa>0$ such that
\begin{align}\label{Q:asymptotic_d=2}
    \forall r>1, \quad \left\vert Q(r)- \kappa K_0(r) \right\vert \lesssim \frac{e^{-2r}}{r}.
\end{align}
In dimension $d=3$, there exists $\kappa>0$ such that
\begin{align}\label{asymp_Q_d3}
    \forall r>1, \quad \left\vert Q(r) - \kappa \frac{e^{-r}}{r} \right\vert \lesssim \frac{e^{-r}}{r^2}. 
\end{align}
\end{propo}

\begin{rema} We will not work directly with \eqref{Q:asymptotic_d=2}. Instead, we use an approximation of $K_0$  by a finite series (see Proposition \ref{propo:asymp_Q} for more details). 
\end{rema} 

The solitary wave solutions of ZK defined in \eqref{def:solwave} are known to be orbitally stable in the energy space $H^1(\mathbb R^d)$, for $d=2$ and $d=3$, since the work de Bouard \cite{deB96} in the $90's$. More recently, the asymptotic stability of the family $\{Q_c\}_{c>0}$ has been proved in dimension $d=2$ by C\^ote, Mu\~noz, Simpson and the first authors in dimension $d=2$ \cite{CMP16} and by Farah, Holmer, Roudenko and Yang in dimension $d=3$ \cite{FHRY23}. The proofs in \cite{CMP16,FHRY23} are a nontrivial extensions of the arguments of Martel and Merle for the gKdV equation in one dimension \cite{MM01,MM05,MM08} to higher dimensions. They are based on monotonicity arguments on oblique hyper-planes and rigidity results classifying the solutions around the solitary waves (Liouville properties). It is worth noting that the virial estimates, which are key ingredient in proving the Liouville properties, rely on numerical arguments (the sign of a scalar product in dimension $2$ and the positivity of a linear operator in dimension $3$ under suitable orthogonality conditions). This is somehow due to the non-explicit character of the ground state $Q$.

Finally, using new virial estimates, Mendez, Mu\~ noz, Poblete and Pozo \cite{MMPP21} prove that the solutions of ZK vanish in some large time-dependent regions that do not contain solitary waves.

\subsection{Collision phenomena for nonlinear dispersive equations} 

Discovered in the second half of the $19^{th}$ century, solitary waves propagating in nonlinear dispersive equations have long remained of interest to only a few mathematicians. It was not until until the $60s$ that these special solutions began to regain attention. To explain some unexpected numerical observations of Fermi, Pasta, Ulam and Tsingou at Los Alamos \cite{FPUT55}, Zabruski, Kruskal \cite{ZK65} performed some numerical simulations on the KdV equation and observed two striking phenomena. First, the solutions of the equation decompose eventually into a  finite sum of solitary waves and a dispersive radiation.  Secondly, two solitary waves evolving at different speed would \lq\lq pass through one another without losing their identity (size or shape)\rq\rq. The solitary waves of the KdV equation were then renamed \emph{solitons} to emphasize their particle-like character. 

These discoveries led to an intense activity to understand the mathematical properties of the KdV equation, culminating in the development of the \emph{inverse scattering method} \cite{gardner1967,Lax68}. In particular, the behaviour of solutions behaving as the sum of two solitary waves at time $-\infty$ and $+\infty$ was investigated by Lax in \cite{Lax68}. We also mention the explicit formulas for multi-solitons of the KdV equation discovered by Hirota in \cite{hirota1971}. These works confirmed the \emph{elasticity} of the collision of two solitary waves of the KdV equation observed numerically in \cite{ZK65}: the  waves remain unchanged in shape and size after the collision. Later on, still relying on the complete integrability theory, Eckhaus and Schuur \cite{EckhausSchuur83} showed the \emph{soliton resolution property} for solutions of KdV evolving from smooth and decaying initial data. 

We are interested here in the collision of two solitary waves with nearly equal speeds. The following phenomenon was observed numerically in \cite{ZK65}: as soon as the highest wave gets reasonably close to the smallest wave, the highest wave begins to shrink, transferring some of its mass to the smallest one until the two waves exchange roles and then separate. LeVeque \cite{LeVeque87} studied the evolution of the explicit solution $u_{c_1,c_2}$ of KdV behaving at time $\pm \infty$ as the sum of two solitons with nearly equal speeds $c_1$ and $c_2$ (with $c_2>c_1$).  If $\mu_0=(c_2-c_1)/(c_1+c_2)$ is small enough, then the solution $u_{c_1,c_2}$ remains close to a sum of two modulated solitary waves for all times with an error of order $\mu_0^2$. Moreover, the distance between the centers of the modulated solitary waves  is always larger than $2|\ln \mu_0|$ going at infinity when $\mu_0$ tends to $0$.

It is worth noting that all the results described above rely heavily on the complete integrability machinery for the KdV equation. The situation for non-integrable equations is far more delicate since there is no explicit formula to describe the evolution of solutions behaving at infinity as the sum of two solitary waves.  For this reason, there are only a few theoritical results that describe the non-integrable collision of two solitary waves. For the generalized KdV equation with quartic power in the nonlinearity, Mizumachi \cite{Mizu03} proved that the solution evolving from an intial datum close in an exponentially weighted space to a sum of two solitary waves of nearly equal speed, the highest one being on the left, will remain close to a sum of two solitary waves for all positive times, and that these waves do not cross each other; in other words, the interaction is \emph{repulsive} as in the integrable case.

The overtaking collision of two copropagating solitary waves of almost equal size, with the leading wave being the smallest, was then investigated for the water wave system by Craig, Guyenne, Hammack, Henderson and Sulem in \cite{CGHHS06} (see figure 14). It is shown numerically and experimentally that the trajectories of the two solitary waves do not cross, the evolution of their heights is monotonic and the collision amplifies the largest of the solitary wave. Moreover, the presence of a residual term after the collision emphasizes its \emph{inelastic} character. 

Some years later,  Martel and Merle \cite{MM11} were able to give a much more precise description of the interaction of two nearly equal solitary waves for the quartic gKdV equation. By starting with a pure $2$-solitary wave $u_{c_1^-,c_2^-}$ at time $-\infty$ satisfying 
\begin{equation*}
\lim_{t\rightarrow - \infty} \left\| u_{c_1^-,c_2^-}(t,\cdot)-Q_{c_2^-} \left( \cdot - c_2^-t-x_2^- \right)-Q_{c_1^-}\left( \cdot - c_1^- t - x_1^-\right)  \right\|_{H^1} =0,
\end{equation*}
for $c_2^->c_1^-$ with $\mu_0=(c_2^--c_1^-)/(c_2^-+c_1^-)>0$ small enough, they proved that the solution $u_{c_1^-,c_2^-}$ decomposes as 
\begin{equation*}
u_{c_1^-,c_2^-}(t,\cdot)=Q_{c_1(t)}(\cdot-z_1(t))+Q_{c_2(t)}(\cdot-z_2(t))+\eta(t),
\end{equation*}
where $\|\eta(t)\|_{H^1} \le C\mu_0^2 |\ln \mu_0|^{\frac12}$, $z_1(t)-z_2(t) \ge c |\ln \mu_0|$, for all $t \in \mathbb R$, $\liminf_{t \to +\infty}\|\eta(t)\|_{H^1} \geq c\mu_0^3$. As in the integrable case, there is a transfer of mass from the highest wave to the smallest one in the interaction region. Note however that in contrast with the integrable case, there is also a loss of mass and energy by radiation, and that the size $c_1^+=\lim_{t \to +\infty}c_1(t)$ of the largest wave at infinity is greater than $c_2^-$, reciprocally $c_2^+=\lim_{t \to +\infty}c_2(t)<c_1^-$. This result thus emphasises the \emph{inelastic} nature of the collision of two solitary waves of nearly equal speeds for the quartic gKdV equation. Moreover, the stability in the energy space $H^1$ of the collision is also proved in \cite{MM11}.  

The situation of the collision of one large and one small solitary waves for the quartic gKdV equation was investigated by Martel and Merle \cite{MMannals11}. In this setting, the largest wave crosses the smallest, but again the authors were able to prove the inelasticity of the collision. These results have been extended for the Benjamin-Bona-Mahony equation (BBM) in \cite{MMM10}, and also for the gKdV equation with a general smooth nonlinearity $f(u)$ in \cite{MMcmp09,MunIMRN10}. In particular, Mu\~noz proved in \cite{MunIMRN10} that if $f(u)$ is analytic in $0$, then the collision is elastic if and only if $f(u)=u^2$, $u^3$ or $u^3-cu^2$ which corresponds to the three known integrable cases in the gKdV setting (the Korteweg-de Vries equation, the modified Korteweg-de Vries equation and the Gardner equation). In this sense, one could argue that the integrability of a focusing nonlinear dispersive equation is related to the elasticity of the collision of two solitary waves. It is worth noting that, as far as we know, there is no result of this kind for KdV-type equations in higher dimensions. We refer nevertheless to \cite{KRS212d,KRS21} for interesting numerical simulations for the Zakharov-Kuznetsov equation in $2$ and $3$-dimensions, displaying in particular the inelasticity of the collision.

There are also a few theoritical results on the collision for non-integrable dispersive equations which are not of KdV type. Perelman \cite{PerelAIHP11} studied the collision of a large and a small solitary waves for a perturbation of the cubic $1$-dimensional nonlinear Schr\"odinger (NLS) equation and proved that  after the collision the smallest wave splits into two outgoing waves that are controlled on large time by the cubic NLS dynamics. In \cite{MMInv18}, Martel and Merle proved the inelasticity of the collision for the $5$-dimensional energy critical wave equation. Moutinho considered the $\Phi^6$-model \cite{MouCMP23,Mou23} and proved the almost elasticity of the collision. Finally, let us mention that the study of the collision of solitary waves is an important tool in proving the soliton resolution property for nonlinear dispersive equations. We refer for example to \cite{DKMacta23,JL23_soliton,CDKM22} and the references therein in the framework of the nonlinear wave equation.

\begin{toexclude}
\smallskip
\red{Citation de Strauss \cite{Str77} : "Solitary waves have also been called "solitons" but, properly speaking, the latter word should be reserved for those special solitary waves which exactly preserve their shapes after interaction."}
\end{toexclude}

\subsection{Statement of the results} 
The main objective of this paper is to study the dynamics of the collision of two solitary waves of nearly equal speeds for the Zakharov-Kuznetsov equation. We are able to describe the evolution of the solution behaving as a sum of $2$-solitary waves of nearly equal speeds at time $t=-\infty$ up to time $t=+\infty$. We show that the solution behaves as a sum of two modulated solitary waves and an error term which is small in $H^1$. Finally, we also prove the stability of this solution for large times around the collision.

Let $0<c_1^-<c_2^-$ be the given speeds at time $-\infty$ and let $\bx_1^-=(x_1^-,\by_1^-)$ and $\bx_2^-=(x_2^-,\by_2^-)$ denote the centers of the solitary waves. We consider the solution $u$ of \eqref{ZK} behaving asymptotically at time $-\infty$ as the sum of these two solitary waves, \textit{i.e.} 
\begin{align} \label{multi_sol:-infty:u}
    \lim_{t\rightarrow - \infty} \left\| u(t,\cdot)-Q_{c_2^-} \left( \cdot - c_2^-t-x_2^-, \cdot -\by_2^- \right)-Q_{c_1^-}\left( \cdot - c_1^- t - x_1^-, \cdot -\by_1^-\right)  \right\|_{H^1} =0.
\end{align}
The existence and uniqueness of such solution was proved by the second author in \cite{Val21}. 

In order to study the evolution of $u$ from time $-\infty$ to $+\infty$, we first symmetrize the problem. Let 
\begin{align*}
& \mu_0=\frac{c_2^--c_1^-}{c_1^-+c_2^-}, \quad c_0=\frac{c_1^-+c_2^-}{2}, \quad \bomega_0=c_0^{\frac12}\frac{\by_2^--\by_1^-}2, \quad x_0=\frac{x_1^-+x_2^-}2,\\
& \by_0=\frac{\by_1^-+\by_2^-}2, \quad t_0=\frac{lc_0^{-\frac12}+x_2^--x_1^-}{c_2^--c_1^-}
\end{align*}
where $l=l(\mu_0)$ is a real number defined below (see Lemma \ref{lemm:pointwise_Z}). We define the new function $v$ by
\begin{equation*} 
u(t,x,y)=c_0v\big(c_0^{\frac32}(t+t_0),c_0^{\frac12}(x-c_0t-x_0),c_0^{\frac12}(\by-\by_0))
\end{equation*}
Then $v$ is the unique solution to 
\begin{equation} \label{ZK:sym}
\partial_tv+\partial_x \left( \Delta v-v+v^2 \right)=0 .
\end{equation}
satisfying 
\begin{equation} \label{multi_sol:-infty}
\lim_{t \to -\infty} \left\| v(t,\cdot) -Q_{1+\mu_0}(\cdot-(\mu_0t-\frac{l}2),\cdot+\frac12 \bomega_0) -Q_{1-\mu_0}(\cdot+\mu_0t-\frac{l}2,\cdot- \frac12 \bomega_0) \right\|_{H^1}=0 .
\end{equation}
Therefore, the study of the solution $u$ of \eqref{ZK} satisfying \eqref{multi_sol:-infty:u} for arbitrary parameters $0<c_1^-<c_2^-$, $x_1^-$, $x_2^-$, $\by_1^-$, $\by_2^-$ is equivalent to the study of the solution $v$ of \eqref{ZK:sym} satisfying the more symmetric asymptotic behaviour \eqref{multi_sol:-infty} at time $-\infty$. 

The dynamics of the difference of the centers of the two solitary waves will be roughly given by the even solution $Z$ of the ODE 
\begin{align} \label{eq:Z_intro}
    \Ddot{Z}(t)= \frac{2}{\langle \Lambda Q, Q \rangle} \int_{\mathbb{R}^d} Q(x + Z(t),\by) \partial_x (Q^2) (x,\by) dxd\by,
\end{align}
satisfying $\displaystyle \lim_{t \to -\infty}(Z(t),\dot{Z}(t))=(+\infty,-2\mu_0)$. We denote by $Z_0:=Z(0)$. If $\mu_0$ is chosen small enough, then $\mu_0 \sim Z_0^{-\frac{d-1}4}e^{-\frac12Z_0}$. We refer to Appendix \ref{app:Z} for a detailed study of the solutions to \eqref{eq:Z_intro}.

\begin{figure}[ht]
    \centering
    \begin{tikzpicture}
            \draw[->,samples=\Num] (-5,0) -- (-1,0); 
            \draw (-1,0) node[right] {$t$};
            \draw[->,samples=\Num] (-3,-0.5) -- (-3,3);
            \draw (-3,3) node[right] {$Z(t)$};

            \draw [domain=-2:2,samples=\Num] plot [variable=\t] (\t-3,{(1.5*\t*\t)/(1+abs(\t))+1});
            \draw (-3,0.7) node[left] {$Z_0$};
    
            \draw[->,samples=\Num] (1,1.5) -- (5,1.5); 
            \draw (5,1.5) node[right] {$t$};
            \draw[->,samples=\Num] (3,-0.5) -- (3,3);
            \draw (3,3) node[right] {$\dot{Z}(t)$};

            \draw [domain=-0.5:0.5,samples=\Num] plot [variable=\t] (4*\t+3,{1.1*\t/(abs(\t)+0.1)+1.5});
            \draw (3,0.5) node[right] {$-2\mu_0$};
            \draw[dashed,red] (1,0.5) -- (3,0.5);
            \draw (3,2.5) node[left] {$2\mu_0$};
            \draw[dashed,red] (3,2.5) -- (5,2.5);
    \end{tikzpicture}
    \caption{Graph of the functions $Z$ and $\dot{Z}$}
\end{figure}
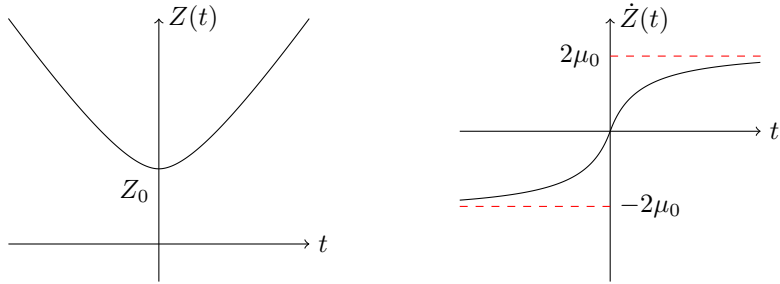

In this setting, we state our main result.

\begin{theo}[Dynamics of the interaction of two solitary waves]\label{maintheo}Let $d=2$ or $3$. There exist positive constants $C$ and $\nu^\star$ such that the following is true. Let $0<\mu_0<\nu^\star$, $\vert \bomega_0 \vert <\mu_0$ and let $v$ be the unique solution to \eqref{ZK:sym} satisfying \eqref{multi_sol:-infty}. Then, there exist $(z_1,z_2,\bomega_1,\bomega_2,\mu_1,\mu_2) \in C^1(\mathbb R : \mathbb R^{2d+2})$ such that 
\begin{equation} \label{def:eta}
\eta(t,\cdot)=v(t,\cdot)-Q_{1+\mu_2(t)}(\cdot-z_2(t),\cdot-\bomega_2(t))-Q_{1+\mu_1(t)}(\cdot-z_1(t),\cdot-\bomega_1(t))
\end{equation}
satisfies, for all $t \in \mathbb R$,
\begin{align} 
    &\|\eta(t)\|_{H^1} + \sum_{j=1}^2 \left( \left|\dot{z}_j(t)-\mu_j(t) \right|+\left\vert \dot{\bomega}_j(t) \right\vert+ \left| \mu_j(t)+\frac{(-1)^j}2 \dot{Z}(t) \right| \right) \le C \mu_0^{\frac32} ; \label{maintheo.1} \\
    & z_1(t)-z_2(t) \ge \frac12 Z_0.\label{maintheo.2} 
\end{align}
Moreover, there exist $\mu_1^+=\lim_{t \to +\infty}\mu_1(t)$ and $\mu_2^+=\lim_{t \to +\infty}\mu_2(t)$ such that 
\begin{align} 
&\lim_{t \to +\infty} \|\eta(t)\|_{H^1(x>-\frac{99}{100}t)}=0; \label{maintheo.4} \\
& 0 \le  \mu_1^+ -\mu_0\le C\mu_0^2 ;\label{maintheo.5} \\ 
& 0\le -\mu_2^+-\mu_0 \le C \mu_0^2 .
\label{maintheo.6}
\end{align}
\end{theo}

\begin{rema}
As in the one dimensional case, it follows from \eqref{maintheo.2} that the distance between the two solitary waves is always large and is tending to $\infty$ as $\mu_0$ goes to $0$. Moreover, it will be clear from the proof that the more precise estimate 
\begin{align*}
    \left\vert \min_{t \in \mathbb{R}} \left( (z_1,\bomega_1)(t) - (z_2,\bomega_2) (t) \right) -(Z_0,\bomega_0) \right\vert \leq C Z_0^2 \mu_0^{\frac34},
\end{align*}
holds, (see (1.17) in \cite{MM11}  in the one dimensional case).
\end{rema}

In the later computations, we will make use of a small positive constant $\rho<\frac1{32}$ independent of $\mu_0$ and define the unique time $T_1>0$ satisfying $Z(T_1)=\rho^{-1} Z_0$. Notice that at time $T_1$, the solution $v$ as defined in Theorem \ref{maintheo} is still close to the sum of the two decoupled solitary waves. We define the time interval $[-T_1,T_1]$ as the \emph{collision region}.

\begin{theo}[Stability in the region of collision] \label{theo:stability}
   Under the notations of Theorem \ref{maintheo} and the definition of $T_1$ above, if a function $w$ satisfies
   \begin{align} \label{cond:theo:stability}
       \left\| w(-T_1) - v(-T_1) \right\|_{H^1} \leq Z_0^{-5}\mu_0^{\frac74},
   \end{align}
   then it holds for any $t\in [-T_1,T_1]$
   \begin{align} \label{est:theo:stability}
       \left\| w(t)-v(t) \right\|_{H^1} \lesssim \mu_0^{\frac74}.
   \end{align}
\end{theo}

\begin{rema}\label{rema:stability}
By using similar arguments as in the proof of Proposition \ref{propo:bootstrap_-T_T}, it is enough to assume that, for some $\mathcal{T} \in [-T_1,T_1]$,
   \begin{align}\label{rema:cond:stability}
       \left\| w(\mathcal{T}) - v(\mathcal{T}) \right\|_{H^1} \leq Z_0^{-5}\mu_0^{\frac74},
   \end{align}
   to obtain \eqref{est:theo:stability}.
\end{rema}

\begin{rema} After the time $T_1$ (and before the time $-T_1$), one can use the stability result of well-prepared multi solitary waves in \cite{PV23} to deduce that the solution $w$ remains close in the energy space $H^1$ to the sum of two modulated solitary waves for all time $t \in \mathbb R$. Even though the speeds of these waves are close to the ones of $v$ in the decomposition in Theorem \ref{maintheo}, the translation parameters $(z_i,\bomega_i)$ may not be close to the ones of $v$, and thus, it is not clear that the solution $w$ remains close to the solution $v$ for all times. 
\end{rema}

\subsection{Strategy of the proofs}

The proofs of Theorems \ref{maintheo} and \ref{theo:stability} extend the road map introduced by Martel and Merle in \cite{MM11} (see also \cite{Mizu03}) to higher dimensions. Below we explain the key ingredients of the proofs and highlight the main differences and new challenges that arise in higher dimensions. For the sake of simplicity, we focus on the $2$-dimensional case in the rest of the paper, since the proofs for the $3$-dimensional case are similar. 

\smallskip \noindent \emph{Step $1$: constructing a suitable approximate solution.}
The first step consists of constructing an approximate solution $V$ to \eqref{ZK:sym} with a good control on the error $\E(V) = \partial_tV - \partial_x \left( -\Delta V + V -V^2 \right)$. At first order, $V$ should be given by a sum of two modulated solitary waves $V_0=R_1+R_2$, where $R_i(t,\bx)=Q_{1+\mu_i(t)}(\bx-\bz_i(t))$, $i=1,2$, for a set of well chosen $C^1$ translation and dilation parameters $(\bz_1,\bz_2, \mu_1,\mu_2)=(z_1,z_2,\omega_1,\omega_2,\mu_1,\mu_2)$. Then, the error is of the form 
\begin{equation*}
\E(V_0)=\sum_{i=1}^2 \overrightarrow{\widetilde{m}_i} \cdot \MRi  +2\partial_x(R_1R_2), \quad \text{where} \quad   \overrightarrow{\widetilde{m}_i} = \begin{pmatrix} -\dot{z}_i + \mu_i   \\ -\dot{\omega}_i  \\ \dot{\mu}_i \end{pmatrix} , \quad \MRi = \begin{pmatrix} \partial_x R_i \\ \partial_y R_i \\ \Lambda R_i \end{pmatrix} .
\end{equation*}
Here, $\MRi$ represents the natural directions associated with the translation and scaling\footnote{We refer to \eqref{def:LambdaR_i} for a precise definition of the scaling operator $\Lambda R_i$.} invariants of the equation (see \eqref{scaling}),  $ \overrightarrow{\widetilde{m}}_i$ is the rough evolution of the dynamical parameters $\Gamma$. Moreover, the  interaction term $2\partial_x(R_1R_2)$ is a plateau of size $z^{-\frac12}e^{-z}$ between $R_1$ and $R_2$, where $z=z_1-z_2$ denotes the distance between the two solitary waves in the tangential spatial coordinate $x$ (see \eqref{eq:bound_Q_first_order}). Thus, it is not localized around each solitary wave as for the quartic gKdV equation in \cite{MM11} . It is worth noting that this first new difficulty is related to the quadratic power in the non-linearity of \eqref{ZK} and not to the higher dimension of the problem. A first new idea is to add the term $F=\left(-\Delta+1\right)^{-1}(2R_1R_2)$ \footnote{In the definition of $F$ in \eqref{def:VA}, we are actually working with $\Rtone$ and $\Rttwo$, which are obtained from  $R_1$ and $R_2$ by rescaling around the size $1$ and translating in the transverse direction so that the centers of the waves are located on the $x$-axis.} to the approximate solution $V_0$. Then, a direct computation shows that 
\begin{equation*}
\E(V_0+F)=\sum_{i=1}^2 \vec{\widetilde{m}}_i \cdot \MRi+2\partial_x((R_1+R_2)F)+\partial_x(F^2)+\partial_t F .
\end{equation*}
Even though the main error term $2\partial_x((R_1+R_2)F)$ is still of size $z^{-\frac12}e^{-z}$, it is now localised around the two solitary waves $R_1$ and $R_2$. 

In the $1$-dimensional case of the quartic gKdV equation, the solitary wave $Q$ is explicit and $Q(\cdot+z)$ has an asymptotic expansion in powers of $e^{-z}$ at infinity. This allows the authors in \cite{MM11} to improve the order of the error around each solitary wave coming from the expansion of $Q(\cdot+z)$ by a factor $e^{-z}$ by solving the remainder terms locally around each waves. The situation is much more delicate in higher dimensions. In dimension $2$, $Q(\cdot+z, \cdot)$ has an asymptotic development for $\vert \bx \vert$ small compared to $z$ of the form\footnote{We refer to Proposition \ref{propo:Q_n_app_gen} for a precise statement.}
\begin{equation*}
\forall \, n \in \mathbb N, \quad Q(x+z,y) \underset{z \to +\infty}{\sim} q_0(\bx) z^{-\frac12}e^{-z}+q_1(\bx) z^{-\frac32}e^{-z}+\cdots+q_n(\bx) z^{-n-\frac{1}2}e^{-z}.
\end{equation*}

Hence, even solving these asymptotic expansions locally around each solitary wave would not improve the error sufficiently to close the estimates.  To bypass this difficulty, we choose to make the error small only in the three natural directions of $\MRi$. Another major new difficulty that arises in higher dimension is that the solitary waves no longer have an explicit formulation. For this reason, we define our approximate solution $V$ by
\begin{equation*}
V = R_1+R_2+ F + \sum_{i=1}^2  \pAi \cdot \NRi \quad \text{with} \quad \pAi(z) := \begin{pmatrix} \alpha_{i}(z) \\ \beta_{i}(z) \\ \gamma_i(z) \end{pmatrix} , \quad \NRi= \begin{pmatrix}  -(-\Delta+1)^{-1} R_i \\ -\partial_x^{-1} (-\Delta+1)^{-1}\partial_y R_i \\ -\partial_x^{-1} (-\Delta+1)^{-1}\Lambda R_i \end{pmatrix} ,
\end{equation*}
where the directions in $\NRi$ are defined in an intrinsic way\footnote{In the definition of $V$ in \eqref{def:VA}, we are actually working with $\NRti$ which corresponds to the directions $\NRi$ rescaled around the size $1$ and whose centers are retranslated in the transverse direction onto the $x$ axis.}. By using $\partial_x(-\Delta+1)\NRi=\MRi$, one can rewrite the error of $V$ on the form 
\begin{equation} \label{EV:intro}
\E(V)=  \sum_{i=1}^2 \mAi \cdot \MRi + T +\partial_xS, \quad \text{where} \quad \mAi := \begin{pmatrix} -\dot{z}_i + \mu_i + \alpha_i  \\ -\dot{\omega}_i +\beta_i \\ \dot{\mu}_i + \gamma_i \end{pmatrix} .
\end{equation}

While the main term in the error $\partial_xS$ is still of order $z^{-\frac12}e^{-z}$, we adjust the parameters in $ \pAi$ so that the scalar products $\langle S, \MRi \rangle $ have smaller order of approximately $e^{-\frac32 z}$ (see \eqref{est:ortho_S} for a precise statement). This choice induces a refined dynamical system on the geometrical parameters given by $\mAi$. Here, we can compute explicitly 
\begin{equation*}
\gamma_1(z)=-\gamma_2(z)=- \langle \Lambda Q,Q\rangle^{-1} \int_{\mathbb R^2} Q(x+z,y)\partial_x(Q^2)(x,y)dxdy .
\end{equation*}
so that the distance $z$ between the two solitary waves can be approximated  by the solution $Z$ of the second order nonlinear ODE 
\begin{equation} \label{ODE:Z:intro}
 \Ddot{Z}(t)= \frac{2}{\langle \Lambda Q, Q \rangle} \int_{\mathbb{R}^2} Q(x + Z(t),y) \partial_x (Q^2) (x,y) dxdy \approx Z(t)^{-\frac12}e^{-Z(t)} ,
\end{equation}
satisfying $\displaystyle \lim_{t \to -\infty}(Z(t),\dot{Z}(t))=(+\infty,-2\mu_0)$. Observe that the factor $Z^{-\frac12}$ is not present in the $1$-dimensional case (see (1.24) in \cite{MM11}). Moreover, it is worth noting that since $\gamma_1(z)=-\gamma_2(z)$ and $\beta_1(z)=\beta_2(z)$, the approximate solution $V$ is actually a sum of the two solitary waves $R_1$ and $R_2$ and a plateau of size $z^{-\frac12}e^{-z}$ localised between the waves (see Figure \ref{Figure:W}).

\smallskip \noindent \emph{Step $2$: modulation theory and energy estimates.}
Once the approximate solution $V$ is constructed, we look at a solution $w$ of \eqref{ZK:sym} close to $V$. By using modulation theory, we adjust the geometrical parameters $z_i$, $\omega_i$ and $\mu_i$ to ensure some orthogonality conditions on the error $\epsilon=w-V$ in the natural directions $\partial_xR_i$, $\partial_yR_i$ and $R_i$, $i=1,2$. This implies directly some bounds for $|\mAi|$ in terms of $\|\epsilon\|_{H^1}$ and the error terms in $\E(V)$. To control $\|\epsilon\|_{H^1}$, we use a virial-energy functional inspired from \cite{MM11}, but modified by adding the term $\int S \epsilon$, where $S$ is the main term in the error $\E(V)$ in \eqref{EV:intro}. This modification of the energy is necessary because, as explained above, we cannot improve the error $\E(V)$ in a significant way as in the one dimensional case. This is the second new ingredient that we have to introduce to deal with the $2$-dimensional case. We refer to \cite{MN20} for a similar idea to construct strong interactions of solitary waves for a Schr\"odinger system with cubic nonlinearity in one dimension. 

\smallskip \noindent \emph{Step $3$: bootstrap estimates.}
With these estimates in hand, we are able to prove Theorem \ref{theo:stability}. We study a solution $w$ close to the approximate solution $V$ in a large time region $[-T_1,T_1]$ around the collision. We prove that the distance between the solitary waves $z(t)$ is well approximated by the solution of the ODE $Z(t)$ to \eqref{ODE:Z:intro} satisfying $\displaystyle \lim_{t \to -\infty}(Z(t),\dot{Z}(t))=(+\infty,-2\mu_0)$ (see Proposition \ref{propo:bootstrap_-T_T}). To this aim, we perform a series of bootstrap estimates on the sizes of the geometrical  parameters and $\|\epsilon\|_{H^1}$. Two new difficulties arise in the two dimensional case. Firstly, the control on the transverse translation parameter given by the modulation estimates is not good enough to close the estimates. To address this issue, we introduce new linear quantities in $\epsilon$ which allow to obtain sharper estimates on the evolution of the transverse parameters. These quantities correspond to scalar products of $\epsilon$ with some non-localised directions. Similar quantities were used in \cite{MM11} in the tangential direction. Note however that in order to use these quantities for a general solution $w$, we have to localise them in a suitable way (see the definition of $\mathcal{K}_i$ in \eqref{def:Kj} and Remark \ref{rema:Ki} for more detail). We refer to \cite{MP17} for a similar idea in the framework of the modified Benjamin-Ono equation. Secondly, the second order nonlinear ODE governing the evolution of the distance between the  two solitary waves cannot be approximated by explicit solutions as in the one dimensional case. Moreover, even an approximation of the ODE \eqref{ODE:Z:intro} by $\Ddot{Z}(t)= Z(t)^{-\frac12}e^{-Z(t)} $ would not provide sufficiently sharp estimates on the quantities $\vert \dot{\mu}_i(t) -  \gamma_i(t) \vert$. For this reason, we perform in Appendix \ref{app:Z} a detailed study of the ODE \eqref{ODE:Z:intro} by using a phase portrait and rely on its Hamiltonian structure to obtain more refined estimates on $|z(t)-Z(t)|$ (see Proposition \ref{propo:H_positive_times} for a precise statement). 

\smallskip \noindent \emph{Step $4$: proof of the main theorem.}
Now, let us give some details of how to prove Theorem \ref{maintheo}. We start with the well-prepared $2$-solitary waves solution of \eqref{ZK:sym} at time $-\infty$, $v$ satisfying \eqref{multi_sol:-infty} and constructed in \cite{Val21}. First, we prove a quantitative orbital stability result around the sum of two solitary waves as long as their centers remain far enough away from each other. It is worth noting that the distance between the solution and the sum of the two solitary waves at a given and potentially large negative time $-t_0$ has to be carefully quantified as a function $\kappa(-t_0)$ which decays faster than $t_0^{-1}$ at infinity. This allows us to bring our solution $v$ close to the collision time $-T_1$. We then apply the stability result through the collision region $[-T_1,T_1]$. At this point the waves have transferred their mass and the highest one is on the right. Finally, we conclude by applying the asymptotic stability result proved in \cite{PV23}. 

\begin{toexclude}
\bigskip

\blue{\begin{itemize}
\item the second order nonlinear ODE governing the evolution of the distance between the  two solitary waves does not have explicit solutions as in the one dimensional case  $\rightarrow$ we study the solutions of this equation by using a phase portrait, $+$ in the second bootstrap, we replace $\mu_0^2$ by $\ddot{Z}$ to integrate in time $+$ $\cdots$
\end{itemize}}

\red{\begin{itemize}
    \item(Accuracy of the ODE governing the distance between the two solitary waves) The solution of the ODE satisfied by $z$ is a level set of the hamiltonian $H$, involving nonlinear terms. Another possibility to obtain at the main order this ODE is to approximate the nonlinear term by a finite linear combination of terms $Z^{-\frac12-k} e^{-Z}$ with $k\in \mathbb{N}$. However, the bound on $\vert \dot{\mu}(t) - 2 \gamma_1(t) \vert$ in \eqref{eq:H'-T_1} would be $Z^{-\frac12 - K}e^{-Z}$ instead of $e^{-\frac32 Z}$, which implies instead of \eqref{eq:H_-T_1_-T_2} and \eqref{eq:z_Z_-T_1_Tstar}
    \begin{align*}
        \left\vert H(t)-2\mu_0^2 \right\vert \leq C Z_0^{-K} Z(t)^{-\frac12 } e^{- Z(t)} \quad \text{and} \quad \left\vert z(t)-Z(t) \right\vert \leq Z_0^{-K}.
    \end{align*}
    We thus obtain a much higher precision on the function $z$ using the nonlinear term in the ODE.
\end{itemize}}
\end{toexclude}

\smallskip \noindent \emph{Summary.}
 We extend the techniques introduced in \cite{MM11} for the quartic generalized Korteweg-de Vries equation to higher dimensions. First, despite the non-explicit nature of the solitary wave $Q$, we are able to construct an approximate solution in an intrinsic way by canceling the error to the equation only in a few natural directions. Then, to control the difference between a solution and the approximate solution, we use modified energy functional and a refined modulation estimate in the transverse variable. Moreover, we rely on the hamiltonian structure of the ODE governing the distance between the waves, which cannot be approximated by explicit solutions, to close the bootstrap estimates on the parameters. It is worth noting that we do not use any numerical method, except for the asymptotic stability result in \cite{PV23}, which relies on the virial estimates proved in \cite{CMP16} and \cite{FHRY23}. We hope that the techniques introduced here are robust and will prove useful in studying the collision phenomena for other focusing non-linear dispersive equations with non-explicit solitary waves. In a future work, we plan to refine our construction and prove the inelasticity of the collision of ZK in dimensions $2$ and $3$.

\bigskip

The paper is organised as follows: in Section \ref{sec:construction_V}, we construct the approximate solution $V$ and derive several estimates on the error $\E(V)$. In Section \ref{sec:est_eps}, we perform the modulation on the geometrical parameters and prove the energy estimates on the error between the solution and the approximate solution. We use these estimates in Section \ref{sec:collision} to prove the main stability result in the collision region. Then, we prove Theorems \ref{maintheo} and \ref{theo:stability} in Section \ref{sec:dynamics}. Finally, we state and derive several useful results in the appendices: we give some properties of the Bessel potential $(-\Delta+1)^{-1}$ in Appendix \ref{App:Bessel}; we give some precise asymptotics of the ground state $Q$ in Appendix \ref{app:asymptotic_Q}; we derive some technical estimates in Appendix \ref{Append:Est} and we study the dynamical system \eqref{ODE:Z:intro} in Appendix \ref{app:Z}.


\subsection*{Notations}

\begin{itemize}

\item Throughout the paper, we will denote by $C$ a positive constant which may change from line to line. The notation $a \lesssim b$ means that $a \le C b$.

\item We use the standard notation for the sets of integers $\mathbb{N}= \left\{ 0,1, \cdots \right\}$ and $\mathbb{N}^{\ast}=\mathbb{N}\setminus \{0\}$.

\item In general, we will denote $\bx =(x,y) \in \mathbb R^2$. Then, $\vert \bx \vert=(x^2+y^2)^{\frac12}$ and $\langle \bx \rangle=(1+\vert \bx \vert^2)^{\frac12}.$

\item We will work with real-valued function $f=f(\bx)=f(x,y): \mathbb R^2 \to \mathbb R$. For $1 \le p \le +\infty$ and $s \in \mathbb R$, $L^p(\mathbb R^2)$, respectively $W^{s,p}(\mathbb R^2)$, denotes the standard Lebesgue spaces, respectively Sobolev spaces. We also write $H^s(\mathbb R^2)=W^{s,2}(\mathbb R^2)$. 
For $f,g \in L^2(\mathbb R^2)$ two real-valued functions, we denote 
by $\langle f,g \rangle =\int_{\mathbb R^2} f(\bx)g(\bx) d\bx$ their scalar product\footnote{Note that there is no risk of confusion with the japanese bracket defined above for an element $\bx \in \mathbb R^2$.}. 

\item We define the function spaces 
\begin{align} \label{def:Y}
\Y := \left\{ f \in \C^\infty (\mathbb{R}^2) : \forall \alpha \in \mathbb{N}^2, \exists n \in \mathbb{N}, \forall \, \bx \in \mathbb R^2 \ \vert \partial_\bx^\alpha f(\bx) \vert \lesssim \langle \bx \rangle^n e^{-\vert \bx \vert} \right\}.
\end{align}
and 
\begin{equation} \label{def:Z}
\mathcal{Z} := \left\{
f \in \C^\infty (\mathbb{R}^2) \cap L^{\infty}(\mathbb R^2) : \  \begin{aligned} & \partial_xf \in \mathcal{Y} \ \text{and} \ \forall k \in \mathbb N,  \forall \bx \in \mathbb R^2 \,  \\ 
    &\text{with} \, x \ge 0, \  \vert \partial_y^k f(\bx) \vert \lesssim  e^{-\frac{15}{16}\vert \bx \vert} 
\end{aligned} \right\}.
\end{equation}


\item For $c>0$, we define the generator of the scaling symmetry  $\Lambda_c$ and its iterated  $\Lambda^2_c$ by
\begin{align}
	& \Lambda_c (f)(\bx) = \Lambda_c f(\bx) := \frac{d}{d\tilde{c}} \left( \tilde{c} f(\sqrt{\tilde{c}} \cdot) \right)_{\vert \tilde{c}=c}(\bx) = \left( \left(1+\frac{1}{2}\bx \cdot \nabla\right)  f\right)\left( \sqrt{c} \bx \right), \label{defi:Lambda}\\
	& \Lambda^2_c f(\bx) := \frac{d^2}{d\tilde{c}^2} \left(\tilde{c}f (\sqrt{\tilde{c}} \cdot) \right)_{\tilde{c}= c}(\bx) = \frac{1}{c} \left( \frac{1}{2} \bx \cdot \nabla \left( 1+\frac{1}{2} \bx \cdot \nabla \right) f \right)\left( \sqrt{c}\bx\right). \label{defi:Lambda_2} 
\end{align}
In the case $c=1$, we denote 
\begin{equation} \label{defi:Lambda_Q}
\Lambda Q=\Lambda_1 Q  \quad \text{and} \quad  \Lambda^2 Q=\Lambda_1^2 Q . 
\end{equation} 


%
%
%

\item We will denote by $(-\Delta+1)^{-1}: L^2(\mathbb R^2) \to H^2(\mathbb R^2)$ the Bessel potential of order $2$, which is defined as a Fourier multiplier by $(1+|{\bf \xi}|^2)^{-1}$. We refer to Appendix \ref{App:Bessel} for several important properties related to $(-\Delta+1)^{-1}$.

\item We define the operator $-\partial_x^{-1}$ on $\mathcal{Y}$ by
\begin{align} \label{defi:d_x_-1}
-\partial_x^{-1}: \mathcal{Y} \to \mathcal{Z}, \ f \mapsto -\partial_x^{-1} f := \int_x^{+\infty} f(\tilde{x},y) d\tilde{x}. 
\end{align}
Observe in particular that, for all $f \in \mathcal{Y},$ 
\begin{equation}  \label{prop:d_x_-1}
 \partial_x (-\partial_x^{-1}) f = -f = (- \partial_x^{-1}) \partial_x f.
\end{equation}

\end{itemize}

\section{Construction of the approximate solution}\label{sec:construction_V}

\subsection{Properties of the linearized operator around the ground state}

The linearized operator $L$ around the ground state is defined by
\begin{align}\label{defi:L}
    L:= -\Delta+1-2Q.
\end{align}
The following proposition gathers some classical results on $L$.

\begin{propo}\label{theo:L} The operator $L: H^2(\mathbb R^2) \subseteq L^2(\mathbb R^2) \to L^2(\mathbb R^2)$ satisfies the following properties.
\begin{enumerate}
    \item[(i)] \emph{Essential spectrum:} $L$ is a self-adjoint operator and its essential spectrum is given by $ \sigma_{ess}(L) = [1,+\infty)$. 
    \item[(ii)] \emph{Non-degeneracy of the spectrum:} $ ker(L) = \text{span} \, \{\partial_x Q, \partial_y Q\}$.
    \item[(iii)] \emph{Negative eigenvalue:} $L$ has a unique negative eigenvalue $-\lambda_0<0$, which is simple, and there exists a radially symmetric positive eigenfunction $\chi_0$associated to $-\lambda_0$. Without loss of generality, $\chi_0$ is chosen of $L^2$-norm equal to $1$. Furthermore, there exists $\kappa_0>0$ such that:
    \begin{align}\label{eq:chi_0}
        \forall r>1, \quad \left\vert \chi_0(r) - \kappa_0 K_0\left( \sqrt{1+\lambda_0}r \right) \right\vert \lesssim r^{-1} e^{-2(1+\lambda_0)^{\frac{1}{2}}r}.
    \end{align}
    \item[(iv)] \emph{Coercivity:} there exists $C>0$ such that, for any $f\in H^2$ in $L^2$ satisfying $\langle f, \partial_xQ \rangle=\langle f, \partial_yQ \rangle=\langle f, Q \rangle =0$:
    \begin{align*}
        \left\langle L f,f \right\rangle \geq C \| f \|_{L^2}^2.
    \end{align*}
    \item[(v)]\emph{Invertibility:} for any function $h \in L^2(\mathbb R^2)$ orthogonal to $\text{span} \, \{\partial_x Q, \partial_y Q\}$, there exists a unique function $f \in H^2(\mathbb R^2)$ orthogonal to $\text{span} \, \{\partial_x Q, \partial_y Q\}$ such that $Lf=h$.  Moreover, if $h$ is radial, then $f$ is also radial and if $h$ is even (resp. odd) with respect to one of its two variables, then $f$ is even (resp. odd) with respect to the same variable.  
    \item[(vi)] \emph{Regularity:} for $h\in L^2$, if $f \in H^2(\mathbb R^2)$ is such that $h=Lf \in \mathcal{Y}$, then $f \in \mathcal{Y}$. 
    \item[(vii)]\emph{Scaling:} $L\Lambda Q=-Q$. 
\end{enumerate}
\end{propo}

\begin{toexclude}
To study the operator $L$, we proceed as follow. $L$, defined as an operator from $L^2$ to $L^2$, satisfies:
\begin{itemize}
    \item The domain of $L$ is $H^2$. On this domain, it is self-adjoint, thus the spectrum is real: $\sigma(L)\subset \mathbb{R}$.
    \item We continue with the essential spectrum $\sigma_{ess}$. The operator $-\Delta$ is self-adjoint from $H^2$ to $L^2$, thus its spectrum is real. However, by the Fourier transform, for any $\lambda>0$, $-\Delta+\lambda$ is a bijection from $H^2$ to $L^2$, thus $\sigma(-\Delta) \subset [0,+\infty)$. Equivalently, $\sigma(-\Delta+1) \subset [1+\infty)$. Since $-3Q^2$ is smooth with an exponential decay, the operator $g \mapsto -3Q^2 g$ is a compact operator from $H^2$ to $L^2$. By a compact perturbation argument, we have $\sigma_{ess}(L) = \sigma_{ess}(-\Delta+1) \subset [1,+\infty)$. 
    \item To prove the equality in the previous spectrum, we prove that the $-\Delta-\lambda$ is not surjective onto $L^2$. Indeed, one can find a function in $L^2$ that is not in the range. This construction is done in the Lecture Notes of Didier, with Rafael, by using the Fourier transform, and finding a function emphasising the pole.
    \item Next, we study the discrete spectrum of $L$. First, we know that $\nabla Q$ is in the kernel of $L$. By the min-max principal, and $\langle Q, LQ \rangle <0$, $L$ has at least one negative eigenvalue, called $-\lambda_0$.
    \item The argument of \cite{Wei85}, Proposition 2.8, asserts that the space of eigenfunctions associated to negative eigenvalue is of dimension at most $1$. Indeed, it is proved that:
    \begin{align*}
        \inf_{f\in H^2, f \perp Q} \langle L f, f \rangle \geq 0.
    \end{align*}
    By this statement, an eigenfunction associated to a negative eigenvalue is orthogonal to $Q$. Then, suppose that $\chi_0$ and $\chi_1$ are two eigenfunctions associated to two negative eigenvalues, thus orthogonal, then there exists a non-trivial linear combination such that $\alpha \chi_0+\beta \chi_1 \perp Q$. This function belies the statement above, thus at least one function among $\chi_0$ and $\chi_1$, and the space of eigenfunctions associated to negative eigenvalues is of dimension $1$.
    \item The eigenfunction associated to $\lambda_0$ can be chosen positive (Perron-Frobenius arguments, see Frank and Lenzmann) and radially symmetric (or use of Strum-Liouville to determine one eigenvalue; see \cite{CGN08}). We denote by $\chi_0$ the positive one with $L^2$-norm equal to $1$.
    \item Coercivity. By the arguments of Weinstein, we have the following inequality:
    \begin{align*}
        \exists C>0, \forall f \in H^1, \quad \int f \chi_0 = \int f \partial_x Q =\int f\partial_y Q =0 \quad \Rightarrow \langle L f, f \rangle \geq C \langle f,f \rangle. 
    \end{align*}
    \item Inverse on the adequate space, 2nd version. Let us denote by $V$ the space of $H^1$ functions satisfying \eqref{cond:orthogonality}, and by $W$ the set of $H^{-1}$ functions satisfying \eqref{cond:orthogonality}. $L$ is bijective from $V$ to $W$. By the Lax-Milgram theorem, $L$ is a bijection from $V$ to $W$. In particular, for any $f\in L^2$ orthogonal to the kernel, there exists a unique function $g\in H^1$ orthogonal to the kernel such that $Lg=f$.
\end{itemize}
\end{toexclude}

\begin{proof}
Assertion (i) follows from classical perturbation arguments such as the Kato-Rellich theorem. The proof of (ii) is more delicate (see \cite{Wei85,CGN08}). We refer to \cite{Wei85} for the proof of (iii) and (iv). Note however, that the proof of \eqref{eq:chi_0} is given in Appendix \ref{app:asymptotic_Q}. 

We now give some details of the proof of (v). The uniqueness part follows directly from (ii). For the existence part, it suffices to prove that any $h \in L^2(\mathbb R^2)$ orthogonal to $\text{Span} \, \{\partial_x Q, \partial_y Q\}$ is in the range of $L$. Equivalently, given $h \in L^2(\mathbb R^2)$ orthogonal to $\text{Span} \, \{\partial_x Q, \partial_y Q\}$, we look for $f\in H^2(\mathbb R^2)$ such that
\begin{align} \label{Prop:L:v}
    \left( id_{L^2} - 2Q\left( -\Delta+1 \right)^{-1} \right) (-\Delta+1)f = h.
\end{align}
Let us denote by $T$ the operator $2Q (-\Delta+1)^{-1} : L^2 \rightarrow L^2$. This operator is compact and $T^\star = (-\Delta+1)^{-1}(2Q \cdot)$. Moreover, observe that $\text{Ker}(id - T^\star)=\text{Ker}(L)$. Thus, it follows from the Fredholm alternative that $h \in R(id-T)=\text{Ker}(id - T^\star)^{\perp}$, so that there exists $\widetilde{f} \in L^2(\mathbb R^2)$ such that $(id-T)\widetilde{f}=h$. Then, $f=(-\Delta+1)^{-1}\widetilde{f} \in H^2(\mathbb R^2)$ and satisfies \eqref{Prop:L:v}, which concludes the existence part of (v).  

Assume now that $h$ is radial. Let $\theta\in [0, 2\pi]$ and $R_\theta$ be the rotation on $\mathbb{R}^2$ of angle $\theta$. Since $Q$ and $\Delta$ are invariant under rotation, we have $L \left( f \circ R_\theta \right)=h$, so that $f-f\circ R_\theta \in \text{Ker}(L)$. Moreover, $f$ is orthogonal to $\text{span}\left\{ \partial_x Q, \partial_y Q \right\}$, so  that $f\circ R_\theta$ is orthogonal to $\text{Span}\left\{ \partial_x Q\circ R_\theta, \partial_y Q\circ R_\theta \right\}=\text{Span}\left\{ \partial_x Q, \partial_y Q \right\}$. Hence, $f=f\circ R_\theta$, and we conclude that $f$ is radial. The proof of the parity with respect to one of the variables follows similarly.

To prove (vi), assume that $h \in \mathcal{Y} \subseteq H^{\infty}(\mathbb R^2)$ and let  $f \in H^2(\mathbb R^2)$ be a solution of $Lf=h$. First, it follows from a standard bootstrap argument that $f \in H^{\infty}(\mathbb R^2)$. To prove the decay properties of $f$, we rewrite $f$ as $f=(-\Delta+1)^{-1}(2Qf+h)$. It is clear from \eqref{eq:bound_Q_first_order} and the Sobolev embedding that $Qf \in \mathcal{Y}$. Therefore, we conclude from Lemma \ref{Bessel:Y} that $f \in \mathcal{Y}$, which finishes the proof of (vi). 

Finally, the proof of $(vii)$ follows directly by differentiating the equation satisfied by $Q_c$ with respect to $c$ and using \eqref{defi:Lambda_Q}.

\end{proof}

Next we introduce some useful functions related to the ground state.  We define  
\begin{equation} \label{def:XYW}
X=-(-\Delta+1)^{-1} Q, \quad Y=-\partial_x^{-1}(-\Delta+1)^{-1} \partial_yQ \quad \text{and} \quad W=-\partial_x^{-1}(-\Delta+1)^{-1} \Lambda Q .
\end{equation} 
Observe that $X \in \mathcal{Y}$ and $Y, \, W \in \mathcal{Z}$ (see Lemma \ref{parity}). To work in a more condensed form, we also introduce the vectors $\MQ$ and $\NQ$ defined by 
\begin{equation} \label{def:NQMQ}
\MQ := \begin{pmatrix} \partial_x Q \\ \partial_y Q \\ \Lambda Q \end{pmatrix} \quad \text{and} \quad \NQ := \begin{pmatrix} X \\ Y \\ W \end{pmatrix} .
\end{equation}

\begin{toexclude}
Below, we derive an invertibility result for $L$, which will be an important tool in the construction of the local profiles around each solitary waves. 

\begin{lemm}\label{lemm:antecedent_1}
	Let $f\in \mathcal{Y}$. There exists a unique $(\beta, \gamma)\in \mathbb{R}^2$ and a unique $\hat{A}\in \Y$ orthogonal to $\{\partial_x Q$, $\partial_y Q\}$ such that
	\begin{align} \label{lemm:antecedant_1}
		L (\hat{A}) =2 Q\partial_x^{-1}(-\Delta+1)^{-1}\left( \beta \partial_y Q +\gamma \Lambda Q \right) +f, 
	\end{align}
The coefficients $(\beta,\gamma)$ are given by
\begin{align} \label{lemm:antecedant_1_ortho}
	\beta= \frac{-1}{\langle \partial_x^{-1} \partial_y Q, \partial_y Q\rangle} \langle f, \partial_yQ \rangle, \quad \text{and} \quad \gamma= \frac{-1}{\langle \Lambda Q, Q \rangle} \langle f, \partial_x Q \rangle.
\end{align}
Furthermore, if $f =f(x,y)$ is even in $y$, then $\hat{A}=\hat{A}(x,y)$ is even in $y$ and $\beta=0$.
\end{lemm} 

\begin{proof}
First, we adjust the coefficients $(\beta,\gamma)$ to obtain the orthogonality of the right-hand side of \eqref{lemm:antecedant_1} with respect to $\{\partial_xQ,\partial_yQ\}$. With the notation \eqref{def:XYW} in hand and thanks to \eqref{eq:ineq2_Q}, \eqref{eq:ineq1_Q}, \eqref{eq:ineq3.1_Q} and \eqref{eq:ineq3_Q}, these orthogonality conditions are equivalent to 
\begin{equation*}
\begin{cases}
\beta  \langle 2QY, \partial_yQ \rangle+\gamma \langle 2QW, \partial_yQ \rangle+\langle f, \partial_yQ \rangle=0,\\
 \beta  \langle 2QY, \partial_xQ \rangle+\gamma \langle 2QW, \partial_xQ \rangle+\langle f, \partial_xQ \rangle=0,
\end{cases}
 \quad \Leftrightarrow \quad 
 \begin{cases}
\beta \langle \partial_x^{-1} \partial_y Q, \partial_y Q\rangle +\langle f, \partial_yQ \rangle=0,\\
 \gamma \langle\Lambda Q, Q \rangle+\langle f, \partial_xQ \rangle=0.
\end{cases}
\end{equation*}
Thus, the choice of $(\beta,\gamma)$ in \eqref{lemm:antecedant_1_ortho} and Proposition \ref{theo:L} (v) ensure the existence of a unique $\hat{A}\in H^2(\mathbb R^2)$ orthogonal to $\{\partial_x Q$, $\partial_y Q\}$ solution of \eqref{lemm:antecedant_1}. 

Moreover, it follows from Lemma \ref{parity} that $2 Q\partial_x^{-1}(-\Delta+1)^{-1}\left( \beta \partial_y Q +\gamma \Lambda Q \right) +f \in \mathcal{Y}$. We deduce then from Proposition \ref{theo:L} (vi) that $\hat{A} \in \mathcal{Y}$.

Finally, if $f$ is even in $y$, then it follows from Proposition \ref{theo:L} (v) that $\hat{A}$ is also even in $y$, so that $\beta=0$ thanks to \eqref{lemm:antecedant_1_ortho}.
\end{proof}
\end{toexclude}

The following lemma explains how to choose some coefficients in order to ensure important orthogonality conditions. 

\begin{lemm}\label{lemm:antecedent_2}
Let $f\in \Y$. There exists a unique vector $\overrightarrow{p}_0= \left( \alpha_0, \beta_0, \gamma_0 \right)^{T} \in \mathbb R^3$ such that
\begin{align} \label{lemm:antecedent_2_ortho}
    2Q \overrightarrow{p}_0 \cdot \NQ +f \perp  \{\partial_x Q, \partial_y Q, \Lambda Q\}.
\end{align}
The coefficients $\alpha_0$, $\beta_0$ and $\gamma_0$ are given by
\begin{align}
    & \alpha_0= \frac{1}{\langle Q,\Lambda Q \rangle +\langle(-\Delta+1)^{-1}Q,Q\rangle} \left( \langle f, \Lambda Q \rangle- \frac{\langle 2Q W, \Lambda Q\rangle}{\langle\Lambda Q, Q \rangle} \langle f, \partial_x Q \rangle \right), \label{def:alpha0}\\
    & \beta_0= \frac{1}{\langle \partial_x^{-1} \partial_y Q, \partial_y Q\rangle} \langle f, \partial_yQ \rangle, \quad \gamma_0= \frac{-1}{\langle \Lambda Q, Q \rangle} \langle f, \partial_x Q \rangle. \label{def:gamma0}
\end{align}
\begin{toexclude}
Moreover, $\alpha_0$, $\beta_0$ and $\gamma_0$ are smooth in $f$ with the topology of $\Y$  
\blue{(why with the topology of $\mathcal{Y}$? Isn't it sufficient in $L^2$? This would be given directly by the next formula)}.
\end{toexclude}
Furthermore, 
\begin{align} \label{lemm:antecedent_2_bound}
    \left\vert \alpha_0 \right\vert + \left\vert \beta_0 \right\vert + \left\vert \gamma_0 \right\vert \lesssim \| f \|_{L^2}.
\end{align}
\end{lemm}

\begin{proof}
First, observe by using the identities in Lemma \ref{propo:id_Q} that 
\begin{align*}
 \langle 2Q \overrightarrow{p}_0 \cdot \NQ +f, \partial_x Q \rangle & = \langle\Lambda Q, Q \rangle \gamma_0+\langle f, \partial_xQ \rangle ; \\ 
 \langle 2Q \overrightarrow{p}_0 \cdot \NQ +f, \partial_y Q \rangle &= -\langle \partial_x^{-1}\partial_y Q, \partial_y Q \rangle \beta_0+\langle f, \partial_yQ \rangle ;\\ 
 \langle 2Q \overrightarrow{p}_0 \cdot \NQ +f, \Lambda Q \rangle &= -\left(\langle Q,\Lambda Q \rangle +\langle(-\Delta+1)^{-1}Q,Q\rangle \right) \alpha_0+\langle 2Q W, \Lambda Q\rangle \gamma_0+\langle f, \Lambda Q \rangle .
\end{align*}
Note from Lemma \ref{propo:id_Q} that $\langle\Lambda Q, Q \rangle$, $ -\langle \partial_x^{-1}\partial_y Q, \partial_y Q \rangle $ and $\langle Q,\Lambda Q \rangle +\langle(-\Delta+1)^{-1}Q,Q\rangle$ are positive quantities. Thus, the choice of $\alpha_0$, $\beta_0$ and $\gamma_0$ as in the statement of the lemma is the unique possible choice to ensure the desired orthogonality relations.  

Finally, estimate \eqref{lemm:antecedent_2_bound} follows from the Cauchy-Schwarz inequality. 
\end{proof}

\subsection{Specific functions to describe the soliton interaction}\label{sec:interaction_notation}
In this subsection, we introduce some important notations and specific functions, which will be useful to construct the ansatz of the interaction of two solitary waves in the next subsections. 
\smallskip

For $i=1,2$, let $\bz_i(t)= (z_i(t),\omega_i(t)) \in \mathbb R^2$, respectively $\mu_i(t) \in \mathbb R$, be $C^1$ functions denoting the centers of the solitons, respectively the sizes of the modulated solitons. We also denote by $\tilde{\bz}_i(t)=(z_i(t),0)$ the recentered variables $\bz_i$ on the axis $\{y=0\}$. Moreover, we denote 
\begin{align}
        & z(t):=z_1(t) - z_2(t), \quad \Bar{z}(t):= z_1(t) +z_2(t), \label{defi:z_zbar} \\
        & \omega(t):= \omega_1(t) - \omega_2(t), \quad \Bar{\omega}(t):= \omega_1(t) + \omega_2(t), \label{defi:omega_omegabar} \\
        & \mu(t):=\mu_1(t) - \mu_2(t), \quad \Bar{\mu}(t):= \mu_1(t) + \mu_2(t). \label{defi:mu_mubar}
\end{align}
Sometimes, when there is no risk of confusion, we drop the variable $t$ of these functions.

\smallskip
We first  introduce a compact notation for the two modulated solitons $R_i$, and their rescaled and recentered versions $\Rti$. More precisely, for $i=1,2$, we denote
\begin{equation}\label{def:R_i}
 R_i(t,\bx)  := Q_{1+\mu_i(t)}(\bx- \bz_i(t)) \quad \text{and} \quad 
 \Rti(t,\bx) := Q (\bx-\tilde{\bz}_i(t)) 
\end{equation}
and 
\begin{equation} \label{def:LambdaR_i}
\Lambda R_i(t,\bx)  := \left(\Lambda_{1+\mu_i(t)} Q \right)  (\bx -\bz_i(t)) 
\quad \text{and} \quad \Lambda \Rti(t,\bx)  := \Lambda Q (\bx-\tilde{\bz}_i(t)) .
\end{equation}
Note in particular from the equation of $Q$ in \eqref{eq:Q} that 
\begin{align}
      &-\Delta R_i + (1+\mu_i) R_i - R_i^2=0, \label{eq:R_i}\\ 
    &-\Delta \tilde{R}_i + \tilde{R}_i - \tilde{R}_i^2=0. \label{eq:R_i_tilde}
\end{align}
We also define the linearized operator around $\tilde{R}_i$ and $R_i$ by
\begin{equation} \label{def:LRi}
 L_{\tilde{R}_i} := -\Delta +1 -2 \tilde{R}_i \quad \text{and} \quad L_{R_i}=-\Delta+(1+\mu_i)-2R_i . 
\end{equation}
Then, it follows from Proposition \ref{theo:L} (viii)  that 
\begin{align} \label{eq:L_Lambda_Q}
	 L_{\tilde{R}_i} (\Lambda \tilde{R}_i) = - \tilde{R}_i \quad \text{and} \quad
  \left(-\Delta+1-2R_i\right) \Lambda R_i=-R_i-\mu_i \Lambda R_i .
\end{align}
\smallskip

Finally, we introduce some useful functions $X_i=X_i(t,\bx)$, $Y_i=Y_i(t,\bx)$ and $W_i=W_i(t,\bx)$ to understand the interactions of the two solitons. For $i=1,2$, we define 
\begin{equation} \label{def:XYW:i}
	X_i := - (-\Delta+1)^{-1} \tilde{R}_i, \quad 
	Y_i  := -\partial_x^{-1} (-\Delta+1)^{-1} \partial_y \tilde{R}_i \ \text{and} \
	W_i  := -\partial_x^{-1} (-\Delta+1)^{-1}\Lambda \tilde{R}_i .
\end{equation}
Observe that the functions $X_i$, $Y_i$ and $W_i$ are the translated of the functions $X$, $Y$ and $W$ introduced in \eqref{def:XYW} by $z_i(t)$ in the first variable. Hence, $X_i \in \mathcal{Y}$ and $Y_i, \, W_i \in \mathcal{Z}$ (see Lemma \ref{parity}). However, for any fixed $y \in \mathbb R$, the limit when $x \to -\infty$ of $Y(x,y)$ and $W(x,y)$ is non zero. For this reason, we define 
\begin{equation} \label{def:hl}
	 h(y) := \int_{-\infty}^{+\infty} (-\Delta+1)^{-1} \partial_y Q(x,y) \, dx, \quad l(y) := \int_{-\infty}^{+\infty} (-\Delta+1)^{-1} \left( \Lambda Q \right)(x,y) \, dx. 
\end{equation}

The particular structure studied in this article consists of almost same sizes and symmetric solitary waves interacting through a plateau that can be approximated by
\begin{align}\label{defi:P}
    P(t,\bx) := W_1(t,\bx) - W_2(t,\bx).
\end{align}

\begin{figure}[ht]
\begin{tikzpicture}[
      declare function={
        R1(\x) = 2*exp(-3*(\x-3)^2);  
        R2(\x) = 2*exp(-3*(\x+3)^2);  
        W1(\x) = (3*rad(-atan(exp(2*(\x-3))))+3*pi/2)/2; 
        W2(\x) = 3*rad(atan(exp(-2*(\x+3))))/2; 
        P(\x) = (W1(\x)-W2(\x))*9/10; 
      }]

    \draw[->] (-6.5,0) -- (6.5,0);
    \draw[->] (0,-0.3) -- (0,3);

    \draw (3.7,1.5) node {$R_1$};
    \draw [domain=-6.5:6.5, samples=\Numbis] plot(\x,{R1(\x)});
    \draw (-2.2,1.2) node {$R_2$};
    \draw [domain=-6.5:6.5, samples=\Numbis] plot(\x,{R2(\x)});
    
    \draw (-4.2,1.9) node {$W_2$};
    \draw [domain=-6.5:6.5, samples=\Numbis,color=red] plot(\x,{W2(\x)});
    \draw (2.3,2.5) node {$W_1$};
    \draw [domain=-6.5:6.5, samples=\Numbis,color=red] plot(\x,{W1(\x)});
    \draw (0.5,1.8) node {$P$};
    \draw [domain=-6.5:6.5, samples=\Numbis,color=blue] plot(\x,{P(\x)});
\end{tikzpicture}
    \caption{Schematic reperesentation of the two solitary waves $R_1$ and $R_2$, of the functions $W_1$ and $W_2$, and of the function $P$ with a plateau between the solitary waves} \label{Figure:W}
\end{figure}
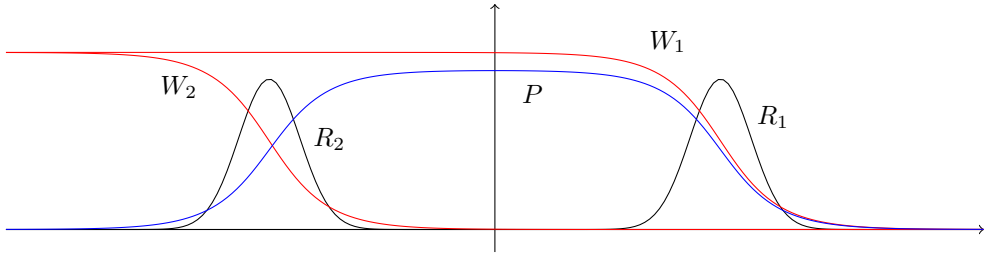

To work in a more condensed form, we introduce the vectors 
\begin{align}\label{defi:MiNi}
   \MRi := \begin{pmatrix} \partial_x R_i \\ \partial_y R_i \\ \Lambda R_i \end{pmatrix}, \quad \MRti:= \begin{pmatrix} \partial_x \Rti \\ \partial_y \Rti \\ \Lambda \Rti \end{pmatrix}, \quad \NRti:= \begin{pmatrix} X_i \\ Y_i \\ W_i \end{pmatrix} \text{ and } \n(y) := \begin{pmatrix} 0 \\ h(y) \\ l(y) \end{pmatrix}.
\end{align}
Then, we observe from the definitions of $\MRti$ and $\NRti$ that, for $i=1,2$,
\begin{align}\label{eq:link_M_N}
     \partial_x (-\Delta+1) \NRti = - \MRti, 
\end{align}
and with \eqref{def:LRi} and the relation \eqref{eq:L_Lambda_Q}
\begin{align}\label{eq:L_MRti}
L_{\Rti}\MRti= \begin{pmatrix} 0 \\ 0 \\ -\Rti \end{pmatrix} \quad \text{and} \quad (-\Delta+1 -2R_i) \MRi = \begin{pmatrix} -\mu_i \partial_x R_i \\ -\mu_i \partial_y R_i \\ -\mu_i \Lambda R_i -R_i \end{pmatrix}
\end{align}

\subsection{Definition and properties of the approximate solution}\label{sec:interaction_profiles}
The goal of the next subsections is to construct a suitable ansatz $V$ for the interaction of the two solitary waves. 

With the notations of Subsection \ref{sec:interaction_notation}, $V=V(t,\bx)$ should be on the form 
\begin{equation} \label{def:V}
V=R_1+R_2+V_A
\end{equation}
and be a good approximation of \eqref{ZK:sym}. To be more precise, we define the error to \eqref{ZK:sym} by
\begin{equation} \label{def:EV}
\E(V) := \partial_tV - \partial_x \left( -\Delta V + V -V^2 \right) .
\end{equation}

The approximate solution $V$ will depend on time through the set of geometrical parameters $\Gamma = (\mu_1, \mu_2, z_1,z_2,\omega_1, \omega_2)$ appearing in the definition of $R_1$ and $R_2$ in \eqref{def:R_i}. These parameters will be adjusted in Section \ref{sec:modulation} by modulation theory. For the moment, we will assume that they satisfy the following rough estimates. 
\begin{assump} \label{hyp:coeff}
Let $I$ be a time interval,  $Z_0^\star \ge 1$ and  $0<\nu_0 \le \nu_0^\star \le 1$, where $Z_0^{*}$ will be chosen large enough and $\nu_0^\star$ will be chosen small enough. Let $\Gamma = (\mu_1, \mu_2, z_1,z_2,\omega_1, \omega_2):I \to \mathbb R^6$ be a $C^1$ function and let $z$ be the real-valued function defined in \eqref{defi:z_zbar}. We assume that, for all $t \in I$,
\begin{align}
   &  z(t)=z_1(t)-z_2(t) \ge Z_0^{\star} \label{hyp:z} \\
    & \vert \omega_1 \vert \leq  \nu_0, \   \vert \omega_2  \vert \leq  \nu_0, \ \vert \mu_1 \vert \leq \nu_0, \ \vert \mu_2 \vert \leq \nu_0  \label{hyp:w1+w2} \\
   &  - z(t) \leq z_2(t) \leq -\frac{1}{4} z(t), \quad  \frac{1}{4} z(t) \leq z_1(t ) \leq  z(t), \label{hyp:z1_z2}
\end{align}
\end{assump}

To gain some insight on how to define the function $V_A$, we start by looking at  $\E(R_1+R_2)$. By differentiating and using \eqref{eq:R_i}, we find
\begin{align} \label{eq:E_R1R2}
\E(R_1+R_2)
	& = \sum_{i=1}^2  \begin{pmatrix} -\dot{z}_i+\mu_i \\ -\dot{\omega}_i \\ \dot{\mu}_i \end{pmatrix} \cdot \MRi + \partial_x \left( 2R_1 R_2\right) .
\end{align}
To cancel out the main contribution of the interaction term  $\partial_x \left( 2R_1 R_2\right)$, we introduce the correction term $V_A$. 

For $i=1,2$, let $\pAi(z) \in \mathbb R^3$ be a $C^1$ vector-valued function defined by
\begin{align}\label{defi:p_q}
   \pAi(z) := \begin{pmatrix} \alpha_{i}(z) \\ \beta_{i}(z) \\ \gamma_i(z) \end{pmatrix} .
\end{align}
Then, we define $V_A$ by
\begin{align} \label{def:VA}
V_A := F_{} + \sum_{i=1}^2  \pAi \cdot \NRti, \quad \text{with} \quad F_{} := (-\Delta +1)^{-1}(2\tilde{R}_1 \tilde{R}_2),
\end{align}
where the vector $\NRti$ is defined in \eqref{defi:MiNi}. We also denote by $V_{A,1}$, respectively $V_{A,2}$, the contributions of $V_A$ around $R_1$, respectively around $R_2$. They are defined by
\begin{align} 
V_{A,1} & := F_{} +  \pAone \cdot \NRtone ; \label{def:VA1} \\ 
V_{A,2} & := F_{} +  \pAtwo \cdot \NRttwo+ \pAone \cdot \n . \label{def:VA2}
\end{align}

The main results in this section are gathered in the following proposition.
\begin{propo}\label{theo:V_A}
Under Assumption \ref{hyp:coeff}, there exist $Z_0^\star \ge 1$ large enough and $0<\nu_0^\star \le 1$ small enough such that the following is true. For $i=1,2$, there exists a $C^1$ vector-valued function  $\pAi=\pAi(z) \in \mathbb R^3$ such that the following properties hold for $t \in I$.
\begin{itemize}
\item[(i)] \emph{Estimate for $V_A$.} $V_A \in C^1(I: H^2(\mathbb R^2) \cap W^{1,\infty}(\mathbb R^2))$ with, for all $t\in I$, $V_A(t,\bx) \to 0$, when $|\bx| \to \infty$, and
\begin{align} 
&\|V_A \|_{L^{\infty}}+\|\nabla V_A \|_{L^{\infty}}  \lesssim z^{-\frac12}e^{-z} ; \label{est:VA:point} \\
& \|V_A \|_{H^2} \lesssim e^{-\frac{15}{16} z} ; \label{est:VA:H2} \\ 
& \|\partial_t V_A  \|_{H^2}  \lesssim e^{-\frac{15}{16} z} \sum_{i=1}^2 |\dot{z}_i|; \label{est:dVAdt} 
\end{align}
\item[(ii)] \emph{Estimate for $V$.} For $i=1,2$,
\begin{align} 
\left\| \vert \MRi \vert (V-R_i) \right\|_{L^2} \lesssim  e^{-\frac{15}{16} z} .\label{est:V-Ri}
\end{align}
\item[(iii)] \emph{Estimates on the parameters.} For $i=1,2$, the vector $\pAi(z)$ defined in \eqref{defi:p_q} satisfies
    \begin{align} 
    &\alpha_2(z)=\alpha_1(z), \quad \beta_2(z)=\beta_1(z)=0, \quad \gamma_2(z)=-\gamma_1(z)\label{p1p2} \\
    &\gamma_i(z)= \frac{-1}{\langle \Lambda Q, Q \rangle} \left\langle Q\left(\cdot-(-1)^i\tilde{\bz}\right), \partial_x(Q^2) \right\rangle \label{Def:gammai} \\
    & \vert \alpha_i(z) \vert  + \vert  \gamma_i(z) \vert\lesssim  z^{-\frac{1}{2}}e^{-z} , \label{Est:piqi} 
    \end{align}
\item[(v)] \emph{Equation of the ansatz $V$.} The error $\E_V=\E(V)$ can be decomposed into
    \begin{align} \label{eq:error_R1_R2_VA}
       \E_V=\E(V)= \E(R_1+R_2+V_A) = \sum_{i=1}^2 \mAi \cdot \MRi + T_{} +\partial_xS_{},
    \end{align}
    where, for $i=1,2$, the directions $\MRi$ are defined in \eqref{defi:MiNi}, the modulation equations $\mAi$ are defined by
    \begin{align}\label{defi:mAi}
    \mAi := \begin{pmatrix} -\dot{z}_i + \mu_i + \alpha_i  \\ -\dot{\omega}_i +\beta_i \\ \dot{\mu}_i + \gamma_i \end{pmatrix} ,
\end{align}
and the error terms $S_{}$ and $T_{}$ satisfy 
\begin{equation} \label{decomp:S}
S_{}=\Sigma_{}+\Psi_{\Sigma_{}}+\Psi_{S_{}}
\end{equation} 
with
\begin{align}
        & \| \Sigma_{}\|_{H^1} \lesssim z^{-\frac1{2}}e^{-z}; \label{est:Sigma} \\
         & \| \Psi_\Sigma \|_{H^1} \lesssim 
         e^{-\frac{15}{16} z}\sum_{i=1}^2 \left( |\mu_i|+|\omega_i| \right); \label{est:PsiSigma} \\
        & \| \Psi_{S_{}}\|_{H^1} \lesssim  e^{-\frac{31}{16}z} +e^{-\frac{15}{16} z} \sum_{i=1}^2\left(\mu_i^2+\omega_i^2 \right) ; \label{est:PsiS} \\
        & \| T_{}  \|_{H^1} \lesssim e^{-\frac{15}{16} z} \sum_{i=1}^2 |\dot{z}_i|+z^{-\frac12}e^{-z}\sum_{i=1}^2 \left( |\mu_i|+|\omega_i| \right) ,\label{est:T} \\
        & \left\| \partial_t S \right\|_{H^1} \lesssim z^{-\frac12} e^{-z} \sum_{i=1}^2 \vert \dot{z}_i \vert + e^{-\frac{15}{16} z}  \sum_{i=1}^2 \left( \vert \dot{\mu}_i \vert + \vert \dot{\omega}_i \vert \right)  + e^{-\frac{15}{16} z} \left( \sum_{i=1}^2 \vert \dot{z}_i \vert \right)  \sum_{i=1}^2\left( \vert \mu_i \vert+ \vert \omega_i \vert \right). \label{est:dt_S}
\end{align}
    \item[(vi)] \emph{Orthogonality estimates for $\Sigma_{}$ and $S$.} For $i=1,2$, the quantity $\Sigma_{}$ satisfies 
    \begin{align}
        & \left\vert \langle \Sigma_{}, \partial_x \Rti \rangle \right\vert +\left\vert \langle \Sigma_{}, \partial_y \Rti \rangle \right\vert +\left\vert \langle \Sigma_{}, \Lambda \Rti \rangle \right\vert \lesssim e^{-\frac{31}{16}z} , \label{est:ortho_Sigma}\\
        & \left\vert \langle S, \partial_x R_i \rangle \right\vert +\left\vert \langle S, \partial_y R_i \rangle \right\vert +\left\vert \langle S, \Lambda R_i \rangle \right\vert \lesssim e^{-\frac{31}{16} z} + e^{- \frac{15}{16} z} \sum_{i=1}^2 \left( \vert \mu_i \vert + \vert \omega_i \vert \right).\label{est:ortho_S}
    \end{align}
\end{itemize}
\end{propo} 

\begin{toexclude}
\blue{\begin{rema}
The distinction between $\Psi_S$ and $\Psi_{\Sigma}$ is not useful in this paper. However, it will be useful to prove the inelasticity in paper $\#3$.
\end{rema}}
\end{toexclude}

The rest of this section is devoted to the proof of Proposition \ref{theo:V_A}.

\subsection{Decomposition of the error \texorpdfstring{$\E(V)$}{EV}}
We compute the error associated to this new approximation
\begin{align*}
\E(R_1+R_2+V_A)= \E(R_1+R_2) +\E(V_A) + \partial_x \left(2(R_1+R_2) V_A \right).
\end{align*}

To compute $\E(V_A)$, we observe by using \eqref{eq:link_M_N} that
\begin{align*}
    -\partial_x \left((-\Delta +1) V_A \right) = - \partial_x \left( 2\Rtone \Rttwo \right)+\sum_{i=1}^2  \pAi \cdot \MRti .
\end{align*}
Hence, we deduce that
\begin{align*}
	\E(V_A)
		& =\partial_tV_A - \partial_x \left( 2\Rtone \Rttwo \right) + \sum_{i=1}^2 \pAi \cdot \MRti  +\partial_x \left( V_A^2 \right) .
\end{align*}
Therefore, by using \eqref{eq:E_R1R2}, we rearrange the error $\E(V)$ in the more concise form \eqref{eq:error_R1_R2_VA}, where
\begin{equation}  \label{def:SA}
    S_{} :=2(R_1 R_2 - \Rtone \Rttwo ) + 2(R_1+R_2) V_A  +V_A^2 
\end{equation}
and    
\begin{equation} \label{def:TA}
    T_{} := \partial_tV_A+\sum_{i=1}^2 \pAi \cdot \left( \MRti-\MRi \right)  
\end{equation}
Now we rearrange the term $S_{}$. First, with the notations \eqref{def:VA1}-\eqref{def:VA2}, we decompose the non-linear term $2(R_1+R_2)V_A$ in the following way
\begin{align*}
    2(R_1+R_2)V_A & =  2 (R_1 -\Rtone + R_2 - \Rttwo) V_A+2 \Rtone V_{A,1} +2 \Rttwo V_{A,2} \\
        & \quad + 2\Rtone \left( \pAtwo \cdot \NRttwo  \right) + 2 \Rttwo \left(\pAone \cdot \left(\NRtone- \n \right) \right) .
\end{align*}
Thus  $S_{}$ can be written as in \eqref{decomp:S} where
 $\Sigma_{}= \Sigma_{1} + \Sigma_{2}$
and
\begin{align}
    \Sigma_{1} & := 2\Rtone \pAone \cdot \NRtone + 2 \Rtone F_{}=2\Rtone V_{A,1}, \label{defi:Sigma1} \\
    \Sigma_{2} & := 2\Rttwo \pAtwo \cdot \NRttwo + 2 \Rttwo F_{} + 2\Rttwo\pAone \cdot \n=2\Rttwo V_{A,2}, \label{defi:Sigma2}\\
    \Psi_{\Sigma} & := 2\left( \mu_1 \Lambda \Rtone + \omega_1 \partial_y \Rtone \right) V_{A,1}+2\left( \mu_2 \Lambda \Rttwo + \omega_2 \partial_y \Rttwo \right)V_{A,2}\noindent  \nonumber \\ & \quad +2\left(\mu_1\Lambda \Rtone+\omega_1\partial_y \Rtone\right)\Rttwo+2\Rtone\left(\mu_2\Lambda \Rttwo+\omega_2\partial_y \Rttwo\right)\label{defi:Psi_Sigma} \\
    \Psi_{S} & := 2\Rtone \left( \pAtwo \cdot \NRttwo  \right) + 2 \Rttwo \left(\pAone \cdot \left(\NRtone- \n(y) \right) \right)+2(R_1 R_2 - \Rtone \Rttwo )+V_A^2 \nonumber\\
        & \quad + 2 (R_1 -\Rtone + R_2 - \Rttwo) V_A - \Psi_{\Sigma} .\label{defi:Psi_S}
\end{align}

\subsection{Choice of the coefficients \texorpdfstring{$\protect\pAi$}{pAi}}

We adjust the coefficients  $\pAi$ such that the contributions $\Sigma_i$ satisfy an adequate orthogonality relation.

\begin{lemm}\label{lemma:Sigma_i}
Under Assumption \ref{hyp:coeff}  with $Z^\star_0>0$ large enough, for $i=1,2$, there exists a $C^1$ vector-valued function $\pAi=\pAi(z)=(\alpha_i(z), \beta_i(z), \gamma_i(z))^T \in \mathbb R^3$ such that, for $t \in I$, 
\begin{align}
    &\langle \Sigma_{i} ,\partial_x \Rti \rangle = \langle \Sigma_{i} , \partial_y \Rti \rangle = \langle \Sigma_{i} , \Lambda \Rti \rangle = 0, \label{ortho:Sigmai}\\
    & \beta_i(z)=0, \quad \gamma_i(z)= \frac{-1}{\langle \Lambda Q, Q \rangle} \left\langle Q\left(\cdot-(-1)^i\tilde{\bz}\right), \partial_x(Q^2) \right\rangle \label{def:gamma_beta}\\ 
    & \vert \alpha_i(z) \vert +\left\vert \alpha'_i(z) \right\vert+ \vert \beta_i(z) \vert +\left\vert \beta'_i(z) \right\vert +\vert \gamma_i(z) \vert+\vert \gamma_i'(z) \vert \lesssim z^{-\frac12}e^{-z}. \label{est:pi}
\end{align}
As a consequence, $\gamma_2(z)=-\gamma_1(z)$ and $\alpha_2(z)=\alpha_1(z)$.
\end{lemm}

\begin{proof}
We start with the case $i=1$.  We translate the equation of $\Sigma_1$ in \eqref{defi:Sigma1} by $-\tilde{\bz}_1=-(z_1,0)$. Then, we will work on the quantity
\begin{align}\label{eq:Sigma1_translated}
	\widehat{\Sigma}_1(\cdot,z):=\Sigma_{1}(\cdot + \tilde{\bz}_1) = 2 Q \pAone \cdot \NQ + 2Q \left(-\Delta+1 \right)^{-1} \left( 2QQ(\cdot +\tilde{\bz})\right).
\end{align}
By Lemma \ref{lemm:antecedent_2}, there exists a unique  $\overrightarrow{p_{1}} =(\alpha_1, \beta_1, \gamma_1)^T  \in \mathbb R^3$ such that 
\begin{align*}
    2Q \overrightarrow{p_{1}} \cdot \NQ + 2Q(-\Delta+1)^{-1}\left( 2Q  Q(\cdot+\tilde{\bz}) \right) \ \perp \
    \left\{ \partial_xQ, \partial_y Q, \Lambda Q \right\} .
\end{align*}
Moreover, it follows from \eqref{eq:Q} and \eqref{def:alpha0}-\eqref{def:gamma0} that the coefficients $\overrightarrow{p_{1}}(z) =(\alpha_1(z), \beta_1(z), \gamma_1(z))^T$ are given by 
\begin{align*}
\alpha_1(z) & = c \left( \langle  2Q(-\Delta+1)^{-1}\left( 2Q  Q(\cdot+\tilde{\bz}) \right), \Lambda Q \rangle- \frac{\langle 2Q W, \Lambda Q\rangle}{\langle\Lambda Q, Q \rangle} \left\langle Q(\cdot+\tilde{\bz}), \partial_x (Q^2)\right\rangle  \right),\\ 
\beta_1(z) &= \left(\langle \partial_x^{-1}\partial_yQ,\partial_yQ \rangle\right)^{-1} \left\langle Q(\cdot+\tilde{\bz}), \partial_y (Q^2)\right\rangle,\\
    \gamma_1(z) 
        & = -\left(\langle \Lambda Q,Q \rangle\right)^{-1} \left\langle Q(\cdot+\tilde{\bz}), \partial_x (Q^2)\right\rangle ,
\end{align*}
with $c=\left(\langle Q,\Lambda Q \rangle +\langle(-\Delta+1)^{-1}Q,Q\rangle\right)^{-1}$.
Observe by parity that $\beta_1(z)=0$.

\smallskip

Now, we deal with the case $i=2$.  We translate the equation of $\Sigma_2$ in \eqref{defi:Sigma2} by $-\tilde{\bz}_2=-(z_2,0)$, so that we will focus on the quantity
\begin{align} \label{eq:Sigma2_translated}
\widehat{\Sigma}_2(\cdot,z):=\Sigma_{2}(\cdot+\tilde{\bz}_2) & := 2Q \pAtwo \cdot \NQ + 2 Q (-\Delta+1)^{-1}\left( 2Q Q(\cdot - \tilde{\bz}) \right) + 2Q\pAone \cdot \n . 
\end{align}
As in the case $i=1$, we want to ensure the orthogonality relations \eqref{ortho:Sigmai} for $i=2$. By Lemma \ref{lemm:antecedent_2}, there exists a unique $\overrightarrow{p_{2}}=(\alpha_2, \beta_2^,\gamma_2)^T \in \mathbb R^3$ such that
\begin{align*}
    2Q \overrightarrow{p_{2}} \cdot \NQ +2 Q \left( -\Delta+1\right)^{-1} \left( 2Q  Q(\cdot-\tilde{\bz})\right) +2Q \overrightarrow{p_{1}} \cdot \n \ \perp \ \left\{ \partial_xQ, \partial_y Q,\Lambda Q \right\}
\end{align*}
Since $\beta_1(z)=0$, $2Q\pAone \cdot \n=2\gamma_1(z)Q l(y) $. Moreover, note from \eqref{def:hl} that $l=l(y)$ is even.
Thus, the coefficients  $\overrightarrow{p_{2}}(z) =(\alpha_2(z), \beta_2(z), \gamma_2(z))^T$ are given by
\begin{align*}
\alpha_2(z) & = c \left( \langle  2Q(-\Delta+1)^{-1}\left( 2Q  Q(\cdot-\tilde{\bz}) \right), \Lambda Q \rangle+2\gamma_1(z) \langle Ql(y),\Lambda Q \rangle \right. \\
    & \quad \quad \left. - \frac{\langle 2Q W, \Lambda Q\rangle}{\langle\Lambda Q, Q \rangle} \left\langle Q(\cdot-\tilde{\bz}), \partial_x (Q^2)\right\rangle  \right),\\ 
\beta_2(z) &= \left(\langle \partial_x^{-1}\partial_yQ,\partial_yQ \rangle\right)^{-1} \left\langle Q(\cdot-\tilde{\bz}), \partial_y (Q^2)\right\rangle,\\
    \gamma_2(z) 
        & = -\left(\langle \Lambda Q,Q \rangle\right)^{-1} \left\langle Q(\cdot-\tilde{\bz}), \partial_x (Q^2)\right\rangle .
\end{align*}
with $c=\left(\langle Q,\Lambda Q \rangle +\langle(-\Delta+1)^{-1}Q,Q\rangle\right)^{-1}$.
Since $Q$ is radial, we have $Q(\bx-\tilde{\bz})=Q(\tilde{\bz}-\bx)$. Hence, it follows by parity that $\beta_2(z)=0$,  $\gamma_2(z)=-\gamma_1(z)$ and 
\begin{align*}
\alpha_2(z) & = c \left( \langle  2Q(-\Delta+1)^{-1}\left( 2Q  Q(\cdot+\tilde{\bz}) \right), \Lambda Q \rangle+ \frac{\langle 2Q (W-l), \Lambda Q\rangle}{\langle\Lambda Q, Q \rangle} \left\langle Q(\cdot+\tilde{\bz}), \partial_x (Q^2)\right\rangle  \right)=\alpha_1(z),
\end{align*}
where we used on the last step that $W(x,y)-l(y)=-W(-x,y)$.

Finally, estimate
\eqref{est:pi}  follows from  the Cauchy-Schwarz inequality and estimate \eqref{est.decay.QzQ}.
This concludes the proof of Lemma \ref{lemma:Sigma_i}.
\end{proof}

\subsection{Estimates for the components of the approximate solution}
Since $\alpha_1(z)=\alpha_2(z)$ and $\gamma_1(z)=-\gamma_2(z)$ (see Lemma \ref{lemma:Sigma_i}), we can rewrite $V_A$ as 
\begin{equation*} 
V_A=F+\alpha_1(z)\sum_{i=1}^2X_i +\gamma_1(z) P ,
\end{equation*}
where $X_1$, $X_2$ and $P$ are defined in \eqref{def:XYW:i} and \eqref{defi:P}, so that $V_A(\bx) \to 0$ as $|\bx| \to \infty$. Finally estimates \eqref{est:VA:point}--\eqref{est:dVAdt} are a consequence of the following claims combined with \eqref{est:pi}.

\begin{claim} \label{claim:P}
Let $P$ be defined as in \eqref{defi:P}. Then $P \in C^1(I:H^{2}(\mathbb R^2) \cap W^{1,\infty}(\mathbb R^2))$. Moreover, we have, for all $t \in I$ and $\bx \in \mathbb R^2$, 
\begin{align} 
& \left\vert P(t,\bx) \right\vert+\left\vert \nabla P(t,\bx) \right\vert \lesssim 1 ; \label{est:P:point}\\
&\left\| P(t,\cdot) \right\|_{H^2} \lesssim z^{\frac32} ; \label{est:P:H2.1}\\ 
&\left\| \partial_tP(t,\cdot) \right\|_{H^2} \lesssim |\dot{z}_1|+|\dot{z}_2|.\label{est:P:H2.2}
\end{align}
\end{claim}

\begin{proof}
First, thanks to the assumptions on $z_1$ and $z_2$ in \eqref{hyp:coeff} and Lemma \ref{comm:dx-1}, we rewrite $P$ as  
\begin{equation*}
P(t,\bx)=\int_{x-z_1}^{x-z_2} \left(-\Delta+1\right)^{-1} \Lambda Q(\tilde{x},y) d\tilde{x}=\left(-\Delta+1\right)^{-1}\int_{x-z_1}^{x-z_2}  \Lambda Q(\tilde{x},y) d\tilde{x} .
\end{equation*}
Since $\left(-\Delta+1\right)^{-1} \Lambda Q(\tilde{x},y) \in \mathcal{Y}$ (see Lemma \ref{Bessel:Y}), we have $\left(-\Delta+1\right)^{-1} \Lambda Q(\tilde{x},y) \lesssim e^{-\frac{|y|}2}e^{-\frac{|\tilde{x}|}2}$, so that 
\begin{equation*}
\left|P(t,\bx) \right| \lesssim \int_{-\infty}^{+\infty}e^{-\frac{|\tilde{x}|}2} d\tilde{x} \lesssim 1, 
\end{equation*}
which proves the first estimate in \eqref{est:P:point}. The proof of the second estimate in \eqref{est:P:point} follows arguing in a similar way.  

Next, we claim that  
\begin{equation} \label{Claim:P.1}
\left\vert \int_{x-z_1}^{x-z_2} \Lambda Q(\tilde{x},y) d\tilde{x}\right\vert \lesssim e^{-\frac{|y|}2} \left(z{\bf 1}_{2z_2<x<2z_1} +e^{-\frac{|x|}4} \right) .
\end{equation}
 which implies \eqref{est:P:H2.1} by using that $(-\Delta+1)^{-1}: L ^2(\mathbb R^2) \to H^2(\mathbb R^2)$ is a bounded operator. 
To prove \eqref{Claim:P.1}, we first observe that $\left|\Lambda Q(\tilde{x},y)\right| \lesssim e^{-\frac{|\tilde{x}|}2}e^{-\frac{|y|}2}$ since $\Lambda Q \in \mathcal{Y}$. On the one hand, in the case  $2z_2<\tilde{x}<2z_1$, we have 
\begin{equation*} 
\left\vert \int_{x-z_1}^{x-z_2} \Lambda Q(\tilde{x},y) d\tilde{x}\right\vert \lesssim e^{-\frac{|y|}2} \int_{x-z_1}^{x-z_2} d\tilde{x}=ze^{-\frac{|y|}2} .
\end{equation*}
On the other hand, in the case where $x>2z_1>0$, we have $x-z_1>x/2$, so that 
\begin{equation*} 
\left\vert \int_{x-z_1}^{x-z_2} \Lambda Q(\tilde{x},y) d\tilde{x}\right\vert \lesssim e^{-\frac{|y|}2} \int_{x/2}^{+\infty} e^{-\frac{|\tilde{x}|}2}d\tilde{x}\lesssim e^{-\frac{|x|}4}e^{-\frac{|y|}2} .
\end{equation*}
The estimate is similar in the case where $x<2z_2<0$.

Finally, observe that 
\begin{equation*}
\partial_tP(t,\bx)=\dot{z}_1 \left(-\Delta+1\right)^{-1}\Lambda \tilde{R}_1-\dot{z}_2 \left(-\Delta+1\right)^{-1}\Lambda \tilde{R}_2 .
\end{equation*}
This, combined with the fact that $\Lambda Q \in \mathcal{Y}$ and $(-\Delta+1)^{-1}: L ^2(\mathbb R^2) \to H^2(\mathbb R^2)$ is a bounded operator concludes the proof of \eqref{est:P:H2.2}.
\end{proof}

\begin{claim} \label{claim:F}
Let $F$ be defined as in \eqref{def:VA}. Then $F \in C^1(I:H^{2}(\mathbb R^2) \cap L^{\infty}(\mathbb R^2))$. Moreover, we have, for all $t \in I$ and $\bx \in \mathbb R^2$, 
\begin{align} 
&\left\vert F(t,\bx) \right\vert +\left\vert \nabla F(t,\bx) \right\vert\lesssim z^{-\frac12}e^{-z} ; \label{est:F:point} \\
&\left\vert \partial_t F(t,\bx) \right\vert +\left\vert \nabla \partial_t F(t,\bx) \right\vert\lesssim \left( \vert \dot{z}_1 \vert + \vert \dot{z}_2 \vert \right) z^{-\frac12}e^{-z} ; \label{est:dtF:point} \\
&\left\| F(t,\cdot) \right\|_{H^2} \lesssim e^{-\frac{15}{16} z} ; \label{est:F:H2.1}\\ 
&\left\| \partial_tF(t,\cdot) \right\|_{H^2} \lesssim \left(|\dot{z}_1|+|\dot{z}_2|\right)e^{-\frac{15}{16} z}.\label{est:F:H2.2}
\end{align}
\end{claim}

\begin{proof}
Estimate \eqref{est:F:point} and \eqref{est:dtF:point} are consequences of \eqref{est.decay.QzQ}, \eqref{coro:G_2.1} and of 
\begin{equation*} 
\partial_tF=-2\dot{z}_1 \left(-\Delta+1\right)^{-1}\left((\partial_x\tilde{R}_1)\tilde{R}_2 \right)-2\dot{z}_2 \left(-\Delta+1\right)^{-1}\left(\tilde{R}_1(\partial_x\tilde{R}_2) \right) .
\end{equation*}
Thus, estimates \eqref{est:F:H2.1} and \eqref{est:F:H2.2} follow combining the fact that $(-\Delta+1)^{-1}: L ^2(\mathbb R^2) \to H^2(\mathbb R^2)$ is a bounded operator with \eqref{est:fg:L2}. 
\end{proof}

Finally, we prove \eqref{est:V-Ri}.
Without loss of generality, we assume that $i=1$. We deduce from the definition of $V$ in \eqref{def:V} and estimates \eqref{est:VA:H2} and \eqref{est:R_1R_2} that 
\begin{equation*} 
\left\| \vert \MRone \vert (V-R_1) \right\|_{L^2} \lesssim \left\| \vert \MRone \vert R_2 \right\|_{L^2}+\|V_{A}\|_{L^2} \lesssim e^{-\frac{15}{16} z} .
\end{equation*}

\subsection{Estimates for the error terms}

\noindent
\emph{Estimate for $\Sigma$.} It follows from the definition of $\Sigma=\Sigma_1+\Sigma_2$ (see \eqref{defi:Sigma1}-\eqref{defi:Sigma2}) that 
\begin{equation*} 
\|\Sigma\|_{H^1} \lesssim \sum_{i=1}^2 \left( |\alpha_i(z)|+|\gamma_i(z)| \right) +\|Q\|_{H^1}\left( \|F\|_{L^{\infty}} +\|\nabla F\|_{L^{\infty}} \right), 
\end{equation*}
which, combined to \eqref{Est:piqi} and \eqref{est:F:point} implies \eqref{est:Sigma}. 

\smallskip
\noindent
\emph{Estimate for $\Psi_{\Sigma}$.} It follows from the definition of $\Psi_{\Sigma}$ in \eqref{defi:Psi_Sigma} that 
\begin{align*} 
\|\Psi_{\Sigma}\|_{H^1} &\lesssim \left(\sum_{i=1}^2 \left( |\alpha_i(z)|+|\gamma_i(z)| \right) + \|F\|_{L^{\infty}}+\|\nabla F\|_{L^{\infty}}\right)\sum_{i=1}^2 \left( |\mu_i|+|\omega_i| \right) \\ 
& \quad + \left( \|\Lambda \Rtone \Rttwo\|_{H^1}+ \| \Rtone \Lambda\Rttwo\|_{H^1}+ \|\partial_y \Rtone \Rttwo\|_{H^1}+\| \Rtone \partial_y\Rttwo\|_{H^1}\right) \sum_{i=1}^2\left( |\mu_i|+|\omega_i| \right),
\end{align*}
which, combined to \eqref{Est:piqi}, \eqref{est:F:point} and \eqref{est:R_1R_2}, implies \eqref{est:PsiSigma}.

\smallskip
\noindent
\emph{Estimate for $\Psi_{S}$.} By using \eqref{defi:Psi_S}, we rewrite $\Psi_S$ as 
\begin{equation} \label{defi:Psi_S.2}
\begin{split}
\Psi_S&=2\Rtone \left( \pAtwo \cdot \NRttwo  \right)+ 2 \Rttwo \left(\pAone \cdot \left(\NRtone- \n(y) \right) \right) +2\sum_{i=1}^2\left(R_i -\Rti -\mu_i \Lambda \Rti-\omega_i\partial_y\Rti\right) V_{A,i} \\
& \quad +V_A^2 +2\left(R_1R_2-\Rtone\Rttwo-\left(\mu_1\Lambda \Rtone+\omega_1\partial_y \Rtone\right)\Rttwo-\Rtone\left(\mu_2\Lambda \Rttwo+\omega_2\partial_y \Rttwo\right)\right) \\  
        &\quad +2 \left(R_1 -\Rtone\right) \left( \pAtwo \cdot \NRttwo  \right)+ 2 \left(R_2 -\Rttwo\right) \left(\pAone \cdot \left(\NRtone- \n \right) \right)
\end{split}
\end{equation}
We estimate separately each term in  \eqref{defi:Psi_S.2}. These estimates are gathered in the next claim, which concludes the proof of \eqref{est:PsiS}. 

\begin{claim} \label{claim:Psi_S} The following estimates hold.
\begin{align} 
&\left\|\Rtone \left( \pAtwo \cdot \NRttwo  \right) \right\|_{H^1}+\left\|\Rttwo \left(\pAone \cdot \left(\NRtone- \n \right) \right) \right\|_{H^1} \lesssim e^{-\frac{31}{16}z} ; \label{est:RipiNi} \\
& \left\|V_A^2\right\|_{H^1} \lesssim e^{-\frac{31}{16}z} ; \label{est:Va2} \\
&\left\|R_1R_2-\Rtone\Rttwo-\left(\mu_1\Lambda \Rtone+\omega_1\partial_y \Rtone\right)\Rttwo-\Rtone\left(\mu_2\Lambda \Rttwo+\omega_2\partial_y \Rttwo\right)\right\|_{H^1}\lesssim e^{-\frac{15}{16}z} \sum_{i=1}^2  \left( \mu_i^2+\omega_i^2 \right) ; 
\label{est:RtR2-tR1tR2} \\ 
&\sum_{i=1}^2\left\|\left(R_i -\Rti -\mu_i \Lambda \Rti-\omega_i\partial_y\Rti\right) V_{A,i}\right\|_{H^1} \lesssim z^{-\frac12}e^{-z}\sum_{i=1}^2 \left(\mu_i^2+\omega_i^2 \right) ; \label{est:Ri-tildeRi_VAi} \\ 
&\left\|(R_1-\Rtone) \left( \pAtwo \cdot \NRttwo  \right) \right\|_{H^1}+\left\|(R_2-\Rttwo) \left(\pAone \cdot \left(\NRtone- \n \right) \right) \right\|_{H^1} \lesssim e^{-\frac{31}{16}z}\sum_{i=1}^2\left( |\mu_i|+|\omega_i|\right) . \label{est:RipiNi_refined} 
\end{align}
\end{claim}

\begin{proof} 
We have from the definition of $\pAi$ in \eqref{defi:p_q} and the ones of $\NRti$ and $\n$ in \eqref{defi:MiNi} that
\begin{align*} 
&\left\|\Rtone \left( \pAtwo \cdot \NRttwo  \right) \right\|_{H^1} \le |\alpha_2(z)| \left\|\Rtone X_2\right\|_{H^1}+|\gamma_2(z)|\left\|\Rtone W_2\right\|_{H^1} ; \\ 
& \left\|\Rttwo \left(\pAone \cdot \left(\NRtone- \n(y) \right) \right) \right\|_{H^1} \le |\alpha_1(z)| \left\|\Rttwo X_1\right\|_{H^1}+|\gamma_1(z)|\left\|\Rttwo \left(W_1-l\right)\right\|_{H^1}.
\end{align*}
Therefore, we conclude the proof of \eqref{est:RipiNi}  by gathering \eqref{est:pi}, \eqref{est:tildeR_X}, \eqref{est:tildeR1_W2} and \eqref{est:tildeR2_W1-l}. 

Estimate \eqref{est:Va2} follows from \eqref{est:VA:point}-\eqref{est:VA:H2}, while estimate \eqref{est:RtR2-tR1tR2} is proved in Lemma \ref{lemm:R_1_R_2}.

Next, we deduce from the definition of $V_{A,1}$ in \eqref{def:VA1} that 
\begin{align*}
\sum_{i=1}^2&\left\|\left(R_i -\Rti -\mu_i \Lambda \Rti-\omega_i\partial_y\Rti\right) V_{A,i}\right\|_{H^1} \\ & 
\lesssim \left(\sum_{i=1}^2\left\|\left(R_i -\Rti -\mu_i \Lambda \Rti-\omega_i\partial_y\Rti\right)\right\|_{H^1}\right)\left(\sum_{i=1}^2 \left(|\alpha_i|+|\gamma_i|\right)+\|F\|_{L^{\infty}}+\|\nabla F\|_{L^{\infty}}\right).
\end{align*}
Thus, we conclude the proof of \eqref{est:Ri-tildeRi_VAi} by using \eqref{est:pi}, \eqref{est:F:point} and \eqref{exp:Ri_H1}.

Finally, the proof of \eqref{est:RipiNi_refined} follows as the one of \eqref{est:RipiNi} by using Remark \ref{est:tildeR_NR_refined} instead of Lemma \ref{est:tildeR_NR}.
\end{proof}
\smallskip
\noindent
\emph{Estimate for $T$.} We deduce from the definition of $T$ in \eqref{def:TA} that 
\begin{equation*} 
\|T\|_{H^1} \lesssim \|\partial_tV_A\|_{H^1}+\sum_{i=1}^2\left(|\alpha_i(z)|\|\partial_xR_i-\partial_x\Rti\|_{H^1}+|\gamma_i(z)|\|\Lambda R_i-\Lambda \Rti\|_{H^1} \right) .
\end{equation*}
Therefore, we conclude the proof of \eqref{est:T} by combining this estimate with \eqref{est:dVAdt}, \eqref{Est:piqi} and \eqref{exp:MRi_H1}.

\smallskip
\noindent
\emph{Estimate for $\partial_t S$.}
The time derivative of $S$ decomposed as in \eqref{decomp:S} gives
\begin{align*}
    \partial_t S & = \partial_t \Sigma + \partial_t \Psi_\Sigma + \partial_t \Psi_S.
\end{align*}
For the first term, from \eqref{defi:Sigma1}, \eqref{defi:p_q}, \eqref{est:pi} and \eqref{est:F:point}-\eqref{est:dtF:point}, it holds
\begin{align*}
    \| \partial_t \Sigma_1 \|_{H^1} \lesssim \left\vert \dot{z} \partial_z \pAone (z) \right\vert + \left\vert  \dot{z}_1 \right\vert \left( \left\vert \pAone \right\vert + \left\| F \right\|_{L^\infty} \right) + \left\| \partial_t F \right\|_{L^\infty} \lesssim \left( \vert \dot{z}_1 \vert + \vert \dot{z}_2 \vert \right) z^{-\frac12} e^{-z}.
\end{align*}
The same estimate holds for $\partial_t \Sigma_2$. Proceeding similarly for $\Psi_\Sigma$ and $\Psi_S$, we get
\begin{align*}
    \left\| \partial_t \Psi_{\Sigma} \right\|_{H^1} & \lesssim \sum_{i=1}^2 \left( \vert \dot{\mu}_i \vert + \vert \dot{\omega}_i \vert \right) e^{-\frac{15}{16}z} + \left( \vert \dot{z}_1 \vert + \vert \dot{z}_2 \vert \right) \sum_{i=1}^{2} \left( \vert \mu_i \vert + \vert \omega_i \vert \right) e^{-\frac{15}{16} z}, \\
    \left\| \partial_t \Psi_S \right\|_{H^1} & \lesssim \left( \vert \dot{z}_1 \vert + \vert \dot{z}_2 \vert \right) \left( e^{-\frac{31}{16}z} + e^{-\frac{15}{16} z} \sum_{i=1}^2 \left( \mu_i^2 + \omega_i^2 \right) \right) + e^{-\frac{15}{16} z} \sum_{i=1}^2 \left( \vert \dot{\mu}_i \vert+ \vert \dot{\omega}_i \vert\right)\sum_{i=1}^2 \left( \vert \mu_i \vert+ \vert \omega_i \vert\right).
\end{align*}
Gathering the three previous inequalities proves \eqref{est:dt_S}.

\smallskip
\noindent
\emph{Estimate for $\langle \Sigma, \MRti \rangle$.} We focus on the term $\langle \Sigma, \partial_x \Rttwo \rangle$. By using the definition of $\Sigma=\Sigma_1+\Sigma_2$ in \eqref{defi:Sigma1}-\eqref{defi:Sigma1} and the orthogonality relation \eqref{ortho:Sigmai}, we have
\begin{align*}
\left|\langle \Sigma, \partial_x \Rttwo \rangle\right|=\left|\langle \Sigma_1, \partial_x \Rttwo \rangle\right|&\le 2\left|\langle \Rtone \pAone \cdot \NRtone,  \partial_x \Rttwo \rangle\right| + 2\left| \langle \Rtone F_{}, \partial_x \Rttwo \rangle\right| \\
& \lesssim \left(|\alpha_1(z)|+|\gamma_1(z)|+\|F\|_{L^{\infty}} \right) \|\Rtone \partial_x\Rttwo\|_{L^1} .
\end{align*}
Hence, we conclude the proof of \eqref{est:ortho_Sigma} by combining this estimate with \eqref{Est:piqi}, \eqref{est:F:point} and \eqref{est:R_1R_2}.

\smallskip
\noindent
\emph{Estimate for $\langle S, \MRi \rangle$.}
The proof of \eqref{est:ortho_S} follows by gathering \eqref{est:ortho_Sigma} with \eqref{est:Sigma}, \eqref{est:PsiSigma}, \eqref{est:PsiS} and \eqref{exp:MRi_H1}.


\section{Decomposition of a solution close to the approximate solution} \label{sec:est_eps}

Consider a solution $v$ of \eqref{ZK:sym} that is close to the approximate solution $V$ constructed in Section \ref{sec:construction_V} on a certain time interval $I \subset \mathbb R$. Recall that is $V$ is on the form $V=R_1+R_2+V_A$ (see \eqref{def:V}), where $R_1$ and $R_2$ are defined in \eqref{def:R_i}. Thus, $V$ depends on time through the set of geometric parameters $\Gamma=(\mu_1,\mu_2,z_1,z_2,\omega_1,\omega_2)$, \textit{i.e.} $V(t,\bx)=V(\bx ;\Gamma(t))$. 

In all this section, we assume that the geometric parameters satisfy the rough estimates in Assumption \ref{hyp:coeff} and that $V_A(t,\bx)=V_A(\bx;\Gamma(t))$ satisfies \eqref{est:VA:point}, \eqref{est:VA:H2}, \eqref{est:dVAdt}, \eqref{est:V-Ri}. We also assume that the error $\E_V$ decomposes as in \eqref{eq:error_R1_R2_VA}, where $\mAi$ is defined in \eqref{defi:mAi}, and that the parameters $\alpha_i$, $\beta_i$ and $\gamma_i$ satisfy \eqref{est:pi}. 

We begin this section by adjusting the geometrical parameters to ensure some important orthogonality relations for the difference $\epsilon$ between the solution $v$ and the approximate solution $V$ with respect to some natural directions. We then state and prove some estimates for the dynamical systems satisfied by the geometric parameters and for the $H^1$ norm of the difference $\epsilon$.

\subsection{Modulation close to the approximate solution}\label{sec:modulation}

For any $Z^0>1$, $\mu^0>0$ and $\eta>0$, we define the set of parameters $\Xi_{Z^0,\mu^0}$ and the tube $\mathcal{U}_{Z^{0}, \mu^0,\eta}$ around the approximate solution $V(\cdot;\Gamma)$ by
\begin{equation*}
    \Xi_{Z^0, \mu^0} := \left\{ \begin{aligned} & \Gamma =(z_1,z_2,\omega_1,\omega_2, \mu_1, \mu_2) \in \mathbb{R}^6 \, : \,  \vert z_1 - z_2 \vert >  Z^0;  \\ &  \vert \mu_1 \vert < \mu^0; \ \vert \mu_2 \vert < \mu^0; \ \vert \omega_1 \vert < \mu^0; \ \vert \omega_2 \vert < \mu^0 \end{aligned} \right\}, 
\end{equation*}

\begin{equation*}
    \mathcal{U}_{Z^{0}, \mu^0,\sigma}:= \left\{ w \in H^1(\mathbb R^2) \, : \, \inf_{\Gamma \in \Xi_{Z^0, \mu^0}} \left\|w -V(\Gamma) \right\|_{H^1} < \sigma  \right\}.
\end{equation*}

\begin{lemm}\label{propo:modulation}
There exists $C>0$, $\sigma^\star>0$, $Z_1^\star>Z_0^\star$ and $\nu_1^\star\in (0, \nu_0^\star)$ such that the following holds. For any $\sigma < \sigma^\star$, $Z^0>Z_1^\star$ and $\mu^0 <\nu_1^\star$, let $w \in C(I:H^1(\mathbb R^2))$ be a solution to \eqref{ZK:sym} such that $w(t) \in \mathcal{U}_{Z^0,\mu^0, \sigma}$ for all $t \in I$. Then there exists a unique $\mathcal{C}^1$ function $\Gamma(t)= \left( z_1(t),z_2(t),\omega_1(t), \omega_2(t), \mu_1(t), \mu_2(t)\right)$ from $I$ to $\Xi_{Z^0,\mu^0}$ such that by defining 
\begin{equation} \label{def:epsilon}
\epsilon(t) := w(t) -V(\Gamma(t)) ,
\end{equation}
we have for all $t \in I$ and for $i=1,2$,
\begin{align}
    & \int \epsilon(t) \partial_x R_i(t) = \int  \epsilon(t) \partial_y R_i(t) = \int \epsilon(t) R_i(t) =0, \label{eps:ortho} \\
    & \vert z_1(t)-z_2(t) \vert > Z^0 -C \sigma, \quad  \vert \omega_i(t) \vert < \mu^0+C\sigma, \quad \vert \mu_i(t) \vert < \mu^0 + C\sigma, \label{est:mod:param} \\
    &  \left\| \epsilon(t) \right\|_{H^1} \leq C \sigma. \label{est:mod:eps} 
\end{align}
Moreover, $\epsilon$ satisfies on $I$ 
\begin{equation} \label{eq:epsilon}
\partial_t\epsilon+\partial_x(\Delta \epsilon-\epsilon +(V+\epsilon)^2-V^2)=-\E_V ,
\end{equation}
where $\E_V$ is defined in \eqref{def:EV}. 
\end{lemm}

\begin{rema}\label{rema:modulation_pointwise}
If we assume moreover that for some $\Gamma^0=(z_1^0,z_2^0,\omega_1^0,\omega_2^0,\mu_1^0,\mu_2^0) \in \Xi_{Z^0, \mu^0}$ and for some $t_0 \in I$, we have $\| w(t_0)-V(\Gamma^0)\|_{H^1} < \sigma$, then in addition to \eqref{eps:ortho}-\eqref{est:mod:eps}, we have for $i=1,2$, that 
\begin{equation} \label{est:mod:param_refined} 
\sum_{i=1}^2\left( \vert z_i(t_0)-z_i^0 \vert +\vert \mu_i(t_0)-\mu_i^0 \vert + \vert \omega_i(t_0) -\omega_i^0 \vert \right) \le C \sigma .
\end{equation}
\end{rema}

\begin{proof} The proof of the existence, uniqueness and continuity of $\Gamma$ relies on classical arguments based on a qualitative version of the implicit function theorem for functions independent of time (see for example Lemma 3.1 in \cite{MM11} or Proposition 2.2 in \cite{Eyc22} and the references therein). 

More precisely, if we define 
\begin{equation*}
\Phi:\mathcal{U}_{Z^{0},\mu^0,\sigma} \times \Xi_{Z^0, \mu^0} \to \mathbb R^6, \ (w,\Gamma) \mapsto \left( \int \epsilon \partial_x R_1,\int  \epsilon \partial_y R_1,\int \epsilon R_1, \int \epsilon \partial_x R_2,\int  \epsilon \partial_y R_2,\int \epsilon R_2 \right),
\end{equation*}
the main point consists in verifying that $d_\Gamma \Phi (w,\Gamma)$ is invertible. This is ensured by decomposing $d_\Gamma \Phi (w,\Gamma)=D+A$, where 
\begin{align*}
    D=\text{diag}\left(-\int (\partial_xQ)^2,-\int (\partial_yQ)^2,\int Q \Lambda Q,-\int (\partial_xQ)^2,-\int (\partial_yQ)^2,\int Q \Lambda Q  \right)
\end{align*}
and by using \eqref{eq:id2_Q} and $\|A\|_{L^{\infty}}\le e^{-\frac12 Z^0}+\mu^0$ thanks to \eqref{est:VA:point}, Lemma \ref{lemm:decompo_rescaled_Q} and Lemma \ref{lemm:fg}. Thus, it follows by taking $Z^0$ large enough, and $\mu^0$ and $\sigma$ small enough that there exist a unique $C^1$ function $\Gamma:\mathcal{U}_{Z^{0},\mu^0,\sigma} \to \Xi_{Z^0, \mu^0}$ such that $\Phi(w,\Gamma(w))=0$, for any $w \in \mathcal{U}_{Z^{0},\mu^0,\sigma}$. 

Now let $w \in C(I:H^1(\mathbb R^2))$ be a solution to \eqref{ZK:sym} such that $w(t) \in \mathcal{U}_{Z^0,\mu^0, \sigma}$ for all $t \in I$. By abuse of notation, we define $\Gamma(t):=\Gamma(v(t))$, so that $\Gamma \in C^0(I)$. The fact that $\Gamma$ is of class $C^1$ follows then from the Cauchy-Lipschitz theorem (see for example Proposition 2.2 in \cite{Eyc22}). Moreover, \eqref{eq:epsilon} is deduced from \eqref{ZK:sym} and \eqref{def:EV}. 

Finally, Remark \ref{rema:modulation_pointwise} follows directly from the uniqueness of the function $\Gamma$ in the neighborhood of $V(\Gamma^0)$. 
\end{proof}

\subsection{Control of the time derivative of the geometric parameters} \label{Subsec:mod:est}
We derive estimates for the differential system satisfied by the geometric parameters $\Gamma=(\mu_1,\mu_2,z_1,z_2,\omega_1,\omega_2)$. 

\begin{propo}\label{propo:dynamical_system}
We have, for all $t \in I$, 
\begin{align}
    & \sum_{i=1}^2 \left(\left\vert -\dot{z}_i +\mu_i + \alpha_i \right\vert + \left\vert \dot{\omega}_i - \beta_i \right\vert \right)  \lesssim \| \partial_xS \|_{L^2} + \| T \|_{L^2} + \| \epsilon \|_{L^2}, \label{eq:ortho_translation}\\
    & \sum_{i=1}^2 \left\vert \dot{\mu}_i + \gamma_i 
    \right\vert \lesssim  \sum_{i=1}^2\left\vert \int S \partial_x R_i \right\vert + \| T \|_{L^2} + \left( e^{-\frac{15}{16} z} + \| \epsilon \|_{L^2} \right)\| \epsilon\|_{L^2}. \label{eq:ortho_velocity}
\end{align}
\end{propo}

\begin{rema}
Estimates \eqref{eq:ortho_translation} and \eqref{eq:ortho_velocity} still hold if one replace $\|T\|_{L^2}$ by $\int |T| \left(|R_1|+|R_2| \right)$ and $\|\partial_x S\|_{L^2}$ by $\int |\partial_xS| \left(|R_1|+|R_2| \right)$ on the right-hand side. 
\end{rema}

\begin{rema} We deduce from Proposition \ref{propo:dynamical_system} some rough bounds on the time derivative of the geometrical parameters $z_i$, $\omega_i$, $\gamma_i$. More precisely, it holds for all $t \in I$, 
 \begin{align}
    & \sum_{i=1}^2 \left(\left\vert \dot{z}_i \right\vert + \left\vert \dot{\omega}_i \right\vert \right)  \lesssim  \sum_{i=1}^2\vert \mu_i \vert +\frac{e^{-z}}{\sqrt{z}}+ \| \partial_xS \|_{L^2} + \| T \|_{L^2} + \| \epsilon \|_{L^2}, \label{est:dot_zi_omegai}\\
    & \sum_{i=1}^2 \left\vert \dot{\mu}_i 
    \right\vert \lesssim  \frac{e^{-z}}{\sqrt{z}}+\sum_{i=1}^2\left\vert \int S \partial_x R_i \right\vert + \| T \|_{L^2} + \left( e^{-\frac{15}{16} z} + \| \epsilon \|_{L^2}   \right)\| \epsilon\|_{L^2}. \label{est:dot_mui}
\end{align}
\end{rema}

\begin{proof}
We first claim that 
\begin{align}
    & \sum_{i=1}^2 \left\vert -\dot{z}_i + \mu_i + \alpha_i \right\vert + \sum_{i=1}^2 \left\vert -\dot{\omega}_i + \beta_i \right\vert \nonumber \\
    & \quad \quad \lesssim e^{-\frac{15}{16} z}\sum_{i=1}^2\left\vert \mAi \right\vert + \left(1+\sum_{i=1} \left( \vert \dot{\mu}_i \vert +\vert \dot{z}_i \vert  + \vert \dot{\omega}_i \vert \right)\right)\| \epsilon \|_{L^2} + \| \partial_xS \|_{L^2} + \| T \|_{L^2}  ; \label{mod:trans.1} \\ 
    & \sum_{i=1}^2\left\vert \dot{\mu}_i + \gamma_i \right\vert \lesssim e^{-\frac{15}{16}z}\sum_{i=1}^2\left\vert \mAi \right\vert +\left( \sum_{i=1}^2\vert \dot{\mu}_i \vert + e^{-\frac{15}{16} z}+\| \epsilon \|_{L^2} \right) \| \epsilon \|_{L^2}   +\sum_{i=1}^2\left\vert \int S \partial_x R_i \right\vert + \| T \|_{L^2}. \label{mod:dil.1}
\end{align}

To prove \eqref{mod:trans.1}, we compute the time derivative of the orthogonality relations in \eqref{eps:ortho}. Let us detail the computations for the first orthogonality relation with $i=1$. By using \eqref{eq:epsilon} and \eqref{eq:error_R1_R2_VA}, we infer that
\begin{align*}
    0   
    =  \int \epsilon \frac{d}{dt} \partial_x R_1 - \int \partial_x \left( \Delta \epsilon -\epsilon + 2V \epsilon + \epsilon^2 \right) \partial_x R_1 - \int \sum_{i=1}^2 \mAi \cdot \MRi \partial_x R_1 - \int \left( \partial_x S + T \right) \partial_x R_1 .
\end{align*}
We handle each terms on the right-hand side of the above identity separately. First, we have
\begin{align*}
    \left\vert \int \epsilon \frac{d}{dt} \partial_x R_1 \right\vert \lesssim  \left( \vert \dot{\mu}_1 \vert +\vert \dot{z}_1 \vert  + \vert \dot{\omega}_1 \vert \right) \left\| \epsilon \right\|_{L^2}. 
\end{align*}
By integrating by parts and using $\| V \|_{L^\infty} + \| \partial_x R_1 \|_{L^\infty} \lesssim 1$, we obtain
\begin{align*}
    \left\vert \int \partial_x \left( \Delta \epsilon - \epsilon +2 V \epsilon + \epsilon^2 \right) \partial_x R_1 \right\vert \lesssim \| \epsilon \|_{L^2} + \| \epsilon \|_{L^2}^2.
\end{align*}
On the one hand, we compute by using the orthogonality relations \eqref{eq:id4_Q}
\begin{align*}
    \int  \mAone \cdot \MRone \partial_x R_1 = \left( -\dot{z}_1 + \mu_1 + \alpha_1 + \lambda_1 \right) \left(1+\mu_1 \right) \| \partial_x Q \|_{L^2}^2 .
\end{align*}
On the other hand, it follows from Lemma \ref{lemm:R_1_R_2} that 
\begin{align*}
    \left\vert \int \mAtwo \cdot \MRtwo \partial_x R_1\right\vert \lesssim \left\vert \mAtwo \right\vert e^{-\frac{15}{16} z}.
\end{align*}
Finally, the Cauchy-Schwarz inequality yields
\begin{align*}
    \left\vert \int \left( \partial_x S + T \right) \partial_x R_1 \right\vert \lesssim \| \partial_xS \|_{L^2}+\| T \|_{L^2} .
\end{align*}
Therefore, we conclude gathering these estimates, and choosing $\sigma^\star$ and $\nu_1^\star$ small enough that
\begin{align*}
    \left\vert -\dot{z}_1 + \mu_1 + \alpha_1 \right\vert \lesssim \left\vert \mAtwo \right\vert e^{-\frac{15}{16} z} + \| \partial_xS \|_{L^2} + \| T \|_{L^2} + \left(1+ \vert \dot{\mu}_1 \vert +\vert \dot{z}_1 \vert  + \vert \dot{\omega}_1 \vert \right) \| \epsilon \|_{L^2} .
\end{align*}
Similar computations with the time derivative of the orthogonality relations $\int \epsilon \partial_y R_1=\int \epsilon \partial_x R_2=\int \epsilon \partial_y R_2=0$ yield the estimates on $\dot{z}_2$, $\dot{\omega}_1$ and $\dot{\omega_2}$ in \eqref{mod:trans.1}.

Next, by taking the time derivative of the orthogonality relation $\int \epsilon R_1=0$, and by using \eqref{eq:epsilon} and \eqref{eq:error_R1_R2_VA}, we find that
\begin{align*}
    0  =  \int \epsilon \frac{d}{dt} R_1 - \int \partial_x \left( \Delta \epsilon -\epsilon + 2V \epsilon + \epsilon^2 \right)  R_1 - \int \sum_{i=1}^2 \mAi \cdot \MRi  R_1 - \int \partial_x S  R_1 - \int T  R_1 .
\end{align*}
By using  the orthogonality conditions $\int \epsilon \partial_x R_1 = \int \epsilon \partial_y R_1=0$ (see \eqref{eps:ortho}), we obtain
\begin{align*} 
\left| \int \epsilon \frac{d}{dt} R_1 \right| = \left| \dot{\mu}_1 \int \epsilon   \Lambda R_1 \right| \lesssim  \vert \dot{\mu}_ 1 \vert \|\epsilon\|_{L^2}.
\end{align*}
 Integration by parts, the second  identity \eqref{eq:L_MRti} and the orthogonality relation $\int \epsilon \partial_x R_1=0$ implies that 
 \begin{equation*} 
\int \partial_x \left( \Delta \epsilon -\epsilon + 2R_1 \epsilon  \right)R_1=-\mu_1 \int \epsilon \partial_xR_1=0 .
 \end{equation*}
 By using \eqref{est:V-Ri}, it follows that
\begin{align*}
    \left\vert \int 2 \epsilon (V-R_1) \partial_xR_1 \right\vert+\left\vert \int \epsilon^2 \partial_x R_1 \right\vert \lesssim e^{-\frac{15}{16} z} \| \epsilon \|_{L^2}+\|\epsilon \|_{L^2}^2.
\end{align*}
On the one hand, the orthogonality relations \eqref{eq:id4_Q} imply
\begin{align*}
    \int \mAone \cdot \MRone R_1 = \left( \dot{\mu}_1 +\gamma_1 \right) \int \Lambda R_1 R_1 = \left( \dot{\mu}_1 + \gamma_1 \right)(1+\mu_1) \int Q \Lambda Q.
\end{align*}
On the other hand, it follows from Lemma \ref{lemm:R_1_R_2} that 
\begin{align*}
    \left\vert \int \mAtwo \cdot \MRtwo R_1 \right\vert \lesssim \vert \mAtwo \vert e^{-\frac{15}{16} z}.
\end{align*}
Gathering the previous inequalities and arguing similarly with the time derivative of the orthogonality relation $\int \epsilon R_2=0$ yield \eqref{mod:dil.1}.

Next, we deduce from \eqref{mod:trans.1}-\eqref{mod:dil.1} and the estimates in Proposition \ref{theo:V_A} the rough bound $\sum_{i=1} \left( \vert \dot{\mu}_i \vert +\vert \dot{z}_i \vert  + \vert \dot{\omega}_i \vert \right) \lesssim 1$. By bootstraping this estimate into \eqref{mod:trans.1}-\eqref{mod:dil.1} and taking $Z_1^\star$ large enough, we deduce that 
\begin{equation} \label{est:mod.1}
\sum_{i=1}^2\left\vert \mAi \right\vert \lesssim \|\epsilon\|_{L^2} +\|\partial_xS\|_{L^2}+\|T\|_{L^2} ,
\end{equation}
which implies \eqref{eq:ortho_translation}. 

Finally, we obtain that $\sum_{i=1}^2\vert \dot{\mu}_i \vert \lesssim e^{-\frac78 z}$, by combining the bounds \eqref{mod:dil.1} and \eqref{est:mod.1} with \eqref{est:Sigma}-\eqref{est:ortho_Sigma}. Therefore, we conclude the proof of \eqref{eq:ortho_velocity} by reinjecting this estimate together with \eqref{est:mod.1} into \eqref{mod:dil.1}. 
\end{proof}

\subsection{Energy estimates}\label{sec:energy_estimates}

In order to study the evolution of $\epsilon$, we follow the approach in \cite{MM11} based on monotonicity formulas for localized portions of the mass and the energy. To this end, we use adequate weight functions $\psi_+$ and $\psi_{-,m}$ and $\psi_{-,e}$ defined as follows. For a fixed $0<\rho<\frac1{32}$, we set 
\begin{align} \label{def:psi}
    \psi(x) := \frac{2}{\pi} \arctan \left( e^{8\rho x}\right),
\end{align}
so that $\displaystyle \lim_{x \to -\infty} \psi(x)=0$, $\displaystyle \lim_{x \to+\infty}\psi(x)=1$, and , for all $x \in \mathbb R$,
\begin{align}
&\psi'(x)=\frac{8 \rho}{\pi \cosh(8\rho x)}>0, \quad 1-\psi(x)=\psi(-x), \label{prop:psi.1} \\
&|\psi''(x)| \le 8\rho \psi'(x), \quad |\psi'''(x)| \le (8\rho)^2 \psi'(x). \label{prop:psi.2}
\end{align}

We define the three weight functions
\begin{align}
    & \psi_+(t,x) := \mu_1(t) \psi(x) + \mu_2(t) \left( 1-\psi (x)\right) \label{defi:psi_+}\\
    & \psi_{-,e}(t,x):= \frac{\psi(x)}{\left(1+ \mu_1(t)\right)^2} + \frac{1-\psi(x)}{\left(1+\mu_2(t) \right)^2},  \label{defi:psi_-e} \\
    & \psi_{-,m}(t,x) := \frac{\mu_1(t)}{\left(1+ \mu_1(t)\right)^2}\psi(x) + \frac{\mu_2(t)}{\left(1+\mu_2(t) \right)^2}\left(1-\psi(x)\right) \label{defi:psi_-m}
\end{align}

The following properties of the weight functions are deduced from \eqref{def:psi}-\eqref{prop:psi.2}.
\begin{claim}
\begin{itemize} 
\item[(i)] Assume $t \in I$ is such that $\mu_1(t) \ge \mu_2(t)$. Then we have for all $x \in \mathbb R$, 
\begin{equation} \label{prop:psi+}
\partial_x\psi_{+}(t,x)=\mu(t) \psi'(x) \ge 0; \quad |\partial_x^2\psi_+(t,x)| \le \frac14 \partial_x\psi_+(t,x); \quad \vert \partial_x^3\psi_+(t,x) \vert \le \frac1{16} \partial_x\psi_+(t,x) .
\end{equation}
\item[(ii)] Assume $t \in I$ is such that $\mu_2(t) \ge \mu_1(t)$. Then we have for all $x \in \mathbb R$,
\begin{align} 
&\partial_x\psi_{-,e}(t,x)=-\mu(t)\frac{2+\mu_1(t)+\mu_2(t)}{(1+\mu_1(t))^2(1+\mu_2(t))^2} \psi'(x) \ge 0,  \label{prop:psi-e.1} 
\\ 
& 1-\psi_{-,e}(t,x)=\mu_1(t) \frac{2+\mu_1(t)}{(1+\mu_1(t))^2}\psi(x)+\mu_2(t) \frac{2+\mu_2(t)}{(1+\mu_2(t))^2}(1-\psi(x)) , \label{prop:psi-e.3}\\
& |\partial_x^2\psi_{-,e}(t,x)| \le \frac14 \partial_x\psi_{-,e}(t,x) \quad \text{and} \quad \vert \partial_x^3\psi_{-,e}(t,x) \vert \le \frac1{16} \partial_x\psi_{-,e}(t,x).  \label{prop:psi-e.2}  \\ 
& \partial_x \psi_{-,m}(t,x)=\mu(t) \frac{1-\mu_1(t)\mu_2(t)}{(1+\mu_1(t))^2(1+\mu_2(t))^2} \psi'(x) \le 0 ; \label{prop:psi-m.1}
\end{align}
\end{itemize}
\end{claim}

We define the localized mass-energy functionals
\begin{align}\label{defi:F_+}
    \mathcal{F}_+ & := \int \left( \frac{\vert \nabla \epsilon \vert^2}{2} + \frac{\epsilon ^2}{2} - \frac{1}{3}\left( (V+\epsilon)^3 - V^3 - 3V^2 \epsilon \right) \right) + \int \frac{\epsilon^2}{2} \psi_+ - \int S \epsilon \\
    \mathcal{F}_- & := \int \left( \frac{\vert \nabla \epsilon \vert^2}{2} + \frac{\epsilon ^2}{2} - \frac{1}{3}\left( (V+\epsilon)^3 - V^3 - 3V^2 \epsilon \right) \right) \psi_{-,e} + \int \frac{\epsilon^2}{2} \psi_{-,m} - \int S \epsilon, \label{defi:F_-}
\end{align}
where $S$ is given in the decomposition of $\E_V$ in \eqref{eq:error_R1_R2_VA}. 

\begin{rema} Following \cite{MM11}, the mass terms $\int \frac{\epsilon^2}{2} \psi_+$, respectively  $\int \frac{\epsilon^2}{2} \psi_{-,m}$ are introduced in $\mathcal{F}_+$, respectively $\mathcal{F}_-$, to cancel out some bad contributions involving the time derivative $\partial_t V$ of the approximate solution $V$. In the case where $\mu_2 \ge \mu_1$, $\partial_x \psi_{-,m}$ has a bad sign (see \eqref{prop:psi-m.1}), and therefore the energy needs also to be localised to compensate some bad contributions coming from the localised mass. 
\end{rema}

\begin{rema}
Compared to \cite{MM11}, the energy functional is now defined with the addition of the term $\int S \epsilon$, which will help to cancel out the main part of the contribution coming from the error term $\E_V$ when computing the time derivative of $\mathcal{F}_+$ and $\mathcal{F}_-$. A similar idea was used by Martel and Nguyen in \cite{MN20} to study strong interactions for a $1$-dimensional system of cubic Schr\"odinger equations (see the term $\mathbf{S}$ in Section $5$ of \cite{MN20}).
\end{rema}

Next, we state some useful estimates involving the approximate solution $V$ and the weight functions, and whose proof is given in Appendix \ref{Append:Est}.  

\begin{claim}  \label{claim:Vpsi}
For $i=1,2$, the following estimates hold:
\begin{align} 
  &\|  \vert \MRi \vert \psi'\|_{L^2} \lesssim e^{-\rho z} ; \label{est:MRipsi'} \\
  &   \| \MRone (\psi -1) \|_{L^2} +  \| \MRtwo \psi\|_{L^2} \lesssim e^{-\rho z} ;\label{est:MRipsi}
  \\
  & 
   \| V  \psi' \|_{H^1}+\| V  \psi' \|_{L^{\infty}}+\| \partial_xV  \psi'\|_{L^{\infty}} \lesssim e^{-\rho z}; \label{est:V dxpsi+} \\ & \| \partial_t V + \partial_x V \psi_+ \|_{H^1} \lesssim \sum_{i=1}^2\left( \vert\dot{z}_i-\mu_i \vert+ |\dot{\mu}_i|+|\dot{\omega}_i| \right)+\left(\sum_{i=1}^2|\mu_i|\right) e^{-\rho z} +e^{-\frac{15}{16}z} \sum_{i=1}^2 |\dot{z}_i| ;\label{est:Vt-ViPsi+} \\
   & \| \partial_t V \psi_{-,e} + \partial_x V \psi_{-,m} \|_{H^1} \lesssim \sum_{i=1}^2\left( \vert\dot{z}_i-\mu_i \vert+ |\dot{\mu}_i|+|\dot{\omega}_i| \right)+e^{-\rho z} \left(\sum_{i=1}^2|\mu_i|+|\dot{z}_i|\right).\label{est:Vtpsi-e-Vpsi-m}
\end{align}
\end{claim}

Before giving some precise estimates for the time derivative of $\mathcal{F}_+$ and $\mathcal{F}_-$, we introduce a time-dependant quantity which will be useful to express these bounds:
\begin{align}\label{defi:theta}
     \Theta := \sum_{i=1}^2 \left\vert \int S \left( \mAi \cdot \MRi \right) \right\vert  + \| S \|_{L^2} \| T \|_{L^2} .
\end{align}
\begin{toexclude}
Before giving some precise estimates for the time derivative of $\mathcal{F}_+$ and $\mathcal{F}_-$, we introduce several time-dependant quantities which will be useful to express these bounds:
\begin{align*}
    & \Theta^0 := \sum_{i=1}^2 \left\vert \int S \left( \mAi \cdot \MRi \right) \right\vert  + \| S \|_{L^2} \| T \|_{L^2};\\
    & \Theta^1_+ :=  \sum_{i=1}^2 \left\vert \mu_i \left( \dot{\mu}_i +\gamma_i \right) \right\vert + e^{-\frac78 z} \sum_{i=1}^2 \vert \mAi \vert  +  \left( \sum_{i=1}^2|\mu_i|\right) \left( e^{-\rho z} \sum_{i=1}^2 \vert \mAi \vert +\left\| S \right\|_{L^2} \right) + \| T \|_{H^1}+ \| \partial_t S \|_{L^2} ;  \\
    & \Theta^2_+ :=  \sum_{i=1}^2 \left(\vert \mAi \vert +\vert \dot{\mu}_i \vert + \vert \dot{\omega}_i \vert + \vert \dot{z}_i-\mu_i \vert \right) +\left(\sum_{i=1}^2|\mu_i|\right) e^{-\rho z} +\|\partial_tV_A\|_{L^2}; 
    \\
    & \Theta^1_- :=  \sum_{i=1}^2 \left\vert \mu_i \left( \dot{\mu}_i +\gamma_i \right) \right\vert + e^{-\frac78 z} \sum_{i=1}^2 \vert \mAi \vert  +  \left( \sum_{i=1}^2|\mu_i|\right) \left( e^{-\rho z} \sum_{i=1}^2 \vert \mAi \vert +\left\| S \right\|_{H^2} \right) + \| T \|_{H^1}+ \| \partial_t S \|_{L^2} ;  \\
    & \Theta^2_- :=  \sum_{i=1}^2 \left(\vert \mAi \vert +\vert \dot{\mu}_i \vert + \vert \dot{\omega}_i \vert +\vert \dot{z}_i-\mu_i \vert \right) +\left(\sum_{i=1}^2|\mu_i|+|\dot{z}_i|\right) e^{-\rho z} +\|\partial_tV_A\|_{L^2} .
\end{align*}
Then 
\begin{align}
   &  \| \epsilon (t) \|_{H^1}^2 \lesssim \mathcal{F}_-(t) +  \| S(t) \|_{H^1}^2, \label{coercivity_F_-}, \\
    & \frac{d}{dt} \mathcal{F}_-(t) \lesssim \Theta^0(t) + \Theta^1_-(t)\| \epsilon(t) \|_{H^1} + \Theta^2_-(t) \| \epsilon(t) \|_{H^1}^2. \label{eq:estimate_F_-}
\end{align}
\end{toexclude}

\begin{propo}\label{propo:functionals}
The following estimates hold true. 
\begin{itemize}
    \item[(i)] \emph{Coercivity}. For all $t \in I$, 
    \begin{equation}
     \| \epsilon  \|_{H^1}^2 \lesssim \mathcal{F}_+ +  \| S \|_{H^1}^2 \quad  \text{and} \quad \| \epsilon  \|_{H^1}^2 \lesssim \mathcal{F}_- +  \| S \|_{H^1}^2 
     \label{coercivity_F}
     \end{equation}
    \item[(ii)] If $t \in I$ is such that $\mu_1(t) \geq \mu_2(t)$, then
\begin{equation} \label{eq:estimate_F_+}
\begin{split}
    \frac{d}{dt} \mathcal{F}_+ &\lesssim \Theta + \|\epsilon\|_{H^1}\left( \left( \sum_{i=1}^2 \vert \mu_i \vert \right) \|S\|_{H^1}+e^{-\frac{15}{16} z} \|\partial_xS \|_{L^2}+ \|T\|_{H^1}+\|\partial_tS\|_{L^2} \right) \\ 
    & \quad + \|\epsilon\|_{H^1}^2 \left(  \left( \sum_{i=1}^2 \vert \mu_i \vert \right) e^{-\rho z}+ e^{-\frac{15}{16} z}+\|\partial_xS\|_{L^2} +\|\epsilon\|_{H^1}\right) .
    \end{split}
\end{equation}
    \item[(iii)] If $t \in I$ is such that $\mu_2(t) \geq \mu_1(t)$, then
\begin{equation} \label{eq:estimate_F_-}
\begin{split}
    \frac{d}{dt} \mathcal{F}_- &\lesssim \Theta + \|\epsilon\|_{H^1}\left( \left( \sum_{i=1}^2 \vert \mu_i \vert \right)\|S\|_{H^2}+e^{-\frac{15}{16} z}\|\partial_xS\|_{L^2} +\|T\|_{H^1}+\|\partial_tS\|_{L^2} \right) \\ 
    & \quad + \|\epsilon\|_{H^1}^2 \left(  \left( \sum_{i=1}^2 \vert \mu_i \vert \right) e^{-\rho z}+ e^{-\frac{15}{16} z}+ \|\partial_xS\|_{L^2}+ \|\epsilon\|_{H^1}\right) .
    \end{split}
\end{equation}
\end{itemize}
\end{propo}

\begin{proof} 
\smallskip
\noindent \emph{Proof of estimate \eqref{coercivity_F}.}  First, for $\mathcal{F}_-$, we use a localization argument (See Appendix A of \cite{MM02} and Appendix A of \cite{Val21}) to obtain almost orthogonality relations. Next, around each $R_i$, the operator in the first integral of $\mathcal{F}_+$ and $\mathcal{F}_-$ essentially behaves like the linearized operator $L_{\tilde{R}_i}$. From the coercivity property of $L$ under suitable orthogonality relations (see Proposition \ref{theo:L} (iv)), we obtain the existence of a constant $\lambda>0$ such that, for any $\epsilon$ satisfying \eqref{eps:ortho}, it holds 
\begin{align*}
    \lambda \| \epsilon \|_{H^1}^2 \leq  \int \frac{ \vert \nabla \epsilon \vert^2 }{2}+ \frac{\epsilon^2}{2} - R_i \epsilon^2.
\end{align*}
Notice finally that $\vert \psi_+ \vert + \vert \psi_{-,e} \vert \lesssim \vert \mu_1 \vert + \vert \mu_2 \vert$ by \eqref{defi:psi_+}. Decreasing $\lambda$ if necessary, estimates \eqref{coercivity_F} follow by choosing $\sigma^\star \gtrsim \| \epsilon \|_{H^1}$ and $\nu_1^\star > \vert \mu_1 \vert + \vert \mu_2 \vert$ small enough and using Young's inequality.

\smallskip
\noindent \emph{Proof of estimate \eqref{eq:estimate_F_+}}
We decompose the time derivative of $\mathcal{F}_+$ as
\begin{align*}
    \frac{d}{dt} \mathcal{F}_+
        & = f_1^+ +f_2^+ + f_3^+,
\end{align*}
where
\begin{align*}
    f_1^+ &:= \int \partial_t \epsilon \left( - \Delta \epsilon + \epsilon -(V+\epsilon)^2 + V^2 \right) - \int S \partial_t \epsilon , \\ 
    f_2^+ &:= -\int \partial_t V \epsilon^2 + \int (\partial_t \epsilon) \epsilon \psi_+, \\
    f_3^+ &:= \int \frac{\epsilon^2}{2} \partial_t \psi_+ -\int \partial_t S \epsilon .
\end{align*}

\emph{Estimate for $f_1^+$.} We claim that 
\begin{equation} \label{est:f1^+}
\left\vert f_{1}^+ \right\vert \lesssim \Theta+\|\epsilon\|_{H^1} \left( \sum_{i=1}^2 \left\vert \mu_i (\Dot{\mu}_i +\gamma_i )\right\vert+e^{-\frac{15}{16} z}\sum_{i=1}^2 \vert \mAi \vert +\|T\|_{H^1}\right)+ \|\epsilon\|_{H^1}^2 \sum_{i=1}^2 \vert \mAi \vert .
\end{equation}

By the equation satisfied by $\epsilon$ in \eqref{eq:epsilon}, we have
\begin{align*}
   f_1^+ 
        & = -  \sum_{i=1}^2 \int \left( -\Delta +1 -2R_i \right) \left( \mAi \cdot \MRi \right) \epsilon + \sum_{i=1}^2 \int \mAi \cdot \MRi \left( 2(V -R_i) \epsilon  + \epsilon^2\right) \\
        & \quad - \int  T \left( -\Delta \epsilon + \epsilon -(V+\epsilon)^2 + V^2 \right) + \int S \left( \sum_{i=1}^2 \mAi \cdot \MRi + T \right)\\ 
        &=:f^+_{1,1}+f^+_{1,2}+f^+_{1,3}+f^+_{1,4}.
\end{align*}
It follows from the second identity in \eqref{eq:L_MRti} and the orthogonality relations \eqref{eps:ortho} that
\begin{align*}
    |f_{1,1}^+| \lesssim \left\vert \sum_{i=1}^2 \int \left( -\Delta + 1 -2 R_i \right) \left( \mAi \cdot \MRi \right) \epsilon \right\vert \lesssim \|\epsilon\|_{L^2}\sum_{i=1}^2 \left\vert \mu_i (\Dot{\mu}_i +\gamma_i )\right\vert.
\end{align*}
Next, we use \eqref{est:V-Ri} to deduce that
\begin{align*}
    \left\vert f_{1,2}^+ \right\vert \lesssim \sum_{i=1}^2 \vert \mAi \vert \left( e^{-\frac{15}{16} z}\| \epsilon \|_{L^2}+ \| \epsilon \|_{L^2}^2 \right) .
\end{align*}
By using the H\"older inequality and the Sobolev embedding $H^1(\mathbb R^2) \hookrightarrow L^3(\mathbb R^2)$ and recalling that $\| V \|_{L^\infty}\lesssim 1$, we find 
\begin{align*}
    \left\vert f_{1,3}^+ \right\vert \lesssim \| T \|_{H^1} \| \epsilon \|_{H^1} + \| T \|_{L^3} \| \epsilon \|_{L^3}^2 \lesssim \| T \|_{H^1} \left(\| \epsilon \|_{H^1}+\|\epsilon\|_{H^1}^2\right) \lesssim  \| T \|_{H^1} \| \epsilon \|_{H^1}.
\end{align*}
Finally, the Cauchy-Schwarz inequality yields
\begin{equation*}
 \left\vert f_{1,4}^+ \right\vert \lesssim \sum_{i=1}^2 \left\vert \int S \left( \mAi \cdot \MRi \right) \right\vert  + \| S \|_{L^2} \| T \|_{L^2} .
 \end{equation*}
 The proof of \eqref{est:f1^+} follows gathering these estimates and the definition of $\Theta$ in \eqref{defi:theta}.

\smallskip
\noindent \emph{Estimate for $f_2^+$}. We claim that 
\begin{equation} \label{est:f2^+}
\begin{split}
f_{2}^+ & +\frac14 \int \left( ( \partial_x \epsilon )^2 + ( \partial_y \epsilon )^2 +\epsilon^2 \right) \partial_x \psi_+ \\
    & \lesssim  \| \epsilon \|_{L^2} \left(\vert \mu \vert \left( \vert \mAone \vert \| \MRone (\psi-1) \|_{L^2} + \vert \mAtwo \vert \| \MRtwo \psi \|_{L^2} \right) + \sum_{i=1}^2 \left\vert \mu_i \left( \dot{\mu}_i +\gamma_i \right) \right\vert \right) \\
    & \quad + \| \epsilon \|_{L^2}  \left( \| S \|_{L^2} + \| T \|_{L^2} \right) \left( \sum_{i=1}^2|\mu_i| \right) + \| \epsilon \|_{H^1}^2 \left(\| \partial_t V + \partial_x V \psi_+ \|_{H^1}+ \| V \partial_x \psi_+ \|_{H^1} \right).
\end{split}
\end{equation}

Indeed, by using the equation for $\epsilon$ and integrating by parts, we rearrange the term $f_2^+$ as
\begin{align*}
    f_2^+
        & = -\int \left( \partial_t V +\partial_x V \psi_+ \right) \epsilon^2 + \int \epsilon^2 V \partial_x \psi_+- \int \partial_x \left( \Delta \epsilon - \epsilon + \epsilon^2 \right) \epsilon \psi_+  -\int \E_V \epsilon \psi_+ \\ &=:f_{2,1}^++f_{2,2}^++f_{2,3}^++f_{2,4}^+ . 
\end{align*}
First, observe from H\"older's inequality and the Sobolev embedding that 
\begin{align*}
    \left\vert f_{2,1}^+ \right\vert + \left\vert f_{2,2}^+ \right\vert \lesssim \left( \| \partial_t V + \partial_x V \psi_+ \|_{H^1} + \| V \partial_x \psi_+ \|_{H^1} \right) \| \epsilon \|_{H^1}^2.
\end{align*}
Next, after integration by parts, we rewrite
\begin{align*}
   f_{2,3}^+ = - \frac{1}{2}\int \left( 3( \partial_x \epsilon )^2 + ( \partial_y \epsilon )^2 + \epsilon^2 \right) \partial_x \psi_+ + \frac{2}{3} \int \epsilon^3 \partial_x\psi_+ + \frac12 \int \epsilon^2 \partial_x^3\psi_+ . 
\end{align*}
By using H\"older's inequality, the Sobolev embedding $H^1(\mathbb R^2) \hookrightarrow L^3(\mathbb R^2)$ and \eqref{prop:psi+}, we observe that
\begin{equation*}
\int \epsilon^3 \partial_x\psi_+ \lesssim \|\epsilon \|_{L^3}\|\epsilon \sqrt{\partial_x\psi_+}\|_{L^3}^2 \lesssim \|\epsilon \|_{H^1} \int \left( \epsilon^2+ (\partial_x \epsilon )^2+( \partial_y \epsilon )^2  \right)\partial_x\psi_+ .
\end{equation*}
Then, we deduce by using \eqref{prop:psi+} and choosing $\|\epsilon\|_{H^1} \lesssim \sigma^\star$ small enough that 
\begin{equation*} 
f_{2,3}^++\frac14 \int \left( ( \partial_x \epsilon )^2 + ( \partial_y \epsilon )^2 + \epsilon^2 \right) \partial_x \psi_+ \leq 0 .
\end{equation*}
Finally, we deal with the last term of $f_2^+$. From the decomposition of the error $\E_V$ in \eqref{eq:error_R1_R2_VA} and the definition of $\psi_+$ in \eqref{defi:psi_+}, we have
\begin{align*}
\left\vert f_{2,4}^+ \right\vert \leq \sum_{i=1}^2 \left\vert \int \mAi \cdot \MRi \epsilon \psi_+ \right\vert + \left( \| S \|_{H^1} + \| T \|_{L^2} \right) \| \epsilon \|_{L^2} \sum_{i=1}^2 |\mu_i|.
\end{align*}
Using the orthogonality conditions \eqref{eps:ortho} and the second identity in \eqref{eq:L_MRti}, we have
\begin{align*}
    \left\vert \int \mAone \cdot \MRone \epsilon \psi_+ \right\vert
    & \leq \left\vert \int \mAone \cdot \MRone \epsilon \mu (\psi -1) \right\vert + \vert \mu_1(\dot{\mu}_1 + \gamma_1 )\vert \left\vert \int \Lambda R_1 \epsilon \right\vert \\
    & \leq \vert \mu \vert \vert \mAone \vert \| \epsilon \|_{L^2} \| \MRone (\psi-1) \|_{L^2} + \left\vert \mu_1 \right\vert \left\vert \left( \dot{\mu}_1 + \gamma_1 \right) \right\vert \| \epsilon \|_{L^2}
\end{align*}
and thus, we conclude arguing similarly for the term $\left\vert \int \mAtwo \cdot \MRtwo \epsilon \psi_+ \right\vert$ that
\begin{align*}
    \left\vert f_{2,4}^+ \right\vert & \leq \| \epsilon \|_{L^2} \left( \vert \mu \vert \left( \vert \mAone \vert \| \MRone (\psi-1) \|_{L^2} + \vert \mAtwo \vert \| \MRtwo \psi \|_{L^2} \right) + \sum_{i=1}^2 \left\vert \mu_i \left( \dot{\mu}_i +\gamma_i \right) \right\vert \right) \\
        & \quad + \| \epsilon \|_{L^2} \left( \| S \|_{L^2} + \| T \|_{L^2} \right) \sum_{i=1}^2 |\mu_i|.
\end{align*}

Hence, the proof of \eqref{est:f2^+} follows by gathering the above estimates. 

\smallskip
\noindent \emph{Estimate for $f_3^+$}. From the definition of $\psi_+$ and the Cauchy-Schwarz inequality, we infer that
\begin{align} \label{est:f3+}
\left\vert f_3^+ \right\vert \lesssim \| \epsilon \|_{L^2}^2 \sum_{i=1}^2\vert \dot{\mu}_i\vert  +\| \epsilon \|_{L^2} \| \partial_t S \|_{L^2} .
\end{align}

\smallskip
Therefore, we conclude the proof of estimate \eqref{eq:estimate_F_+} by combining \eqref{est:f1^+}, \eqref{est:f2^+}, \eqref{est:f3+} with \eqref{eq:ortho_translation}-\eqref{eq:ortho_velocity}, \eqref{est:dot_mui} and \eqref{est:MRipsi}-\eqref{est:V dxpsi+}-\eqref{est:Vt-ViPsi+}.

\medskip
\noindent \emph{Proof of estimate \eqref{eq:estimate_F_-}}. After integrating by parts, we rearrange the time derivative of $\mathcal{F}_-$ as
\begin{align*}
    \frac{d}{dt} \mathcal{F}_-
        & = f_1^-+f_{2}^- +f_3^- + f_4^- ,
\end{align*}
where
\begin{align*}
f_1^- &:= \int \partial_t \epsilon \left( - \Delta \epsilon + \epsilon -(V+\epsilon)^2 + V^2 \right) - \int S \partial_t \epsilon ; \\
    f_{2}^- & := \int \partial_t \epsilon \left( - \Delta \epsilon + \epsilon -(V+\epsilon)^2 + V^2 \right) (\psi_{-,e}-1)- \int \partial_t \epsilon \partial_x \epsilon  \partial_x \psi_{-,e}  ;\\  
    f_3^- &:= -\int \partial_t V \epsilon^2 \psi_{-,e} + \int (\partial_t \epsilon) \epsilon \psi_{-,m}; \\
    f_4^- & := \int \frac{\epsilon^2}{2} \partial_t \psi_{-,m} -\int \partial_t S \epsilon + \int \left( \frac{\vert \nabla \epsilon \vert^2}{2} + \frac{\epsilon^2}{2} - \frac{1}{3} \left( 3V \epsilon^2 + \epsilon^3\right) \right) \partial_t \psi_{-,e}.
\end{align*}

\smallskip
\noindent \emph{Estimate for $f_1^-$}. Observe that $f_1^-=f_1^+$, so that we deduce from \eqref{est:f1^+} that 
\begin{equation} \label{est:f1^-}
\left\vert f_{1}^- \right\vert \lesssim \Theta+\|\epsilon\|_{H^1} \left( \sum_{i=1}^2 \left\vert \mu_i (\Dot{\mu}_i +\gamma_i )\right\vert+e^{-\frac{15}{16} z}\sum_{i=1}^2 \vert \mAi \vert +\|T\|_{H^1}\right)+\|\epsilon\|_{H^1}^2 \sum_{i=1}^2 \vert \mAi \vert .
\end{equation}

\smallskip
\noindent \emph{Estimate for $f_2^-$}. We claim that 
\begin{equation} \label{est:f2^-}
\begin{split}
 f_{2}^-  &+\frac{3}{8} \int \left( |\partial_x \nabla \epsilon |^2+ 3(\partial_x\epsilon)^2+(\partial_y \epsilon)^2+\epsilon^2\right) \partial_x \psi_{-,e} \\ 
 & \lesssim \| \epsilon \|_{H^1}\left(\sum_{i=1}^2|\mu_i|\right) \left( \sum_{i=1}^2\left( \left\vert \mu_i \left( \dot{\mu}_i +\gamma_i \right) \right\vert+ \vert \mAi \vert e^{- \rho z} \right)+ \|\partial_xS\|_{H^1} + \| T \|_{H^1}   \right) \\ 
 & \quad +\| \epsilon \|_{H^1}^2\left(\sum_{i=1}^2|\mu_i| \right)\left(\sum_{i=1}^2 \vert \mAi \vert +e^{-\rho z}\right) .
 \end{split}
\end{equation}

Indeed, by using the equation \eqref{eq:epsilon} satisfied by $\epsilon$ and performing an integration by parts, we rewrite $f_2^-$ as 
\begin{align*}
f_2^- & = -\frac{1}{2} \int \left( -\Delta \epsilon + \epsilon -(V+ \epsilon)^2+V^2 \right)^2 \partial_x \psi_{-,e} - \int (\partial_x\epsilon)\partial_x\left( -\Delta \epsilon + \epsilon -(V+ \epsilon)^2+V^2 \right) \partial_x \psi_{-,e}\\ & \quad  - \int \left( -\Delta \epsilon + \epsilon -(V+\epsilon)^2 +V^2\right) \E_V (\psi_{-,e}-1)+\int (\partial_x\epsilon) \E_V \partial_x \psi_{-,e} \\ 
&=:f_{2,1}^-+f_{2,2}^-+f_{2,3}^-+f_{2,4}^- .
\end{align*}

First, observe from Young's inequality that 
\begin{equation*}
f_{2,1}^- \le -\frac{7}{16} \int \left( -\Delta \epsilon + \epsilon\right)^2 \partial_x \psi_{-,e}+ C\int \left(V^2\epsilon^2+\epsilon^4 \right) \partial_x\psi_{-,e}. 
\end{equation*}
On the one hand, more integrations by parts yield
\begin{align*}
-\frac{7}{16} \int \left( -\Delta \epsilon + \epsilon\right)^2 \partial_x \psi_{-,e} & =-\frac{7}{16} \int \left( (-\Delta \epsilon)^2 + 2|\nabla \epsilon|^2+\epsilon^2\right) \partial_x \psi_{-,e}+\frac7{16} \int \epsilon^2 \partial_x^3\psi_{-,e}  .
\end{align*}
On the other hand, we deduce from H\"older's inequality, the Sobolev embedding and \eqref{prop:psi-e.2} that 
\begin{align*} 
\int V^2 \epsilon^2\partial_x\psi_{-,e} &\lesssim \|V\partial_x\psi_{-,e}\|_{H^1}\|\epsilon\|_{H^1}^2, \\ 
\int \epsilon^4\partial_x\psi_{-,e} &\lesssim \|\epsilon\|_{H^1}^2\|\epsilon \sqrt{\partial_x \psi_{-,e}}\|_{H^1}^2\lesssim \|\epsilon\|_{H^1}^2\int \left(|\nabla \epsilon|^2+\epsilon^2\right) \partial_x \psi_{-,e} .
\end{align*}
Hence, it follows by using \eqref{prop:psi-e.2} and taking $\|\epsilon\|_{H^1} \lesssim \sigma^\star$ small enough that 
\begin{equation*}
f_{2,1}^- +\frac{3}{8} \int \left( (-\Delta \epsilon)^2 + 2|\nabla \epsilon|^2+\epsilon^2\right) \partial_x \psi_{-,e} \lesssim \|V\partial_x\psi_{-,e}\|_{H^1}\|\epsilon\|_{H^1}^2 .
\end{equation*}

We argue similarly for $f_{2,2}^-$. After some integration by parts, we obtain that 
\begin{align*} 
f_{2,2}^-
    & =-\int \left( (\partial_x^2 \epsilon)^2 +(\partial_{xy}^2 \epsilon)^2 +(\partial_x\epsilon)^2\right) \partial_x \psi_{-,e}+\frac12 \int (\partial_x\epsilon)^2\partial_x^3 \psi_{-,e}+2\int \epsilon (\partial_x\epsilon)^2\partial_x \psi_{-,e} \\
    & \quad +2\int (\partial_x\epsilon) \partial_x(V \epsilon) \partial_x \psi_{-,e} .
\end{align*}
We deduce from H\"older's inequality, the Sobolev embedding and \eqref{prop:psi-e.2} that 
\begin{align*} 
& \left|\int (\partial_x\epsilon) \partial_x(V \epsilon) \partial_x \psi_{-,e} \right| \lesssim \left(\|V\partial_x\psi_{-,e}\|_{L^{\infty}}+\|\partial_xV\partial_x\psi_{-,e}\|_{L^{\infty}}\right)\|\epsilon\|_{H^1}^2, \\ 
& \left|\int \epsilon (\partial_x\epsilon)^2\partial_x \psi_{-,e} \right| \lesssim \|\epsilon\|_{H^1}\|\partial_x\epsilon \sqrt{\partial_x \psi_{-,e}}\|_{H^1}^2 \lesssim \|\epsilon\|_{H^1}\int \left( (\partial_x^2 \epsilon)^2 +(\partial_{xy}^2 \epsilon)^2 +(\partial_x\epsilon)^2\right) \partial_x \psi_{-,e} .
\end{align*}
Hence, it follows by using \eqref{prop:psi-e.2} and taking $\|\epsilon\|_{H^1} \lesssim \sigma^\star$ small enough that
\begin{equation*}
f_{2,2}^- +\frac{1}{2} \int \left( (\partial_x^2 \epsilon)^2 +(\partial_{xy}^2 \epsilon)^2 +(\partial_x\epsilon)^2\right) \partial_x \psi_{-,e} \lesssim \left(\|V\partial_x\psi_{-,e}\|_{L^{\infty}}+\|\partial_xV\partial_x\psi_{-,e}\|_{L^{\infty}}\right)\|\epsilon\|_{H^1}^2 .
\end{equation*}
Next, we deal with the term $f_{2,3}^-$. Recall the notation of $L_{R_i}$ in \eqref{def:LRi}. We infer from H\"older's inequality and the Sobolev embedding that
\begin{align*}
    \left\vert f_{2,3}^- \right\vert 
     & \lesssim \sum_{i=1}^2\left\vert \int (-\Delta \epsilon + \epsilon -2R_i\epsilon) 
     \mAi \cdot \MRi (\psi_{-,e} - 1)\right\vert  + \sum_{i=1}^2 \left\vert \int \left( V-R_i \right) \mAi \cdot \MRi \epsilon \left( \psi_{-,e}-1 \right) \right\vert \\ 
        & \quad +\left(\| \E_V \|_{H^1} \| \epsilon \|_{H^1}^2 + \| \epsilon \|_{H^1} \left( \| \partial_xS \|_{H^1} + \| T \|_{H^1} \right)\right) \left(\| \psi_{-,e} -1 \|_{L^\infty}+ \| \partial_x\psi_{-,e}\|_{L^\infty}\right).
\end{align*}
Using the orthogonality conditions \eqref{eps:ortho}, the second identity in \eqref{eq:L_MRti} and \eqref{prop:psi-e.3}, we have that 
\begin{align*} 
\left\vert \int (-\Delta \epsilon \right. & \left. + \epsilon -2R_1\epsilon) \mAone \cdot \MRone (\psi_{-,e} - 1)\right\vert  \\
& \lesssim \left(\sum_{i=1}^2|\mu_i|\right)\left\vert \int (-\Delta \epsilon + \epsilon -2R_1\epsilon) \mAone \cdot \MRone (\psi - 1)\right\vert+\mu_1^2\vert (\dot{\mu}_1 + \gamma_1 )\vert \left\vert \int \Lambda R_1 \epsilon \right\vert
\\ & \lesssim \left(\sum_{i=1}^2|\mu_i|\right)\|\epsilon\|_{H^1}\vert \mAone \vert \left\| \MRone(1-\psi) \right\|_{H^1}+\mu_1^2\vert \left( \dot{\mu}_1 +\gamma_1 \right)\vert \|\epsilon\|_{H^1} .
\end{align*}
We argue similarly for the term involving $\MRtwo$. Therefore, by using  $\|\partial_x \psi_{e,-}\|_{L^{\infty}}+\|\psi_{e,-}-1\|_{L^{\infty}} \le \sum_{i=1}^2|\mu_i|$, we conclude  that  
\begin{align*} 
|f_{2,3}^-| 
    &\lesssim \| \epsilon \|_{H^1}  \left(\sum_{i=1}^2|\mu_i| \right) \left( \vert \mAone \vert \| \MRone (\psi-1) \|_{L^2} + \vert \mAtwo \vert \| \MRtwo \psi \|_{L^2}  + \sum_{i=1}^2 \left\vert \mu_i \left( \dot{\mu}_i +\gamma_i \right) \right\vert \right) \\ 
    & \quad  +\| \epsilon \|_{H^1}  \left(\sum_{i=1}^2|\mu_i| \right) \left( \sum_{i=1}^2 \vert \mAi \vert \left\| (V-R_i) \MRi \right\|_{L^2} + \| \partial_xS \|_{H^1} + \| T \|_{H^1}+ \| \epsilon \|_{H^1} \left\| \E_V \right\|_{H^1}\right).
\end{align*}

Finally, we estimate the term $f_{2,4}^-$ by using the Cauchy-Schwarz inequality and the identity \eqref{prop:psi-e.1}. More precisely, we have 
\begin{equation*} 
|f_{2,4}^-|  \lesssim \|\epsilon\|_{H^1} \left( \sum_{i=1}^2 \vert \mu_i \vert \right)
\left( \sum_{i=1}^2\vert \mAi \vert \| \MRi \psi' \|_{L^2} +\|\partial_xS\|_{L^2}+\|T\|_{L^2}\right) .
\end{equation*}

Therefore, we conclude the proof of \eqref{est:f2^-} by combining these estimates with \eqref{est:V-Ri}, \eqref{est:MRipsi'}, \eqref{est:MRipsi} and \eqref{est:V dxpsi+}.

\smallskip
\noindent \emph{Estimate for $f_3^-$}. We claim that 
\begin{equation} \label{est:f3^-}
\begin{split}
|f_3^-| &\lesssim \frac{5}{16} \int \left( 3 (\partial_x \epsilon)^2 + (\partial_y \epsilon )^2 + \epsilon^2 \right) \partial_x \psi_{-,e} \\ & \quad 
+\| \epsilon \|_{H^1} \left(\left( \sum_{i=1}^2|\mu_i| \right) \left( \sum_{i=1}^2 e^{-\rho z}\vert \mAi \vert +\| S \|_{L^2} + \| T \|_{L^2}\right) + \sum_{i=1}^2 \left\vert \mu_i \left( \dot{\mu}_i +\gamma_i \right) \right\vert \right) \\ & 
\quad +\|\epsilon \|_{H^1}^2\left(\sum_{i=1}^2\left( \vert\dot{z}_i-\mu_i \vert+ |\dot{\mu}_i|+|\dot{\omega}_i| \right)+\left(\sum_{i=1}^2|\mu_i|+|\dot{z}_i|\right) e^{-\rho z}  \right) .
\end{split}
\end{equation}

By using the equation of $\epsilon$ in \eqref{eq:epsilon}, we decompose $f_3^-$ as
\begin{align*}
    f_3^- &= - \int \left( \partial_t V \psi_{-,e} + \partial_x V\psi_{-,m} \right) \epsilon^2+ \int V\epsilon^2 \partial_x \psi_{-,m} - \int \partial_x \left( \Delta \epsilon - \epsilon+\epsilon^2\right) \epsilon \psi_{-,m}  - \int \E_V \epsilon \psi_{-,m} \\ 
    &=: f_{3,1}^-+f_{3,2}^-+f_{3,3}^-+f_{3,4}^- .
\end{align*}
Firstly, it follows from H\"older's inequality and the Sobolev embedding that 
\begin{equation*} 
\vert f_{3,1}^- \vert +\vert f_{3,2}^- \vert\lesssim \left(\| \partial_t V \psi_{-,e} + \partial_x V \psi_{-,m} \|_{H^1}+\|V\partial_x\psi_{-,m}\|_{H^1}\right) \| \epsilon \|_{H^1}^2 .
\end{equation*}
Secondly, we rewrite after integrating by parts 
\begin{equation*} 
f_{3,3}^-=-\frac{1}{2} \int \left( 3 (\partial_x \epsilon)^2 + (\partial_y \epsilon )^2 + \epsilon^2 \right) \partial_x \psi_{-,m} + \frac12 \int \epsilon^2 \partial_x^3 \psi_{-,m} + \frac23\int \epsilon^3 \partial_x \psi_{-,m}
\end{equation*}
Note that the first term on the right-hand side of the above identity has a bad sign since $\partial_x \psi_{-,m} \le 0$. However, it follows from \eqref{prop:psi-e.1} and \eqref{prop:psi-m.1} that $|\partial_x\psi_{-,m}|+|\partial_x^3\psi_{-,m}|\le \frac{9}{16} \partial_x\psi_{-,e}$. Moreover, we deduce from H\"older's inequality and the Sobolev embedding that 
\begin{equation*}
\left| \int \epsilon^3 \partial_x \psi_{-,m} \right| \lesssim \|\epsilon\|_{H^1} \int \left(|\nabla \epsilon|^2+\epsilon^2 \right) \partial_x \psi_{-,e} .
\end{equation*}
Hence, we have, recalling that $\|\epsilon\|_{H^1} \lesssim \sigma^\star$ is chosen small enough, 
\begin{equation*} 
|f_{3,3}^-| \le \frac{5}{16} \int \left( 3 (\partial_x \epsilon)^2 + (\partial_y \epsilon )^2 + \epsilon^2 \right) \partial_x \psi_{-,e} ,
\end{equation*}
so that the contribution of $|f_{3,3}^-|$ will be absorbed by the second term on the left-hand side of \eqref{est:f2^-}. Thirdly, we bound the contribution $|f_{3,4}^-|$ arguing as with $|f_{2,4}^+| $. More precisely, we obtain that 
\begin{align*}
    \left\vert f_{3,4}^- \right\vert 
        & \leq \| \epsilon \|_{H^1} \left( \vert \mu \vert \left( \vert \mAone \vert \| \MRone (\psi-1) \|_{L^2} + \vert \mAtwo \vert \| \MRtwo \psi \|_{L^2} \right) + \sum_{i=1}^2 \left\vert \mu_i \left( \dot{\mu}_i +\gamma_i \right) \right\vert \right)\\
        & \quad \quad + \left\| \epsilon \right\|_{H^1} \left( \| S \|_{L^2} + \| T \|_{L^2} \right) \sum_{i=1}^2 |\mu_i|  .
\end{align*}
Therefore, we conclude the proof of \eqref{est:f3^-} by combining the above estimates with \eqref{est:MRipsi}, \eqref{est:V dxpsi+} and \eqref{est:Vtpsi-e-Vpsi-m}.

\smallskip
\noindent \emph{Estimate for $f_4^-$}.
We infer from H\"older's inequality, the Sobolev embedding and the definitions of $\psi_{-,e}$ and $\psi_{-,m}$ that
\begin{align} \label{est:f4^-}
\left\vert f_4^- \right\vert \lesssim \| \epsilon \|_{H^1}^2 \sum_{i=1}^2\vert \dot{\mu}_i\vert  +\| \epsilon \|_{L^2} \| \partial_t S \|_{L^2} ,
\end{align}
where we also used that $\|\epsilon\|_{H^1} \lesssim \sigma^\star$, where $\sigma^\star$ is chosen small enough. 

\medskip 
Therefore, we conclude the proof of \eqref{eq:estimate_F_-} by combining \eqref{est:f1^-}-\eqref{est:f4^-} with \eqref{eq:ortho_translation}-\eqref{eq:ortho_velocity}, and \eqref{est:dot_zi_omegai}-\eqref{est:dot_mui}.
\end{proof}


\section{Stability of the collision of two nearly equal solitary waves}\label{sec:collision}

In this section, we study the evolution of a solution close to the approximate solution $V$ constructed in Section \ref{sec:construction_V}. Before studying the stability of two solitary-waves, we introduce in the first subsection a function $Z$ which approximates the distance between the two solitary waves. Then, we state the main stability result in Subsection  \ref{sec:collision_region}, which is proved in the following subsections.

\subsection{The function \texorpdfstring{$Z$}{Z}} \label{sub:sec:Z}
We gather below several properties of the function $Z$ that are proved in Appendix \ref{app:Z}.

\begin{lemm}[Pointwise properties of $Z$]\label{lemm:pointwise_Z}
Let $0<\mu_0 \leq \left( 2 \langle \Lambda Q, Q \rangle^{-1} \| Q \|_{L^3}^3 \right)^\frac12$. There exists a unique even function $Z$ in $\mathcal{C}^2(\mathbb{R})$ satisfying
\begin{align} \label{eq:Z}
    \Ddot{Z}(t)= \frac{2}{\langle \Lambda Q, Q \rangle} \int_{\mathbb{R}^2} Q(x + Z(t),y) \partial_x (Q^2) (x,y) dxdy,
\end{align}
with the condition
\begin{align*}
    \lim_{t\rightarrow -\infty} ( Z(t), \dot{Z}(t) ) = (+\infty,-2\mu_0).
\end{align*}
Furthermore, $Z$ satisfies the following properties.
\begin{itemize}
    \item[(i)](Convexity) $Z$ is a convex function with a positive minimum at $0$.
    \item[(ii)](Asymptote) There exist a real constant $l=l(\mu_0)\in \mathbb{R}$ and some positive constants $c=c(\mu_0)$ and $K=K(\mu_0)$ such that for $t<0$
\begin{align} \label{Z_asympt}
    \left\vert Z(t) + 2\mu_0 t -l \right\vert \leq K e^{ct}.
\end{align}
\end{itemize}
\end{lemm}

\begin{proof}
The existence and uniqueness of $Z$ are proved in Corollary \ref{coro:initial_condition} and Lemma \ref{lemm:asymp_Z}. The proof of existence and the value of the asymptote is given in Lemma \ref{lemm:asymp_Z}.
\end{proof}

\begin{lemm}[Uniform estimates on $Z$]\label{lemm:uniform_Z}
There exist $\nu_2^\star>0$ and two positive constants $c$ and $C$ such that for any $0<\mu_0< \nu_2^{\star}$, $Z$ defined as in Lemma \ref{lemm:pointwise_Z}, the following holds true. 
\begin{itemize}
    \item[(i)](Behaviour at $0$) Let $Z_0:=Z(0)$, then $\mu_0 \sim Z_0^{-\frac14}e^{-\frac12Z_0}$. More precisely
    \begin{align} \label{est:mu0_Z0}
        cZ_0^{-\frac12} e^{-Z_0} \leq \mu_0^2 \leq C Z_0^{-\frac12} e^{-Z_0}.
    \end{align}
    \item[(ii)](Large time $T_1$) Let $T_1$ be the unique positive time such that
    \begin{align}\label{defi:T_1}
        Z(T_1)= \rho^{-1} Z_0,
    \end{align}
    where $0<\rho<\frac1{32}$ has been fixed in Section \ref{sec:energy_estimates}. Then
    \begin{align}\label{ineq:dot_Z(T_1)}
        \dot{Z}(T_1) \mu_0^{-1} \geq 1.
    \end{align}
\end{itemize}
\end{lemm}

\begin{proof}
The first point is proved in Lemma \ref{lemm:l_1_Z_0}. Second, from \eqref{equation_Z_dot} with $l_1=2\mu_0$, we have
\begin{align*}
    \dot{Z}(T_1)^2 = 4\mu_0^2 - \frac{4}{\langle \Lambda Q, Q \rangle} \int Q(\bx -(Z(T_1),0)) Q^2(\bx) d\bx, 
\end{align*}
so that, by using Proposition \ref{propo:approx_e^-Z}, Lemma \ref{lemm:l_1_Z_0} and $\frac{2}{\rho} >10$,
\begin{align}\label{est:dot_Z_T_1}
    \left\vert \dot{Z}(T_1)^2 - 4\mu_0^2 \right\vert \lesssim e^{-\frac{1}{\rho}Z_0} \lesssim \mu_0^{10}.
\end{align}
Hence, choosing $\nu_2^{\star}$ small enough provides the desired inequality.
\end{proof}

\subsection{Dynamics of  a solution close to a \texorpdfstring{$2$}{2}-solitary waves structure in the collision region}\label{sec:collision_region}

In this subsection, we describe the behaviour of a solution close to the approximate solution $V$ constructed in Section \ref{sec:construction_V} in the collision region. The main result of this section is stated below.

We recall that $0<\rho <\frac{1}{32}$ has been defined in Section \ref{sec:energy_estimates}. For $\mu_0>0$ small, let $Z$ as defined in Lemma \ref{lemm:pointwise_Z} and define $Z_0:=Z(0)$. Recall also from \eqref{est:mu0_Z0} that $\mu_0 \sim Z_0^{-\frac14}e^{-\frac12Z_0}$. 

\begin{toexclude}
We denote by $v(t)$ the unique solution of \eqref{ZK:sym} satisfying
\begin{align*}
    \lim_{t\rightarrow -\infty} \left\| v(t) - \sum_{i=1}^2 Q_{1+(-1)^i\mu_0} \left( \cdot - \frac{(-1)^i}{2} \left(2\mu_0 t +l\right) , \cdot + \frac{(-1)^i}{2} \omega_0 \right) \right\|_{H^1(\mathbb{R}^2)} =0,
\end{align*}
and let $T_1>0$ be such that $Z(T_1)= \rho^{-1} Z_0$ as in Lemma \ref{lemm:uniform_Z}.

\red{Peut-on citer ce theoreme avec $v(T_1)$ proche de $V(T_1)$? Ca permettrait d'avoir une proposition plus generale, que l'on pourrait appliquer a notre solution particuliere $v$, et reutiliser cette proposition dans le Theorem 1.2 pour decrire la stabilite.}
\end{toexclude}

\begin{propo}\label{propo:bootstrap_-T_T}
There exist some constants $C>0$ and $\nu_3^\star$ such that the following holds. Let $0<\mu_0<\nu_3^\star$, $\vert \omega_0 \vert < \mu_0$, $w^1\in H^1(\mathbb{R}^2)$ and $\Gamma^1=(z_1^1,z_2^1,\omega_1^1,\omega_2^1,\mu_1^1,\mu_2^1)$ satisfying 
\begin{equation}\label{propo:bootstrap_-T_T:ini}
\begin{aligned} 
    &  \left\| w^1 - V(\Gamma_1) \right\|_{H^1} + \sum_{i=1}^2 \left\vert \mu_i^1 + \frac{(-1)^i}2 \dot{Z}(-T_1) \right\vert  \leq Z_0^{-4} \mu_0^{\frac74}, \\
    & \sum_{i=1}^2\left\vert z_i^1 + \frac{(-1)^i}2 Z(-T_1) \right\vert \leq Z_0^{-3} \mu_0^{\frac34}, \quad \sum_{i=1}^2 \left\vert \omega_i^1 + \frac{(-1)^i}{2} \omega_0 \right\vert \leq Z_0^{-2} \mu_0^{\frac54}.
\end{aligned}
\end{equation}
By denoting $w$ the solution to \eqref{ZK:sym} with initial condition $w(-T_1)=w^1$, there exists $(\Gamma(t),\epsilon(t))\in \mathcal{C}^1$ such that for any $t\in[-T_1,T_1]$ 
\begin{align}
        & w(t,\bx) = V(t,\bx) + \epsilon (t,\bx), \quad  \Gamma(t) = \left( z_1(t),z_2(t),\omega_1(t),\omega_2(t), \mu_1(t), \mu_2(t) \right), \nonumber \\
        &\begin{aligned}
            & \| \epsilon(t) \|_{H^1} \leq C \mu_0^{\frac74}, && \left\vert \mu(t) - \dot{Z}(t) \right\vert \leq C Z_0 \mu_0^{\frac74}, && \left\vert \bar{\mu}(t) \right\vert \leq C \mu_0^\frac74,  \\
            & \left\vert z(t) - Z(t) \right\vert \leq C Z_0^2 \mu_0^\frac34, && \left\vert \omega(t) - \omega_0 \right\vert + \left\vert \bar{\omega} (t) \right\vert \leq C Z_0 \mu_0^{\frac54},  && \left\vert \bar{z}(t) \right\vert \leq C Z_0 \mu_0^{\frac34}.
    \end{aligned}\label{est:-T_1_T_1}
\end{align}
where $z(t)$, $\Bar{z}(t)$, $\omega(t)$, $\Bar{\omega}(t)$, $\mu(t)$ and $\Bar{\mu}(t)$ are defined in \eqref{defi:z_zbar}, \eqref{defi:omega_omegabar} and \eqref{defi:mu_mubar}.
\end{propo}

The rest of this section is dedicated to the proof of this proposition. Let us begin with an overview of the proof. 

\begin{toexclude}
On the time $(-\infty,-T_1]$, we use the qualitative orbital stability of the previous subsection to get that $V( \cdot ; \Gamma)$ is a good approximation of the two solitary waves $v$ (see Corollary \ref{coro:quantified_orbital_stability}). 
\end{toexclude}
The two solitary waves exchange their mass on the time interval $[-T_1,T_1]$. In particular, at the \lq\lq switching time\rq\rq, the solitary wave on the left becomes the smallest one and after the collision heads to the left, whereas the one on the right becomes the highest one and after the collision heads to the right. To describe this whole phenomenon in the region $[-T_1,T_1]$, we split the time interval into three regions based on a fixed time $T_2$ to define later.
\begin{itemize}
    \item In the first region $[-T_1,-T_2]$, the solitary wave on the left is higher than the one on the right. The estimates on the error terms are optimal in the sense that they cannot be improved significantly due to the error of approximation between $w$ and $V(\Gamma)$. 
    
    \item In the second region $[-T_2,T_2]$, a small interval around $0$, the smallest solitary wave becomes the highest. The estimates on the different parameters are based on the smallness of the interval.
    
    \item In the third region $[T_2,T_1]$, we deal with the two solitary waves going away one from another. We thus integrate the accumulated error term on an interval of size $T_1-T_2$, until the two solitary waves are far enough one from another to enter the orbital stability regime.
\end{itemize}

\vspace{0.5cm}

Throughout these parts, we will need our solution to satisfy Assumption \ref{hyp:coeff} and to use Lemma \ref{propo:modulation}, Lemma \ref{lemm:pointwise_Z} and Lemma \ref{lemm:uniform_Z}. Lemmata \ref{propo:H_positive_times} and \ref{propo:H_negative_times} will be used to estimate the parameters controlled by an approximate ODE system. In that context, by using the constant $C$ as defined in Lemma \ref{lemm:l_1_Z_0}, we fix the following constraint
\begin{align}\label{defi:nu_3_star}
    \nu_3^\star \leq \frac12 \min \left( \sigma^\star, \nu_1^\star, C (Z_1^\star)^{-\frac12} e^{-Z_1^\star}, \nu_2^\star, \nu_6^\star \right).
\end{align}

As long as $w(t)$ stays close to the approximation $V$, we introduce the decomposition $(\Gamma(t),\epsilon(t))$ as constructed in Lemma \ref{propo:modulation}. In particular, we have from \eqref{propo:bootstrap_-T_T:ini} at time $t=-T_1$,
\begin{align}
    & \left\| \epsilon(-T_1) \right\|_{H^1} + \left\vert \mu(-T_1) - \dot{Z}(-T_1) \right\vert + \left\vert \bar{\mu}(-T_1) \right\vert \leq Z_0^{-3} \mu_0^{\frac74} \label{eq:mu_T_1} \\
    & \left\vert z(-T_1)- Z(-T_1) \right\vert \leq Z_0^{-2} \mu_0^{\frac34}, \quad \left\vert \bar{z}(-T_1) \right\vert \leq Z_0^{-1} \mu_0^{\frac34}, \quad \left\vert \omega(-T_1) - \omega_0 \right\vert + \left\vert \bar{\omega}(-T_1) \right\vert \leq Z_0^{-1} \mu_0^{\frac54}, \label{eq:z_omega_T_1} \\
    & \mathcal{F}_-(-T_1) \leq Z_0^{-6} \mu_0^{\frac72}. \label{eq:F_-_T_1}
\end{align}

Following the previous discussion on splitting $[-T_1,T_1]$ into different regions, we define, for a parameter $\eta$ small to be fixed later, the time $T_2$ as the unique positive time satisfying
\begin{align}\label{defi:T_2}
    Z(T_2)= Z_0+\eta^2,
\end{align}
and the time $T_3$ as the unique time in $(T_2,T_1)$ such that
\begin{align}\label{defi:T_3}
    \mu_0^2 = Z_0 Z(T_3)^{-\frac12} e^{-Z(T_3)}.
\end{align}

From Lemmata \ref{lemm:uniform_Z} and \ref{lemm:dot_Z_T''} we have
\begin{align}\label{ineq:dot_Z(T_2)}
    \dot{Z}(T_1)\mu_0^{-1} \geq 1 \quad \text{and} \quad \dot{Z}(T_2) \gtrsim \eta \mu_0.
\end{align} 

Let $C_1>1$ to be chosen later, and consider the bootstrap estimates on each of the two regions:
\begin{itemize}
    \item For any $t \in [-T_1,T_2]$,
    \begin{equation}\label{bootstrap_-T_1}
            \| \epsilon(t) \|_{H^1} + \left\vert \bar{\mu}(t)\right\vert \leq \mu_0^{\frac74}, \quad \left\vert \bar{z}(t) \right\vert \leq \mu_0^{\frac34}, \quad
            \left\vert \omega(t) - \omega_0 \right\vert + \left\vert \bar{\omega}(t)\right\vert \leq \mu_0^{\frac54}.
    \end{equation}
    \begin{itemize}
        \item for any $t\in [-T_1, -T_3]$
        \begin{equation}\label{bootstrap_-T_1_-T_3}
            \left\vert \mu(t) - \dot{Z}(t) \right\vert \leq Z_0^{-1} \mu_0^{\frac74}, \quad \left\vert z(t) - Z(t) \right\vert \leq Z_0^{-1} \mu_0^{\frac34} .
        \end{equation}
        \item for any $t\in [-T_3, T_2]$
        \begin{equation}\label{bootstrap_-T_3_T_2}
            \left\vert \mu(t) - \dot{Z}(t) \right\vert \leq \mu_0^{\frac74}, \quad \left\vert z(t) - Z(t) \right\vert \leq \mu_0^{\frac34}.
        \end{equation}
    \end{itemize}
    \item For any $t \in [T_2,T_1]$,
    \begin{equation}\label{bootstrap_T_2}
        \begin{aligned}
            & \| \epsilon(t) \|_{H^1} + \left\vert \bar{\mu}(t) \right\vert \leq C_1 \mu_0^{\frac74}, \quad \left\vert \bar{z}(t) \right\vert \leq C_1 Z_0 \mu_0^{\frac34}, \quad\left\vert \omega(t) - \omega_0 \right\vert + \left\vert \bar{\omega} (t) \right\vert \leq C_1 Z_0 \mu_0^{\frac54} \\
            & \left\vert \mu(t) - \dot{Z}(t) \right\vert \leq C_1 Z_0 \mu_0^{\frac74}, \quad \left\vert z(t) - Z(t) \right\vert \leq C_1 Z_0^2 \mu_0^{\frac34}, .
        \end{aligned}
    \end{equation}
\end{itemize}

We prove the  proposition by a bootstrap argument. Let us define the maximal time $T^\star$ so that the bootstrap estimates are satisfied
\begin{equation*}
    T^\star := \sup \left\{ t \in [-T_1;T_1] \text{ such that }\eqref{est:-T_1_T_1} \text{ and } \eqref{bootstrap_-T_1}-\eqref{bootstrap_T_2} \text{ hold on } [-T_1;t) \right\}.
\end{equation*}

Note first that estimates \eqref{bootstrap_-T_1}-\eqref{bootstrap_-T_1_-T_3} hold at time $-T_1$ by \eqref{eq:mu_T_1}-\eqref{eq:z_omega_T_1}. By continuity of $t\mapsto w(t)$ in $H^1$ and of $\Gamma$, $T^\star$ is well-defined.

Now, we aim at proving that $T^\star=T_1$ by choosing $C_1$ large enough (independent of $\mu_0$) and $\mu_0<\nu_4^\star$ small enough (possibly depending on $C_1$). We argue by contradiction by assuming $T^\star < T_1$. The proof of Proposition \ref{propo:bootstrap_-T_T} is given in the next subsections.

\subsection{Estimates in the bootstrap setting}

Since \eqref{bootstrap_-T_1}-\eqref{bootstrap_T_2} hold on $[-T_1,T^\star)$, it follows from the mean value theorem and by choosing $\nu_4^\star$ small enough that
\begin{align}\label{ineq:z-Z}
    \left\vert z(t)^{-\frac12} e^{-z(t)} - Z(t)^{-\frac12} e^{-Z(t)} \right\vert \lesssim \left\vert z(t)-Z(t) \right\vert Z(t)^{-\frac12} e^{-Z(t)}.
\end{align}
Moreover, as a consequence of Proposition \ref{theo:V_A} we deduce the following inequalities on $t\in [-T_1,T^\star)$
\begin{align}\label{est:S_-T_1_T_1_rough}
    & \left\| S(t) \right\|_{H^2} \lesssim Z(t)^{-\frac12} e^{-Z(t)} + \mu_0 e^{-\frac{15}{16} Z(t)}, \\
    & \left\| T(t) \right\|_{H^1} \lesssim \mu_0 Z(t)^{-\frac12} e^{-Z(t)} + e^{-\frac{15}{16}Z(t)} \sum_{i=1}^2 \vert \dot{z}_i\vert, \label{est:T_-T_1_T_1_rough} \\
    & \sum_{i=1}^2 \left\vert \langle S(t), \MRi(t) \rangle \right\vert \lesssim \mu_0 e^{-\frac{15}{16}Z(t)}, \label{est:S_MRi_-T_1_T_1_rough}
\end{align}

Note that the implicit constants appearing in the former estimates are independent of $C_1$ and $Z_0$. It will be the case for all the estimates in the rest of this section.

Next, we reformulate the estimates on the derivatives of the geometrical parameters obtained in Proposition \ref{propo:dynamical_system} in the bootstrap setting. 

\begin{lemm}[Simplified dynamical system]\label{lemm:sim_dyn_sys}
Recall the notations of $z$, $\bar{z}$, $\omega$, $\bar{\omega}$, $\mu$ and $\bar{\mu}$ in \eqref{defi:z_zbar}-\eqref{defi:mu_mubar}. The following system holds on $[-T_1, T^\star)$
\begin{align}
    & \left\vert \dot{z} - \mu \right\vert + \left\vert \dot{\bar{z}} - \bar{\mu} \right\vert + \left\vert \dot{\omega} \right\vert + \left\vert \dot{\bar{\omega}} \right\vert \lesssim Z^{-\frac12} e^{-Z} +\mu_0 e^{-\frac{15}{16} Z}+ \| \epsilon \|_{L^2}, \label{eq:simplified_z_omega} \\
    & \left\vert \dot{\mu} + 2\gamma_1(z) \right\vert + \left\vert \dot{\bar{\mu}}\right\vert \lesssim \mu_0  e^{-\frac{15}{16}Z}+ \| \epsilon \|_{L^2}^2. \label{eq:simplified_mu}
\end{align}
Furthermore, we have
\begin{align}\label{eq:bound_dot_z_1_dot_z_2_mu_1_mu_2}
      & \left\vert \dot{z}_1 \right\vert + \left\vert \dot{z}_2 \right\vert + \left\vert \mu_1 \right\vert + \left\vert \mu_2 \right\vert + \left\vert \mu \right\vert + \vert \omega_1 \vert + \vert \omega_2 \vert \lesssim \mu_0.
\end{align}
\end{lemm}

\begin{proof}
From Proposition \ref{propo:dynamical_system}, \eqref{Est:piqi}, \eqref{est:mu0_Z0}, \eqref{est:S_-T_1_T_1_rough}-\eqref{est:S_MRi_-T_1_T_1_rough} and $\vert \mu_i \vert \lesssim \vert \dot{Z}(t) \vert + \mu_0 \lesssim \mu_0$ (see Appendix \ref{app:Z}), we have
\begin{align*}
    \sum_{i=1}^2 \left( \left\vert - \dot{z}_i + \mu_i \right\vert + \left\vert \dot{\omega_i} \right\vert \right) \lesssim Z^{-\frac12} e^{-Z} + \mu_0e^{-\frac{15}{16}Z} + \| \epsilon \|_{L^2} \text{ and } \sum_{i=1}^2 \left\vert \dot{\mu}_i + \gamma_i \right\vert \lesssim \mu_0 e^{-\frac{15}{16} Z} + \| \epsilon \|_{L^2}^2.
\end{align*}
The proof of \eqref{eq:simplified_z_omega}-\eqref{eq:simplified_mu} ends with the definition of the parameters in \eqref{defi:z_zbar}-\eqref{defi:mu_mubar} and with \eqref{Def:gammai} that ensures $\gamma_1(z) = -\gamma_2(z)$.

We now continue with the inequality on $\dot{z}_1$. From \eqref{defi:z_zbar}, we have
\begin{align*}
    2\left\vert \dot{z}_1 \right\vert \leq \left\vert \dot{z} \right\vert + \left\vert \dot{\bar{z}} \right\vert \leq \left\vert \dot{z}- \mu \right\vert + \left\vert \mu -\dot{Z} \right\vert + \left\vert \dot{Z} \right\vert + \left\vert \dot{\bar{z}} - \bar{\mu} \right\vert + \left\vert \bar{\mu} \right\vert
\end{align*}
which implies \eqref{eq:bound_dot_z_1_dot_z_2_mu_1_mu_2} by using \eqref{eq:simplified_z_omega}, the estimates on $\mu-\dot{Z}$ and $\bar{\mu}$ in \eqref{bootstrap_-T_1_-T_3}-\eqref{bootstrap_T_2} and $\vert \dot{Z}(t) \vert \lesssim \mu_0$. We proceed similarly to prove the inequalities on $\dot{z}_2$, $\mu_1$, $\mu_2$ and $\mu$ in \eqref{eq:bound_dot_z_1_dot_z_2_mu_1_mu_2}. The estimates on $\omega_i$ follow from \eqref{est:-T_1_T_1} and the assumption that $\vert \omega_0 \vert \leq \mu_0$ in Proposition \ref{propo:bootstrap_-T_T}
\end{proof}

Gathering \eqref{est:dt_S}, \eqref{est:S_-T_1_T_1_rough}-\eqref{est:S_MRi_-T_1_T_1_rough} and Lemma \ref{lemm:sim_dyn_sys}, we obtain simplified estimates on the source terms:
\begin{align}\label{est:S_-T_1_T_1}
    & \left\| S(t) \right\|_{H^2} \lesssim Z(t)^{-\frac12} e^{-Z(t)} + \mu_0 e^{-\frac{15}{16} Z(t)}, \\
    & \left\| T(t) \right\|_{H^1} + \sum_{i=1}^2 \left\vert \langle S(t), \MRi(t) \rangle \right\vert + \left\| \partial_t S(t) \right\|_{H^1} \lesssim \mu_0 e^{-\frac{15}{16}Z(t)}. \label{est:T_-T_1_T_1}
\end{align}

Finally, we derive a refined estimate  for the control of the derivative of the transverse translation parameter $\dot{\omega}_i$.
To this aim, we introduce the quantity 
\begin{equation} \label{def:Kj}
K_i(t,\bx)=\tilde{K}_i(t,\bx)\chi_{\mu_0}(x) \quad \text{with} \quad \widetilde{K}_i(t,\bx)=\int_{-\infty}^x \partial_yR_i(t,\tilde{x},y)d\tilde{x} ,
\end{equation}
where $\chi_{\mu_0}(x)=\chi(\mu_0 x)$ and $\chi \in C^{\infty}(\mathbb R)$ satisfies $0 \le \chi \le 1$, $\chi(x)=1$ for $x \le 1$ and $\chi(x)=0$ for $x \ge 2$. Note in particular that 
\begin{equation} \label{bounds:chimu_0}
 \|{\bf{1}}_{(-\mu_0^{-1},+\infty)}\chi_{\mu_0}\|_{L^2} \lesssim \mu_0^{-\frac12} \quad \text{and} \quad \|\chi_{\mu_0}'\|_{L^2} \lesssim \mu_0^{\frac12} .
\end{equation}
Then, we deduce from the pointwise bounds on $\left| \partial_yQ(\bx) \right| \lesssim e^{-|\bx|}$ (see Proposition \ref{propo:asymp_Q}) and the Cauchy-Schwarz inequality that 
\begin{equation} \label{est:Kj}
\|K_i \|_{L^{\infty}} \lesssim 1 \quad \text{and} \quad \|K_i\|_{L^2} \lesssim \mu_0^{-\frac12} .
\end{equation}

We define the quantity 
\begin{equation} \label{def:calKj}
\mathcal{K}_i(t)=\int \epsilon(t,\bx) K_i(t,\bx) \, d\bx . 
\end{equation}

\begin{propo}
It holds, for all $t \in [-T_1,T^{\star})$ and $j=1,2$, 
\begin{equation} \label{est:calKj} 
\|\mathcal{K}_i\|_{L^2} \lesssim  \mu_0^{-\frac12}\|\epsilon\|_{L^2} 
\end{equation}
and 
\begin{equation} \label{deriv:est:calK1}
\begin{split}
\left|\frac{d}{dt}\mathcal{K}_1-c_{Q}\dot{\omega}_1 \right|&+ \left|\frac{d}{dt}\mathcal{K}_2-c_{Q} \dot{\omega}_1 -2c_{Q} \dot{\omega}_2 \right| \\ &\lesssim  \|\epsilon\|_{L^2}\left(\mu_0^{\frac12}+e^{-\frac{15}{16} z}+\|\epsilon\|_{L^2}\right) +\mu_0^{-\frac12}\|T\|_{L^2}+\mu_0^{\frac12}\|S\|_{L^2}
+\sum_{i=1}^2\left\vert \int S \partial_y R_i \right\vert \\ & \quad +\sum_{i=1}^2\left(\left|-\dot{z}_i+\mu_i+\alpha_i\right|+|\dot{\omega}_i|\right) \left( e^{-\frac{15}{16} z}+\mu_0 \right)+\sum_{i=1}^2|\dot{\mu}_i+\gamma_i| ,
\end{split}
\end{equation}
where $c_{Q}=\frac12\int_y\left(\int_x \partial_yQ(x,y)dx\right)^2dy >0$.
\end{propo}

\begin{rema} \label{rema:Ki}The functions $K_i$, used to improve the equation on $\dot{\omega}_i$, are an adaptation of the functions $J_i$ introduced in Section 5.1 of \cite{MM11} to the transverse direction. Note however that here, instead of using the decay on the right of the particular solution $v$ defined in \eqref{coro:quantified_orbital_stability}, we choose to cut the functions $K_i$ on the right, so that the argument applies for a general solution $w$. We refer to \cite{MP17} for a similar idea in the context of the modified Benjamin-Ono equation. 
\end{rema}

\begin{proof} Estimate \eqref{est:calKj} follows directly from \eqref{est:Kj} and the Cauchy-Schwarz inequality. 

We use \eqref{eq:epsilon} and \eqref{def:EV} to compute, for $i=1,2$,
\begin{align*} 
\frac{d}{dt}\mathcal{K}_i & = \int \partial_t\epsilon K_i+\int \epsilon \partial_t K_i \\ 
&=k_1^i+k_2^i+k_3^i+k_4^i+k_5^i+k_6^i+k_7^i+k_8^i+k_9^i+k_{10}^i,
\end{align*}
where 
\begin{equation*} 
\begin{aligned}
k_1^i&:=\int \partial_x \left(-\Delta \epsilon +\epsilon-2R_i\epsilon\right) K_i , 
&k_2^i&:=-\int \partial_x \left(2(R_1+R_2+V_A-R_i) \epsilon+\epsilon^2\right) K_i, \\ 
k_3^i&:=-(-\dot{z}_1+\mu_1+\alpha_1) \int \partial_xR_1 K_i  ,
&k_4^i&:=-(-\dot{z}_2+\mu_2+\alpha_2) \int \partial_xR_2 K_i, \\ 
k_5^i&:=\dot{\omega}_1 \int \partial_yR_1 K_i  ,
&k_6^i&:=\dot{\omega}_2 \int \partial_yR_2 K_i ,\\ 
k_7^i&:=-(\dot{\mu}_1+\gamma_1) \int \Lambda R_1 K_i , 
&k_8^i &:=-(\dot{\mu}_2+\gamma_2) \int \Lambda R_2 K_i ,\\ 
k_9^i&:= -\int (\partial_xS+T) K_i ,
&k_{10}^i&:=\int \epsilon \partial_tK_i.
\end{aligned}
\end{equation*}

\noindent \textit{Estimate for $k_1^i$}. We claim that, for $i=1,2$, 
\begin{equation} \label{est:k1}
|k_1^i| \lesssim \|\epsilon\|_{L^2} \left(\mu_0^{\frac12}+|\mu_i| \right) .
\end{equation}

By integrating by parts and recalling the definition of $L_{R_i}$ in \eqref{def:LRi}, we have 
\begin{align*}
k_1^1
    &=-\int (L_{R_1}\epsilon) \partial_y R_1 \chi_{\mu_0}+\mu_1 \int \epsilon \partial_yR_i \chi_{\mu_0} -\int \left(-\Delta \epsilon +\epsilon-2R_1\epsilon\right) \widetilde{K}_1 \chi_{\mu_0}' \\ 
    &:=k_{1,1}^1+k_{1,2}^1+k_{1,3}^1 .
\end{align*}
More integration by parts, \eqref{bounds:chimu_0} and the identity $L_{R_1}\partial_yR_1=0$ yield
\begin{align*}
|k_{1,1}^1| &\lesssim \left| \int \epsilon \partial^2_{xy}R_1 \chi_{\mu_0}'\right|+\left| \int \epsilon \partial_{y}R_1 \chi_{\mu_0}''\right| \lesssim \|\epsilon\|_{L^2} \mu_0 ; \\ 
|k_{1,2}^1| &\lesssim  \|\epsilon\|_{L^2} |\mu_1| ; \\ 
|k_{1,3}^1| &\lesssim \|\epsilon\|_{L^2}\|\chi_{\mu_0}'\|_{L^2} \lesssim \|\epsilon\|_{L^2} \mu_0^{\frac12} ,
\end{align*}
which concludes the proof of \eqref{est:k1} in the case $i=1$. The proof in the case $i=2$ is similar. 

\smallskip
\noindent \textit{Estimate for $k_2^i$}. We claim that, for $i=1,2$,
\begin{equation} \label{est:k2}
|k_2^i| \lesssim \|\epsilon\|_{L^2} \left(\mu_0^{\frac12} +e^{-\frac{15}{16}z} + \|\epsilon\|_{L^2} \right).
\end{equation} 

Indeed, we observe integrating by parts and using H\"older's inequality that 
\begin{align*}
|k_2^1|
    &\lesssim  \left| \int \left( 2\left(R_2+V_A\right)\epsilon + \epsilon^2 \right) \partial_y R_1 \chi_{\mu_0}\right|+\left|\int \left( 2 \left(R_2+V_A\right)\epsilon + \epsilon^2 \right)  \widetilde{K}_1 \chi_{\mu_0}' \right| \\ 
    &\lesssim \|\epsilon\|_{L^2} \left( \|V_A\|_{L^2}+\|R_1\partial_yR_2\|_{L^2}+\|\chi'_{\mu_0}\|_{L^2} + \|\epsilon\|_{L^2} \right) ,
\end{align*}
which combined to \eqref{est:VA:H2}, \eqref{est:R_1R_2} and \eqref{bounds:chimu_0} yields \eqref{est:k2} in the case $i=1$. The proof in the case $i=2$ is similar.

\smallskip
\noindent \textit{Estimate for $k_3^i$}. By using Lemma \ref{parity}, $k_3^1=0$. We claim that 
\begin{equation} \label{est:k3}
|k_3^2| \lesssim  \left|-\dot{z}_1+\mu_1+\alpha_1\right| \left(\mu_0+e^{-\frac{15}{16} z} \right) .
\end{equation} 

By integrating by parts, using the orthogonality relation $\langle Q ,\partial_y Q \rangle=0$ (see \eqref{eq:id4_Q}), and the pointwise bound $Q(\bx) \lesssim e^{-|\bx|}$, we have
\begin{align*}
|k_3^1|& \le \left|-\dot{z}_1+\mu_1+\alpha_1\right| \left( \left|\int R_1 \partial_yR_1 (1-\chi_{\mu_0})\right|+
\left|\int R_1 \widetilde{K}_1 \chi_{\mu_0}'\right| \right) \\ 
& \lesssim  \left|-\dot{z}_1+\mu_1+\alpha_1\right| \left( e^{-\mu_0^{-1}}+\mu_0 \right) ,
\end{align*}
since, on the time interval $[-T_1,T_1]$,
\begin{equation} \label{est:zi:bootstrap}
|z_1(t)| \le 2|Z(T_1)| =2\rho^{-1}Z_0 \le 8\rho^{-1} \left|\ln (\mu_0)\right| \le \frac1{8\mu_0} .
\end{equation}
On the other hand, we have, after integrating by parts, 
\begin{align*}
|k_3^2|& \le \left|-\dot{z}_1+\mu_1+\alpha_1\right| \left( \left|\int R_1 \partial_yR_2 \chi_{\mu_0}\right|+
\left|\int R_1 \widetilde{K}_2 \chi_{\mu_0}'\right| \right) \\ 
& \lesssim  \left|-\dot{z}_1+\mu_1+\alpha_1\right| \left( \|R_1 \partial_yR_2\|_{L^2}+\mu_0 \right) ,
\end{align*}

Combining these estimates to \eqref{est:R_1R_2} implies \eqref{est:k3}.

\smallskip
\noindent \textit{Estimate for $k_4^i$}. By using Lemma \ref{parity}, $k_4^2=0$. By arguing exactly as for \eqref{est:k3}, we have
\begin{equation} \label{est:k4}
|k_4^1| \lesssim  \left|-\dot{z}_2+\mu_2+\alpha_2\right| \left(\mu_0+e^{-\frac{15}{16} z} \right).
\end{equation} 

\smallskip
\noindent \textit{Estimate for $k_5^1$}. We claim that 
\begin{equation} \label{est:k51}
\left|k_5^1-c_Q \dot{\omega}_1\right| \lesssim  \left|\dot{\omega}_1\right|\left( e^{-\frac1{\mu_0}} + \vert \mu_1 \vert \right),
\end{equation} 
where $c_Q=\frac12\int_y\left(\int_x \partial_yQ(x,y)dx\right)^2dy >0$.

Indeed, 
\begin{align*} 
k_5^1=-\dot{\omega}_1 \int \partial_yR_1 \tilde{K}_1+\dot{\omega}_1 \int \partial_yR_1 \tilde{K}_1 (1-\chi_{\mu_0}) .
\end{align*}
On the one hand, by using \eqref{est:zi:bootstrap}, we bound the second term on the right-hand side of the above inequality by $\left|\dot{\omega}_1\right| e^{-\frac1{\mu_0}}$. On the other hand, a direct computation provides that 
\begin{equation*}
\left|\int \partial_yR_1 \tilde{K}_1-\frac12\int_y\left(\int_x \partial_yQ(x,y)dx\right)^2dy \right| \lesssim |\mu_1|,
\end{equation*}
which implies \eqref{est:k51}.

\smallskip
\noindent \textit{Estimate for $k_5^2$}.  We claim that 
\begin{equation} \label{est:k52}
\left|k_5^2-2c_{Q} \dot{\omega}_1\right| \lesssim  \left|\dot{\omega}_1\right| \left( e^{-\frac{15}{16} z} + \sum_{i=1}^2 \left( \vert \mu_i \vert + \vert \omega_i \vert \right) \right).
\end{equation}

Indeed, we decompose $k_{5}^2$ as 
\begin{align*} 
k_{5}^2
    &=\dot{\omega}_1\int \partial_yR_1\left(\int_{-\infty}^{+\infty} \partial_yR_2 (\tilde{x},y)d\tilde{x}\right) -\dot{\omega}_1\int \partial_yR_1 \left(\int_{x}^{+\infty} \partial_yR_2 (\tilde{x},y) d\tilde{x}\right) \\
    & \quad+ \int \partial_yR_1(x,y)\left(\int_{-\infty}^{x} \partial_yR_2(\tilde{x},y) d\tilde{x}\right) \left(1-\chi_{\mu_0}\right)
\end{align*}
By arguing as in the proofs of Lemmata \ref{lemm:R_1_R_2} and \ref{est:tildeR_NR}, we have 
\begin{equation*} 
\left|\int \partial_yR_1\left(\int_{x}^{+\infty} \partial_yR_2d\tilde{x}\right) \right| \lesssim e^{-\frac{15}{16} z} .
\end{equation*}
Moreover, we infer, by using the pointwise bound $\vert \partial_yQ(\bx) \vert \lesssim e^{-|\bx|}$
\begin{equation*} 
\left|\int \partial_yR_1\left(\int_{-\infty}^{x} \partial_yR_2d\tilde{x}\right) \left(1-\chi_{\mu_0}\right) \right| \lesssim e^{-\mu_0^{-1}} .
\end{equation*}
Finally, observe with Lemma \ref{lemm:decompo_rescaled_Q} that 
\begin{equation*} 
\left|\int \partial_yR_1 (x,y) \left(\int_{-\infty}^{+\infty} \partial_yR_2 (\tilde{x},y)d\tilde{x}\right)dx dy-\int_y \left(\int_x \partial_yQ(x,y)dx\right)^2 dy \right| \lesssim \sum_{i=1}^2 \left( |\mu_i| + \vert \omega_i \vert \right).
\end{equation*}
Gathering the previous estimates yields \eqref{est:k52}.

\noindent \textit{Estimate for $k_6^1$}. We have that 
\begin{equation} \label{est:k6}
|k_6^1| \le  \left|\dot{\omega}_2\right|\left|\int \partial_yR_2K_1\right| \lesssim \left|\dot{\omega}_2\right|e^{-\frac{15}{16} z},
\end{equation} 
where the last estimate is obtained arguing as in the proofs of Lemmata \ref{lemm:R_1_R_2} and \ref{est:tildeR_NR}.

\smallskip
\noindent \textit{Estimate for $k_6^2$}. Arguing exactly as for $k_5^1$, we infer that 
\begin{equation} \label{est:k62}
\left|k_6^2-c_{Q} \dot{\omega}_2\right| \lesssim  \left|\dot{\omega}_2\right| \left(  e^{-\frac1{\mu_0}} 
 + \vert \mu_2 \vert \right).
\end{equation} 

\smallskip
\noindent \textit{Estimate for $k_7^i$ and $k_8^i$}. It follows directly from the fact that $\Lambda Q$ is an integrable function that, for $i=1,2$,
\begin{equation} \label{est:k78}
\sum_{i=1}^2\left(|k_7^i|+|k_8^i|\right) \lesssim  \sum_{i=1}^2\left|\dot{\mu}_i+\gamma_i\right|.
\end{equation}

\smallskip
\noindent \textit{Estimate for $k_9^i$}. We claim that, for $i=1,2$,
\begin{equation} \label{est:k9}
|k_9^i| \lesssim  \left| \int S\partial_yR_i\right|+\|S\|_{L^2}\mu_0^{\frac12}+\|T\|_{L^2}\mu_0^{-\frac12} .
\end{equation} 

Indeed, we find after integrating by parts,
\begin{align*}
 |k_9^i| &\le \left| \int S\partial_yR_i\right|+\left| \int S\partial_yR_i(1-\chi_{\mu_0})\right| +  \left| \int S\widetilde{K}_i \chi_{\mu_0}'\right|+\left| \int T\widetilde{K}_i \chi_{\mu_0}\right| \\
 & \lesssim \left| \int S\partial_yR_i\right|+\|S\|_{L^2} \left(\|\chi_{\mu_0}'\|_{L^2}+e^{-\frac1{\mu_0}} \right)+\|T\|_{L^2}\|\chi_{\mu_0}\|_{L^2},
\end{align*} 
which combined to \eqref{bounds:chimu_0} yields \eqref{est:k9}. 

\smallskip
\noindent \textit{Estimate for $k_{10}^i$}. We claim that 
\begin{equation} \label{est:k10}
|k_{10}^i| \lesssim  \|\epsilon\|_{L^2}\left(|\dot{\mu}_i \vert +|\dot{z}_i|+|\dot{\omega}_i|\right)  \mu_0^{-\frac12} .
\end{equation} 

Indeed, by using $\partial_t\partial_yR_i=\dot{\mu_1}\partial_y\Lambda R_i-\dot{z}_1\partial^2_{xy}R_i-\dot{\omega}_1\partial^2_{yy}R_i$ and the Cauchy-Schwarz inequality, we have that 
\begin{equation*}
|k_{10}^i| \lesssim  \|\epsilon\|_{L^2}\left(|\dot{\mu}_i \vert +|\dot{z}_i|+|\dot{\omega}_i|\right) \|{\bf 1}_{(-\mu_0^{-1},+\infty)}\chi_{\mu_0}\|_{L^2},
\end{equation*}
which implies \eqref{est:k10} in view of \eqref{bounds:chimu_0}. 

\smallskip
Therefore, we conclude the proof of  \eqref{deriv:est:calK1} by combining the previous estimates to the rough bounds \eqref{eq:bound_dot_z_1_dot_z_2_mu_1_mu_2}. 
\end{proof}

Gathering \eqref{deriv:est:calK1}, Proposition \ref{propo:dynamical_system} and \eqref{est:S_-T_1_T_1}-\eqref{est:T_-T_1_T_1} , we obtain the simplified bound
\begin{equation} \label{deriv:est:calK1_simpl}
\left|\frac{d}{dt}\mathcal{K}_1-c_Q\dot{\omega}_1 \right| + \left|\frac{d}{dt}\left( \mathcal{K}_2-2 \mathcal{K}_1 \right) -c_{Q} \dot{\omega}_2 \right| \lesssim \mu_0^{\frac12} e^{-\frac{15}{16}Z} + \|\epsilon\|_{L^2}\left(\mu_0^{\frac12}+e^{-\frac{15}{16} Z} \right) +\|\epsilon\|_{L^2}^2.
\end{equation}

\subsection{Bootstrap of \texorpdfstring{\eqref{bootstrap_-T_1}-\eqref{bootstrap_-T_3_T_2}}{(4.3)-(4.3)} on \texorpdfstring{$[-T_1,-T_2]$}{[-T1,-T2]}.} In this subsection, we prove that $T^\star > -T_2$ by improving \eqref{bootstrap_-T_1}-\eqref{bootstrap_-T_3_T_2} on $[-T_1,T^\star)$. In every step of the proof below, we will take $\mu_0>0$ small enough possibly depending on $\eta$. 

\smallskip
\noindent \textit{Closing the bound on $\|\epsilon\|_{H^1}$ on $[-T_1,T^\star)$}.
From its definition in \eqref{defi:theta}, the following bound on $\Theta$ holds, with \eqref{defi:mAi}, Proposition \ref{propo:dynamical_system}, \eqref{est:S_-T_1_T_1}-\eqref{est:T_-T_1_T_1} and \eqref{bootstrap_-T_1},
\begin{align}\label{eq:bound_theta_-T_1}
    \Theta(t) \lesssim  \mu_0^{\frac{11}{4}} e^{-\frac{15}{16}Z(t)}.
\end{align}
By the estimate of the functional $\mathcal{F}_-$ in Proposition \ref{propo:functionals}, we have, from \eqref{est:mu0_Z0}, \eqref{bootstrap_-T_1}, \eqref{est:S_-T_1_T_1}-\eqref{est:T_-T_1_T_1}, \eqref{eq:bound_theta_-T_1}, \eqref{eq:bound_dot_z_1_dot_z_2_mu_1_mu_2} and by choosing $\mu_0$ small enough,
\begin{align}\label{est:F_-_-T_1}
    \frac{d}{dt} \mathcal{F}_- \lesssim \mu_0^{\frac{11}{4}} e^{-\frac{15}{16}Z}+\mu_0^{\frac{9}{2}} e^{-\rho Z}+\mu_0^{\frac{21}4} . 
\end{align}
On the time interval $[-T_1,-T_2]$, with \eqref{ineq:dot_Z(T_2)}, we have $-\dot{Z}(t) \gtrsim \eta \mu_0$. Thus by multiplying the right-hand side of the previous inequality by $-\dot{Z}(t) \eta^{-1} \mu_0^{-1}$, an integration from $-T_1$ to $t$ provides with \eqref{defi:T_1}
\begin{align}\label{est:F_-_-T_1_-T_2}
    \mathcal{F}_-(t) - \mathcal{F}_-(-T_1) \lesssim \eta^{-1} \left(\mu_0^{\frac{7}{4}} e^{-\frac{15}{16}Z(t)}+\mu_0^{\frac{7}{2}} e^{-\rho Z(t)}+Z_0 \mu_0^{\frac{17}4}  \right) \lesssim Z_0^{-6} \mu_0^{\frac72},
\end{align}
by using \eqref{est:mu0_Z0}. Thus, using \eqref{coercivity_F}, \eqref{est:S_-T_1_T_1} and \eqref{eq:F_-_T_1}, we obtain
\begin{align}\label{eq:eps_-T_1_T_star}
    \forall t \in [-T_1,T^\star), \quad \| \epsilon (t) \|_{H^1}^2 \lesssim Z_0^{-6} \mu_0^{\frac72}.
\end{align}

\smallskip
\noindent \textit{Closing the bounds on $\bar{\mu}$, $\bar{z}$ on $[-T_1,T^\star)$}. From \eqref{eq:simplified_z_omega}-\eqref{eq:simplified_mu} and \eqref{eq:eps_-T_1_T_star}, we deduce
\begin{align*}
    \left\vert \dot{\bar{\mu}} \right\vert \lesssim \mu_0 e^{-\frac{15}{16} Z} + Z_0^{-6} \mu_0^{\frac72}\quad \text{and} \quad \left\vert \dot{\bar{z}} -\bar{\mu} \right\vert \lesssim Z_0^{-3} \mu_0^{\frac74}.
\end{align*}
After multiplying the first inequality by $-\dot{Z}(t) \eta^{-1}\mu_0^{-1}$, an integration from $-T_1$ to $t$ provides, with \eqref{eq:mu_T_1},
\begin{align}\label{eq:bar_mu_-T_1_Tstar}
    \left\vert \bar{\mu}(t) \right\vert \lesssim \left\vert \bar{\mu}(-T_1) \right\vert + \eta^{-1} \left( e^{-\frac{15}{16}Z(t)} + Z_0^{-5}\mu_0^{\frac52} \right) \lesssim Z_0^{-3} \mu_0^{\frac74}.
\end{align}
Arguing similarly for $\bar{z}$, one obtains
\begin{align}\label{eq:bar_z_-T_1_Tstar}
    \left\vert \bar{z}(t) \right\vert \lesssim \left\vert \bar{z}(-T_1) \right\vert +  \eta^{-1} Z_0^{-2} \mu_0^{\frac34} \lesssim Z_0^{-1}\mu_0^{\frac34}.
\end{align}

\smallskip
\noindent \textit{Closing the bounds on $\omega$ and $\bar{\omega}$ on $[-T_1,T^\star)$}. From \eqref{deriv:est:calK1_simpl}, with $c_Q \neq 0$, it holds 
\begin{equation*}
\left|\frac{d}{dt}\mathcal{K}_1-c_Q\dot{\omega}_1 \right| + \left|\frac{d}{dt} \left( \mathcal{K}_2 - 2 \mathcal{K}_1 \right) -c_{Q} \dot{\omega}_2 \right| \lesssim  Z_0^{-3} \mu_0^{\frac94}.
\end{equation*}
After multiplying by $-\dot{Z}(t) \eta^{-1} \mu_0^{-1}$, an integration from $-T_1$ to $t$ gives, with \eqref{est:calKj}, \eqref{eq:eps_-T_1_T_star} and \eqref{eq:z_omega_T_1}
\begin{align}
    \left\vert \omega_1(t) - \frac12 \omega_0 \right\vert + \left\vert \omega_2(t) + \frac12 \omega_0 \right\vert 
    & \lesssim \left\vert \omega_1(-T_1) - \frac12 \omega_0 \right\vert + \left\vert \omega_2(-T_1) + \frac12 \omega_0 \right\vert + Z_0^{-3} \mu_0^{\frac54} + \eta^{-1}  Z_0^{-2} \mu_0^{\frac54} \nonumber \\
    & \lesssim Z_0^{-1} \mu_0^{\frac54}.\label{eq:omega_-T_1}
\end{align}

\smallskip
\noindent \textit{Closing the bounds on $\mu$ and $z$ on $[-T_1,T^\star)$ with $T^{\star} \le -T_3$}. We recall that $T_3$ is defined in \eqref{defi:T_3}. We infer from \eqref{eq:simplified_z_omega}-\eqref{eq:simplified_mu} and \eqref{eq:eps_-T_1_T_star} that
\begin{align}\label{est:mu_dot_z_dot_-T_1_T_star}
    \left\vert \dot{\mu} + 2\gamma_1(z) \right\vert \lesssim \mu_0 e^{-\frac{15}{16} Z} + Z_0^{-6} \mu_0^{\frac72}\quad \text{and} \quad \left\vert \dot{z} - \mu \right\vert \lesssim Z_0^{-3} \mu_0^{\frac74}.
\end{align}
From \eqref{Def:gammai} and \eqref{propo:approx_e^-Z.1} we infer
\begin{align}\label{eq:gamma_1_z_Z}
    \left\vert \gamma_1(z(t)) - \gamma_1(Z(t)) \right\vert \lesssim \vert z(t)-Z(t) \vert \sup_{\vert Y- Z(t) \vert \leq \vert Z(t)-z(t) \vert} \left\vert \gamma_1'(Y) \right\vert \lesssim \vert z(t)-Z(t) \vert \frac{e^{-Z(t)}}{Z(t)^{\frac12}}.
\end{align}
Moreover, $\ddot{Z}=-2\gamma_1(Z)$ (see \eqref{Def:gammai} and \eqref{eq:Z}). Thus, it follows from \eqref{bootstrap_-T_1_-T_3}
\begin{align*}
    \left\vert \dot{\mu} -\ddot{Z} \right\vert \lesssim \left\vert z-Z\right\vert Z^{-\frac12} e^{-Z} + \mu_0 e^{-\frac{15}{16} Z} + Z_0^{-6} \mu_0^{\frac72} \lesssim Z_0^{-1}\mu_0^{\frac34} Z^{-\frac12} e^{-Z} + \mu_0 e^{-\frac{15}{16} Z} + Z_0^{-6} \mu_0^{\frac72}
\end{align*}
By multiplying by $-\dot{Z}(t)\eta^{-1}\mu_0^{-1}$, an integration provides
\begin{align} 
    \left\vert \mu(t) - \dot{Z}(t) \right\vert & \lesssim \left\vert \mu(-T_1) - \dot{Z}(-T_1)  \right\vert + \eta^{-1}\left( Z_0^{-1} \mu_0^{-\frac14} Z(t)^{-\frac12} e^{-Z(t)} + e^{-\frac{15}{16}Z(t)} + Z_0^{-5} \mu_0^{\frac52} \right) \nonumber \\
        & \lesssim \eta^{-1} Z_0^{-2} \mu_0^{\frac{7}{4}}, \label{eq:mu_-T_1_Tstar}
\end{align}
where the last bounds follows by using that $Z(t)^{-\frac12}e^{Z(t)} \le Z_0^{-1}\mu_0^2$ on $[-T_1,-T_3]$ (see the definition of $T_3$ in \eqref{defi:T_3}) and from \eqref{eq:mu_T_1}. Then, we deduce gathering the above estimate, \eqref{eq:mu_T_1} and \eqref{est:mu_dot_z_dot_-T_1_T_star} that
\begin{align*}
    \left\vert \dot{z} -\dot{Z} \right\vert \le  \left\vert \dot{z} -\mu \right\vert+\left\vert \mu -\dot{Z} \right\vert\lesssim Z_0^{-3}\mu_0^{\frac74}+\eta^{-1}\left( Z_0^{-1} \mu_0^{-\frac14} Z(t)^{-\frac12} e^{-Z(t)} + e^{-\frac{15}{16}Z(t)} + Z_0^{-5} \mu_0^{\frac52} \right), 
\end{align*}
which, after a multiplication by $-\dot{Z}(t)\eta^{-1}\mu_0^{-1}$, implies with \eqref{eq:z_omega_T_1} and the definition of $T_3$ in \eqref{defi:T_3}
\begin{align}
    \left\vert z(t) -Z(t) \right\vert 
    & \lesssim \left\vert z(-T_1) - Z(-T_1) \right\vert + \eta^{-1} Z_0^{-2} \mu_0^{\frac34} \nonumber \\
    & \qquad + \eta^{-2}\left( Z_0^{-1} \mu_0^{-\frac54} Z(t)^{-\frac12} e^{-Z(t)} + \mu_0^{-1} e^{-\frac{15}{16}Z(t)} + Z_0^{-4} \mu_0^{\frac32} \right)  \lesssim \eta^{-2} Z_0^{-2} \mu_0^{\frac{3}{4}}. \label{eq:z_-T_1_Tstar} 
\end{align}

\smallskip
By using the bounds \eqref{eq:eps_-T_1_T_star}, \eqref{eq:bar_mu_-T_1_Tstar}, \eqref{eq:bar_z_-T_1_Tstar}, \eqref{eq:omega_-T_1}, \eqref{eq:mu_-T_1_Tstar} and \eqref{eq:z_-T_1_Tstar}, we strictly improve the bounds \eqref{bootstrap_-T_1} and \eqref{bootstrap_-T_1_-T_3} on the time interval $[-T_1,T^\star)$ with $T^{\star} \le -T_3$. This implies that $T^{\star} > -T_3$.

\smallskip
\noindent \textit{Closing the bounds on $\mu$ and $z$ on $[-T_3,T^\star)$}. In order to study the dynamical parameters $\mu$ and $z$, we introduce the approximate Hamiltonian $H$ defined by
\begin{align}\label{defi:H(t)}
    H(t) := \frac{1}{2} \mu(t)^2 + \frac{2}{\langle \Lambda Q, Q \rangle} \int_{\mathbb{R}^2} Q(x+z(t), y) Q^2(x,y)dxdy.
\end{align}

We claim that
\begin{align}\label{eq:H(-T_1)}
    \left\vert H(-T_3) -2\mu_0^2 \right\vert \lesssim \eta^{-1} Z_0^{-1}\mu_0^{\frac34} Z(-T_3)^{-\frac12} e^{- Z(-T_3)}.
\end{align}
Indeed, \eqref{eq:mu_-T_1_Tstar} at time $-T_3$ provides
\begin{align}\label{eq:mu2-dot_z2_-T_1}
    \left\vert \mu(-T_3)^2 - \dot{Z}(-T_3)^2 \right\vert 
    & = \left\vert \left( \mu(-T_3)- \dot{Z}(-T_3) + 2\dot{Z}(-T_3) \right) \left( \mu(-T_3) - \dot{Z}(-T_3)\right) \right\vert \nonumber \\
    & \lesssim \eta^{-1} Z_0^{-2} \mu_0^{\frac{11}{4}}.
\end{align}
With \eqref{propo:approx_e^-Z.1} and \eqref{eq:z_omega_T_1} we have
\begin{align*}
    \left\vert \int \left( Q(\cdot +z(-T_3), \cdot) - Q(\cdot +Z(-T_3),\cdot) \right) Q^2 \right\vert \lesssim Z_0^{-1} \mu_0^{\frac34} Z(-T_3)^{-\frac12} e^{- Z(-T_3)}.
\end{align*}
Gathering those computations with \eqref{est:mu0_Z0}, the definition of $Z$ in Lemma \ref{lemm:pointwise_Z} and the fact that $H(Z(t),\dot{Z(t)})=2\mu_0^2$ (see Lemma \ref{lemm:hamiltonian}) yield \eqref{eq:H(-T_1)}.

On $[-T_3;T^\star)$, we compute the time derivative of $H$ by using \eqref{p1p2} and \eqref{Def:gammai},
\begin{align}
    \left\vert \dot{H}(t) \right\vert 
        & = \left\vert \dot{\mu}(t) \mu(t) + \dot{z}(t) 2 \gamma_1(z(t)) \right\vert \nonumber\\
        & \leq \left\vert \dot{\mu}(t) +2 \gamma_1(z(t)) \right\vert \left\vert \mu(t) \right\vert + 2 \left\vert\gamma_1(z(t)) \right\vert \left\vert \dot{z}(t) - \mu(t) \right\vert. \label{eq:H'-T_1}
\end{align}
From \eqref{eq:simplified_z_omega}-\eqref{eq:bound_dot_z_1_dot_z_2_mu_1_mu_2}, \eqref{eq:eps_-T_1_T_star} and \eqref{bootstrap_-T_1}, it follows
\begin{align}\label{eq:H_-T_1_-T_2}
    \left\vert \dot{H} \right\vert \lesssim \mu_0^2 e^{-\frac{15}{16} Z} + Z_0^{-3} \mu_0^{\frac74} Z^{-\frac12} e^{-Z} \lesssim Z_0^{-3} \mu_0^{\frac74} Z^{-\frac12} e^{-Z},
\end{align}
where we used in the last inequality that $\mu_0^2 \leq Z_0 Z(t)^{-\frac12} e^{-Z(t)} $ for $t \in [-T_3,-T_2]$.
We multiply the right-hand side of \eqref{eq:H_-T_1_-T_2} by $- \dot{Z}(t) \eta^{-1} \mu_0^{-1}$ (recall \eqref{ineq:dot_Z(T_2)}), integrate from $-T_3$ to $t$ and use \eqref{eq:H(-T_1)} to get
\begin{align*}
    \left\vert H(t) - 2\mu_0^2 \right\vert \lesssim \eta^{-1} Z_0^{-1} \mu_0^{\frac34} Z(t)^{-\frac12} e^{-Z(t)}.
\end{align*}
From \eqref{ineq:z-Z} and \eqref{propo:approx_e^-Z.1}, we have 
\begin{align*}
    Z(t)^{-\frac12} e^{-Z(t)} \lesssim \frac{2}{\langle \Lambda Q, Q \rangle} \int_{\mathbb{R}^2} Q(x+z(t),y) Q^2(x,y) dx dy.
\end{align*}
Moreover, we infer from \eqref{eq:simplified_z_omega}, \eqref{eq:eps_-T_1_T_star} and recalling $\vert \dot{Z}(t) \vert \leq 2\mu_0$,
\begin{align*}
    \left\vert \mu(t)^2 -\dot{z}(t)^2 \right\vert
        &\leq \left\vert \mu(t)-\dot{z}(t) \right\vert \left( \left\vert \dot{z}(t)- \mu(t) \right\vert + 2 \left\vert \mu(t) - \dot{Z}(t) \right\vert +2 \vert \dot{Z}(t) \vert \right) \lesssim Z_0^{-2} \mu_0^{\frac34} Z^{-\frac12}(t) e^{-Z(t)}.
\end{align*}
It follows from the three previous inequalities and \eqref{ineq:z-Z} that, by defining  $\nu := C \eta^{-1} Z_0^{-1} \mu_0^{\frac34}$ for a constant $C$ large enough,
\begin{align*}
\MoveEqLeft
    \frac12 \dot{z}(t)^2 + \frac{2(1-\nu)}{\langle \Lambda Q, Q \rangle} \int_{\mathbb{R}^2} Q(x+z(t),y) Q^2(x,y) dx dy \\
    & \qquad \leq 2\mu_0^2 \leq \frac12 \dot{z}(t)^2 + \frac{2(1+\nu)}{\langle \Lambda Q, Q \rangle} \int_{\mathbb{R}^2} Q(x+z(t),y) Q^2(x,y) dx dy.
\end{align*}
We can thus apply Proposition \ref{propo:H_negative_times} with $h=2\mu_0^2$ and $\varepsilon_0 = Z_0^{-1} \mu_0^{\frac34} $ (see \eqref{bootstrap_-T_1_-T_3}) to obtain that for any $t\in[-T_3,T^\star)$,
\begin{align}\label{eq:z_Z_-T_1_Tstar}
    \left\vert z(t) - Z(t) \right\vert \lesssim \left( \nu +\varepsilon_0 \right) \lesssim \eta^{-1} Z_0^{-1} \mu_0^{\frac34}.
\end{align}

To close the bound on $\mu$, we use \eqref{eq:simplified_mu}, with \eqref{eq:z_Z_-T_1_Tstar} in \eqref{eq:gamma_1_z_Z}, \eqref{eq:eps_-T_1_T_star} and $\mu_0$ small so that $\eta^{-1} \mu_0 \leq 1$ in \eqref{eq:eps_-T_1_T_star}, to obtain
\begin{align*}
  \left\vert \dot{\mu} -\ddot{Z} \right\vert=  \left\vert \dot{\mu} + 2\gamma_1(Z) \right\vert \lesssim  \mu_0  e^{-\frac{15}{16}Z}+Z_0^{-6}\mu_0^{\frac72} + \eta^{-1} Z_0^{-1} \mu_0^{\frac34} Z^{-\frac12} e^{-Z} \lesssim \eta^{-1} Z_0^{-1} \mu_0^{\frac34} Z^{-\frac12} e^{-Z}.
\end{align*}
After multiplying the right-hand side by $-\dot{Z}(t) \eta^{-1} \mu_0^{-1} \gtrsim 1$, an integration from $-T_3$ to $t$ provides, with the initial condition \eqref{bootstrap_-T_1_-T_3},
\begin{align}\label{eq:mu_dotZ_-T_1_Tstar}
    \left\vert \mu(t) - \dot{Z}(t) \right\vert \lesssim  \left\vert \mu(-T_3) - \dot{Z}(-T_3) \right\vert + \eta^{-2} Z_0^{-1} \mu_0^{-\frac14} Z^{-\frac12} e^{-Z} \lesssim \eta^{-2} Z_0^{-1} \mu_0^{\frac74}.
\end{align}

By \eqref{eq:eps_-T_1_T_star}, \eqref{eq:bar_mu_-T_1_Tstar}, \eqref{eq:bar_z_-T_1_Tstar}, \eqref{eq:omega_-T_1}, \eqref{eq:z_Z_-T_1_Tstar} and \eqref{eq:mu_dotZ_-T_1_Tstar}, we strictly improve \eqref{bootstrap_-T_1} by taking $Z_0$ large enough depending on $\eta$, and thus by continuity $T^\star >-T_2$. 

\subsection{Bootstrap of \texorpdfstring{\eqref{bootstrap_-T_1}}{(4.19)} and \texorpdfstring{\eqref{bootstrap_-T_3_T_2}}{(4.21)} on \texorpdfstring{$[-T_2,T_2]$}{[-T2;T2]}.} Assuming $T^\star<T_2$, we argue on $[-T_2,T^\star)$ in order to prove by contradiction that $T^\star >T_2$. On this time interval, the inequality $-\dot{Z}(t) \geq c \eta \mu_0$ does not hold anymore. We bypass this issue by using the following inequalities for $t\in [-T_2,T_2]$, deduced from \eqref{est:mu0_Z0}, \eqref{defi:T_2}, \eqref{propo:approx_e^-Z.1} and \eqref{defi:eq_Z},
\begin{align}\label{ineq:Z_T_2}
    \mu_0^2 \lesssim Z(t)^{-\frac12} e^{-Z(t)} \lesssim \ddot{Z}(t).
\end{align}

\smallskip
\noindent \textit{Closing the bounds on $\mu$ and $\bar{\mu}$ on $[-T_2,T^\star)$}. From \eqref{eq:simplified_mu}, \eqref{eq:gamma_1_z_Z}, \eqref{bootstrap_-T_1}, \eqref{bootstrap_-T_3_T_2}, \eqref{ineq:Z_T_2}, we have for $t\in [-T_2,T^\star)$,
\begin{align*}
    \left\vert \dot{\mu} + 2\gamma_1(Z) \right\vert + \left\vert \dot{\bar{\mu}} \right\vert \lesssim  \left\vert z-Z \right\vert Z^{-\frac12} e^{-Z} + \mu_0 e^{-\frac{15}{16}Z} + \mu_0^{\frac72} \lesssim \mu_0^{\frac34} \ddot{Z}.
\end{align*}
From \eqref{eq:bar_mu_-T_1_Tstar} and \eqref{eq:mu_dotZ_-T_1_Tstar}, an integration from $-T_2$ to $t$ provides, with the upper bound $\left\vert \dot{Z}(-T_2) \right\vert \lesssim \eta \mu_0$ (see Lemma \ref{lemm:dot_Z_T''}),
\begin{align}\label{eq:mu_dot_Z_-T_2_T_2}
    \left\vert \mu(t) -\dot{Z}(t) \right\vert + \left\vert \bar{\mu}(t) \right\vert \lesssim \eta^{-2} Z_0^{-1} \mu_0^{\frac74} + \mu_0^{\frac34} \left\vert \dot{Z}(-T_2) \right\vert \lesssim \left( \eta^{-2} Z_0^{-1}+ \eta \right) \mu_0^{\frac74}.
\end{align}

\smallskip
\noindent \textit{Closing the bounds on $z$ and $\bar{z}$ on $[-T_2,T^\star)$}. From \eqref{eq:simplified_z_omega}, \eqref{eq:mu_dot_Z_-T_2_T_2}, \eqref{ineq:Z_T_2} and \eqref{bootstrap_-T_1}, we have for $t\in [-T_2,T^\star)$,
\begin{align*}
    \left\vert \dot{z}(t) -\dot{Z}(t) \right\vert + \left\vert \dot{\bar{z}}(t) \right\vert \lesssim Z^{-\frac12} e^{-Z} +\mu_0 e^{-\frac{15}{16} Z}+ \mu_0^{\frac74} \lesssim \mu_0^{-\frac14} \ddot{Z}(t),
\end{align*}
and an integration from $-T_2$ to $t$ provides with \eqref{eq:z_Z_-T_1_Tstar} and \eqref{eq:bar_z_-T_1_Tstar} at time $-T_2$, 
\begin{align}\label{eq:z_Z_-T_2_T_2}
    \left\vert z(t) -Z(t) \right\vert + \left\vert \bar{z}(t) \right\vert \lesssim \eta^{-1} Z_0^{-1} \mu_0^{\frac34} + \mu_0^{-\frac14} \vert \dot{Z}(-T_2) \vert \lesssim \left( \eta^{-1}Z_0^{-1} +\eta \right) \mu_0^{\frac34}.
\end{align}

\smallskip
\noindent \textit{Closing the bounds on $\|\epsilon\|_{H^1}$ on $[-T_2,T^\star)$}.
On $[-T_2,T^\star)$, since $\left\vert \dot{\mu}(t) - \ddot{Z}(t) \right\vert \leq C \mu_0^{\frac34} \ddot{Z}(t)$, by taking $\mu_0$ small we have $\dot{\mu}(t)> c\mu_0^2$ for some small constant $c$ and $\mu$ is an increasing function. Since $-\dot{Z}(-T_2)= \dot{Z}(T_2)\geq c \eta \mu_0$ and $\left\vert \mu(t) - \dot{Z}(t)\right\vert \leq C \mu_0^{\frac74}$, we have that $\mu(-T_2) \leq -c\eta\mu_0$ and $\mu(T_2) \geq c\eta\mu_0$. Hence, we deduce from the intermediate value theorem, assuming $T^\star$ large enough, that there exists a unique time $t_0$, such that $\mu(t_0)=0$. In particular, it implies $\mu(t) <0$ for $t<t_0$ and $\mu(t)>0$ for $t>t_0$.

To control $\epsilon$ on $[-T_2,T^\star)$, we use the functional $\mathcal{F}_-$ on the time interval $[-T_2,t_0]$ and the functional $\mathcal{F}_+$ on the interval $[t_0,T^\star)$. Arguing as in \eqref{est:F_-_-T_1} and using that $0<\rho<\frac1{32}$, we obtain the following bound
\begin{align*}
    \frac{d}{dt} \mathcal{F}_-(t) \lesssim \mu_0^{\frac{11}{4}} e^{-\frac{15}{16}Z} + \mu_0^{\frac92} e^{-\rho Z} +\mu_0^{\frac{21}{4}} \lesssim \mu_0^{\frac52 + \rho } \ddot{Z}.
\end{align*}
If $T^\star>t_0$, the same bound holds for $\frac{d}{dt}\mathcal{F}_+$ on $[t_0,T^\star)$. In both cases, an integration from $-T_2$ to $T^\star$, the fact that $\left\vert \dot{Z}(-T_2) \right\vert \lesssim \eta \mu_0$ and the estimate at $-T_2$ in \eqref{est:F_-_-T_1_-T_2} give the following bound on $\epsilon$ for $t\in[-T_2,T^\star)$ 
\begin{align}\label{eq:eps_-T_2_T_2}
    \| \epsilon (t) \|_{H^1}^2 \lesssim \eta Z_0^{-4} \mu_0^{\frac72}.
\end{align}

\smallskip
\noindent \textit{Closing the bounds on $\omega$ and $\bar{\omega}$ on $[-T_2,T^\star)$}. From \eqref{deriv:est:calK1_simpl}, with $c_1 \neq 0$, it holds with \eqref{ineq:Z_T_2}
\begin{equation}
\left|\frac{d}{dt}\mathcal{K}_1-c_Q\dot{\omega}_1 \right| + \left|\frac{d}{dt} \left( \mathcal{K}_2 - 2 \mathcal{K}_1 \right) -c_{Q} \dot{\omega}_2 \right|  \lesssim \mu_0^{\frac{9}{4}} \lesssim \mu_0^{\frac{1}{4}} \ddot{Z}.
\end{equation}
By using $\left\vert \dot{Z}(-T_2) \right\vert \lesssim \eta \mu_0$, an integration from $-T_2$ to $t$ gives, with \eqref{est:calKj}, \eqref{eq:eps_-T_2_T_2} and \eqref{eq:omega_-T_1}
\begin{align}
    \left\vert \omega_1(t) - \frac12 \omega_0 \right\vert + \left\vert \omega_2(t) + \frac12 \omega_0 \right\vert & \lesssim \left\vert \omega_1(-T_2) - \frac12 \omega_0 \right\vert + \left\vert \omega_2(-T_2) + \frac12 \omega_0 \right\vert + \eta^{\frac12} Z_0^{-2} \mu_0^{\frac54} + \mu_0^{\frac{1}{4}}\vert \dot{Z}(-T_2)\vert \nonumber \\ 
    & \lesssim \eta \mu_0^{\frac54}.\label{eq:omega_-T_2}
\end{align}

\smallskip
We now choose $\eta$ small enough so that \eqref{eq:mu_dot_Z_-T_2_T_2}, \eqref{eq:z_Z_-T_2_T_2}, \eqref{eq:eps_-T_2_T_2} and \eqref{eq:omega_-T_2} strictly improves \eqref{bootstrap_-T_1} and \eqref{bootstrap_-T_3_T_2} on $[-T_2;T^\star)$, and thus we conclude that $T^\star>T_2$.

\subsection{Bootstrap of \texorpdfstring{\eqref{bootstrap_T_2}}{(4.40)} on \texorpdfstring{$[T_2,T_1]$}{[T2,T1]}} We now prove that $T^\star= T_1$ by improving \eqref{bootstrap_T_2} on $[T_2,T^\star)$. We recall that the constant $\eta$ is fixed. In every step of the proof below, we will take $\mu_0>0$ small enough possibly depending on $C_1$. 

\smallskip
\noindent \textit{Closing the bound on $\|\epsilon\|_{H^1}$ on $[T_2,T^\star)$}. From \eqref{defi:theta}, it holds, with \eqref{defi:mAi}, \eqref{bootstrap_T_2}, \eqref{eq:simplified_z_omega}-\eqref{eq:simplified_mu} and \eqref{est:S_-T_1_T_1}-\eqref{est:T_-T_1_T_1},
\begin{align}\label{eq:bound_theta_T_2}
    \Theta(t) \lesssim C_1 \mu_0^{\frac{11}{4}} e^{-\frac{15}{16}Z(t)}.
\end{align}
By using the function $\mathcal{F}_+$ defined in \eqref{defi:F_+}, we get, with \eqref{eq:estimate_F_+}, \eqref{eq:bound_dot_z_1_dot_z_2_mu_1_mu_2},  \eqref{bootstrap_T_2}, \eqref{est:S_-T_1_T_1}-\eqref{est:T_-T_1_T_1} and \eqref{eq:bound_theta_T_2},
\begin{align*}
    \frac{d}{dt} \mathcal{F}_+(t) \lesssim C_1  \mu_0^{\frac{11}4} e^{-\frac{15}{16} Z(t)} + C_1^2 \mu_0^{\frac92} e^{-\rho Z (t)} + C_1^3 \mu_0^{\frac{21}{4}}.
\end{align*}
Using that $\dot{Z}$ is an increasing function, \eqref{ineq:dot_Z(T_2)} implies that $\dot{Z}(t) \gtrsim \mu_0$ for $t\geq T_2$ (recall that $\eta$ is fixed). By multiplying the previous term by $\dot{Z}(t) \mu_0^{-1} \gtrsim 1$ and integrating over $[T_2,t)$ with $\left\vert \mathcal{F}_+(T_2) \right\vert \lesssim \mu_0^{\frac72}$ (see \eqref{bootstrap_-T_1}), we obtain for any $t \in [T_2,T^\star)$,
\begin{align} \label{eq:eps_T_2_improved}
    \| \epsilon (t) \|_{H^1}^2 \lesssim  \mu_0^{\frac72}+C_1 \mu_0^{\frac74}e^{-\frac{15}{16}Z(T_2)} + C_1^2 \mu_0^{\frac72}e^{-\rho Z(T_2)}+C_1^3 \mu_0^{\frac{17}4} Z(t) \lesssim \mu_0^{\frac72}
\end{align}
by using that $T_1 \geq t$ in \eqref{defi:T_1} so that $Z(t)\leq Z(T_1) \lesssim Z_0$ and $Z(T_2) \ge Z_0$.

\smallskip
\noindent \textit{Closing the bound on $\bar{\mu}$ and $\bar{z}$ on $[T_2,T^\star)$}. 
From \eqref{eq:simplified_z_omega}-\eqref{eq:simplified_mu} and \eqref{eq:eps_T_2_improved}, we have
\begin{align}\label{eq:mu_mu_bar_T_2}
   \left\vert \dot{\bar{\mu}}\right\vert \lesssim \mu_0 e^{-\frac{15}{16} Z} + \mu_0^{\frac72}, \quad \text{and} \quad   \left\vert \dot{\bar{z}} - \bar{\mu} \right\vert \lesssim  Z^{-\frac12}e^{-Z}+\mu_0^{\frac74}.
\end{align}
By multiplying the right-hand side of \eqref{eq:mu_mu_bar_T_2} by $\dot{Z}(t)\mu_0^{-1} \gtrsim 1$, we get with \eqref{bootstrap_-T_1} at time $T_2$, for any $t\in [T_2;T^\star)$,
\begin{align}\label{eq:bar_mu_T_2_T_4}
    \left\vert \bar{\mu}(t) \right\vert \lesssim \mu_0^{\frac74}  \quad \text{and} \quad \left\vert \bar{z}(t)\right\vert \lesssim Z_0 \mu_0^{\frac34}.
\end{align}

\smallskip
\noindent \textit{Closing the bound on $\omega$ and $\bar{\omega}$ on $[T_2,T^\star)$}. 
From \eqref{deriv:est:calK1_simpl}, with $c_Q \neq 0$, it holds with \eqref{eq:eps_T_2_improved}
\begin{equation}
\left|\frac{d}{dt}\mathcal{K}_1-c_Q\dot{\omega}_1 \right| + \left|\frac{d}{dt} \left( \mathcal{K}_2 - 2 \mathcal{K}_1 \right) -c_{Q} \dot{\omega}_2 \right| \lesssim \mu_0^{\frac{9}{4}} .
\end{equation}
By using $\dot{Z}(t) \gtrsim \mu_0$, an integration from $T_2$ to $t$ gives, with \eqref{est:calKj}, \eqref{eq:eps_T_2_improved} and \eqref{eq:omega_-T_2}
\begin{align}\label{eq:omega_T_2}
    \left\vert \omega_1(t) - \frac12 \omega_0 \right\vert + \left\vert \omega_2(t) + \frac12 \omega_0 \right\vert \lesssim   Z_0\mu_0^{\frac54}.
\end{align}

\smallskip
To improve the bounds on $\mu$ and $z$, we introduce the time $T_4 \in (T_2,T_1)$ defined as the unique time such that
\begin{align}\label{defi:T_4}
    \mu_0^2 = Z_0 M Z(T_4)^{-\frac12} e^{-Z(T_4)},
\end{align}
where $M>10$ is a constant to be chosen later. Note that $T_4$ is well-defined thanks to the intermediate value theorem.

\smallskip
\noindent \textit{Closing the bound on $\mu$ and $z$ on $[T_2,T^\star)$ with $T^\star<T_4$}. 
As in the proof of \eqref{bootstrap_-T_3_T_2}, we use the functional $H$ defined in \eqref{defi:H(t)}. We claim that
\begin{align}\label{eq:H(T_2)}
    \left\vert H(T_2) - 2\mu_0^2 \right\vert \lesssim \mu_0^{\frac{11}4}.
\end{align}
Indeed, as in the proof of \eqref{eq:H(-T_1)}, we use \eqref{bootstrap_-T_3_T_2} at time $t=T_2$ to obtain
\begin{align*}
    \left\vert \mu(T_2)^2 - \dot{Z}(T_2)^2 \right\vert \lesssim \mu_0^{\frac{11}4}
\end{align*}
and with \eqref{propo:approx_e^-Z.1} and \eqref{bootstrap_-T_3_T_2} at time $t=T_2$
\begin{align*}
    \left\vert \int \left( Q(\cdot + z(T_2), \cdot) - Q(\cdot + Z(T_2), \cdot) Q^2 \right) \right\vert \lesssim \mu_0^{\frac34}  Z(T_2)^{-\frac12} e^{- Z(T_2)} \lesssim \mu_0^{\frac{11}4},
\end{align*}
which concludes the proof of \eqref{eq:H(T_2)}.

The time derivative of $H$ is bounded, by using \eqref{eq:H'-T_1}, \eqref{eq:simplified_z_omega}-\eqref{eq:simplified_mu}, \eqref{eq:bound_dot_z_1_dot_z_2_mu_1_mu_2} and \eqref{eq:eps_T_2_improved}
\begin{align*}
\left\vert \dot{H} \right\vert \lesssim \mu_0^2 e^{-\frac{15}{16}Z}+\mu_0^{\frac74}Z^{-\frac12}e^{-Z}+\mu_0^{\frac92}.
\end{align*}
After multiplying the right-hand side of the above inequality by $\dot{Z}(t)\mu_0^{-1} \gtrsim 1$, an integration from $T_2$ to $t$ provides, with $Z(t) \leq Z(T_1) \lesssim Z_0$ and \eqref{eq:H(T_2)},
\begin{align*}
    \left\vert H(t) - 2 \mu_0^2 \right\vert \lesssim  \mu_0^{\frac{11}4} \lesssim M Z_0 \mu_0^{\frac34} Z(t)^{-\frac12} e^{-Z(t)},
\end{align*}
where we used the definition of $T_4$ in \eqref{defi:T_4} in the last step. 
From \eqref{ineq:z-Z}, \eqref{bootstrap_T_2} and \eqref{propo:approx_e^-Z.1}, we have 
\begin{align*}
    Z(t)^{-\frac12} e^{-Z(t)} \lesssim \frac{2}{\langle \Lambda Q, Q \rangle} \int_{\mathbb{R}^2} Q(x+z(t),y) Q^2(x,y) dx dy ,
\end{align*}
which, combined to \eqref{eq:simplified_z_omega}, \eqref{eq:eps_T_2_improved}, \eqref{bootstrap_T_2}, \eqref{defi:T_4} and $\dot{Z}(t) \lesssim \mu_0$, yields
\begin{align*}
    \left\vert \mu(t)^2 -\dot{z}(t)^2 \right\vert
        & \leq \left\vert \mu(t)-\dot{z}(t) \right\vert \left( \left\vert \dot{z}(t)- \mu(t) \right\vert + 2 \left\vert \mu(t) - \dot{Z}(t) \right\vert +2 \dot{Z}(T_1) \right) \\
        & \lesssim  \mu_0^{\frac{11}4} \lesssim M Z_0 \mu_0^{\frac34} Z(t)^{-\frac12} e^{-Z(t)}. 
\end{align*}
It follows from the three previous inequalities that, by defining $\nu := C M Z_0 \mu_0^{\frac34} $ for a constant $C$ large enough,
\begin{align*}
\MoveEqLeft
    \frac12 \dot{z}(t)^2 + \frac{2(1-\nu)}{\langle \Lambda Q, Q \rangle} \int_{\mathbb{R}^2} Q(x+z(t),y) Q^2(x,y) dx dy \\
    & \qquad \leq 2\mu_0^2 \leq \frac12 \dot{z}(t)^2 + \frac{2(1+\nu)}{\langle \Lambda Q, Q \rangle} \int_{\mathbb{R}^2} Q(x+z(t),y) Q^2(x,y) dx dy.
\end{align*}
With \eqref{bootstrap_-T_3_T_2} at time $T_2$, we can apply Proposition \ref{propo:H_positive_times} with $\epsilon_0 = C\mu_0^{\frac34}$ to obtain for any $t\in [T_2,T^\star)$,
\begin{align}\label{eq:z_T_2_T_4}
    \left\vert z(t) - Z(t) \right\vert \lesssim M Z_0 \mu_0^{\frac34}.
\end{align}
Thus from \eqref{eq:gamma_1_z_Z}, \eqref{eq:simplified_mu} and \eqref{eq:eps_T_2_improved}, we have
\begin{align*}
   \left\vert \dot{\mu} -\ddot{Z} \right\vert = \left\vert \dot{\mu} + 2\gamma_1(Z) \right\vert \lesssim MZ_0 \mu_0^{\frac34} Z^{-\frac12} e^{-Z} +\mu_0 e^{-\frac{15}{16}Z} + \mu_0^{\frac72}.
\end{align*}
We multiply the right-hand side by $\dot{Z}(t)\mu_0^{-1} \gtrsim 1$, integrate from $T_2$ to $t$ and use  \eqref{bootstrap_-T_3_T_2} to find
\begin{align}\label{eq:mu_T_2_T_4}
    \left\vert \mu(t) - \dot{Z}(t) \right\vert \lesssim M Z_0 \mu_0^{\frac{7}{4}}.
\end{align}
By choosing $C_1>1$ large enough dependent on $M$, using \eqref{eq:eps_T_2_improved}, \eqref{eq:bar_mu_T_2_T_4}, \eqref{eq:omega_T_2}, \eqref{eq:z_T_2_T_4} and \eqref{eq:mu_T_2_T_4}, we strictly improve the bootstrap inequalities \eqref{bootstrap_T_2}. Hence, we conclude that $T^\star > T_4$.

\smallskip
\noindent \textit{Closing the bound on $\mu$ and $z$ on $[T_4,T^\star)$}. 
On this time interval we will need to fix $M$ large enough. From \eqref{eq:gamma_1_z_Z} and \eqref{bootstrap_T_2}, we have
\begin{align*}
    \left\vert \gamma_1(z) - \gamma_1(Z)\right\vert \lesssim \vert z- Z \vert \left\vert \gamma_1'(Z) \right\vert \lesssim C_1 Z_0^2 \mu_0^{\frac34} Z^{-\frac12} e^{-Z},
\end{align*}
and thus from \eqref{eq:simplified_mu} and \eqref{eq:eps_T_2_improved},
\begin{align*}
    \left\vert \dot{\mu} - \ddot{Z} \right\vert =\left\vert \dot{\mu} + 2\gamma_1(Z) \right\vert \lesssim \mu_0e^{-\frac{15}{16}Z}+\mu_0^{\frac72}+C_1 Z_0^2 \mu_0^{\frac34} Z^{-\frac12} e^{-Z}.
\end{align*}
By multiplying the right-hand side by $\dot{Z}(t)\mu_0^{-1} \gtrsim 1$, we integrate from $T_4$ to $t\in [T_4, T^\star)$ with the initial condition \eqref{eq:mu_T_2_T_4} at time $T_4$ to obtain
\begin{align*}
    \left\vert \mu(t) - \dot{Z}(t)\right\vert \lesssim \left( C_1 M^{-1} + M \right) Z_0 \mu_0^\frac74.
\end{align*}
As a consequence, using \eqref{eq:simplified_mu}, we have
\begin{align*}
    \left\vert \dot{z}(t) - \dot{Z}(t) \right\vert \lesssim \left( C_1 M^{-1} + M\right) Z_0 \mu_0^{\frac74}.
\end{align*}
We multiply the right-hand side by $\dot{Z}(t)\mu_0^{-1} \gtrsim 1$, use \eqref{eq:z_T_2_T_4}  and integrate from $T_4$ to $t \in [T_4,T^\star)$ to find
\begin{align*}
    \left\vert z(t) - Z(t) \right\vert \lesssim \left( C_1 M^{-1} + M \right) Z_0^2 \mu_0^{\frac34}.
\end{align*}
Now, we choose $M$ (independent of $C_1$) large enough and then $C_1$ large so that the two following inequalities are satisfied
\begin{align*}
    \left\vert \mu(t) - \dot{Z}(t) \right\vert \leq \frac{C_1}{2} Z_0 \mu_0^{\frac74}, \quad \left\vert z(t) - Z(t) \right\vert \leq \frac{C_1}{2} Z_0^2 \mu_0^{\frac34}.
\end{align*}

Increasing $C_1$ if necessary in \eqref{eq:eps_T_2_improved}, \eqref{eq:bar_mu_T_2_T_4} and \eqref{eq:omega_T_2} strictly improves the bounds of \eqref{bootstrap_T_2} on $[T_4,T^\star)$. As a consequence, $T^\star=T_1$, which concludes the proof of Proposition \ref{propo:bootstrap_-T_T}.

\section{Dynamics of the two-solitary waves structure} \label{sec:dynamics}

The aim of this section is to study the dynamics of the two-solitary waves structure $v$ solution to \eqref{ZK:sym} behaving at time $t=-\infty$ as a sum of two well-prepared symmetric solitary waves (see \eqref{multi_sol:-infty}). 
In the first subsection, we prove a quantitative orbital stability result on a finite time interval as long as the two solitary waves remain far enough. In Subsection \ref{sub:sec:maintheo}, we prove Theorem \ref{maintheo} by describing the evolution of the solution $v$ from $-\infty$ to $+\infty$. Finally, Subsection \ref{sub:sec:theo:stability} is devoted to the proof of Theorem \ref{theo:stability}.

In all this section, for $\mu_0$ small, we use the function $Z$ defined in Lemma \eqref{lemm:pointwise_Z} with the asymptote $2\mu_0t +l$ at $+\infty$ and its properties stated in Subsection \ref{sub:sec:Z}.

\subsection{Orbital stability of the \texorpdfstring{$2$}{2}-solitary waves structure for large negative time}
We state below the main result of this subsection. 

\begin{propo}\label{propo:quantified_orbital_stability}
There exist $\nu_4^\star>0$ and $C>0$ such that the following holds. Let $\mu_0\in (0, \nu_4^\star)$ and $\kappa$ be a continuous increasing function $\kappa : \mathbb{R}_- \rightarrow (0,1)$ satisfying 
\begin{align} \label{def:kappa}
    \lim_{t\rightarrow -\infty} t \kappa(t)=0.
\end{align}
Let $Z$, $T_1$, $l$ and $K$ be defined as in Lemma \ref{lemm:pointwise_Z} and Lemma \ref{lemm:uniform_Z}. There exists a time $\tilde{t}_0$ satisfying
\begin{align}\label{def:t_0}
    -\tilde{t}_0 <\min(-T_1, -\frac{l+K}{\mu_0})
\end{align}
such that, if $t_0> \tilde{t}_0$, $\vert \omega_0 \vert \leq \mu_0$ and $w \in \mathcal{C}([-t_0,-T_1], H^1(\mathbb{R}^2))$ is a solution to \eqref{ZK:sym} satisfying
\begin{align}\label{cond:init_stability}
    \left\| w(-t_0) - \sum_{i=1}^2 Q_{1+(-1)^i\mu_0} \left(\cdot+ \frac{(-1)^i}{2}Z(-t_0),\cdot + \frac{(-1)^i}{2} \omega_0 \right) \right\|_{H^1} \leq \kappa(-t_0) \mu_0,
\end{align}
then there exists a unique function $\Gamma=(z_1,z_2,\omega_1,\omega_2,\mu_1,\mu_2) \in \mathcal{C}^1([-t_0,-T_1] : \mathbb{R}^6)$, such that by defining $\epsilon$ as in \eqref{def:epsilon}, we have, for any $t\in [-t_0,-T_1]$ and for any $i = 1,2$,
    \begin{equation}\label{estimates:stability}
        \begin{aligned}
            & \left\| \epsilon(t) \right\|_{H^1} + \left\vert \mu_i(t) +\frac{(-1)^i}{2} \dot{Z}(t) \right\vert  \leq C \left( \kappa(-t_0) \mu_0 + e^{-\frac{15}{16} Z(t)} \right), \\
            & \left\vert z_i(t) + \frac{(-1)^i}{2}Z(t) \right\vert + \left\vert \omega_i(t) + \frac{(-1)^i}{2} \omega_0 \right\vert \leq C \left(\kappa(-t_0) t_0 + Z(t)^{-\frac14} e^{-\frac12 Z(t)} \right).
        \end{aligned}
    \end{equation}
    Moreover, the following inequality holds, for any $i=1,2$ and $t \in [-t_0,T_1]$
     \begin{align}\label{est:z_i_omega_i_-t_0_-T_1}
         \left\vert \dot{z}_i(t)  - \mu_i(t) \right\vert + \vert \dot{\omega}_i(t) \vert \leq C \left( \kappa(-t_0) \mu_0 + e^{-\frac{15}{16} Z(t)} \right).
     \end{align}
\end{propo}

\begin{rema} 
The neighbourhood of the $2$-solitary waves for which the stability result applies is quantified by the function $\kappa$ introduced in Proposition \ref{propo:quantified_orbital_stability}.
\end{rema}

\begin{proof}
For any $t\geq -t_0$, as long as $w(t)$ stays close to a rescaled approximation $V(\Gamma)$, we introduce the decomposition $(\Gamma(t), \epsilon(t))$ of the solution $w(t)$ as constructed in Lemma \ref{propo:modulation}. Let us decompose the solution at time $t=-t_0$. From assumption \eqref{cond:init_stability}, set \begin{align*}
\Gamma_0:=\left( \frac{1}{2}Z(-t_0),-\frac{1}{2}Z(-t_0), \frac{\omega_0}2,-\frac{\omega_0}2, \frac12 \dot{Z}(-t_0),-\frac12 \dot{Z}(-t_0)\right).    
\end{align*}
Thus, by continuity of $w$ in $H^1$, we obtain from \eqref{est:VA:H2} and the triangle inequality that 
\begin{align*}
    w(t) \in \mathcal{U}_{ \frac12 Z(t_0),2\mu_0,\eta}, \quad \text{where} \quad \eta = \kappa(-t_0) \mu_0 + C e^{-\frac{15}{16} Z(-t_0)} ,
\end{align*}
for $t$ in a small neighbourhood of $t_0$. By choosing the parameter $\nu_4^\star$ small enough, one can then apply Lemma \ref{propo:modulation}. As noticed in Remark \ref{rema:modulation_pointwise}, we have 
\begin{align}
    & \| \epsilon (-t_0) \|_{H^1} \lesssim \kappa(-t_0)\mu_0 + e^{- \frac{15}{16} Z(-t_0)}, \nonumber\\
    & \sum_{i=1}^2 \left( \left\vert z_i(-t_0)+ \frac{(-1)^i}{2} Z(-t_0)  \right\vert + \left\vert \omega_i(-t_0)+\frac{(-1)^i}{2}\omega_0 \right\vert  + \left\vert \mu_i(-t_0) + \frac{(-1)^i}{2} \dot{Z}(-t_0)\right\vert \right) \nonumber \\
        & \quad \lesssim \kappa(-t_0)\mu_0 + e^{- \frac{15}{16} Z(-t_0)} . \label{est_bootstrap_mu_i_-t_0}
\end{align}
In particular, by the definition of $\mathcal{F}_-$ in \eqref{defi:F_-}, we have
\begin{align}\label{est_bootstrap_F_-_-t_0}
    \mathcal{F}_- (-t_0) \lesssim \left( \kappa(-t_0)\mu_0 + e^{-\frac{15}{16} Z(-t_0)} \right)^2.
\end{align}

We now prove the rest of the proposition by a bootstrap argument. For a positive constant $C_\star>0$ to be chosen later, we introduce the bootstrap estimates
\begin{align}
     & \left\| \epsilon(t) \right\|_{H^1} + \sum_{i=1}^2 \left\vert \mu_i(t) +\frac{(-1)^i}{2} \dot{Z}(t) \right\vert \leq C_\star \left( \kappa(-t_0) \mu_0 + e^{-\frac{15}{16} Z(t)} \right), \label{estimates:stability_bootstrap} \\
     & \sum_{i=1}^2 \left( \left\vert z_i(t) + \frac{(-1)^i}{2}Z(t) \right\vert + \left\vert \omega_i(t) + \frac{(-1)^i}2 \omega_0 \right\vert \right) \leq C_\star \left( \kappa(-t_0) t_0 + Z(t)^{-\frac14} e^{-\frac12 Z(t)} \right). \label{estimates:stability_bootstrap_z_i}
\end{align}
We define
\begin{equation*}
    T^\star := \sup \left\{ T \in [-t_0,-T_1) \, : \,
        \text{for all }t \in [-t_0,T), \quad w(t) \in \mathcal{U}_{2Z_1^\star, \frac12 \nu_1^\star ,\frac12 \sigma^\star} \text{ and } \eqref{estimates:stability_bootstrap}- \eqref{estimates:stability_bootstrap_z_i} \text{ hold}
     \right\},
\end{equation*}
where $Z_1^\star$, $\nu_1^\star$ and $\sigma^\star$ are the parameters defined in Lemma \ref{propo:modulation}. Observe from \eqref{cond:init_stability} and the continuity of $t\mapsto w(t)$ in $H^1$ that $T^\star$ is well-defined. 

To prove that $T^\star=-T_1$, for $C_\star$ large enough, we argue by contradiction and assume that $T^\star<-T_1$. Then estimates \eqref{estimates:stability_bootstrap}-\eqref{estimates:stability_bootstrap_z_i} hold on $[-t_0,T^\star)$. We will prove that these estimates can be strictly improved on $[-t_0,T^\star)$ to obtain a contradiction.

On the interval $[-t_0,T^\star)$, we have, recalling $z(t)=z_1(t)-z_2(t)$,
\begin{align} \label{estimates:stability_bootstrap_z}
     \left\vert z(t)-Z(t) \right\vert \lesssim C_\star \left( \kappa(-t_0) t_0 + Z(t)^{-\frac14} e^{-\frac12 Z(t)} \right).
\end{align}
Note that by choosing $\tilde{t}_0$ large and $\nu_3^\star$ both dependent on $C_\star$, we obtain $\vert z(t)-Z(t) \vert \leq 1$. Hence, we deduce from the mean value theorem
\begin{align}\label{est:z_Z_-infty}
    \left\vert z(t)^{-\frac12}e^{-z(t)}- Z(t)^{-\frac12} e^{-Z(t)}\right\vert \lesssim \vert z(t)-Z(t) \vert Z(t)^{-\frac12}e^{-Z(t)}.
\end{align}
As a consequence, it follows from Proposition \ref{theo:V_A} (v) and (vi) that, for $t\in [-t_0,T^\star)$,
\begin{align*}
    & \left\| S(t) \right\|_{H^2} \lesssim Z(t)^{-\frac12} e^{-Z(t)}+ e^{-\frac{15}{16} Z(t)} \sum_{i=1}^2 \left(\vert \mu_i(t) \vert+ \vert \omega_i(t) \vert \right) , \\
    & \left\| T(t) \right\|_{H^1} + \left\vert \left\langle S,\partial_x R_i \right\rangle \right\vert + \left\vert \left\langle S,\partial_y R_i \right\rangle \right\vert + \left\vert \left\langle S,\Lambda R_i \right\rangle \right\vert \\
    & \quad \lesssim e^{-\frac{31}{16} Z(t)} + e^{-\frac{15}{16} Z(t)} \sum_{i=1}^2 \left( \vert \mu_i(t) \vert+ \vert \omega_i (t) \vert + \vert \dot{z}_i(t) \vert \right), 
\end{align*}
which provides with \eqref{estimates:stability_bootstrap}, \eqref{estimates:stability_bootstrap_z_i} and recalling that $\vert \dot{Z}(t) \vert \leq 2 \mu_0$ (see Lemma \ref{lemm:pointwise_Z}),
\begin{align}
    & \| S(t) \|_{H^2} \lesssim Z(t)^{-\frac12} e^{-Z(t)} + \left( \mu_0 +C_\star \kappa (-t_0) t_0 \right) e^{-\frac{15}{16} Z(t)}, \label{est:S_-T_T} \\
    & \left\| T(t) \right\|_{H^1} + \left\vert \left\langle S,\partial_x R_i \right\rangle \right\vert + \left\vert \left\langle S,\partial_y R_i \right\rangle \right\vert + \left\vert \left\langle S,\Lambda R_i \right\rangle \right\vert \nonumber \\
    & \quad \lesssim \mu_0  e^{-\frac{15}{16} Z(t)} + C_\star \left( Z(t)^{-\frac12} e^{-\frac{23}{16} Z(t)} + \kappa(-t_0) t_0 Z(t)^{-\frac12} e^{- Z(t)} \right) + e^{-\frac{15}{16} Z(t)} \sum_{i=1}^2 \vert \dot{z}_i (t) \vert. \label{est:T_-T_T}
\end{align}

From \eqref{eq:ortho_translation}-\eqref{eq:ortho_velocity} combined with \eqref{Def:gammai}, \eqref{eq:Z}, \eqref{estimates:stability_bootstrap}, \eqref{estimates:stability_bootstrap_z}-\eqref{est:T_-T_T} and \eqref{propo:approx_e^-Z.1}, we obtain the quantified dynamical system, for $i=1,2$,
\begin{align}
    \left\vert \dot{z}_i(t) + \frac{(-1)^i}{2} \dot{Z}(t) \right\vert & + \left\vert \dot{\omega_i}(t) \right\vert \lesssim C_\star \left( \kappa(-t_0) \mu_0 + e^{-\frac{15}{16} Z(t)} \right), \label{dynamical_system_z_i_-T_1} \\
    \left\vert \dot{\mu}_i(t) + \frac{(-1)^i}{2} \ddot{Z}(t) \right\vert & \lesssim \mu_0 e^{-\frac{15}{16} Z(t)} + C_\star Z(t)^{-\frac12} e^{-\frac{23}{16} Z(t)} + C_\star \kappa(-t_0) \mu_0 e^{-\frac{15}{16} Z(t)} \nonumber \\
        & \quad + C_\star \kappa(-t_0) t_0 Z(t)^{-\frac12} e^{- Z(t)} + C_\star^2 \left( \kappa(-t_0) \mu_0 \right)^2 - \dot{Z}(t) e^{-\frac{15}{16} Z(t)}. \label{dynamical_system_mu_i_-T_1} 
\end{align}

First, we estimate $\mu_i$, for $i=1,2$. By multiplying the three first terms of the right-hand side of \eqref{dynamical_system_mu_i_-T_1} by $-\dot{Z}(t)\mu_0^{-1} \geq 1$ (see \eqref{ineq:dot_Z(T_1)}) and by integrating from $-t_0$ to $t$, we have combining with \eqref{est_bootstrap_mu_i_-t_0},
\begin{align*}
    \left\vert \mu_i(t) + \frac{(-1)^i}{2} \dot{Z}(t) \right\vert
    & \lesssim \kappa(-t_0) \mu_0 + e^{-\frac{15}{16} Z(t)} + C_\star \left( \frac{e^{-\frac{23}{16} Z(t)}}{\mu_0 Z^{\frac12} (t)} + \kappa(-t_0) e^{-\frac{15}{16} Z(t)} + \frac{\kappa(-t_0)t_0 e^{- Z(t)}}{\mu_0 Z(t)^{\frac12}} \right) \\
        & \quad  + C_\star^2 \frac{ Z(-t_0)}{\mu_0} \kappa(-t_0)^2 \mu_0^2.
\end{align*}
Now we have from \eqref{def:t_0} and \eqref{def:kappa} 
\begin{align} \label{est:Z(-t0)}
Z(-t_0) \lesssim t_0 \mu_0, \quad \text{so that} \quad Z(-t_0)\kappa(-t_0) \lesssim \mu_0,
\end{align}
and from  the definition of $T_1$ in \eqref{defi:T_1}, \eqref{est:mu0_Z0} and $0<\rho<\frac1{32}$, for any $t \in [-t_0,T^{\star})$, 
\begin{align} \label{est:e-Z(t)}
e^{-Z(t)} \lesssim e^{-\frac{Z_0}{\rho}} \lesssim \mu_0^{16} .
\end{align}
Hence, we deduce by choosing $\nu_3^\star$ small enough and $C_\star$ large enough that
\begin{align}\label{eq:bound_mu_i_improved}
    \left\vert \mu_i(t) + \frac{(-1)^i}{2} \dot{Z}(t) \right\vert \leq \frac{C_\star}{4} \left( \kappa(-t_0) \mu_0 + e^{-\frac{15}{16} Z(t)} \right).
\end{align}

Now, we turn to the bounds on $z_i$ and $\omega_i$ for $i=1,2$. By multiplying \eqref{dynamical_system_z_i_-T_1} by $- \dot{Z}(t) \mu_0^{-1}  \ge 1$, an integration from $-t_0$ to $t$ yields, using the initial bound \eqref{est_bootstrap_mu_i_-t_0} and \eqref{est:Z(-t0)}-\eqref{est:e-Z(t)},
\begin{align*}
\MoveEqLeft
    \left\vert z_i(t) + \frac{(-1)^i}{2} Z(t) \right\vert + \left\vert \omega_i(t) +\frac{(-1)^i}{2}\omega_0 \right\vert \\
    & \qquad \lesssim \kappa(t_0)\mu_0 + e^{-\frac{15}{16} Z(-t_0)}+ C_\star \left( \kappa(-t_0) Z(-t_0) +  \frac{1}{\mu_0} e^{-\frac{15}{16} Z(t)} \right).
\end{align*}
Therefore, by choosing $ \mu_0 \leq \nu_3^\star$ small enough, we conclude that  
\begin{align*}
    \left\vert z_i(t) + \frac{(-1)^i}{2} Z(t) \right\vert +\left\vert \omega_i(t) - \frac{(-1)^i}{2}\omega_0 \right\vert \leq \frac{C_\star}{4} \left( \kappa(-t_0) t_0 + Z(t)^{-\frac14} e^{-\frac12 Z(t)} \right), 
\end{align*}
which strictly improves the bound \eqref{estimates:stability_bootstrap_z_i}.

To estimate $\epsilon(t)$, we use the functional $\mathcal{F}_-$ defined in \eqref{defi:F_-}. From the definition of $\Theta$ in \eqref{defi:theta}, using that $\kappa(-t_0)t_0$ is uniformly bounded on $\mathbb{R}$, combining \eqref{ineq:dot_Z(T_1)}, \eqref{estimates:stability_bootstrap}, \eqref{est:S_-T_T}-\eqref{est:T_-T_T},  \eqref{est:e-Z(t)} and taking $\mu_0$ small enough, we obtain for any $t \in [-t_0,T_1]$
\begin{align}\label{est:theta_t_0_-T_1} 
    \Theta(t) \lesssim C_\star \mu_0 e^{-\frac{15}{8} Z(t)} + C_\star \kappa(-t_0) \mu_0^2 e^{-\frac{15}{16} Z(t)} + C_\star^2 \kappa(-t_0)^2 \mu_0^2 e^{-\frac{15}{16} Z(t)}.
\end{align}

We obtain combining \eqref{eq:estimate_F_-}, \eqref{estimates:stability_bootstrap}, \eqref{estimates:stability_bootstrap_z}-\eqref{est:z_Z_-infty},\eqref{est:theta_t_0_-T_1}, 
\begin{align*}
    \frac{d}{dt} \mathcal{F}_-(t) 
        & \lesssim C_\star \mu_0 e^{-\frac{15}{8} Z(t)} + C_\star \kappa(-t_0) \mu_0^2 e^{-\frac{15}{16} Z(t)} + C_\star \kappa(-t_0)^2 t_0 \mu_0^2 e^{-\frac{15}{16} Z(t)} \\
        & + C_\star^2 \mu_0 \kappa(-t_0) t_0 e^{-\frac{15}{8} Z(t)} + C_\star^3 \kappa(-t_0)^3 t_0 \mu_0^2 e^{-\frac{15}{16} Z(t)} + C_\star^2 \kappa(-t_0) \mu_0^3 e^{-\rho Z(t)} + C_\star^3 \kappa(-t_0)^3 \mu_0^3
\end{align*}

We integrate from $-t_0$ to $t$ after multiplying the right-hand side of the above inequality by $-\dot{Z}(t) \mu_0^{-1} >1$ to obtain
\begin{align*}
    & \mathcal{F}_-(t) - \mathcal{F}_-(-t_0) \lesssim C_\star e^{-\frac{15}{8} Z(t)} + C_\star \kappa(-t_0) \mu_0 e^{-\frac{15}{16} Z(t)} + C_\star^2 \kappa(-t_0)^2 t_0 \mu_0 e^{-\frac{15}{16} Z(t)} \\
    & \quad + C_\star^2 \kappa(-t_0) t_0 e^{-\frac{15}{8} Z(t)} + C_\star^3 \kappa(-t_0)^3 t_0 \mu_0 e^{-\frac{15}{16} Z(t)} + C_\star^2 \kappa(-t_0)^2 \mu_0^2 e^{-\rho Z(t)} + C_\star^3 \kappa(-t_0)^3 \mu_0^2 Z(-t_0).
\end{align*}

Observe from Young's inequality that 
\begin{align*}
    C_\star^2 \kappa(-t_0)^2 t_0 \mu_0 e^{-\frac{15}{16} Z(t)} + C_\star^2 \kappa(-t_0) t_0 e^{-\frac{15}{8} Z(t)} \lesssim C_\star^2 \kappa(-t_0) t_0 \left( \left( \kappa(-t_0 \mu_0 \right)^2 + e^{-\frac{15}{8} Z(t)} \right).
\end{align*}
Moreover, we have by using \eqref{est:e-Z(t)} by choosing $\mu_0$ small depending on $\rho$ and $C_\star$,
\begin{align*}
     C_\star^2 \kappa(-t_0)^2 \mu_0^2 e^{-\rho Z(t)} \lesssim C_\star^2 \kappa(-t_0)^2 \mu_0^2 \mu_0^{16 \rho} \lesssim C_\star \left( \kappa(t) \mu_0 \right)^2.
\end{align*}
Thus it follows from \eqref{est:Z(-t0)}, \eqref{est:e-Z(t)}, the above estimates, the coercivity of $\mathcal{F}_-$ in \eqref{coercivity_F}, the bound at time $-t_0$ of $\mathcal{F}_-$ in \eqref{est_bootstrap_F_-_-t_0} and taking $\mu_0 \leq \nu_4^\star$ small enough that there exists a positive universal constant $C$ such that
\begin{align*}
    \| \epsilon (t) \|_{H^1}^2 \leq \left( C C_\star + C C_\star^2 \kappa(-t_0)t_0\right) \left(\kappa(-t_0)^2 \mu_0^2 + e^{-\frac{15}{8} Z(t)} \right).
\end{align*}
Therefore, by taking $t_0$ large enough such that $C \kappa(-t_0) t_0 \leq \frac18$ and $C_\star$ large enough such that $C_\star > 8C $, and using \eqref{eq:bound_mu_i_improved} the bound \eqref{estimates:stability_bootstrap} is strictly improved on $[-t_0,T^\star)$.

Thus, by a standard continuity argument, we have just contradicted the definition of $T^\star$, which concludes the proof of the first part of the proposition. Note that \eqref{est:z_i_omega_i_-t_0_-T_1} is a direct consequence of the quantified dynamical system \eqref{dynamical_system_z_i_-T_1} by choosing $C$ large enough.
\end{proof}

\subsection{Proof of Theorem \ref{maintheo}.} \label{sub:sec:maintheo}

Let $v$ be the pure $2$-solitary waves at $-\infty$, \textit{i.e.} $v$ is the unique solution to \eqref{ZK:sym} satisfying \eqref{multi_sol:-infty}. We refer to \cite{Val21} for the existence and uniqueness of the function $v$. In this section, we rely on the approximate solution $V$ constructed in Section \ref{sec:construction_V} to describe the evolution of $v$ from $-\infty$ to $+\infty$. 

Recall that the positive time $T_1$ defined in \eqref{defi:T_1} satisfies $Z(T_1)= \rho^{-1} Z_0$, where $\rho \in (0;\frac{1}{32})$ has been fixed in the proof of Proposition \ref{propo:bootstrap_-T_T}. Notice that, by \eqref{eq:comparison_l_1_e^Z_0} and the bound on $\rho$ in \eqref{defi:T_1}, it holds
\begin{align}
    \forall t \in \mathbb{R}, & \quad Z(t)^{-\frac12} e^{-Z(t)} \leq Z_0^{-\frac12} e^{-Z_0} \lesssim \mu_0^{2}, \label{est:Z(t)_R}\\
    \forall \vert t \vert \geq T_1, & \quad Z(t)^{-\frac12} e^{-Z(t)} \leq \mu_0^{10}.\label{est:Z(t)_T_1_infty}
\end{align}

\begin{coro}\label{coro:quantified_orbital_stability}
Let $\nu_4^\star$ be defined in Proposition \ref{propo:quantified_orbital_stability}. There exists a constant $C>0$ such that the following holds.
Let $\mu_0 <\nu_4^\star$ and $v(t)$ be the unique solution of \eqref{ZK:sym} satisfying \eqref{multi_sol:-infty}.
Then, there exists a function $\Gamma: (-\infty,-T_1] \rightarrow \mathbb{R}^6$ such that the orthogonality conditions \eqref{eps:ortho} hold on $(-\infty,-T_1]$, and for all $t\leq -T_1$, we have
\begin{align}
    & \left\| v(t) - V(\Gamma(t)) \right\|_{H^1(\mathbb{R}^2)} + \sum_{i=1}^2 \left( \left\vert \mu_i(t) +\frac{(-1)^i}{2} \dot{Z}(t) \right\vert + \left\vert \dot{z}_i(t)  - \mu_i(t) \right\vert + \vert \dot{\omega}_i(t) \vert \right) \leq C e^{-\frac{15}{16} Z(t)}, \label{ineq:-infty_-T_1.1}\\
    & \sum_{i=1}^2 \left( \left\vert z_i(t) + \frac{(-1)^i}{2}Z(t) \right\vert + \left\vert \omega_i(t) + \frac{(-1)^i}{2} \omega_0 \right\vert \right) \leq C Z(t)^{-\frac14} e^{-\frac12 Z(t)}. \label{ineq:-infty_-T_1.2}
\end{align}
\end{coro}

\begin{proof}
It was proved in \cite{Val21} that there exist two constants $C>0$ and $\delta>0$, and a time $T_0>0$ such that, for any $t\in (-\infty,-T_0)$, 
\begin{align*}
    \left\| v(t) - \sum_{i=1}^2 Q_{1+ (-1)^{i}\mu_0} \left( \cdot - \frac{(-1)^i}{2}\left(2\mu_0 t +l \right) , \cdot + \frac{(-1)^i}{2} \omega_0 \right) \right\|_{H^1} \leq C e^{\delta t}.
\end{align*}
By the triangle inequality and Lemma \ref{lemm:pointwise_Z} (ii), there exists $0<\delta_1 \le \delta$ such that for any $t\in (-\infty,-T_0)$, 
\begin{align*}
    \left\| v(t) - \sum_{i=1}^2 Q_{1+ (-1)^{i}\mu_0} \left( \cdot + \frac{(-1)^i}{2}Z(t) , \cdot + \frac{(-1)^i}{2} \omega_0 \right) \right\|_{H^1} \leq C e^{\delta_1 t} .
\end{align*}
By defining the function $\kappa(t):= e^{\frac{\delta_1}2 t}$ and $\tilde{T}_0>T_0$ such that $C e^{-\frac{\delta_1}2 \tilde{T}_0} \le \mu_0 < \nu_4^\star$, Proposition \ref{propo:quantified_orbital_stability} yields for any $-t_0<\min(-\tilde{t}_0,-\tilde{T}_0)$ the existence of a unique function $\Gamma$ defined on $(-t_0,-T_1]$ such that estimates \eqref{estimates:stability} hold on $(-t_0,-T_1]$. By uniqueness, the function $\Gamma$ is well-defined on $(-\infty,-T_1]$ and estimates \eqref{estimates:stability} hold for any $t_0$ large. This remark concludes the proof of Corollary \ref{coro:quantified_orbital_stability} by letting $t_0 \to +\infty$ in \eqref{estimates:stability}-\eqref{est:z_i_omega_i_-t_0_-T_1}.
\end{proof}

We now prove a sharper version of Theorem \ref{maintheo} on $(-\infty,T_1]$ by comparing the solution $v$ to the approximate solution $V$.
\begin{propo}\label{mainpropo}
There exist some positive constants $C$ and $\nu_5^\star$ such that the following is true. Let $0<\mu_0<\nu_5^\star$ and let $v$ be the unique solution to \eqref{ZK:sym} satisfying \eqref{multi_sol:-infty}. Then, there exists a unique $\Gamma=(z_1,z_2,\omega_1,\omega_2,\mu_1,\mu_2) \in C^1( (-\infty,T_1] : \mathbb R^6)$ that satisfies for all $t \leq T_1$, by defining $\epsilon(t,\cdot):=v(t,\cdot)-V(\Gamma(t), \cdot)$,
\begin{equation}\label{mainpropo.1}
	\begin{aligned}
		& \| \epsilon(t) \|_{H^1} \leq C \mu_0^{\frac74}, && \left\vert \mu(t) - \dot{Z}(t) \right\vert + \left\vert \dot{z}_i(t) - \mu_i(t) \right\vert  + \left\vert \dot{\omega}_i(t) \right\vert \leq C Z_0 \mu_0^{\frac74}, && \left\vert \bar{\mu}(t) \right\vert \leq C \mu_0^{\frac74}  \\
		& \left\vert z(t) - Z(t) \right\vert \leq C Z_0^2 \mu_0^{\frac34}, && \left\vert \omega(t) - \omega_0 \right\vert + \left\vert \bar{\omega} (t) \right\vert \leq C Z_0 \mu_0^{\frac54}, && \left\vert \bar{z}(t) \right\vert \leq C Z_0 \mu_0^{\frac34}.
	\end{aligned}
\end{equation}
\end{propo}

\begin{proof}
We proceed by splitting the time interval $(-\infty;T_1]$ into the two intervals $(-\infty,-T_1)$ and $[-T_1,T_1]$.

On the time interval $(-\infty;-T_1]$, Corollary \ref{coro:quantified_orbital_stability} with $\nu_5^\star < \nu_4^\star$ provides the existence of $\Gamma=(z_1,z_2,\omega_1,\omega_2,\mu_1,\mu_2)$ defined on $(-\infty,-T_1]$ satisfying inequalities \eqref{ineq:-infty_-T_1.1}-\eqref{ineq:-infty_-T_1.2} for any $t\leq -T_1$. This implies that \eqref{propo:bootstrap_-T_T:ini} is verified for $w^1=v(-T_1)$ and $\Gamma^1=\Gamma(-T_1)$.

On the time interval $[-T_1,T_1]$, we  adjust $\nu_5^\star< \nu_3^\star$ to apply Proposition \ref{propo:bootstrap_-T_T}. Then, there exists $C>0$ and $\Gamma_2$ defined on $[-T_1,T_1]$ such that \eqref{est:-T_1_T_1} holds on the whole interval $[-T_1,T_1]$. Notice that $\Gamma_2$ is a $\mathcal{C}^1$-extension of $\Gamma$  by the uniqueness of the modulation theory in Lemma \ref{propo:modulation}. We thus denote by $\Gamma$ the function defined on $(-\infty,T_1]$. Finally, \eqref{mainpropo.1} on $[-T_1,T_1]$ is a direct application of \eqref{est:-T_1_T_1} and of the quantified dynamical system in \eqref{eq:simplified_z_omega}.
\end{proof}

Notice that the previous proposition holds on $(-\infty,T_1]$ by comparing the solution $v$ with the precise approximation $V$. It is also possible to compare the solution $v$ with the sum of two translated and rescaled ground-states, which corresponds to the first part of Theorem \ref{maintheo}.

\begin{coro}\label{coro:maincoro}
    There exist some positive constants $C$ and $\nu_5^\star$ such that the following is true. Let $0<\mu_0<\nu_5^\star$ and let $v$ be the unique solution to \eqref{ZK:sym} satisfying \eqref{multi_sol:-infty}. Then, there exists a unique $\Gamma=(z_1,z_2,\omega_1,\omega_2,\mu_1,\mu_2) \in C^1( (-\infty,T_1] : \mathbb R^6)$  that satisfies for all $t \leq T_1$, by defining $\eta(t,\cdot):=v(t,\cdot)-\sum_{i=1}^2 Q_{1+\mu_i(t)}(\cdot- (z_i(t),\omega_i(t)))$
\begin{equation}\label{mainpropo.2}
	\begin{aligned}
		& \| \eta(t) \|_{H^1} \leq C \mu_0^{\frac32}, && \left\vert \mu(t) - \dot{Z}(t) \right\vert + \left\vert \dot{z}_i(t) - \mu_i(t) \right\vert  + \left\vert \dot{\omega}_i(t) \right\vert \leq C \mu_0^{\frac32}, && \left\vert \bar{\mu}(t) \right\vert \leq C \mu_0^{\frac32}  \\
		& \left\vert z(t) - Z(t) \right\vert \leq C Z_0^2 \mu_0^{\frac34}, && \left\vert \omega(t) - \omega_0 \right\vert + \left\vert \bar{\omega} (t) \right\vert  \leq C  Z_0 \mu_0^{\frac54}, && \left\vert \bar{z}(t) \right\vert \leq C Z_0 \mu_0^{\frac34}.
	\end{aligned}
\end{equation}
\end{coro}

\begin{proof}
    By applying Proposition \ref{mainpropo} to $v$, we define the function $\Gamma$ such that \eqref{mainpropo.1} holds on $(-\infty,T_1]$. It implies in particular that $z(t)>\frac{8}{9} Z(t)$ on $(-\infty,-T_1]$.Since $\frac{9}{5}<2$, \eqref{est:Z(t)_R} implies $e^{-Z(t)} \leq \mu_0^{\frac{9}{5}}$. Thus, it follows from \eqref{def:V} and \eqref{est:VA:H2} that, for any $ t \leq T_1$,
\begin{align*}
    \left\| V_A(\Gamma(t)) \right\|_{H^1} \lesssim e^{-\frac{15}{16} z(t)} \leq e^{-\frac56 Z(t)} \lesssim \mu_0^\frac32.
\end{align*}
From the estimates on $\epsilon$ and $z_i$ in \eqref{mainpropo.1}, it holds, for any $t \le T_1$,
\begin{align*}
	 \left\| v(t)- \sum_{i=1}^2 Q_{1+\mu_i(t)} \left( \cdot - \bz_i(t) \right) \right\|_{H^1} \leq \left\| v(t)- V(\Gamma(t)) \right\|_{H^1} + \| V_A(\Gamma(t)) \|_{H^1} \lesssim \mu_0^{\frac32}.
\end{align*}
We now apply a modulation result around the two solitary waves. Since the statement is similar to the one of Lemma \ref{propo:modulation} and to Proposition \ref{propo:dynamical_system}, we omit the proof. Then, we obtain the existence of a function $\tilde{\Gamma}=(\tilde{z}_1,\tilde{z}_2,\tilde{\omega}_1,\tilde{\omega}_2,\tilde{\mu}_1,\tilde{\mu}_2)$ such that, by defining
\begin{align*}
    \tilde{R}_i(t,\bx) := Q_{1+\tilde{\mu}_i(t)} (\cdot- (\tilde{z}_i(t),\tilde{\omega}_i(t))) \quad \text{and} \quad \eta(t) := v(t) - \sum_{i=1}^2 \tilde{R}_i(t) ,
\end{align*}
we have for all $t\leq T_1$ and $i=1,2$,
\begin{align*}
    & \int \eta (t) \partial_x \tilde{R}_i(t) = \int \eta(t) \partial_y \tilde{R}_i (t) = \int \eta(t) \tilde{R}_i(t)=0, \\
    & \left\vert \dot{\tilde{z}}_i(t)- \tilde{\mu}_i(t) \right\vert + \left\vert \dot{\tilde{\omega}}_i(t) \right\vert + \| \eta (t) \|_{H^1} \lesssim \mu_0^{\frac32}.
\end{align*}
Furthermore, a pointwise estimate at any time $t \leq T_1$ around two ground-states as in Remark \ref{rema:modulation_pointwise} around $V$ provides for any $i=1,2$
\begin{align*}
    \left\vert \mu_i(t) - \tilde{\mu}_i(t) \right\vert + \left\vert z_i(t) - \tilde{z}_i(t) \right\vert + \left\vert \omega_i(t) - \tilde{\omega}_i(t) \right\vert \lesssim \mu_0^{\frac32}.
\end{align*}
Combining those estimates with \eqref{mainpropo.1} concludes the proof of estimates \eqref{mainpropo.2}.
\end{proof}

On the time interval $[T_1,+\infty)$, we apply the orbital and asymptotic stability results proved by the authors in \cite{PV23} in the case of a function $w$ close to a sum of two solitary waves whose velocities are bounded by $\underline{c}=\frac12$ and $\bar{c}=\frac32$. Note that in \cite{PV23}, the result is stated for a solution $u$ of \eqref{ZK}. Here by using a translation in space, we restate Theorem 1.1 in \cite{PV23} in the  setting of a solution $w$ to \eqref{ZK:sym}.

\begin{theo}[\cite{PV23}, Theorem 1.1]\label{theo:orbital} There exist some positive constants $k$, $K$ and $A$ such that the following is true. For two velocities $1+\mu_1^0$ and $1+\mu_2^0$ satisfying $\frac12 <1+\mu_2^0< 1+\mu_1^0< \frac32$ and $\sigma:=\mu_1^0-\mu_2^0<\frac12$, we define $\alpha^\star:= k\sigma$ and $Z^\star := K | \ln \sigma|$. Let $\alpha<\alpha^\star$, $Z> Z^\star$, $(z_i^0, \omega_i^0) \in \mathbb{R}^2$, $i=1,2$, and $w_0 \in H^1(\mathbb{R}^2)$ satisfying
	\begin{align}\label{eq:IC}
		\left\| w_0 - \sum_{i=1}^2 Q_{1+\mu_i^0} \left( \cdot - \left( z_i^0, \omega_i^0 \right) \right) \right\|_{H^1} < \alpha \quad \text{and} \quad z_1^0-z_2^0 >Z,
	\end{align}
	and let $w \in C(\mathbb R: H^1(\mathbb R^2))$ be the solution of \eqref{ZK:sym} evolving from the initial condition $w(T_1)=w_0$. Then, there exist functions $(z_i, \omega_i, 1+\mu_i) \in \mathcal{C}^{1}(\mathbb{R}_+ : \mathbb{R}^2 \times (0,+\infty))$, $i= 1,2$ such that: 
	\begin{itemize}
		\item[(i)] (Orbital stability) for all $t \ge 0$,
		\begin{align}
			&\left\| w(t) - \sum_{i=1}^2 Q_{1+\mu_i(t)} \left( \cdot - (z_i(t), \omega_i(t)) \right) \right\|_{H^1} + \sum_{i=1}^2 \left( \left\vert \dot{z}_i(t) -\mu_i(t) \right\vert + \left\vert \dot{\omega}_i(t) \right\vert \right) \leq A \left( \alpha + e^{-\frac{1}{32} \sqrt{\underline{c}}Z }\right); \label{eq:u_orbital} \\
			& z_1(t)-z_2(t)\geq \frac12(Z+ \sigma t) ;\label{defi:z} \\
			& \sum_{i=1}^2 \left\vert \mu_i(t) - \mu_i^0 \right\vert \leq  A \alpha .\label{eq:bound_c_i_t}
		\end{align}		
		
		\item[(ii)] (Asymptotic stability) for $i=1,2$ the limit $\mu_i^+=\lim_{+\infty} \mu_i(t)$ exists and  
		\begin{align}
			& \left\vert \mu_i^+ - \mu_i^0 \right\vert \leq A \left( \alpha + e^{-\frac1{32} \sqrt{\underline{c}}Z } \right); \\
			&\lim_{t\rightarrow +\infty} \left\| w(t) - \sum_{i=1}^2 Q_{1+\mu_i(t)} \left( \cdot - \left( z_i(t), \omega_i(t) \right) \right) \right\|_{H^1\left(x> -\frac{199}{200}t\right)} = 0 ; \label{conv_asymp}
			\\ & (\dot{z}_i(t),\dot{\omega}_i(t)) \underset{t \rightarrow + \infty}{\longrightarrow} (\mu_i^+,0). \nonumber
		\end{align}
	\end{itemize}
\end{theo}

We check that the hypotheses of Theorem \ref{theo:orbital} are satisfied for our solution $v$. We choose the initial velocities $1+\mu_i^0=1+ \mu_i(T_1)$ and the initial positions $(z_i^0,\omega_i^0)=(z_i(T_1),\omega_i(T_1))$. It holds $-\frac12 < \mu_2^0 < \mu_1^0 < \frac12$ from the estimates on $\mu$ and $\bar{\mu}$ in \eqref{mainpropo.2} and \eqref{ineq:dot_Z(T_1)}. Concerning $\sigma=\mu_1(T_1)-\mu_2(T_1)=\mu(T_1)$, since $\dot{Z}$ is an increasing function from $-2\mu_0$ to $2\mu_0$, we have with the estimate on $\mu(T_1)$ in \eqref{mainpropo.1} that $\sigma \leq 3\mu_0 <3\nu_5^\star \leq \frac12$. Furthermore, by \eqref{ineq:dot_Z(T_1)}, it holds $\sigma \geq \frac12 \mu_0$ for $\nu_5^\star$ small enough. As a consequence, we deduce
\begin{align}\label{eq:lower_bound_alpha}
	\alpha^\star > \frac12 k \mu_0 \quad \text{and} \quad Z^\star < K \vert \ln(\frac12 \mu_0) \vert.
\end{align}

We now check that the condition \eqref{eq:IC} is satisfied at the initial time $T_1$. From the estimate on $\eta$ in \eqref{mainpropo.2}, it holds
\begin{align*}
	 \left\| v(T_1)- \sum_{i=1}^2 Q_{1+\mu_i(T_1)} \left( \cdot - \bz_i(T_1) \right) \right\|_{H^1} \leq C \mu_0^{\frac32} < \alpha^\star
\end{align*}
where the last bound holds with $\nu_5^\star$ small enough and the lower bound on $\alpha^\star$ in \eqref{eq:lower_bound_alpha}. Since $Z(T_1)= \rho^{-1} Z_0$, with $\rho \in (0,\frac1{32})$, we infer from the estimate on $z(T_1)$ in \eqref{mainpropo.2} and the inequality 
$Z_0 \gtrsim \vert \ln (\mu_0) \vert$ (see \eqref{est:mu0_Z0}), that $z(T_1) > Z^\star$ provided $\rho$ small enough (independent of $\mu_0$). 

Since all the conditions are satisfied and $e^{-\frac{1}{32}\sqrt{\underline{c}} \frac{1}{\rho} Z_0} \leq \mu_0^2$, Theorem \ref{theo:orbital} provides the existence of three functions $(z_i,\omega_i,1+\mu_i) \in \mathcal{C}([T_1;+\infty): \mathbb{R}^2 \times (0,+\infty))$, such that, for all $t \ge T_1$,
\begin{align}
    &\left\| v(t) - \sum_{i=1}^2 Q_{1+\mu_i(t)} \left( \cdot - (z_i(t), \omega_i(t)) \right) \right\|_{H^1} + \sum_{i=1}^2 \left( \left\vert \dot{z}_i(t) -\mu_i(t) \right\vert + \left\vert \dot{\omega}_i(t) \right\vert +\left\vert \mu_i(t) - \mu_i(T_1) \right\vert \right) \lesssim \mu_0^{\frac32} ; \label{est:maincoro.1}\\
    & z(t)=z_1(t)-z_2(t)\geq \frac12 \left( \frac{Z_0}{\rho}+ \mu(T_1) (t-T_1) \right), \label{est:maincoro.2}
\end{align}
and for $i=1,2$ the limit $\mu_i^+=\lim_{+\infty} \mu_i(t)$ exists and
\begin{align}
    &\lim_{t\rightarrow +\infty} \left\| v(t) - \sum_{i=1}^2 Q_{1+\mu_i(t)} \left( \cdot - \left( z_i(t), \omega_i(t) \right) \right) \right\|_{H^1(x>-\frac{99}{100} t)} = 0 ;\label{est:maincoro.3}\\
    & \left\vert \mu_i^+ - \mu_i(T_1) \right\vert \lesssim \mu_0^{\frac32}, \quad (\dot{z}_i(t), \dot{\omega}_i(t)) \underset{t \rightarrow + \infty}{\longrightarrow} (\mu_i^+,0). \label{est:maincoro.4}
\end{align}

The proof of \eqref{maintheo.1}-\eqref{maintheo.4} is thus an application of Corollary \ref{coro:maincoro} and \eqref{est:maincoro.1}-\eqref{est:maincoro.4}. Notice that the function $\Gamma$ is $\mathcal{C}^1$ around $T_1$ since its construction on both time intervals results from the same modulation property around the sum of two rescaled and translated ground-state.

\begin{toexclude}
\begin{rema}[The limit of the energy functional $\mathcal{F}_+$] \red{a preciser}
To prove this orbital stability result, one could wonder it the bound on the time derivative of energy functional $\mathcal{F}_+$ in \eqref{eq:estimate_F_+} is enough to get the bound on the error at all time. However, the time derivative of the functional is bounded by cubic terms in the error, which are bounded by a constant. Integrating this constant over the infinite time interval $[T_1;+\infty)$ precludes the bound of the orbital stability.
\end{rema}
\end{toexclude}

\smallskip
To finish the proof of Theorem \ref{maintheo}, it only remains to prove \eqref{maintheo.5}-\eqref{maintheo.6}. 
These inequalities are direct consequences of the following lemma.

\begin{lemm}\label{lemm:bound_limits_velocities}
    For $i=1,2$, it holds
    \begin{align*}
       0 \le  \mu_0^{-1} \limsup_{t \rightarrow + \infty} \| \eta(t) \|_{H^1}^2 \lesssim (-1)^{i+1}\mu_i^+-\mu_0 \lesssim \mu_0^{-1} \liminf_{t \rightarrow + \infty} \| \eta(t) \|_{H^1}^2 \lesssim \mu_0^2.
    \end{align*}
\end{lemm}

The counterpart of this lemma for the gKdV equation is Lemma 4.1 in \cite{MM11}, whose proof relies on the conservation of the mass and the energy for the solution $v$. We give here an slightly different proof, still relying on the same arguments. 

\begin{proof}
    From \eqref{def:eta}, \eqref{maintheo.4} and \eqref{est:maincoro.2}, the energy and the mass (defined in \eqref{conserved_quantities}) of the error $\eta(t)$ have finite limits as $t\rightarrow +\infty$, and we define 
    \begin{align*}
        m_\eta^+ := \lim_{t\rightarrow + \infty} M(\eta(t)) \quad \text{and} \quad e_\eta^+ := \lim_{t \rightarrow +\infty} E(\eta(t)) .
    \end{align*}
    Thus, by using the conservation of the mass and the energy of the solution $v$ (see \eqref{conserved_quantities}) between the time $-\infty$ and $+\infty$, the existence of the limits $\mu_i^+=\lim_{+\infty} \mu_i(t)$ and \eqref{multi_sol:-infty}, we infer that
    \begin{align*}
        & M(Q_{1+\mu_0}) + M(Q_{1-\mu_0}) = M(Q_{1+\mu_1^+}) + M(Q_{1+\mu_2^+}) + m_\eta^+, \\
        & E(Q_{1+\mu_0}) + E(Q_{1-\mu_0}) = E(Q_{1+\mu_1^+}) + E(Q_{1+\mu_2^+}) + e_\eta^+,
    \end{align*}
    which together with $M(Q_c)=c M(Q)$ and $E(Q_c)=c^2E(Q)$ yields
    \begin{align}\label{eq:mass_rescaled}
        0 = \left( \mu_1^+ + \mu_2^+ \right) M(Q) + m_\eta^+ \quad \text{and} \quad 2\mu_0^2 E(Q) = \left( 2 (\mu_1^+ + \mu_2^+) + (\mu_1^+)^2 +(\mu_2^+)^2 \right) E(Q) + e_\eta^+.
    \end{align}
    Combining the two identities in \eqref{eq:mass_rescaled} with $E(Q)=-\frac14\int Q^2=-\frac14 M(Q)$ (see Appendix B in \cite{CMP16}), we obtain, for $i=1,2$,
    \begin{align*}
        (\mu_i^+)^2 -\mu_0^2 = (M(Q))^{-1}\left(m_\eta^+\left(1-\mu_i^+-(2M(Q))^{-1}m_\eta^+\right)+2e_\eta^+\right).
    \end{align*}
    On the one hand, we observe combining \eqref{est:dot_Z_T_1}, \eqref{mainpropo.2} and \eqref{est:maincoro.4} that 
    \begin{equation} \label{est:mui-mu+}
    \left|\mu_1^+-\mu_0 \right|+\left|\mu_2^++\mu_0 \right| \lesssim \mu_0^{\frac32} .
    \end{equation}
    On the other hand, it follows from the Sobolev embedding $L^3(\mathbb{R}^2) \hookrightarrow H^1(\mathbb{R}^2)$ and \eqref{maintheo.1} that
    \begin{align*}
        \limsup_{t\rightarrow + \infty} \| \eta(t) \|_{H^1}^2 \lesssim \frac12 m_{\eta}^++e_\eta^+ \lesssim \liminf_{t\rightarrow + \infty} \| \eta(t) \|_{H^1}^2 .
    \end{align*}
   Therefore, we deduce by choosing $\mu_0$ small enough that, for $i=1,2$,
   \begin{align*}
        \limsup_{t\rightarrow + \infty} \| \eta(t) \|_{H^1}^2 \lesssim (\mu_i^+)^2 -\mu_0^2 \lesssim \liminf_{t\rightarrow + \infty} \| \eta(t) \|_{H^1}^2 ,
    \end{align*}
   which combined to \eqref{est:mui-mu+} concludes the proof of Lemma \ref{lemm:bound_limits_velocities}.
\end{proof}

\subsection{Proof of Theorem \texorpdfstring{\ref{theo:stability}}{1.6}} \label{sub:sec:theo:stability} Let $w$ be a solution to \eqref{ZK:sym} satisfying \eqref{cond:theo:stability}. Observe from Corollary \ref{coro:quantified_orbital_stability} and the triangle inequality that $w(-T_1)$ satisfies the conditions \eqref{propo:bootstrap_-T_T:ini}. Hence, we deduce from Proposition \ref{propo:bootstrap_-T_T} that $w$ satisfies \eqref{est:-T_1_T_1} on $[-T_1,T_1]$, which together with Proposition \ref{mainpropo} and the triangle inequality concludes the proof of Theorem \ref{theo:stability}.

We now briefly explain how to prove Remark \ref{rema:stability}. By slightly refining the arguments in the proof of Proposition \ref{propo:bootstrap_-T_T}, we can prove that, for any $t \in [-T_1,T_1]$,
\begin{align}\label{est:-T_1_T_1:refined}
    & v(t,\bx) = V(t,\bx) + \epsilon (t,\bx), \qquad \Gamma(t) = \left( z_1(t),z_2(t), \omega_1(t), \omega_2(t), \mu_1(t), \mu_2(t) \right), \\
    & \begin{aligned}
        & \| \epsilon(t) \|_{H^1} \leq Z_0^{-5} \mu_0^{\frac74}, && \left\vert \mu(t) - \dot{Z}(t) \right\vert \leq C Z_0^{-5} \mu_0^{\frac74}, && \left\vert \bar{\mu}(t) \right\vert \leq C Z_0^{-5} \mu_0^\frac74,  \\
        & \left\vert z(t) - Z(t) \right\vert \leq C Z_0^{-3} \mu_0^\frac34, && \left\vert \omega(t) - \omega_0 \right\vert + \left\vert \bar{\omega} (t) \right\vert \leq C Z_0^{-4}  \mu_0^{\frac54},  && \left\vert \bar{z}(t) \right\vert \leq C Z_0^{-4} \mu_0^{\frac34}.
    \end{aligned} \nonumber
\end{align}

Let $\mathcal{T} \in [-T_1,T_1]$ and $w$ be a solution to \eqref{ZK:sym} satisfying \eqref{rema:cond:stability}. By \eqref{est:-T_1_T_1:refined} and the triangle inequality, $w(\mathcal{T})$ satisfies 
\begin{equation*}
\| w(\mathcal{T})-V(\mathcal{T}) \|_{H^1} \leq Z_0^{-4} \mu_0^{\frac74}. 
\end{equation*}
Then, we deduce arguing as in the proof of Proposition \ref{propo:bootstrap_-T_T} on $[\mathcal{T},T_1]$ that there exists $\tilde{\Gamma}$ such that, for all $t \in [\mathcal{T},T_1]$,
\begin{align*}
    & w(t,\bx) = V(\bx;\tilde{\Gamma}(t)) + \widetilde{\epsilon} (t,\bx), \qquad \widetilde{\Gamma}(t) = \left( \widetilde{z}_1(t),\widetilde{z}_2(t), \widetilde{\omega}_1(t),\widetilde{\omega}_2(t), \widetilde{\mu}_1(t), \widetilde{\mu}_2(t) \right), \\
    & \begin{aligned}
        & \| \widetilde{\epsilon}(t) \|_{H^1} \leq C \mu_0^{\frac74}, && \left\vert \widetilde{\mu}(t) - \dot{Z}(t) \right\vert \leq C Z_0 \mu_0^{\frac74}, && \left\vert \bar{\widetilde{\mu}}(t) \right\vert \leq C \mu_0^\frac74,  \\
        & \left\vert \widetilde{z}(t) - Z(t) \right\vert \leq C Z_0^2 \mu_0^\frac34, && \left\vert \widetilde{\omega}(t) - \omega_0 \right\vert + \left\vert \bar{\widetilde{\omega}} (t) \right\vert \leq C Z_0 \mu_0^{\frac54},  && \left\vert \bar{\widetilde{z}}(t) \right\vert \leq C Z_0 \mu_0^{\frac34}.
    \end{aligned}
\end{align*}
Now, consider the function $w^{\dagger}(t,\bx)=w(-t,-\bx)$. Applying the same arguments as before to $w^{\dagger}$, we deduce on $[- \mathcal{T},T_1]$:
\begin{align*}
    & w^{\dagger}(t,\bx) = V(\bx;\Gamma^{\dagger}(t)) + \epsilon^{\dagger} (t,\bx), \qquad \Gamma^{\dagger}(t) = \left( z_1^{\dagger}(t),z_2^{\dagger}(t),\omega_1^{\dagger}(t),\omega_2^{\dagger}(t), \mu_1^{\dagger}(t), \mu_2^{\dagger}(t) \right), \\
    & \begin{aligned}
        & \| \epsilon^{\dagger}(t) \|_{H^1} \leq C \mu_0^{\frac74}, && \left\vert \mu^{\dagger}(t) - \dot{Z}(t) \right\vert \leq C Z_0 \mu_0^{\frac74}, && \left\vert \bar{\mu}^{\dagger}(t) \right\vert \leq C \mu_0^\frac74,  \\
        & \left\vert z^{\dagger}(t) - Z(t) \right\vert \leq C Z_0^2 \mu_0^\frac34, && \left\vert \omega^{\dagger}(t) - \omega_0 \right\vert + \left\vert \bar{\omega}^{\dagger} (t) \right\vert \leq C Z_0 \mu_0^{\frac54},  && \left\vert \bar{z}^{\dagger}(t) \right\vert \leq C Z_0 \mu_0^{\frac34}.
    \end{aligned}
\end{align*}
We conclude the proof of \eqref{rema:cond:stability} by gathering the above estimates.


\appendix

\section{The Bessel potential \texorpdfstring{$(-\Delta+1)^{-1}$}{(-O+1)-1}} \label{App:Bessel}In this appendix, we derive some useful properties of the operator $(-\Delta+1)^{-1}$, which is naturally defined for function $f \in L^2(\mathbb R^2)$ as a Fourier multiplier  by 
\begin{equation} \label{def:-Delta+1}
    (-\Delta+1)^{-1}f =\mathcal{F}^{-1}\left((1+4\pi^2|\cdot|^2)^{-1} \widehat{f} \right) .
\end{equation}

In order to study the properties of $(-\Delta+1)^{-1}$, we introduce the modified Bessel function of second kind $K_0$ and recall several properties (see for example \cite{AS64} paragraph 9.7.2 and \cite{AS61}). First, $K_0: \mathbb R_+^* \to \mathbb R_+$ is a positive function  which solves the equation 
\begin{align}\label{defi:eq_bessel_mod}
    \partial_r^2 f + \frac{1}{r} \partial_r f - f =0 . 
\end{align}
Moreover, 
\begin{equation} \label{asympt:K0:0}
K_0(r) \underset{r \to 0}{\sim} |\ln(r)|, 
\end{equation}
and there exists $C>0$, such that for all $i\in \{0,1\}$ and $r>1$,
\begin{align}\label{eq:asymp_K_0_first_order}
    \left\vert K_0^{(i)}(r) - (-1)^i \sqrt{\frac{\pi}{2}} \frac{e^{-r}}{r^{\frac{1}{2}}} \right\vert \leq C \frac{e^{-r}}{r^{\frac{3}{2}}}.
\end{align}

\begin{toexclude}
Let us recall the Bessel differential equation, with coefficient $\nu \in \mathbb{R}$
\begin{align}\label{defi:eq_bessel_mod}
    \partial_r^2 f + \frac{1}{r} \partial_r f - \frac{r^2-\nu^2}{r^2} f =0, \quad f: \mathbb{R}_+^* \rightarrow \mathbb{R}.
\end{align}
We recall the definition of the modified Bessel function of the first kind $I_\nu$, defined as functions if $-\nu\notin \mathbb{N}$
\begin{align*}
    I_{\nu} (r) = \sum_{m=0}^{+\infty} \frac{1}{m! \Gamma(m+\nu+1)} \left( \frac{r}{2}\right)^{2m+\nu}.
\end{align*}
For $\nu \notin \mathbb{N}\cup \{ 0\}$, the two modified Bessel functions of the first kind $I_\nu$ and $I_{-\nu}$ are linearly independent. For $n \in \mathbb{N} \cup\{0 \}$, the modified Bessel function of the first kind $I_n$ and the modified Bessel's function of the second kind $K_n$ form a linearly independent set of solutions of \eqref{defi:eq_bessel_mod} with $\nu=n$, where $K_n$ is defined by
\begin{align}\label{defi:K_n}
    K_n(r) := \lim_{\nu \rightarrow n} \frac{\pi}{2} \frac{I_{-\nu}(r)-I_{\nu}(r)}{\sin(\nu \pi)}.
\end{align}
\end{toexclude}

We also recall that the kernel of the Bessel potential in dimension $d=2$, defined by
$G_2(\bx) = \mathcal{F}^{-1} \left( \frac{1}{1+4\pi^2 \vert \bx \vert^2} \right),$ can be expressed in terms of the Bessel function $K_0$ through the formula (see \cite{AS61}, and also \cite{Ste70}), 
\begin{align} \label{def:G2}
    G_2(\bx) = \frac{1}{4\pi} \int_0^{+\infty} e^{-\pi \frac{\vert \bx \vert^2}{\delta}} e^{-\frac{\delta}{4\pi}} \frac{d\delta}{\delta} = \frac{1}{2 \pi} K_0(\vert \bx \vert).
\end{align} 
We observe from \eqref{asympt:K0:0} and \eqref{eq:asymp_K_0_first_order} that $G_2$ is integrable over $\mathbb R^2$. Thus, it follows from \eqref{def:-Delta+1} and the Young theorem on convolution that for any $f\in L^2(\mathbb R^2)$
\begin{equation}  \label{def:conv}
(-\Delta+1)^{-1}f=G_2 \star f. 
\end{equation}

First, we prove that $(-\Delta+1)^{-1}$ preserves the parity in any variable. 
\begin{lemm} \label{Bessel:even}
Let $f\in L^2(\mathbb R^2)$. 
\begin{itemize}
\item{} If $f$ is radial, then $(-\Delta+1)^{-1}f$ is also radial. 
\item{}If $f=f(x,y)$ is even, respectively odd, in its first variable $x$, then $(-\Delta+1)^{-1}f$ is also even, respectively odd, in its first variable.
\item{}If $f=f(x,y)$ is even, respectively odd, in its second variable $y$, then $(-\Delta+1)^{-1}f$ is also even, respectively odd, in its second variable.
\end{itemize}
\end{lemm}

\begin{proof}
Since $G_2$ is radial, the proof of Lemma \ref{Bessel:even} follows directly form identity \eqref{def:conv}. 
\end{proof}

Next, we prove that the operator $(-\Delta+1)^{-1}$ preserves the decay in $\mathcal{Y}$ whose definition is given in \eqref{def:Y}. 
\begin{lemm} \label{Bessel:Y}
Let $h \in \mathcal{Y}$. Then $(-\Delta+1)^{-1}h \in \mathcal{Y}$. 
\end{lemm}

\begin{proof} Let $h \in \mathcal{Y}$ and $\alpha \in \mathbb N^2$. Then, by definition of $\mathcal{Y}$, there exists $n \in \mathbb N$ such that $|\partial_{\bx}^{\alpha}h(\bx)| \lesssim \langle \bx \rangle^ne^{-|\bx|}$.  By \eqref{def:G2}, we have
\begin{align} \label{Bessel:Y.1}
\left| e^{|\bx|}(-\Delta+1)^{-1} \partial_{\bx}^{\alpha}h(\bx) \right| \lesssim e^{|\bx|} & \int_{|\by| \le 2|\bx|}K_0(|\by|) |\partial_{\bx}^{\alpha} h(\bx-\by)| d\by \nonumber \\
    & \quad \quad +e^{|\bx|} \int_{|\by| >2|\bx|} K_0(|\by|)| \partial_{\bx}^{\alpha} h(\bx-\by)|d\by =I(\bx)+II(\bx) .
\end{align}

To deal with the first contribution on the right-hand side of \eqref{Bessel:Y.1}, we observe from the triangle inequality that $e^{|\bx|}\le e^{|\by|}e^{|\bx-\by|}$ and $\langle \bx-\by\rangle^n \lesssim \langle \bx\rangle^n$ in the region $|\by|\le 2|\bx|$. Thus, we deduce after passing in polar coordinates and using \eqref{asympt:K0:0} and \eqref{eq:asymp_K_0_first_order} that
\begin{align*}
\MoveEqLeft
I(\bx) \lesssim \left(\int_{|\by| \le 2|\bx|}K_0(|\by|)e^{|\by|}\langle \bx-\by \rangle^n d\by \right) \sup_{\by \in \mathbb R^2} \left(\langle \by \rangle^{-n} e^{|\by|}|\partial_{\by}^{\alpha}h(\by)| \right)  \\
    & \lesssim \left( \int_0^1 |\ln r|rdr+\int_1^{2|\bx|} r^{\frac12} dr \right) \langle \bx \rangle^n \lesssim \langle \bx \rangle^{n+2} .
\end{align*}
In the region $|\by|>2|\bx|$, we use the trivial bound $e^{|\bx|}\le e^{\frac{|\by|}2}$. Hence, it follows from \eqref{eq:asymp_K_0_first_order} that 
\begin{equation*}
II(\bx) \lesssim \left(\int_{|\by| > 2|\bx|}K_0(|\by|)e^{\frac{|\by|}2} d\by \right) \sup_{\by \in \mathbb R^2} \left(|\partial_{\by}^{\alpha}h(\by)| \right)  \lesssim \left( \int_1^{+\infty} e^{-\frac{r}2}r^{\frac12} dr \right) \lesssim 1 .
\end{equation*}
These two estimates conclude the proof of Lemma \ref{Bessel:Y}.
\end{proof}

Next, we prove that the operator $(-\Delta+1)^{-1}$ preserves polynomial and exponential\footnote{with a factor in the exponential strictly smaller than $1$ (see Remark \ref{rem:exp}).} growth at infinity. Consider the  vector space $E$ defined by
\begin{align*}
    E:= \left\{ f \in \mathcal{C}^\infty(\mathbb{R}^2) \ : \ \forall \alpha \in \mathbb{N}^2, \quad  N(\partial_\bx^\alpha f) <\infty \right\} \quad \text{with} \quad N(f):=\sup_{\bx \in \mathbb{R}^2} \left\vert f(\bx) e^{-\frac{\vert \bx \vert}{2}}\right\vert.
\end{align*}

Equipped with the family of semi-norms $\left(N(\partial_\bx^{\alpha}\cdot)\right)_{\alpha \in \mathbb{N}^2}$, $E$ is a Fr\'echet space. In particular, $E$ contains the functions whose all derivatives have polynomial growth at infinity. We first prove that space $E$ is invariant under the action of the operator $(-\Delta+1)^{-1}$.

\begin{propo}\label{propo:G_2}
The operator $(-\Delta+1):E \rightarrow E$ is a homeomorphism, whose  inverse operator is defined by
\begin{align}\label{defi:G_2}
     (-\Delta+1)^{-1}: E \to E, \quad f \mapsto G_2 \star f.
\end{align}
\end{propo} 

\begin{rema} \label{rem:exp} From the proof, this property remains true for functions growing at infinity at most as $e^{\delta\vert \bx \vert}$, for any $0<\delta<1$.
\end{rema}

\begin{proof}
First, by the definition of $E$, the operator $(-\Delta+1)$ is a continuous map from $E$ into itself. To prove that the application is bijective, we use the explicit definition of the inverse.

Let $f\in E$ and $\alpha \in \mathbb N^2$. We verify that $G_2 \star \partial_{\bx}^{\alpha}f$ is a well-defined function belonging to $E$. Indeed, by using the triangle inequality, we have $e^{-\frac{\vert \bx \vert}{2}} \le e^{-\frac{\vert \bz \vert}{2}} e^{\frac{\vert \bz-\bx \vert}{2}}$, so that from \eqref{def:G2}
\begin{align*}
    N(G_2 \star \partial_{\bx}^{\alpha}f) \le \sup_{\bx \in \mathbb{R}^2} \int_{\mathbb R^2} G_2(\bx- \bz) \vert \partial_{\bx}^{\alpha}f(\bz)\vert e^{-\frac{\vert \bx \vert}{2}}  d\bz \lesssim 
    \left(\int_{\mathbb R^2} K_0(\vert\by \vert)  e^{\frac{\vert \by \vert}{2}}  d\by \right) N(\partial_{\bx}^{\alpha}f) .
\end{align*}
Moreover, it follows from the asymptotics of $K_0$ at $0$ in \eqref{asympt:K0:0} and at infinity in \eqref{eq:asymp_K_0_first_order} that $\int_{\mathbb R^2} K_0(\vert\by \vert)  e^{\frac{\vert \by \vert}{2}}  d\by \lesssim 1$. Therefore, we conclude that
\begin{align*}
N(G_2 \star \partial_{\bx}^{\alpha}f) \lesssim N(\partial_{\bx}^{\alpha}f) .
\end{align*}

One can see from \eqref{def:-Delta+1} and \eqref{def:G2} that, for all $f \in \mathcal{S}(\mathbb R^2)$, 
\begin{align*}
    (-\Delta+1) \left(G_2 \star f \right)=G_2 \star (-\Delta+1)f  = f.
\end{align*}
Since the Schwartz space $\mathcal{S}(\mathbb{R}^2)$ is dense in $E$ for the topology induced by the semi-norms $(\Tilde{N}_\beta)_{\beta\in \mathbb{N}^2}$, where for all $f \in E$, $\Tilde{N}_\beta(f) := \sup_{\bx \in \mathbb{R}^2} \left\vert \partial_\bx^\beta f(\bx) e^{-\frac{3}{4}\vert \bx \vert }\right\vert$, we obtain that the convolution with $G_2$ is the inverse of $(-\Delta+1)$ on $E$. As a consequence, $(-\Delta+1)$ is a homeomorphism from $E$ to itself.
\end{proof}

The next result is a consequence from the fact that the operator $(-\Delta+1)^{-1}$ also preserves the polynomial growth at infinity.

\begin{coro}\label{coro:G_2}
Let $f\in E$ and let $z \gg 1$ be a fixed parameter.
\begin{itemize}
\item[(i)]
Assume that there exists a positive constant $c=c(z)$ and $n \in \mathbb N$ such that, for all $\bx \in \mathbb R^2$, $|f(\bx)| \le c(z)\langle \bx \rangle^n$. Then, for all $\bx \in \mathbb R^2$,
\begin{align} \label{coro:G_2.1}
    \left\vert (-\Delta+1)^{-1}f(\bx) \right\vert \lesssim c(z)\langle \bx \rangle^{n}.
\end{align}
\item[(ii)] Assume that there exist positive constants $c_1=c_1(z)$, $c_2=c_2(z)$, $k=k(z)$ with $k(z) \gg 1$  and $n \in \mathbb N$  such that, for all $\bx \in \mathbb R^2$, $|f(\bx)| \le \langle \bx \rangle^n \left( c_1(z) +c_2(z) \mathbf{1}_{|\bx|>k(z)} \right)$. Then, for all $\bx \in \mathbb R^2$,
\begin{align} \label{coro:G_2.2}
    \left\vert (-\Delta+1)^{-1}f(\bx) \right\vert \lesssim \langle \bx \rangle^{n} \left(c_1(z)+c_2(z)e^{-\frac{k(z)}4} + c_2(z)\mathbf{1}_{|\bx|>\frac{k(z)}2}\right) .
\end{align}
\end{itemize}
\end{coro}

\begin{proof}
By using the asymptotic expansions of $K_0$ at $0$ and at infinity (see \eqref{asympt:K0:0} and \eqref{eq:asymp_K_0_first_order}) and the rough estimate $\langle \bx-\by \rangle^{n} \lesssim \langle \bx \rangle^{n}\langle \by \rangle^{n}$, we have that
\begin{align} \label{est:coro:G_2}
   \left\vert (-\Delta+1)^{-1}f(\bx) \right\vert =\left\vert \int_{\mathbb R^2} G_2(\by) f(\bx-\by) d\by \right\vert \lesssim c(z) \left\langle   \bx  \right\rangle^{n} \int_{\mathbb R^2} K_0(\vert \by \vert)\left\langle  \by  \right\rangle^{n}d\by \lesssim c(z)\langle \bx \rangle^{n} .
\end{align}
which concludes the proof of \eqref{coro:G_2.1}.

The proof of \eqref{coro:G_2.2} follows the same strategy as the one of \eqref{coro:G_2.1}. Arguing as above, we have 
\begin{align*}
\left\vert (-\Delta+1)^{-1}f(\bx) \right\vert 
    & \lesssim c_1(z)\left\langle \bx \right\rangle^{n} \int_{\mathbb R^2} K_0(\vert \by \vert)\left\langle  \by  \right\rangle^{n}d\by+c_2(z)\left\langle \bx  \right\rangle^{n} \int_{\mathbb R^2} K_0(\vert \by \vert)\left\langle  \by  \right\rangle^{n} \mathbf{1}_{|\bx-\by|>k(z)} d\by \\
    & =I(\bx)+II(\bx).
\end{align*}
By using the asymptotic expansions of $K_0$ at $0$ and at infinity, we see that $\vert I(\bx) \vert \lesssim c_1(z)\left\langle   \bx  \right\rangle^{n}$. To handle the contribution $II(\bx)$, we split the domain of integration into the regions $|\by| \le k(z)/2$ and $|\by|>k(z)/2$, so that 
\begin{align*}
II(\bx) & =c_2(z)\left\langle   \bx  \right\rangle^{n} \int_{|\by| \le \frac{k(z)}2} K_0(\vert \by \vert)\left\langle  \by  \right\rangle^{n} \mathbf{1}_{|\bx-\by|>k(z)} d\by \\
    & \quad \quad + c_2(z)\left\langle \bx  \right\rangle^{n} \int_{|\by| > \frac{k(z)}2} K_0(\vert \by \vert)\left\langle  \by  \right\rangle^{n} \mathbf{1}_{|\bx-\by|>k(z)} d\by \quad =II_1(\bx)+II_2(\bx) .
\end{align*}
 On the one hand, in the region $|\by| \le k(z)/2$, we have $|\bx| \ge k(z)/2$, so that
\begin{equation*}
II_1(\bx) \lesssim c_2(z)\left\langle   \bx  \right\rangle^{n}\mathbf{1}_{|\bx|>\frac{k(z)}2}  \int_{\mathbb R^2} K_0(\vert \by \vert)\left\langle  \by \right\rangle^{n}d\by \lesssim c_2(z)\left\langle \bx  \right\rangle^{n}\mathbf{1}_{|\bx|>\frac{k(z)}2} .
\end{equation*}
On the other hand, in the region $|\by| > k(z)/2$, we observe from the decay property of  $K_0$ at infinity in  \eqref{eq:asymp_K_0_first_order} that 
\begin{equation*}
II_2(\bx) \lesssim c_2(z)\left\langle   \bx  \right\rangle^{n} \int_{\vert \by \vert > \frac{k(z)}2} e^{-\vert \by \vert} \left\langle  \by  \right\rangle^{n}d\by \lesssim c_2(z) e^{-\frac{k(z)}4}\left\langle   \bx  \right\rangle^{n} .
\end{equation*}
Therefore, we conclude the proof of \eqref{coro:G_2.2} gathering these estimates. 
\end{proof}

Finally, we prove that $(-\Delta+1)^{-1}$ commutes with the operator $\partial_x^{-1}$ defined in \eqref{defi:d_x_-1}.
\begin{lemm} \label{comm:dx-1}
Let $f \in \mathcal{Y}$. Then $\partial_x^{-1}(-\Delta+1)^{-1}f=(-\Delta+1)^{-1}\partial_x^{-1}f$.
\end{lemm} 

\begin{proof} The proof follows from \eqref{defi:G_2}, the integrability of $G_2$ and Fubini's theorem. 
\end{proof}

\section{Properties of the ground-state \texorpdfstring{$Q$}{Q}}\label{app:asymptotic_Q}

\subsection{Basic properties of \texorpdfstring{$Q$}{Q} and related functions}

In this subsection, we derive basic properties of the ground state $Q$ of \eqref{eq:Q}, the function $\Lambda Q$ defined in \eqref{defi:Lambda_Q} and other related functions ($X$, $Y$ and $W$ are defined in \eqref{def:XYW}, while $\Lambda^2Q$ is defined in \eqref{defi:Lambda_Q}).

\begin{lemm} \label{parity} The following properties hold true.
\begin{itemize}
\item[(i)] $Q, \, \Lambda Q, \, \Lambda^2 Q, \, X \in \mathcal{Y}$ and  $Y, \, W \in \mathcal{Z}$.
\item[(ii)] $Q$, $\Lambda Q$ and $X$ are radial function, $W$ is even in its second variable and $Y$ is odd in its second variable.
\end{itemize}
\end{lemm}

\begin{proof} 
The proof of (i) follows from Proposition \ref{propo:Q} and Lemma \ref{Bessel:even}, while the proof of (ii) follows from Proposition \ref{propo:Q} and Lemma \ref{Bessel:Y}.
\end{proof}

\begin{toexclude}
\begin{propo}[Decay of $Q$ and of associated functions]
The following identities hold:
\begin{align}
    Q, \Lambda Q \in \mathcal{Y}.
\end{align}
\end{propo}
\end{toexclude}

\begin{lemm}\label{propo:id_Q}
The following identities hold true.
\begin{align}
    \int Q^2 &= \int Q; \label{eq:id1_Q} \\
     \langle \Lambda Q, Q \rangle &=  \frac{1}{2}\int Q;\label{eq:id2_Q} \\
      \langle Q, \partial_x Q \rangle&=\langle Q, \partial_y Q \rangle=\langle \partial_x Q, \partial_y Q \rangle = \langle \partial_x Q, \Lambda Q \rangle = \langle \partial_y Q, \Lambda Q \rangle =0; \label{eq:id4_Q} \\
     \langle X ,\partial_x Q \rangle &= \langle Y ,\partial_x Q \rangle = \langle W ,\partial_y Q \rangle = \langle X ,\partial_y Q \rangle = \langle Y ,\Lambda Q \rangle = 0; \label{eq:id5_Q} \\ 
     \langle 2QX, \partial_x Q \rangle &=  \langle 2QX, \partial_y Q \rangle = 0; \label{eq:ineq1.1_Q} \\
     \langle 2QX, \Lambda Q \rangle &=   - \left\langle Q,\Lambda Q \right\rangle - \left\langle (-\Delta+1)^{-1} Q, Q \right\rangle <0; \label{eq:ineq1_Q}\\
     \langle 2QY, \partial_x Q \rangle &=  \langle 2QY, \Lambda Q \rangle = 0;
     \label{eq:ineq2.1_Q} \\
     \langle 2QY, \partial_y Q \rangle &=-\langle \partial_x^{-1}\partial_y Q, \partial_y Q \rangle >0; \label{eq:ineq2_Q} \\
     \langle 2QW, \partial_y Q \rangle &= 0; \label{eq:ineq3.1_Q} \\ 
      \langle 2QW, \partial_x Q \rangle &=  \frac{1}{2}\int Q >0. \label{eq:ineq3_Q}
\end{align}
\end{lemm}

\begin{rema}
We were not able to compute explicitely the value of $\langle 2Q W, \Lambda Q \rangle$. The value of this coefficient (positive, negative or null) does not impact the rest of the computations, see for instance Lemma \ref{lemm:antecedent_2}. 
\end{rema}

\begin{proof}
The identity \eqref{eq:id1_Q} is a consequence of the equation of $Q$ in \eqref{eq:Q} and integration by parts. For \eqref{eq:id2_Q}, with the definition of $\Lambda Q$ in \eqref{defi:Lambda_Q} and \eqref{eq:id1_Q}:
\begin{align*}
    \langle \Lambda Q, Q \rangle = \int \Lambda Q Q = \left( \frac{d}{dc} \int \frac{Q_c^2}{2} \right)_{c=1} = \frac{1}{2}\int Q^2= \frac{1}{2}\int Q.
\end{align*}

The proofs of \eqref{eq:id4_Q}, \eqref{eq:id5_Q}, \eqref{eq:ineq1.1_Q}, \eqref{eq:ineq2.1_Q} and \eqref{eq:ineq3.1_Q} follow from directly from Lemma \ref{parity} (ii). 

To prove \eqref{eq:ineq1_Q}, we observe from Proposition \ref{theo:L} (vii) that  $2Q\Lambda Q  =(-\Delta +1) \Lambda Q +Q$. Hence,
\begin{align*}
    \left\langle 2QX, \Lambda Q \right\rangle = -\left\langle (-\Delta+1)^{-1} Q, (-\Delta+1) \Lambda Q +Q\right\rangle=- \left\langle Q,\Lambda Q \right\rangle - \left\langle (-\Delta+1)^{-1} Q, Q \right\rangle <0 ,
\end{align*}
thanks to \eqref{eq:id2_Q} and Plancerel's identity. 

Now, we prove \eqref{eq:ineq2_Q}. We deduce from the identity $(-\Delta+1)Q=Q^2$ and Lemma \ref{comm:dx-1} that
\begin{align*}
    \left\langle 2Q Y, \partial_y Q \right\rangle  = - \left\langle \partial_x^{-1} (-\Delta+1)^{-1} \partial_y Q,  \left( - \Delta+1 \right) \partial_yQ \right\rangle =-  \left\langle \partial_x^{-1} \partial_y Q, \partial_y Q \right\rangle.
\end{align*}
Using an integration by parts in $x$ and that $y\mapsto \partial_y Q (x,y)$ is odd, we compute
\begin{align*}
    - \langle \partial_x^{-1} \partial_y Q, \partial_y Q \rangle = \frac{1}{2}\int_{\mathbb{R}} \left( \int_{\mathbb{R}} \partial_y Q (x,y) dx \right)^2 dy=\int_0^{+\infty}\left( \int_{\mathbb{R}} \partial_y Q (x,y) dx \right)^2 dy >0,
\end{align*}
since $\partial_y Q (x,y)<0$ for all $x \in \mathbb R$, $y >0$.

Finally, we use the identity $(-\Delta+1)Q=Q^2$, Lemma \ref{comm:dx-1} and integration by parts to deduce that
\begin{align*}
    \left\langle 2Q W, \partial_x Q \right\rangle  =  -\left\langle \partial_x^{-1} (-\Delta+1)^{-1} \Lambda Q,  \left( - \Delta+1 \right) \partial_xQ \right\rangle =  \left\langle \Lambda Q,  Q \right\rangle,
\end{align*}
which combined with \eqref{eq:id2_Q} concludes the proof of \eqref{eq:ineq3_Q}.
\end{proof}

\begin{lemm}\label{lemm:decompo_rescaled_Q}
    Let $0<\nu<\frac1{32}$. Let $\mu>0$, $(z,\omega) \in \mathbb R^2$ such that $0< \mu  \leq \nu$ and $\vert (z,\omega) \vert \leq \nu$. Then, for any $\bx=(x,y) \in \mathbb R^2$,
    \begin{align*}
        \left\vert Q_{1+\mu}(\bx+(z,\omega)) - Q(\bx) - \mu \Lambda Q(\bx) - z \partial_x Q(\bx) -\omega \partial_y Q(\bx) \right\vert \lesssim \left( \mu^2 +z^2 + \omega^2 \right) e^{-\frac{31}{32} \vert \bx \vert}.
    \end{align*}
\end{lemm}

\begin{proof}
First, observe from the Taylor formula that 
\begin{align}
     & \begin{aligned}
        \left\vert Q_{1+\mu}(\bx) - Q(\bx) -\mu \Lambda Q(\bx) \right\vert
        & = \left\vert \mu^2 \int_{0}^1(1-s) \frac{d^2}{dc^2} \left( Q_{c}(\bx)\right)_{\vert c= 1+s\mu} ds \right\vert \\
            & \qquad\lesssim \mu^2 \sup_{\vert \Tilde{\mu} \vert \leq \delta} \left\vert \Lambda^2_{1+\Tilde{\mu}} Q (\bx) \right\vert \lesssim \mu^2 e^{-(1-\delta) \vert \bx \vert} ;
        \end{aligned}
    \nonumber \\
    & \begin{aligned}
        \left\vert Q(\bx+(z,\omega))- Q(\bx) - (z,\omega) \cdot \nabla Q(\bx) \right\vert 
        & = \left\vert \int_{0}^1 (1-s) \frac{d^2}{ds^2} Q(\bx+s (z, \omega)) ds \right\vert \\
            & \qquad \lesssim \left(z^2 +\omega^2 \right) e^{- \vert \bx\vert} ; 
        \end{aligned}
    \label{est:DL_Q_x_direction}\\ 
    & \vert\mu \vert \left\vert \Lambda Q(\bx) - \Lambda Q(\bx+(0,\omega)) \right\vert = \left\vert \mu \omega \int_{0}^1 (1-s) \partial_y \Lambda Q (\bx+(0, s\omega)) ds \right\vert \lesssim \vert \mu \omega \vert e^{-\vert \bx \vert}. \nonumber
\end{align}
Therefore, the proof of Lemma \ref{lemm:decompo_rescaled_Q} follows by choosing $0<\nu<\frac1{32}$ and using the triangle inequality. 
\end{proof}

\subsection{Asymptotic expansions of \texorpdfstring{$Q$}{Q} and \texorpdfstring{$\chi_0$}{X0}}
In this sub-section, we give the asymptotic development of $Q$ and its derivatives at infinity at any finite order $k$.

To define the asymptotic expansion of $Q$, we introduce the functions $L_k$
\begin{align}\label{defi:L_k}
L_k(r) := \frac{e^{-r}}{\sqrt{r}} \sum_{l=0}^k \frac{a_l}{r^l}, \quad \text{ with } a_0:=\sqrt{\frac{\pi}{2}} \text{ and for }l \geq 1, \quad a_l:= \sqrt{\frac{\pi}{2}} \frac{(-1)^l}{8^l l!}\prod_{j=1}^l(2j-1)^2.
\end{align}

The main result is stated in the following proposition.

\begin{propo}\label{propo:asymp_Q}
There exists a constant $\kappa>0$, so that for any $k\in \mathbb{N}$, there exists a constant $C=C_k>0$ such that
\begin{align*}
    \forall i \in \{0,1,2\}, \forall r \geq 1, \quad \left\vert Q^{(i)}(r) -\kappa L_k^{(i)}(r)\right\vert \leq C \frac{e^{-r}}{r^{k+\frac{3}{2}}}.
\end{align*}
\end{propo}

\begin{rema}
    By choosing $i=0$ and $k=0$, we recover the first order asymptotic expansion of $Q$ in \eqref{eq:bound_Q_first_order}.
\end{rema}
\begin{rema}
    The radius of convergence of $\sum a_l r^l$ is equal to $0$. For this reason, the approximation of $Q$ in Proposition \ref{propo:asymp_Q} can only hold at a finite order $k$.
\end{rema}
\begin{rema}
    The proof of the asymptotic expansion of $Q$ in dimension $3$ given in \eqref{asymp_Q_d3} relies on the exact same arguments as the ones in dimension $2$.
\end{rema}

The goal of this subsection is to give a proof of Proposition \ref{propo:asymp_Q}. It is divided into several lemmas. 

We rely on the rough bound of the ground-state proved in \cite{BL83}
\begin{equation} \label{rough:bound}
    \exists \delta \in (0,1), \exists C>0, \forall r>1, \quad Q(r) \leq C e^{-\delta r}.
\end{equation}
Let us first improve the decay of $Q$ in \eqref{rough:bound} by using the same arguments as in \cite{Str77,BL83}.
\begin{lemm}\label{lemm:improved_decay_Q}
    There exists $R_0>0$ $C>0$ such that
    \begin{align*}
        \exists C>0, \forall r>R_0, \quad Q(r) \leq C\frac{e^{-r}}{r^{\frac{1}{4}}}.
    \end{align*}
\end{lemm}

\begin{proof}
    \underline{\emph{First step:}} We prove that $f(r):=-3 \frac{Q'}{r}- \frac{Q}{r^2}-4Q^2 - \frac{Q}{r}<0$, for $r$ large enough. Indeed, taking the derivative of this function gives
    \begin{align*}
        f'(r)
            &= Q' \left( -\frac{1}{r} + \frac{5}{r^2} -8Q \right) + Q \left(-\frac{3}{r} + \frac{1}{r^2} + \frac{2}{r^3} +3\frac{Q}{r} \right) \\
        \partial_r( e^{-r} f)
            & = e^{-r} \left( Q' \left( \frac{2}{r} +\frac{5}{r^2} -8Q \right) + Q \left( -\frac{2}{r} +\frac{2}{r^2} + \frac{2}{r^3} + 4Q + 3 \frac{Q}{r} \right) \right) < 0,
    \end{align*}
    
    where the last inequality holds for $r \geq R_0$ with $R_0$ large enough, by the signs of $Q>0$ and of the derivative $Q'<0$ and the exponential decay in \eqref{eq:bound_Q_first_order}. Since $f$ tends to $0$ at $+\infty$, an integration from $r>R_0$ to $+\infty$ yields
    \begin{align*}
        \forall r >R_0, \quad f(r) > 0.
    \end{align*}
    
    \underline{\emph{Second step:}} We prove that $\frac{Q}{r} + 4Q+4Q'<0$, for $r$ large enough. As before, the derivative gives
    \begin{align*}
        \partial_r \left( \frac{Q}{r} + 4 Q +4Q'\right) & = - 3\frac{Q'}{r} - \frac{Q}{r^2} -4 Q^2 + 4Q + 4Q' \\
        \partial_r \left( e^{-r} \left( \frac{Q}{r} + 4 Q +4Q' \right) \right) & = e^{-r} \left( -3 \frac{Q'}{r} - \frac{Q}{r^2} -4Q^2 - \frac{Q}{r} \right)=e^{-r}f(r) > 0,
    \end{align*}
    where the last inequality comes from the first step. With a null limit at infinity, an integration from $r>R_0$ to $\infty$ yields
    \begin{align*}
        \forall r \geq R_0, \quad \frac{Q}{r} + 4Q +4Q' < 0.
    \end{align*}
    
    \underline{\emph{Third step:}} Bound on the required term. By direct computation with the previous step
    \begin{align*}
        \partial_r \left( r^{\frac{1}{2}}e^{2r}Q^2 \right) = \frac{r^{\frac{1}{2}}}{2} e^{2r} Q \left( \frac{Q}{r} +4Q+ 4Q' \right)<0.
    \end{align*}
    Thus the positive function $r^{\frac{1}{2}}e^{2r}Q^2(r)$ decays at infinity, and is bounded on $\mathbb{R}_+^*$ by a constant $c>0$, which concludes the proof of the lemma.
\end{proof}

With the same proof, we can obtain that $Q(r) \leq c r^{-\alpha}e^{-r}$, for any $\alpha \in (0, \frac{1}{2})$. However, Lemma \ref{lemm:improved_decay_Q} is sufficient in the rest of the appendix to obtain a precise asymptotic expansion of $Q$.

We prove that the asymptotic behaviour of the function $Q$ is related to the one of the modified Bessel function $K_0$ in \eqref{eq:asymp_K_0_first_order}. 
We precise the asymptotic behaviour of $K_0$ at $+\infty$ by the functions $L_k$ defined in \eqref{defi:L_k}, inspired by \cite{AS64, Bow58}.
\begin{lemm}[Asymptotic behaviour of $K_0$]\label{lemm:asymp_K_0}
    For any $k\in \mathbb{N}$ and $i \in \{ 0 \} \cup \mathbb{N}$, there exists a constant $C=C_{k,i}>0$ such that
    \begin{align*}
        \forall r \geq 1, \quad \left\vert K_0^{(i)}(r) -L_k^{(i)}(r) \right\vert \leq C \frac{e^{-r}}{r^{k+\frac{3}{2}}}.
    \end{align*}
\end{lemm}

\begin{proof}
By computation, we notice that
\begin{align}\label{eq:K_over_L_k}
\partial_r^2 \left( \frac{K_0}{L_k} \right) + \left( \frac{1}{r} + 2 \frac{L_k'}{L_k} \right) \partial_r \left( \frac{K_0}{L_k} \right) = \frac{K_0}{L_k^2} \left( - L_k'' - \frac{L_k'}{r} + L_k \right).
\end{align}
A computation on the finite sums gives us
\begin{align}\label{eq:L_k_error}
- L_k'' -\frac{L_k'}{r} +L_k= \frac{e^{-r}}{\sqrt{r}}\left( \sum_{n=1}^{k+1} - \frac{a_{n-1}}{r^n}(2n-2) + \sum_{n=2}^{k+2}- \frac{a_{n-2}}{r^n}\left( n- \frac{3}{2}\right)^2 \right)= \frac{e^{-r}}{\sqrt{r}} \frac{-a_k}{r^{k+2}}\left( k+ \frac{1}{2} \right)^2.
\end{align}
By using this development in \eqref{eq:K_over_L_k}, the bound \eqref{eq:asymp_K_0_first_order} and \eqref{defi:L_k}, and implicit constants independent of $k$ and $r$
\begin{align}
\left\vert \partial_r \left( e^{\ln(r)+2 \ln(L_k)} \partial_r \left( \frac{K_0}{L_k} \right) \right) \right\vert = \left\vert K_0 \frac{e^{-r}}{r^{k+ \frac{3}{2}}} a_k \left( k+ \frac{1}{2}\right)^2 \right\vert \lesssim \frac{e^{-2r}}{r^{k+2}}. \label{eq:bound_to_improve}
\end{align}

By integrating from $r$ to $R>r$
\begin{align}\label{eq:derivative_K_Lk}
\left\vert r L_k^2(r)\partial_r \left( \frac{K_0(r)}{L_k(r)} \right) +C(R) \right\vert \lesssim \frac{e^{-2r}}{r^{k+2}}.
\end{align}
Notice that $C(R) \rightarrow 0$ as $R\rightarrow +\infty$, by the asymptotic behaviours in \eqref{eq:asymp_K_0_first_order} and \eqref{defi:L_k}:
\begin{align*}
\vert C(R) \vert = \left\vert R \left(K_0'(R) L_k(R)- K_0(R) L_k'(R) \right) \right\vert \leq R (\vert K_0'(R) \vert +\vert L_k'(R)\vert )( \vert L_k(R) \vert + K_0(R)) \underset{ R \rightarrow + \infty}{\rightarrow} 0.
\end{align*}

Since $\displaystyle \lim_{r\rightarrow +\infty} \frac{K_0(r)}{L_k(r)}=1$, a second integration gives
\begin{align*}
    \left\vert K_0(r) - L_k(r)\right\vert \lesssim \frac{\vert L_k(r) \vert }{r^{k+1}} \lesssim \frac{e^{-r}}{r^{k+\frac{3}{2}}}.
\end{align*}

For the asymptotic behaviour of the derivative of $K_0$, it suffices to develop \eqref{eq:derivative_K_Lk} to get the asymptotic development of $K_0'$ up to order $k+\frac{3}{2}$. For the higher order derivative, we proceed by induction on the equation satisfied by $K_0$ and the derivatives of the equation satisfied by $L_k$ in \eqref{eq:L_k_error}.
\end{proof}

\begin{lemm}[Approximation of $Q$ by $K_0$]\label{lemm:closeness_Q_K_0}
There exists a constant $\kappa>0$ and a constant $c>0$, such that
\begin{align*}
\forall i \in \{0,1,2\}, \ \forall r>1, \quad \vert Q^{(i)}(r)-\kappa K_0^{(i)}(r) \vert \leq c \frac{e^{-2r}}{r}.
\end{align*}
\end{lemm}

\begin{proof}
The proof relies on the same arguments as the previous proof. By direct computations, 
\begin{align}\label{eq:Q_over_L_k}
\partial_r^2 \left( \frac{Q}{K_0} \right) + \left( \frac{1}{r} + 2 \frac{K_0'}{K_0} \right) \partial_r \left( \frac{Q}{K_0} \right) = - \frac{Q^2}{K_0}.
\end{align}
By using the asymptotic bound in Lemma \ref{lemm:improved_decay_Q} and the asymptotic equivalent in Lemma \ref{lemm:asymp_K_0}, we find that
\begin{align}\label{eq:closeness_Q_K0_1}
\partial_r \left( e^{\ln(r)+2 \ln(K_0)} \partial_r \left( \frac{Q}{K_0} \right) \right) = -Q^2 r K_0 , \quad \text{with} \quad 0 \leq Q^2 r K_0 \lesssim e^{-3r}.
\end{align}

By \eqref{eq:closeness_Q_K0_1}, the function $rK_0^2 \partial_r \left( \frac{Q}{K_0}\right)$ is a decreasing and bounded function for $r$ large enough, and so it has a finite limit when $r \to +\infty$, denoted by $A_{\infty}$. Hence, integrating \eqref{eq:closeness_Q_K0_1} from $r$ to $+\infty$ and using \eqref{eq:asymp_K_0_first_order} yield
\begin{align}\label{eq:derivative_Q_K_0}
-e^{-3r} \lesssim  A_\infty - r K_0^2(r)\partial_r \left( \frac{Q(r)}{K_0(r)} \right) \leq 0 \quad  \Leftrightarrow \quad -e^{-r} \lesssim \frac{ A_\infty}{rK_0^2(r)} -\partial_r \left( \frac{Q(r)}{K_0(r)} \right) \leq 0.
\end{align}
Another integration from $r$ to $R$ yields
\begin{align}\label{eq:estimate_Q_K_0}
    -e^{-r} \lesssim \int_r^R \frac{A_\infty}{r' K_0(r')^2} dr' + \frac{Q(r)}{K_0(r)} - \frac{Q(R)}{K_0(R)} \leq 0.
\end{align}

From Lemma \ref{lemm:improved_decay_Q} and \ref{lemm:asymp_K_0}, $\frac{Q(R)}{K_0(R)}=O(\sqrt{R})$ and  $\frac{1}{r'K_0(r')^2} \sim e^{2r}$ at $\infty$, so that by \eqref{eq:estimate_Q_K_0}, it must be the case that $A_\infty=0$. Thus by \eqref{eq:derivative_Q_K_0}, the function $\frac{Q}{K_0}$ increases at infinity, and by \eqref{eq:estimate_Q_K_0}, it is bounded, so it has a finite limit denoted $\kappa$ at infinity. Notice that by the signs of $Q$ and $K_0$, we have $\kappa>0$. So far we have obtained
\begin{align*}
    -\frac{e^{-{2r}}}{\sqrt{r}} \lesssim Q(r) -\kappa K_0(r) \leq 0 .
\end{align*}

This result almost corresponds to the one of the lemma. To improve it, it suffices to notice that the bound in $Q$ in Lemma \ref{lemm:improved_decay_Q} can be improved with the last inequality by $Q(r)\lesssim r^{-\frac{1}{2}} e^{-r}$. By redoing the computations from the beginning of the proof with this new bound, we obtain the result of the lemma.

Now we prove the equivalent of the derivative. By \eqref{eq:derivative_Q_K_0}
\begin{align*}
& \left\vert Q'(r) K_0(r) - Q(r) K_0'(r) \right\vert \lesssim \frac{e^{-3r}}{r^{\frac{3}{2}}}, \\
& \left\vert Q'(r)- K_0'(r) \right\vert \lesssim \frac{\left\vert \left( Q(r)-K_0(r)\right) K_0'(r)\right\vert}{K_0(r)} + \frac{e^{-3r}}{r^{\frac{3}{2}}K_0(r)} \lesssim \frac{e^{-2r}}{r}.
\end{align*}
To obtain the next order derivative, it suffices to use the equation satisfied by $Q$ and the previous estimates.
\end{proof}

\begin{proof}[Proof of Proposition \ref{propo:asymp_Q}]
The proof is given by the triangle inequality with Lemma \ref{lemm:asymp_K_0} and Lemma \ref{lemm:closeness_Q_K_0}.
\end{proof}

\begin{proof}[Proof of \eqref{eq:chi_0}]
The proof that $\chi_0$ is close to a rescaled version of $K_0$ is similar to the one of Lemma \ref{lemm:closeness_Q_K_0}. Indeed, the rotationally invariant function $\chi_0$ satisfies the following equations, with $\lambda_0>0$:
\begin{align*}
    & -\partial_r^2 \chi_0(r) - \frac{1}{r} \partial_r \chi_0(r) + \chi_0(r) -2Q(r) \chi_0(r) +\lambda_0 \chi_0(r) =0, \\ 
    & -\partial_r^2 \left( \chi_0 \left( \frac{r}{\sqrt{1+\lambda_0}} \right) \right) - \left( \frac{1}{r} \partial_r \right) \left( \chi_0 \left( \frac{r}{\sqrt{1+\lambda_0}} \right)\right) \\
        & \qquad + \chi_0 \left( \frac{r}{\sqrt{1+\lambda_0}} \right) -\frac{2}{\sqrt{1+\lambda_0}}Q\left(\frac{r}{\sqrt{1+\lambda_0}}\right) \chi_0\left(\frac{r}{\sqrt{1+\lambda_0}}\right) =0.
\end{align*}
With the same computations as in Lemma \ref{lemm:improved_decay_Q}, we obtain the exponential decay of $\chi_0(\frac{\cdot}{\sqrt{1+\lambda_0}})$. As Lemma \ref{lemm:closeness_Q_K_0}, we obtain that:
\begin{align*}
    \exists \kappa_0>0, \exists c>0, \forall r>1,\quad \left\vert \chi_0\left(\frac{r}{\sqrt{1+\lambda_0}} \right) - \kappa_0 K_0 (r) \right\vert \leq c \frac{e^{-2r}}{r}.
\end{align*}
A last change of variable concludes the approximation \eqref{eq:chi_0}.
\end{proof}

\subsection{Asymptotic expansion of a translated ground-state}\label{sec:Q_n_app}

We now continue with the asymptotic expansion of a ground-state translated in the first direction. We recall that $Q$ is the ground state associated to \eqref{eq:Q}.

\begin{propo}\label{propo:Q_n_app_gen}
    There exist polynomials $\left(P_q(\bx)\right)_{q\in \mathbb{N}^{\ast}}$ even in $y$ and of degree at most $2q$ such that the following holds. Let us define the functions
    \begin{align*}
        \forall n\in \mathbb{N}^{\ast}, \forall z>1, \forall \bx \in \mathbb{R}^2, \quad  Q_{app}^n(\bx,z):= \kappa \frac{e^{-z}}{z^{\frac{1}{2}}} e^{-x} \left( a_0 + \sum_{1\leq q \leq n} \frac{P_q(\bx)}{z^q}\right) ,
    \end{align*}
    where $\kappa>0$ is given by Lemma \ref{lemm:closeness_Q_K_0} and $a_0=\sqrt{\frac{\pi}2}$.
    Then, for any $n \in \mathbb N^{\star}$, there exists $C=C_n>0$, such that, for any $\bj=(j_x,j_y)$ in $\{0,1,2\}^2$,
    \begin{align*}
        \forall \beta \in (0,\frac{1}{2}), \forall z>4, \forall \vert \bx \vert \leq z^{\beta}, \quad \left\vert \partial_\bx^{\bj} Q(\bx+ (z,0))- \partial_{\bx}^{\bj}Q^n_{app} (\bx,z) \right\vert \leq C \frac{e^{-z-x}}{z^{\frac{1}{2}+(n+1)(1-2\beta)}}.
    \end{align*}
\end{propo}

\begin{proof}
By the condition $z>4$, we have $z- z^\beta>1$ for $\beta <\frac{1}{2}$. Thus Proposition \ref{propo:asymp_Q} applies to $Q(\bx +(z,0))$\footnote{Since $Q$ is radial, we recall the identifications $Q(\bx) = Q(\vert \bx \vert)$ in the statement of Proposition \eqref{propo:asymp_Q}}, and we have for any $k\in \mathbb{N}$
\begin{align} \label{propo:Q_n_app_gen.1}
    \left\vert Q(\bx +(z,0)) - \kappa L_k(\vert (x+z,y)\vert ) \right\vert \leq C_k \frac{e^{-\vert (x+z,y) \vert}}{\vert (x+z,y) \vert^{k+ \frac{3}{2}}} \leq C_k\frac{e^{-z-x}}{z^{k+ \frac{3}{2}}},
\end{align}
where for the last inequality, we have used $\vert x+ z \vert = x+z \geq z-z^\beta > 1$ for $\vert x \vert \leq z^\beta$.

It suffices to develop the function $L_k(\vert (x+z,y)\vert )$, for some $k \in \mathbb{N}$ large enough, to conclude the proof of the proposition. We introduce the rescaled variables $X:= \frac{x}{z}$, $Y:= \frac{y}{z}$ and $\Tilde{Y}:= \frac{y^2}{z}$. Note that $\vert X \vert <1$, $\vert Y \vert<1$ and $\vert \Tilde{Y}\vert <1$ for $\vert \bx \vert <z^\beta$ under the condition $\beta<\frac{1}{2}$. Let us first begin with the power terms in $L_k$. We define
\begin{align*}
    S_l(X,Y):= \frac{a_l}{z^{\frac{1}{2}+l}} \frac{1}{(1+X)^{\frac{1}{2}+l}} \left( 1+ \frac{Y^2}{(1+X)^2}\right)^{-(\frac{1}{2}+l)}
\end{align*}
so that
\begin{align*}
    L_k(\vert \bx+ (z,0) \vert)= \exp(- \left\vert \bx +(z,0)\right\vert)\sum_{l=0}^k S_l(X,Y).
\end{align*}
The power series of the functions $S_l$ can be written for $(X,Y)\in (-1,1)^2$ in the form
\begin{align}\label{eq:S_l_decomposition}
    S_l(X,Y) = \frac{1}{z^{\frac{1}{2}+l}} \sum_{(s,t)\in \mathbb{N}^2} b_{l,s,t} X^sY^{2t},
\end{align}
for some real coefficients $b_{l,s,t}$ with $b_{l,0,0}= a_l$, for any $l \in \mathbb N$. We continue with the expansion of the exponential term in $L_k$. By using the same variables, we rewrite the exponential term as
\begin{align*}
    \exp \left(- \vert \bx +(z,0) \vert \right) = \exp\left(-
    \left( z+x \right) R(X,Y) \right),
\end{align*}
with
\begin{align*}
    R(X,Y) := \left( 1+ \frac{Y^2}{(1+X)^2} \right)^{\frac{1}{2}}.
\end{align*}
We use the power series expansion of $R$ for $(X,Y)\in (-1,1)^2$, with real coefficients $c_{j,m}$,
\begin{align*}
    R(X,Y) = 1+ \sum_{j =1}^{+\infty} Y^{2j} \sum_{m=0}^{+\infty} c_{j,m} X^{m}.
\end{align*}
By using this expansion, we have
\begin{align*}
    \exp \left( - \vert \bx +(z,0) \vert \right) & = e^{-z-x} \exp \left( - z \left( 1 + \frac{x}{z}\right) \left( R(X,Y)-1\right)\right) \\
    & = e^{-z-x} \exp \left( -(1+X) \Tilde{Y}\left(\sum_{j=0}^{+\infty} Y^{2j} \sum_{m=0}^{+\infty} c_{j+1,m} X^m \right) \right).
\end{align*}
Notice that the function defined by the double sum in the previous term converges for $(X,Y)\in (-1,1)^2$. We deduce the power series expansion for any $(X,Y,\Tilde{Y}) \in (-1,1)^3$, for some real coefficients $d_{j,m,p}$ with $d_{0,0,0}=1$,
\begin{align}\label{eq:exp_decomposition}
    \exp \left( - \vert \bx +(z,0) \vert \right) = e^{-z-x} \sum_{(j,m,p) \in \mathbb{N}^3} d_{j,m,p} Y^{2j} X^m \Tilde{Y}^p.
\end{align}
With \eqref{eq:S_l_decomposition} and \eqref{eq:exp_decomposition} in hands, we define
\begin{align*}
    I_q:=\left\{ \left( l,s,t,j,m,p\right) \in \mathbb{N}^6; \quad l + s +2t+ 2j+m+p =q \right\},
\end{align*}
and the polynomials $P_q$ by
\begin{align*}
    \forall q \in \mathbb{N}, \quad P_q(\bx) := \sum_{\left( l,s,t,j,m,p\right) \in I_q} b_{l,s,t} d_{j,m,p} x^{s+m}y^{2t+2j+2p}.
\end{align*}
The polynomials $P_q$ are even in $y$, at most of order $2k$, and the first polynomial $P_0$ is given by $P_0(\bx)=a_0$. 

For a given $n>0$, we choose $k>n$ such that $z^{-(k+\frac32)} \lesssim z^{-(\frac12+(n+1)(1-2\beta))}$. Hence, it follows from the above identities that
\begin{align} \label{propo:Q_n_app_gen.2}
    \left\vert L_k \left( \vert \bx +(z,0)\vert \right) - \frac{e^{-z-x}}{z^{\frac{1}{2}}} \left(a_0+\sum_{q=1}^n \frac{P_q(x,y)}{z^q} \right)\right\vert \lesssim \frac{e^{-z-x}}{z^{\frac{1}{2}+(n+1)}} \left( x^{n+1} +y^{2(n+1)} \right)\lesssim \frac{e^{-z-x}}{z^{\frac{1}{2} +(n+1)(1-2\beta)}},
\end{align}
under the conditions $|(x,y)| \le z^{\beta}$ and $\beta<\frac{1}{2}$.

Therefore, we conclude the proof of Proposition \ref{propo:Q_n_app_gen} in the case $\bj=(0,0)$ by gathering \eqref{propo:Q_n_app_gen.1} and \eqref{propo:Q_n_app_gen.2}.
The proof for $|\bj| \ge 1$ is similar.
\end{proof}

\section{Estimates involving exponentially decaying functions} \label{Append:Est}This appendix is dedicated to derive some useful estimates on exponentially decaying functions. 

\begin{lemm} \label{lemm:fg}
Let $f, \, g \in \mathcal{Y}$ and let $Z^{*} \ge 1 $ be chosen large enough. Assume that  $z>Z^{*}$ and recall the notation $\tilde{\bz}=(z,0)$. Then, the following estimates hold. 
\begin{itemize} 
\item[(i)] \emph{Pointwise estimate.}
 There exists $n \in \mathbb N$ such that,
for all $\bx=(x,y) \in \mathbb R^2$, 
 \begin{align} 
   \left| f\left(\bx+\frac{\tilde{\bz}}2\right)g\left(\bx-\frac{\tilde{\bz}}2\right)\right| \lesssim \min\left\{ e^{-z} \left( z^n+\langle \bx \rangle^n \right) ,e^{-\frac{31}{32}z} \left(e^{-\frac{|y|}{64}}{\bf 1}_{-\frac{z}2<x<\frac{z}2} +e^{-\frac1{64} \left\vert \left(x, y\right) \right\vert}  \mathbf{1}_{|x|>\frac{z}{2}}\right) \right\} . \label{est:fg:point}
\end{align}
\item[(ii)] \emph{$L^p$-estimate.} For any $1 \le p < \infty$,
\begin{align} 
   \left\| f\left(\cdot+\frac{\tilde{\bz}}2\right)g\left(\cdot-\frac{\tilde{\bz}}2\right)\right\|_{L^p} \lesssim e^{-\frac{15}{16}z}  . \label{est:fg:L2}
\end{align}
\end{itemize}
\end{lemm}

\begin{proof}We start by proving \eqref{est:fg:point}. In the case where $x>\frac{z}2$, we have $|\bx + \frac{\tilde{\bz}}{2} | \ge |x+\frac{z}2|>z$ and  $|\bx + \frac{\tilde{\bz}}{2}| \ge \frac{1}{\sqrt{2}}(|x|+|y|) \ge \frac1{\sqrt{2}}|(x,y)|$. Thus, it follows from the definition of $\mathcal{Y}$ in \eqref{def:Y} that, for some $n \in \mathbb N$,
\begin{equation*} 
\left| f\left(\bx+\frac{\tilde{\bz}}2\right)g\left(\bx-\frac{\tilde{\bz}}2\right)\right| 
\lesssim e^{-\left\vert \bx + \frac{\tilde{\bz}}{2} \right\vert} \left\langle\left\vert \bx + \frac{\tilde{\bz}}{2} \right\vert \right\rangle^n 
\lesssim \min\left\{ e^{-z} \left(z^n+ \langle \bx \rangle^n \right), e^{-\frac{31}{32}z}e^{-\frac1{64} \left\vert \left(x, y\right) \right\vert} \right\} .
\end{equation*}
The case where $x<-\frac{z}2$ is treated in a similar way. In the case where $-\frac{z}2 \le x \le \frac{z}2$, we have $|\bx + \frac{\tilde{\bz}}{2}| \ge |x+\frac{z}2|=x+\frac{z}2$ and $|\bx - \frac{\tilde{\bz}}{2}| \ge |x-\frac{z}2|=\frac{z}2-x$, so that it follows for some $n \in \mathbb N$,
\begin{equation*} 
\begin{split}
\left| f\left(\bx+\frac{\tilde{\bz}}2\right)g\left(\bx-\frac{\tilde{\bz}}2\right)\right| 
 &\lesssim \min\left\{ e^{-z} \left(z^n+ \langle \bx \rangle^n \right),e^{-\frac{31}{32}(x+\frac{z}2)}e^{-\frac1{64}|y|} e^{-\frac{31}{32}(\frac{z}2-x)}e^{-\frac1{64}|y|} \right\} \\ & 
\lesssim \min\left\{ e^{-z} \left(z^n+ \langle \bx \rangle^n \right),e^{-\frac{31}{32}z}e^{-\frac1{64} \left\vert y \right\vert} \right\} .
\end{split}
\end{equation*}
The proof of \eqref{est:fg:point} follows then gathering these estimates, while the proof of \eqref{est:fg:L2} is a direct consequence of \eqref{est:fg:point} and integration over $\mathbb R^2$. 
\end{proof}

\begin{rema} \label{rema:QzQ}By arguing as in the proof of Lemma \ref{lemm:fg} (i) and using Proposition \ref{propo:asymp_Q}, we deduce that for any $z>Z^{\star}$, $\bj$ and $\bf{k}$ in $\mathbb N^2$ with $|\bj|, |\bf{k}| \le 2$, and for all $\bx=(x,y) \in \mathbb R^2$, 
\begin{equation} \label{est.decay.QzQ}
\left| \partial^{\bj}Q\left(\bx+\tilde{\bz}\right)\partial^{\bf{k}}Q(\bx)\right| \lesssim e^{-z}z^{-\frac12} .
\end{equation}
\end{rema}

\begin{rema} \label{rema:weakfg}
Estimate \eqref{est:fg:L2} still holds under the weaker assumption that $f$, $g \in C^{\infty}(\mathbb R^2)$ satisfy $|f(\bx)|+|g(\bx)| \lesssim e^{-\frac{31}{32}|\bx|}$. The proof is identical to the one of \eqref{est:fg:L2}. 
\end{rema}

\begin{lemm}\label{lemm:R_1_R_2}
There exist $0<\nu^{*} \le 1$ small enough and $Z^{*} \ge 1$ large enough such that the following holds. Let $R_i(\bx) :=Q_{1+\mu_i}\left( x-z_i, y-\omega_i \right)$ and $\Rti(\bx) :=Q\left( x-z_i, y\right)$, with 
$\vert \mu_i \vert \le \nu^{*}$, $\vert \omega_i \vert \le \nu^{*}$  and $\vert z\vert > Z^{*}$, where $z=z_1-z_2$. Then, for $i=1,2$,
\begin{align} 
&\left|R_i(\bx)- \Rti(\bx)-\mu_i\Lambda \Rti(\bx)-\omega_i\partial_y \Rti(\bx) \right| \lesssim e^{-\frac{31}{32}|\bx-\tilde{\bz}_i|}\left(\mu_i^2+\omega_i^2\right) ; \label{exp:Ri_point} \\ 
&\left\|R_i- \Rti-\mu_i\Lambda \Rti-\omega_i\partial_y \Rti \right\|_{H^1} \lesssim 
\mu_i^2+\omega_i^2  .\label{exp:Ri_H1}  
\end{align}
We also have, 
\begin{equation}\label{exp:MRi_H1}
\left\|\partial_xR_i- \partial_x\Rti \right\|_{H^1}
+\left\|\partial_yR_i- \partial_y\Rti \right\|_{H^1}
+\left\|\Lambda R_i- \Lambda \Rti \right\|_{H^1}\lesssim |\mu_i|+|\omega_i| .
\end{equation}
Moreover, for any $\bj$ and $\bf{k}$ in $\mathbb N^2$ with $|\bj|, |\bf{k}| \le 2$, $1 \le p <\infty$
\begin{align} 
   &\left\|  \partial^\bj R_1 \partial^{\bf{k}} R_2 \right\|_{L^p} + \left\|  \partial^\bj R_1 \partial^{\bf{k}} \Lambda R_2 \right\|_{L^p} + \left\|  \partial^\bj \Lambda R_1 \partial^{\bf{k}} R_2 \right\|_{L^p} +\left\|  \partial^\bj \Lambda R_1 \partial^{\bf{k}} \Lambda R_2 \right\|_{L^p} \lesssim e^{-\frac{15}{16}z} ;\label{est:R_1R_2} \\ 
   & \left\|R_1R_2-\Rtone\Rttwo-\left(\mu_1\Lambda \Rtone+\omega_1\partial_y \Rtone\right)\Rttwo-\Rtone\left(\mu_2\Lambda \Rttwo+\omega_2\partial_y \Rttwo\right)\right\|_{H^1}\lesssim e^{-\frac{15}{16}z} \sum_{j=1}^2  \left( \mu_j^2+\omega_j^2 \right) . \label{est:R1R2.2}
\end{align}
\end{lemm}

\begin{proof}
Estimate \eqref{exp:Ri_point} is a direct application of Lemma \ref{lemm:decompo_rescaled_Q}. Then it follows by integrating \eqref{exp:Ri_point} over $\mathbb R^2$ that
\begin{equation*}
\left\|R_i- \Rti-\mu_i\Lambda \Rti-\omega_i\partial_y \Rti \right\|_{L^2} \lesssim 
\mu_i^2+\omega_i^2  .
\end{equation*}
The estimates for the other pieces of the left-hand term in \eqref{exp:Ri_H1} are obtained arguing similarly.  

Estimate \eqref{exp:MRi_H1} is proved arguing as in the proof of estimate \eqref{exp:Ri_H1} after deriving expansions for $\partial_xR_i$, $\partial_yR_i$ and $\Lambda R_i$ similar to the one for $R_i$ in \eqref{exp:Ri_point}.

We explain how to deal with  the estimate for $\left\|\int R_1 R_2\right\|_{L^p}$ in \eqref{est:R_1R_2}. The estimates for the other terms in \eqref{est:R_1R_2} are similar.  By using the decomposition of $R_i$ in \eqref{exp:Ri_point} and estimate \eqref{est:fg:L2} together with Remark 
\ref{rema:weakfg}, we have
\begin{align*}
    \left\| \int R_1 R_2  - \int \left( \Rtone - \mu_1 \Lambda \Rtone - \omega_1 \partial_y \Rtone \right) \left( \Rttwo- \mu_2 \Lambda \Rttwo -\omega_2 \partial_y \Rttwo \right) \right\|_{L^p} \lesssim e^{-\frac{15}{16} z}\sum_{i=1}^2\left(\mu_i^2 + \omega_i^2\right).
\end{align*}
Therefore, we conclude the proof of \ref{est:R_1R_2} by using \eqref{est:fg:L2}.

Now we prove \eqref{est:R1R2.2}. By the triangle inequality, 
\begin{align*}
    &\left\|R_1R_2-\Rtone\Rttwo- \left(\mu_1\Lambda \Rtone+\omega_1 \partial_y \Rtone\right) \Rttwo-\Rtone \left(\mu_2\Lambda \Rttwo+\omega_2 \partial_y \Rttwo\right)\right\|_{H^1} \\ 
    & \lesssim \left\|\left(R_1-\Rtone-\mu_1\Lambda \Rtone-\omega_1\partial_y \Rtone\right) \Rttwo\right\|_{H^1} +\left\|\Rtone \left(R_2-\Rttwo-\mu_2\Lambda \Rttwo-\omega_2\partial_y \Rttwo\right)\right\|_{H^1} \\
        & \qquad +\left\|\left(R_1-\Rtone\right)\left(R_2-\Rttwo\right) \right\|_{H^1} .
\end{align*}
Then, we conclude the proof of \eqref{est:R1R2.2} by using \eqref{est:fg:L2} together with Remark \ref{rema:weakfg}.
\end{proof}

\begin{lemm} \label{est:tildeR_NR}
 Let $Z^{*} \ge 1 $ be chosen large enough. Assume that  $z=z_1-z_2>Z^{*}$, $-z \le z_2 \le -\frac14 z$, $\frac14z \le z_1 \le z$, and recall the definition of $X_i$, $W_i$ and $l$ in \eqref{def:XYW:i} and \eqref{def:hl}. Then, the following estimates hold. 
\begin{align}
&\left\|\Rtone X_2\right\|_{H^1}+ \left\|\Rttwo X_1\right\|_{H^1} \lesssim e^{-\frac{15}{16} z} ; \label{est:tildeR_X}\\
& \left\|\Rtone W_2\right\|_{H^1}+ \left\|\Rtone Y_2\right\|_{H^1}\lesssim e^{-\frac{15}{16} z} ;\label{est:tildeR1_W2} \\
& \left\|\Rttwo\left(W_1-l\right) 
\right\|_{H^1}+\left\|\Rttwo\left(Y_1-h\right) 
\right\|_{H^1} \lesssim e^{-\frac{15}{16} z} . \label{est:tildeR2_W1-l}
\end{align}
\end{lemm}

\begin{proof}
Observe that $X_i(x,y)=(-\Delta+1)^{-1}Q(x-z_i,y)$ with $(-\Delta+1)^{-1}Q \in \mathcal{Y}$ (see Lemma \ref{Bessel:Y}). Thus, estimate \eqref{est:tildeR_X} follows combining \eqref{est:fg:L2} with a change of variable. 

Next, we prove estimate \eqref{est:tildeR1_W2}. We have from the definition of $W_2$ and $Y_2$ in \eqref{def:XYW:i}, the one of $\Lambda Q$ in \eqref{defi:Lambda_Q} and lemma \ref{Bessel:Y} that 
\begin{equation*}
\left|\Rtone W_2 (x,y) \right|+\left|\Rtone Y_2 (x,y) \right| \lesssim e^{-|(x-z_1,y)|} \int_{x-z_2}^{+\infty}e^{-\frac{15}{16}|\tilde{x}|}d\tilde{x}  .
\end{equation*}
In the case where $x>z_1$, we have $|x-z_2|=x-z_2 >z$, so that  
\begin{equation*}
\left|\Rtone W_2 (x,y) \right|+\left|\Rtone Y_2 (x,y) \right| \lesssim  e^{-\frac{15}{16}z}e^{-|(x-z_1,y)|}  .
\end{equation*}
In the case where $z_2<x<z_1$, we use that $|(x-z_1,y)| \ge \frac{31}{32}(z_1-x)+\frac1{32}|y|$. Hence, 
\begin{equation*}
\left|\Rtone W_2 (x,y) \right| +\left|\Rtone Y_2 (x,y) \right|\lesssim  e^{-\frac{31}{32}z}e^{-\frac1{32}|y|}{\bf 1}_{z_2 <x<z_1}  .
\end{equation*}
In the case where $x<z_2$, we have $|x-z_1|=z_1-x>z$, so that 
\begin{equation*}
\left|\Rtone W_2 (x,y) \right|+\left|\Rtone Y_2 (x,y) \right| \lesssim  e^{-\frac{15}{16}z}e^{-\frac1{16}|(x-z_1,y)|} .
\end{equation*}
We deduce combining these estimates and integrating over $\mathbb R^2$ that 
\begin{equation*}
\left\|\Rtone W_2\right\|_{L^2}+\left\|\Rtone Y_2\right\|_{L^2} \lesssim e^{-\frac{15}{16} z} .
\end{equation*}
The  estimate  
\begin{equation*}
\left\|\nabla \left(\Rtone W_2\right)\right\|_{L^2}+\left\|\nabla \left(\Rtone Y_2\right)\right\|_{L^2} \lesssim e^{-\frac{15}{16} z}
\end{equation*}
is proved in a similar way. This concludes the proof of \eqref{est:tildeR1_W2}. 

Finally, we observe from the definitions of $W_1$, $Y_1$  in \eqref{def:XYW:i} and  $l$ and $h$ \eqref{def:hl} that 
\begin{align*}
&\left|\left(\Rttwo\left(W_1-l\right) \right)(x,y)\right|=\left|Q(x-z_2,y)\int_{-\infty}^{x-z_1}\left(-\Delta+1\right)^{-1}\Lambda Q(\tilde{x},y) d\tilde{x} \right| \\ 
&\left|\left(\Rttwo\left(Y_1-h\right) \right)(x,y)\right|=\left|Q(x-z_2,y)\int_{-\infty}^{x-z_1}\left(-\Delta+1\right)^{-1}\partial_yQ(\tilde{x},y) d\tilde{x} \right| 
\end{align*}
and thus
\begin{align*}
   \left|\left(\Rttwo\left(W_1-l\right) \right)(x,y)\right| + \left|\left(\Rttwo\left(Y_1-h\right) \right)(x,y)\right| \lesssim e^{-|(x-z_2,y)|} \int_{-\infty}^{x-z_1}e^{-\frac{15}{16}|\tilde{x}|} d\tilde{x}.
\end{align*}
Hence, the proof of \eqref{est:tildeR2_W1-l} follows arguing as in the proof of \eqref{est:tildeR1_W2}. 
\end{proof}

\begin{rema} \label{est:tildeR_NR_refined}
Under the assumptions of Lemmata \ref{lemm:R_1_R_2} and \ref{est:tildeR_NR}, we have
\begin{align*}
\MoveEqLeft
\left\|(R_1-\Rtone) X_2\right\|_{H^1}+ \left\|(R_2-\Rttwo) X_1\right\|_{H^1} + \left\|(R_1-\Rtone) W_2\right\|_{H^1}+  \left\|(R_2-\Rttwo)\left(W_1-l\right) \right\|_{H^1} \\
& \qquad + \left\|(R_1-\Rtone) Y_2\right\|_{H^1} + \left\|(R_2-\Rttwo)\left(Y_1-h\right) \right\|_{H^1} \lesssim e^{-\frac{15}{16} z}\sum_{i=1}^2 \left(|\mu_i|+|\omega_i| \right) . 
\end{align*}
These estimates are proved by using \eqref{exp:Ri_point} and arguing as in the proof of Lemma \ref{est:tildeR_NR}.
\end{rema}

Next, we give the proof of Claim \ref{claim:Vpsi}, which relies on the following estimates.  
\begin{lemm} \label{lemm:fpsi} Let $f\in C^{\infty}(\mathbb R^2)$ be such that $|f(\bx)| \lesssim e^{-\frac78|\bx|}$, for all $\bx \in \mathbb R^2$, and let $Z^{*} \ge 1 $ be chosen large enough. Assume that  $z_1>Z^{*}$ and $z_2<-Z^{*}$ with for any $i$, $\vert z_i \vert > \frac14 z$ and recall the notation $\tilde{\bz_i}=(z_i,0)$, for $i=1,2$. Then, the following estimates hold. 
 For all $\bx=(x,y) \in \mathbb R^2$,
\begin{align} 
  & \left| f\left(\bx-\tilde{\bz}_1\right)(1-\psi)(x)\right| \lesssim e^{- \rho z}e^{-\frac1{8}\vert (x-z_1,y)\vert}  ; \label{est:f(1-psi)}\\ 
  & \left| f\left(\bx-\tilde{\bz}_2\right)\psi(x)\right| \lesssim e^{-\rho z}e^{-\frac1{8}\vert (x-z_2,y)\vert}; \label{est:fpsi} \\
  & \left| f\left(\bx-\tilde{\bz}_i\right)\psi'(x)\right|\lesssim e^{-\rho z} e^{-\frac1{8}\vert (x-z_i,y)\vert}, \quad i=1,2 . \label{est:fi_psi'}
\end{align}
\end{lemm}

\begin{proof}
We only prove \eqref{est:f(1-psi)}, since the proofs of \eqref{est:fpsi} and \eqref{est:fi_psi'} are similar. 

In the region $x>z_1/2$, observe from the definition of $\psi$ and the second identity in \eqref{prop:psi.1}, that $1-\psi(x)=\psi(-x) \lesssim e^{-8\rho x} \lesssim e^{-4 \rho z_1}$. Since $f \in \mathcal{Y}$, it follows that 
\begin{equation*} 
 \left| f\left(\bx-\tilde{\bz}_1\right)(1-\psi)(x)\right| \lesssim e^{-\rho z}e^{-\frac12\vert (x-z_1,y)\vert} .
\end{equation*}
In the region $x<z_1/2$, we have that $|(x-z_1,y)| \ge |x-z_1| \ge \frac{|z_1|}2$. Thus, 
\begin{equation*} 
 \left| f\left(\bx-\tilde{\bz}_1\right)(1-\psi)(x)\right| \lesssim e^{-\frac3{8} z_1}e^{-\frac18\vert (x-z_1,y)\vert} .
\end{equation*}
Combining these two estimates yields \eqref{est:f(1-psi)}.
\end{proof} 

\begin{proof}[Proof of Claim \ref{claim:Vpsi}]
The proofs of estimates \eqref{est:MRipsi} and \eqref{est:MRipsi'} are consequences of Lemma \ref{lemm:fpsi} and Lemma \ref{lemm:decompo_rescaled_Q}. 

By using the definition of $V$ in \eqref{def:V}, we observe that 
\begin{equation*} 
V \psi'= \sum_{i=1}^2R_i\psi'+ V_A \psi' ,
\end{equation*}
so that the proof of \eqref{est:V dxpsi+} follows from \eqref{est:VA:point}, \eqref{est:VA:point}, \eqref{est:fi_psi'} and Lemma \ref{lemm:decompo_rescaled_Q}.

By using the definitions of $V$ in \eqref{def:V} and of $\psi_+$ in \eqref{defi:psi_+}, we decompose $\partial_tV+\partial_xV \psi_+$ as
\begin{align*} 
    \partial_tV+\partial_xV \psi_+ 
    & = \sum_{i=1}^2 \left( -(\dot{z}_i-\mu_i)\partial_xR_i-\dot{\omega}_i\partial_yR_i+\dot{\mu}_i\Lambda R_i \right) \\
    & \qquad + (\mu_1-\mu_2) \left((\psi-1)\partial_xR_1-\psi\partial_xR_2 \right)+\partial_tV_A+\partial_xV_A\psi_+ .
\end{align*}
Thus, we conclude the proof of estimate \eqref{est:Vt-ViPsi+} by gathering \eqref{est:VA:H2}, \eqref{est:dVAdt}, \eqref{est:f(1-psi)}, \eqref{est:fpsi} with $0<\rho<\frac1{32}$ and Lemma \ref{lemm:decompo_rescaled_Q}.

Finally, by using the definitions of $V$ in \eqref{def:V} and of $\psi_{-,e}$ and $\psi_{-,m}$ in \eqref{defi:psi_-e}-\eqref{defi:psi_-m}, we decompose $\partial_t V \psi_{-,e} + \partial_x V \psi_{-,m}$ as 
\begin{align*} 
\MoveEqLeft
\partial_t V \psi_{-,e} + \partial_x V \psi_{-,m} \\
    &= -(\dot{z}_1-\mu_1)\frac{\psi}{(1+\mu_1)^2}\partial_xR_1-(\dot{z}_2-\mu_2)\frac{1-\psi}{(1+\mu_2)^2}\partial_xR_2 -(\dot{z}_1-\mu_2)\frac{1-\psi}{(1+\mu_2)^2} \partial_xR_1\\ &\quad -(\dot{z}_2-\mu_1)\frac{\psi}{(1+\mu_1)^2} \partial_xR_2 +\sum_{i=1}^2 \left(-\dot{\omega}_i\partial_yR_i+ \dot{\mu}_i\Lambda R_i \right)\psi_{-,e}+\partial_tV_A \psi_{-,e}+\partial_xV_A\psi_{-,m} .
\end{align*}
Hence, we conclude the proof of \eqref{est:Vtpsi-e-Vpsi-m} by using \eqref{est:VA:H2}, \eqref{est:dVAdt}, \eqref{est:f(1-psi)}, \eqref{est:fpsi} with $0<\rho<\frac1{32}$ and Lemma \ref{lemm:decompo_rescaled_Q}, and recalling from 
Assumption \ref{hyp:coeff} that $|\mu_1|+|\mu_2| \le C\nu_0^\star$, for $\nu_0^\star$ small enough. 
\end{proof}

\begin{toexclude}
\bigskip 
\blue{To be continued}

\begin{lemm}\label{lemm:product_Q}
For $z>Z_0$, the following estimate holds for $\bx \in \mathbb{R}^2$
\begin{align*}
    & Q(\bx) Q(\bx-(z,0)) \lesssim e^{z_2+x}\mathbf{1}_{x \leq z_2} + \\
    & \int_{\bx \in \mathbb{R}^2} Q(\bx) Q(\bx-(z,0)) d\bx \lesssim 
\end{align*}
For $f,g\in \Y$, with $\vert f(\bx) \vert + \vert g(\bx) \vert \lesssim \langle \bx \rangle^{n}e^{-\vert \bx \vert}$, we have
\begin{align}\label{eq:product_Y}
    & \left\vert f\left(\bx+\left(\frac{z}{2},0\right)\right) g\left(\bx-\left(\frac{z}{2},0\right)\right) \right\vert \lesssim \left( \right)  \mathbf{1}_{2 \vert x \vert >z} + \left( \right) \mathbf{1}_{ z- \vert y \vert \leq 2\vert x\vert \leq z} + \left( \right) \mathbf{1}_{\vert y \vert + 2 \vert x \vert \leq z } \\
    & \left\vert \partial_x^{-1} f\left( x+\frac{z}{2},y\right) g\left( x-\frac{z}{2},y\right) \right\vert \lesssim  \left( \langle z\rangle +\left\langle \left( x-\frac{z}{2},y \right) \right\rangle\right)^{2n} e^{-z - \left\vert \left( x-\frac{z}{2},y \right) \right\vert}\mathbf{1}_{x>\frac{z}{2}} +\langle z \rangle^{2n} e^{-z} \mathbf{1}_{-\frac{z}{2} < x <\frac{z}{2}- \frac{3}{8}\vert y \vert} \nonumber \\
        & \quad + \left( \langle z \rangle^{2n}+ \langle y \rangle^{2n} \right) e^{-z-\frac{y}{2}} \mathbf{1}_{\max(-\frac{z}{2}, \frac{z}{2}- \frac{3}{8} \vert y \vert )\leq x \leq \frac{z}{2}} + \left( \langle z \rangle + \left\langle \left( x-\frac{z}{2},y \right) \right\rangle \right)^{2n} e^{- \left\vert \left( x-\frac{z}{2},y\right)\right\vert} \mathbf{1}_{x<-\frac{z}{2}} \label{eq:product2}
\end{align}
\blue{We can replace the last inequality by the following
\begin{align*}
    \forall x\in \left( -\frac{z}{2} , \frac{z}{2}\right), \quad \left\vert \partial_x^{-1} f\left(x+\frac{z}{2},y \right) g\left( x- \frac{z}{2},y\right) \right\vert \lesssim \left( \langle z\rangle^{2n} + \langle y \rangle^{2n} \right) \exp \left( - \frac{3}{4} z - \frac{1}{2}y\right).
\end{align*}
}
\end{lemm}

\begin{proof}
For \eqref{eq:product_Y}, 

We continue with the computations for \eqref{eq:product2}. First, we notice that if $f(\bx) \lesssim \langle \bx \rangle^{n} e^{-\vert \bx \vert}$, then
\begin{align*}
    \left\vert \partial_x^{-1} f(\bx)\right\vert \lesssim \min \left(1, \left( \langle x \rangle^{n} +\langle y \rangle^{n} \right) e^{-x} \right)
\end{align*}
The multiplication of $\partial_x^{-1}f\left( x+\frac{z}{2},y\right)$ by $g\left( x-\frac{z}{2},y \right)$ with the help of \red{ref} gives \eqref{eq:product2}.

\blue{The last proof is dealt using the following inequalities.
\begin{align*}
    \left\vert \partial_x^{-1} f(\bx) \right\vert \lesssim \partial_x^{-1} \left(\left( \langle x \rangle^{n} + \langle y \rangle^n \right) e^{-\frac{1}{4}\vert y\vert -\frac{3}{4}\vert x \vert} \right) \lesssim \left( \langle x \rangle^{n} + \langle y \rangle^n \right) e^{-\frac{1}{4}\vert y\vert -\frac{3}{4}\vert x \vert} \quad \text{and} \quad \left\vert g(\bx) \right\vert \lesssim \langle \bx \rangle^{n} e^{-\frac{3}{4}\vert x \vert - \frac{1}{4} \vert y \vert}.
\end{align*}
Next, by using the product of the two functions, for $x \in [-\frac{z}{2},\frac{z}{2}]$,
\begin{align*}
    \left\vert \partial_x^{-1} f\left( x+\frac{z}{2},y\right) g \left( x-\frac{z}{2},y \right) \right\vert \lesssim \left( \langle z \rangle^{2n} + \langle y \rangle^{2n}\right) \exp \left( -\frac{1}{2}y- \frac{3}{4}\left( \frac{z}{2}- x \right) - \frac{3}{4} \left( \frac{z}{2}+x \right) \right).
\end{align*}
}
\end{proof}

\begin{lemm}
For $z>\red{Z_0}$ The following estimates hold on $\mathbb{R}$
\begin{align}\label{est:non_lin1}
     &\left\vert (-\Delta+1)^{-1} \left( Q(\bx) Q(\bx-(z,0))\right) \right\vert \lesssim \frac{e^{-z}}{\sqrt{z}}, \\
     & \left\vert (-\Delta+1)^{-1} \left( Q(\bx) Q_{app}^n (\bx,z)\right) \right\vert \lesssim \frac{e^{-z}}{\sqrt{z}} \sum_{k=0}^n \frac{\langle \bx\rangle^{2k}}{z^k}, \label{est:non_lin2} \\
     & \left\vert (-\Delta+1)^{-1} \left( Q(\bx)\left(  Q(\bx-(z,0))-Q_{app}^n (\bx,z)\right) \right) \right\vert \lesssim   \label{est:non_lin3}
\end{align}
\end{lemm}

\begin{proof}
For \eqref{est:non_lin1}, 
\end{proof}
\end{toexclude}

\section{Solutions to the second order dynamical system}\label{app:Z}

The goal of this appendix is to study the solutions $Z$ of the autonomous ODE
\begin{align}\label{defi:eq_Z}
    \Ddot{Z}(t)= \frac{2}{\langle \Lambda Q, Q \rangle} \int_{\mathbb{R}^2} Q(x + Z(t),y) \partial_x (Q^2) (x,y) dxdy.
\end{align}
The function $Z$ is an approximation of the distance between the two solitary waves. In particular, at $-\infty$ the two solitary waves are far away one from another and are repulsing at a velocity $-2\mu_0$, and thus, we should look for the solutions of \eqref{defi:eq_Z} satisfying $\lim_{t\rightarrow -\infty} (Z(t),\dot{Z}(t))=(+\infty,-2\mu_0)$. We also expect the collision to be almost symmetric, in the sense that $Z$ should be even. 

\begin{rema} \label{rem:non_explicit_solutions}
In the one dimensional case of the gKdV equation, the authors in \cite{MM11} use the first order asymptotic of \eqref{defi:eq_Z}, $\ddot{Z}=ce^{-Z}$, which admits explicit solutions.  
\begin{toexclude}
for some positive constant $c$: this equation can be found by replacing in \eqref{defi:eq_Z} the term $Q(\cdot+Z,\cdot)$ by its first order asymptotic expansion where $\partial_x(Q^2)$ is located, thus around $0$ (see the asymptotic expansion in dimension $2$ in Proposition \eqref{propo:Q_n_app_gen}). In the $1$-dimensional context, the solution $Z$ is explicit and the behaviour of the asymptotic expansion at different points can be given in term of the solution. 
\end{toexclude}
In the $2$-dimensional case, the equation and its asymptotics cannot be solved explicitly. Instead, we describe the set of solutions by using ODE arguments.
\end{rema}

\subsection{Phase portrait of the solutions}

In this subsection, we consider for a positive constant $\kappa$, the ODE
\begin{align}\label{defi:eq_Z_general}
    \ddot{Z}(t) = \kappa \int_{\mathbb{R}^2} Q(x+Z(t),y) \partial_x \left(Q^2\right)(x,y)dx dy ,
\end{align}
which corresponds to \eqref{defi:eq_Z} in the case $\kappa=2\langle \Lambda Q, Q \rangle^{-1}$. 
We define the associated Hamiltonian function
\begin{align}\label{defi:H_static}
    H (Y_0,Y_1) := \frac{Y_1^2}{2} + \kappa \int_{\mathbb{R}^2} Q(x+Y_0,y) Q^2(x,y) dx dy.
\end{align} 

\begin{lemm} \label{lemm:hamiltonian}
The Hamiltonian $H(Z(t),\dot{Z}(t))$ is preserved along the trajectories of \eqref{defi:eq_Z_general}. 

More precisely, let $I$ be an interval containing $0$ and let $Z:I \to \mathbb R$ be a $C^2$ solution to \eqref{defi:eq_Z_general} with initial datum $(Z(0),\dot{Z}(0))=(Y_0,Y_1) \in \mathbb R^2$. Then, $H(Z(t),\dot{Z}(t))=H(Y_0,Y_1)$, for all $t \in I$.
\end{lemm}

\begin{proof}
We observe that for any solution $Z$ of \eqref{defi:eq_Z_general}, $\frac{d}{dt}H(Z(t),\dot{Z}(t))=0$. Hence, we conclude the proof by integrating between $0$ and $t$.
\end{proof}

\begin{propo}\label{propo:phase_portrait}
Let $(Y_0,Y_1)\in \mathbb{R}^2$. Then, there exists a unique solution $Z \in C^2(\mathbb R : \mathbb R)$ of \eqref{defi:eq_Z_general} satisfying $\left(Z(0),\dot{Z}(0)\right)=(Y_0,Y_1)$.  
Furthermore, the phase portrait of \eqref{defi:eq_Z_general} is given by Figure \ref{phase_portrait}.
\end{propo}

\begin{figure}
    \centering
    \begin{tikzpicture}
        \draw [->] (-7,0) -- (7,0);
        \draw [->] (0,-2.75) -- (0,2.75);
        \draw (7,0) node[right] {$Z$};
        \draw (0,2.75) node[right] {$\dot{Z}$};
    
        \draw[red] (0,0) node {$\bullet$};
    
            \draw[blue,->,>=latex,domain=0.3:1,samples=\Num] plot [variable=\t] (\t*\t,{rad(atan(\t))*1.1}); 
            \draw[blue,domain=1:2.6,samples=\Num] plot [variable=\t] (\t*\t,{rad(atan(\t))*1.1});
            
            \draw[blue,domain=0.3:1,samples=\Num] plot [variable=\t] (\t*\t,{-rad(atan(\t))*1.1}); 
            \draw[blue,->,>=latex,domain=2.6:1,samples=\Num] plot [variable=\t] (\t*\t,{-rad(atan(\t))*1.1});
            
            \draw[blue,domain=0.3:1,samples=\Num] plot [variable=\t] (-\t*\t,{rad(atan(\t))*1.1}); 
            \draw[blue,->,>=latex,domain=2.6:1,samples=\Num] plot [variable=\t] (-\t*\t,{rad(atan(\t))*1.1});
            
            \draw[blue,->,>=latex,domain=0.3:1,samples=\Num] plot [variable=\t] (-\t*\t,{-rad(atan(\t))*1.1}); 
            \draw[blue,domain=1:2.6,samples=\Num] plot [variable=\t] (-\t*\t,{-rad(atan(\t))*1.1});
    
            \draw [->,>=latex,domain=-2.4:1,samples=\Num] plot [variable=\t] (1+\t*\t,{rad(atan(\t))});
            \draw [domain=1:2.4,samples=\Num] plot [variable=\t] (1+\t*\t,{rad(atan(\t))});
            
            \draw [->,>=latex,domain=-2.2:1,samples=\Num] plot [variable=\t] (2+\t*\t,{rad(atan(\t))/2});
            \draw [domain=1:2.2,samples=\Num] plot [variable=\t] (2+\t*\t,{rad(atan(\t))/2});
            
            \draw [->,>=latex,domain=-1.7:1,samples=\Num] plot [variable=\t] (4+\t*\t,{rad(atan(\t))/3});
            \draw [domain=1:1.7,samples=\Num] plot [variable=\t] (4+\t*\t,{rad(atan(\t))/3});
            
            \draw [->,>=latex,domain=-2.4:1,samples=\Num] plot [variable=\t] (-1-\t*\t,{-rad(atan(\t))});
            \draw [domain=1:2.4,samples=\Num] plot [variable=\t] (-1-\t*\t,{-rad(atan(\t))});
            
            \draw [->,>=latex,domain=-2.2:1,samples=\Num] plot [variable=\t] (-2-\t*\t,{-rad(atan(\t))/2});
            \draw [domain=1:2.2,samples=\Num] plot [variable=\t] (-2-\t*\t,{-rad(atan(\t))/2});
            
            \draw [->,>=latex,domain=-1.7:1,samples=\Num] plot [variable=\t] (-4-\t*\t,{-rad(atan(\t))/3});
            \draw [domain=1:1.7,samples=\Num] plot [variable=\t] (-4-\t*\t,{-rad(atan(\t))/3});

            \draw [->,>=latex,domain=-6.9:1,samples=\Num] plot [variable=\t] (\t,{1.5*(\t*\t+1)/(\t*\t+2)});
            \draw [domain=1:6.9,samples=\Num] plot [variable=\t] (\t,{1.5*(\t*\t+1)/(\t*\t+2)});
            
            \draw [->,>=latex,domain=-6.9:1,samples=\Num] plot [variable=\t] (\t,{2*(\t*\t+1)/(\t*\t+2)});
            \draw [domain=1:6.9,samples=\Num] plot [variable=\t] (\t,{2*(\t*\t+1)/(\t*\t+2)});
            
            \draw [->,>=latex,domain=-6.9:1,samples=\Num] plot [variable=\t] (\t,{0.25+2.5*(\t*\t+2)/(\t*\t+4)});
            \draw [domain=1:6.9,samples=\Num] plot [variable=\t] (\t,{0.25+2.5*(\t*\t+2)/(\t*\t+4)});
            
            \draw [->,>=latex,domain=-6.9:1,samples=\Num] plot [variable=\t] (-\t,{-1.5*(\t*\t+1)/(\t*\t+2)});
            \draw [domain=1:6.9,samples=\Num] plot [variable=\t] (-\t,{-1.5*(\t*\t+1)/(\t*\t+2)});
            
            \draw [->,>=latex,domain=-6.9:1,samples=\Num] plot [variable=\t] (-\t,{-2*(\t*\t+1)/(\t*\t+2)});
            \draw [domain=1:6.9,samples=\Num] plot [variable=\t] (-\t,{-2*(\t*\t+1)/(\t*\t+2)});
            
            \draw [->,>=latex,domain=-6.9:1,samples=\Num] plot [variable=\t] (-\t,{-0.25-2.5*(\t*\t+2)/(\t*\t+4)});
            \draw [domain=1:6.9,samples=\Num] plot [variable=\t] (-\t,{-0.25-2.5*(\t*\t+2)/(\t*\t+4)});
    \end{tikzpicture}
    \caption{Phase portrait of equation \eqref{defi:eq_Z_general}. In red: the constant solution. In blue: the four solutions defined on $\mathbb{R}$ ending at $(0,0)$ when $t$ goes to $-\infty$ or $+\infty$. In black: all the other trajectories.}\label{phase_portrait}
\end{figure}
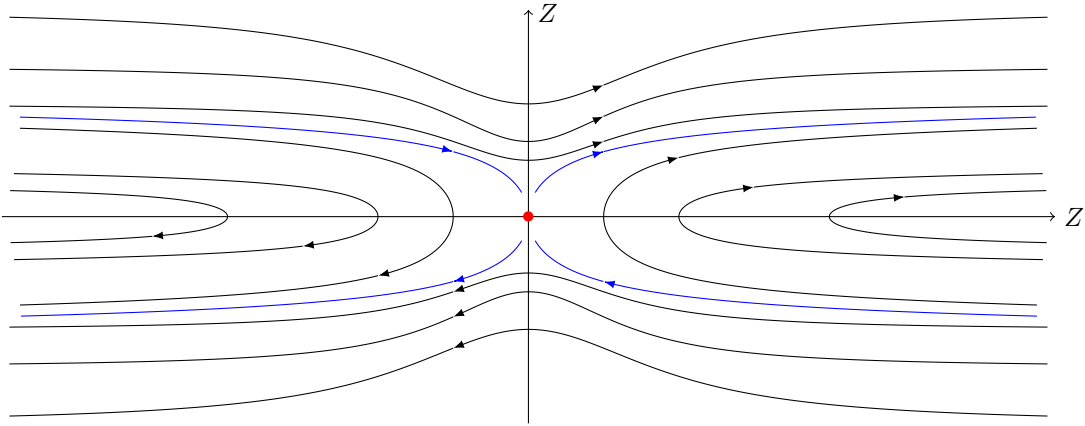

\begin{proof}
First, notice for $\tilde{Y}_0, Y_0 \in \mathbb R$ that
\begin{toexclude}
\begin{align*}
    \left\vert Q(x+Y_0,y)- Q(x+\tilde{Y}_0,y) \right\vert \leq \int_{\tilde{Y}_0}^{Y_0} \left\vert \partial_x Q(x+s,y) \right\vert ds \leq \left\vert Y_0 - \tilde{Y}_0 \right\vert \| \partial_x Q \|_{L^\infty},
\end{align*}
and thus
\end{toexclude}
\begin{align*}
    \left\vert \int Q(x+Y_0,y)\partial_x Q^2(x,y)dxdy - \int Q(x+\tilde{Y}_0,y) \partial_x Q^2(x,y)dxdy \right\vert \leq \| \partial_x Q \|_{L^\infty} \| Q \|_{H^1}^2 \left\vert Y_0 -\tilde{Y}_0 \right\vert.
\end{align*}

Therefore, for any initial condition $(Y_0,Y_1)\in \mathbb{R}^2$, we deduce from the Cauchy-Lipschitz-Picard theorem the existence of a unique globally defined solution.

We now describe the phase portrait of the solutions. Let us first focus on initial conditions $(Y_0,Y_1) \in \mathbb{R}_+^*\times \mathbb{R}_+^*$. For such initial conditions, we claim that the curve $\left\{ (Z(t), \dot{Z}(t)) : t \in \mathbb{R}_+\right\}$ remains in the domain $\mathbb{R}_+^*\times \mathbb{R}_+^*$ and that $(Z(t),\dot{Z}(t)) \to (+\infty,l_1)$ when $t \to+\infty$, for some positive $l_1$. 

Indeed, it follows form \eqref{defi:eq_Z_general} that 
\begin{align*}
 \ddot{Z}(t) = \kappa \int_{y \in \mathbb{R}} \int_{x \ge 0} \left(Q(x+Z(t),y)-Q(-x+Z(t),y) \right) \partial_x \left(Q^2\right)(x,y)dx dy >0
\end{align*}
since $\partial_x(Q^2)$ is odd in $x$ and $Q$ is positive and decreasing for $x \ge 0$. Thus $Z$ is convex and $\dot{Z}(t)>Y_1>0$ on $\mathbb R_+$. This proves that the curve remains in this domain. Furthermore, $Z$ is bounded by below by $Y_1t$, so that $Z(t)$ tends to $+\infty$ at $+\infty$. 

To study the limit of $\dot{Z}$, we use the Hamiltonian defined in \eqref{defi:H_static}. Observe that for a fixed $Y_1$ (respectively $Y_0$), the function $H(\cdot, Y_1)$ is strictly decreasing (resp. $H(Y_0,\cdot)$ is strictly increasing). We deduce from Lemma \ref{lemm:hamiltonian} and \eqref{est.decay.QzQ} that
\begin{align}\label{limit_hamiltonian}
    \lim_{t\rightarrow +\infty} \dot{Z}(t)^2 = \lim_{t\rightarrow +\infty} 2H(Z(t), \dot{Z}(t)) = 2H(Y_0,Y_1) .
\end{align}

Observe that each curve of the phase portrait in the region $\mathbb{R}_+^\star \times\mathbb{R}_+^\star$ corresponds to a different Hamiltonian. Indeed, if two initial conditions $(Y_0,Y_1)$ and $(\tilde{Y}_0,\tilde{Y}_1)$ are associated to the same Hamiltonian $H(Y_0,Y_1)=H(\tilde{Y}_0,\tilde{Y}_1)$, let us denote by $Z$ and $\tilde{Z}$ the two corresponding solutions. Assume without loss of generality that $Y_0 < \tilde{Y}_0$. Then, $Z$ is bijective from $[0,+\infty)$ to $[Y_0,+\infty)$, and thus $Z(t_1)=\tilde{Y}_0$ for some $t_1>0$. Since $H(Z(t_1),\dot{Z}(t_1))=H(Y_0,Y_1)=H(\tilde{Z}(0),\dot{\tilde{Z}}(0))$ and $Y_1 \mapsto H(\tilde{Y}_0,Y_1)$ is strictly increasing on $\mathbb{R}_+$, we have $\dot{\tilde{Z}}(0)=\tilde{Y}_1=\dot{Z}(t_1)$. By uniqueness of the Cauchy-Lipschitz-Picard theorem, we conclude that the two solutions are the same up to a time translation.

By \eqref{limit_hamiltonian}, we conclude that each curve of the phase portrait has a different limit for $\dot{Z}$ at infinity. 

We continue by describing the set of solutions on the boundary of the domain of initial conditions $(Z(0),\dot{Z}(0))=(Y_0,Y_1)\in \mathbb{R}_+^*\times \mathbb{R}_+^*$. First, notice that if $(Y_0,Y_1)=(0,0)$, then the solution is constant $Z(t)=0$. In the case where $(Y_0,Y_1)=(Z_0,0)$ with $Z_0>0$, then by strict convexity of $Z$, the graph of $(Z(t),\dot{Z}(t))$ for any time $t>0$ remains in $\mathbb{R}_+^*\times \mathbb{R}_+^*$. In the case $(Y_0,Y_1)=(0,\mu_1)$ with $\mu_1>0$, we have $Z(t)>0$ close to $0$ and thus the graph remains in the same domain.

We deduce the description of whole set of solutions by symmetry from the set of solutions corresponding to initial conditions in $\mathbb{R}_+\times \mathbb{R}_+$. This concludes the proof of the proposition.
\end{proof}


\begin{toexclude}
\begin{rema}
By similar computations, we can also prove that for $Z>Z_0^*$, we have
\begin{align}\label{eq:monotonic}
    \int_{\mathbb{R}^2} \partial_x Q(x+Z,y) \partial_x Q^2(x,y) dxdy \leq -\frac{\kappa a_0}{2} \frac{e^{-Z}}{\sqrt{Z}}\int_{\mathbb{R}^2} e^{-x} \partial_x(Q^2)(x,y) dxdy <0.
\end{align}
\end{rema}
\begin{proof}[Alternative proof for a solution with an initial condition at $-\infty$]
With the initial conditions at $-\infty$, by multiplying the equation \eqref{defi:eq_Z} by $\dot{Z}$ and integrating, we have
\begin{align}\label{equation_Z_dot}
    2 \mu_0^2 - \frac{1}{2} \dot{Z}(t)^2 = \frac{2}{\langle \Lambda Q, Q \rangle} \int_{\mathbb{R}^2} Q(x+Z(t),y) Q^{2}(x,y) dxdy.
\end{align}

We first prove the existence of a solution for $\mu_0$ small enough. By defining $Z_0^*$ as in the remark above, we define $\mu_0^*$ by the relation \eqref{defi:Z_0} with $Z_0^*$ instead of $Z_0$. We now fix $\mu_0\leq \mu_0^*$ and define $Z_0$ satisfying relation \eqref{defi:Z_0}, and we obtain $Z_0>Z_0^*$. We claim that there exists a unique even solution of \eqref{defi:eq_Z} with the initial condition $(Z,\dot{Z})(0)=(Z_0,0)$ and that this function is also a solution to the equation of the lemma.

Notice that for $Z>Z_0^*$ the function $Z \mapsto \int Q(x+Z,y) \partial_x (Q^2(x,y))dxdy$ is positive because, by \eqref{eq:monotonic}, we have 
\begin{align*}
    \frac{d}{dZ} \int_{\mathbb{R}^2} Q(x+Z,y) \partial_x Q^2(x,y) dxdy <0, \quad \lim_{Z\rightarrow +\infty} \int_{\mathbb{R}^2} Q(x+Z,y) \partial_x Q^2(x,y) dxdy =0.
\end{align*}


Since $\Ddot{Z}>0$, the function $Z$ is strictly convex and the minimum is achieved at $0$. Since $\dot{Z}$ is bounded by \eqref{equation_Z_dot}, the solution $Z$ is globally defined. Note that by local uniqueness, $Z$ is even in $t$. To check the conditions at $-\infty$, we notice that since $Z$ decreases on $\mathbb{R}_-$, for $\epsilon$ small enough we have
\begin{align*}
    0<2\mu_0^2- \frac{2}{\langle \Lambda Q, Q \rangle} \int_{\mathbb{R}^2} Q(x+Z_0+\epsilon,y) Q^2(x,y) dx dy < \frac{1}{2} \dot{Z}(t)^2,
\end{align*}
and thus $Z(t)$ goes to $+\infty$ at $-\infty$. Plugging in this condition in the equation satisfied by $\dot{Z}$, we obtain
\begin{align*}
    \lim_{t\rightarrow -\infty} \dot{Z}(t)=-2\mu_0.
\end{align*}

As a consequence, we obtain $Z(t)\leq - C t$ for some positive constant $C$ for $t\leq -T_0$ for $T_0$ large enough. We also have $\left( 2\mu_0 -\dot{Z}(t) \right) \left( 2\mu_0 +\dot{Z}(t)\right)>0$, and thus $2\mu_0+\dot{Z}(t)$ is positive on an interval $(-\infty,-T_0)$ increasing $T_0$ if necessary. For the next order term, by using the equation satisfied by $\dot{Z}$, we have
\begin{align*}
\left\vert \left( 2\mu_0 -\dot{Z}(t) \right) \left( 2\mu_0 +\dot{Z}(t)\right) \right\vert \lesssim e^{C t},
\end{align*}
and thus $\dot{Z}(t) +2\mu_0 = O_{-\infty}(e^{Ct})$. An integration from $t<-T_0$ to $-T_0$ for $T_0$ large enough provides
\begin{align*}
    \left\vert Z(t) +2\mu_0t -\left( Z(-T_0) +2\mu_0 T_0 \right) \right\vert \lesssim e^{CT_0}.
\end{align*}
By monotonicity of $t\mapsto Z(t)+2\mu_0 t$ close to $-\infty$, we obtain that $Z(t)-2\mu_0t$ converges to a finite limit $l$. Notice that by convexity, $Z$ is above its asymptote at $-\infty$. We obtain the condition of the lemma
\begin{align*}
    \lim_{t\rightarrow -\infty} \left( Z(t) + 2\mu_0 t -l \right)=0.
\end{align*}

We now prove the uniqueness of the equation of the lemma for values of $\mu_0$ lower that $\frac{1}{2} \mu_0^*$. The following arguments are due the interpretation of a phase portrait. Suppose that for $\mu_0 < \frac{1}{2} \mu_0^*$, there exists another even solution $Y$ to \eqref{defi:eq_Z} with the condition
\begin{align*}
    \lim_{t\rightarrow -\infty} (Y(t),\dot{Y}(t)) = (+\infty, -2\mu_0).
\end{align*}
By continuity of $Y$ and $\dot{Y}$, there exists $t_0$ close to $-\infty$ such that $\dot{Y}(t_0) \in (-\mu_0^*,0)$ and $Y(t_0)$ large enough such that
\begin{align*}
    \frac{1}{\langle \Lambda Q, Q \rangle} \int_{\mathbb{R}^2} Q(x + Y(t_0),y) Q^2(x,y) dxdy \leq\frac{1}{4} (\mu_0^*)^2.
\end{align*}
Now consider the following final velocity
\begin{align*}
    \tilde{\mu}^2_0 := \frac{\dot{Y}(t_0)^2}{4} + \frac{1}{\langle \Lambda Q, Q \rangle} \int_{\mathbb{R}^2} Q(x+Y(t_0),y) Q^2(x,y) dx dy \leq \frac{1}{2} (\mu_0^*)^2 < (\mu_0^*)^2,
\end{align*}
and denote by $\tilde{Y}$ the solution constructed at the previous step associated to the velocity $\tilde{\mu}_0$. In particular, since $\dot{\tilde{Y}}$ strictly increases on $(-\infty;0]$ from $-2\tilde{\mu}_0^2$ to $0$, there exists a unique time $t_1$ such that $\dot{\tilde{Y}}(t_1)=\dot{Y}(t_0)$. By the relation \eqref{equation_Z_dot}, we also have $\tilde{Y}(t_1)=Y(t_0)$. Since the equation is autonomous, by uniqueness around one initial condition, $\tilde{Y}$ is a time-translated version of $Y$. Since the two functions are convex and even, they are equal. Thus $Y$ is one of the solution constructed at the previous step, and by the uniqueness of the asymptotic development, we have $\tilde{\mu}_0=\mu_0$. It concludes the proof of the uniqueness of the even solution.
\end{proof}
\end{toexclude}

\begin{coro}\label{coro:initial_condition}
Let $l_1\in \mathbb{R}_+^*$. There exists a unique solution, up to time translation, to \eqref{defi:eq_Z_general} with the condition
\begin{align} \label{coro:initial_condition.1}
    \lim_{t\rightarrow -\infty} ( Z(t), \dot{Z}(t) ) = (+\infty,-l_1).
\end{align}
\end{coro}

\begin{proof}
Let us split the different possibilities of $l_1$, and compare with the hamiltonian \eqref{defi:H_static}. We begin with the existence of a solution.
\begin{itemize}
    \item $l_1^2 < 2\kappa \| Q \|_{L^3}^3$. In this case, define $Y_0$ as the unique positive real coefficient such that
    \begin{align}\label{defi:Z_0_l_1}
        l_1^2 = 2\kappa \int_{\mathbb{R}^2} Q(x+Y_0,y) Q^2(x,y) dxdy.
    \end{align}
    The solution $Z$ associated to the initial condition $(Y_0,0)$ has the hamiltonian $H(Y_0,0)=\frac{l_1^2}{2}$, and thus is a solution to the problem by \eqref{limit_hamiltonian}.
    \item $l_1^2 > 2\kappa \| Q \|_{L^3}^3$. In this case, define $Y_1$ as the unique negative real coefficient such that to
    \begin{align*}
         Y_1^2 =l_1^2 -2\kappa \| Q\|_{L^3}^3.
    \end{align*}
    The solution $Z$ associated to the initial condition $(0,Y_1)$ has the hamiltonian $H(0,Y_1)=\frac{l_1^2}{2}$, and is a solution to the problem by \eqref{limit_hamiltonian}. 
    \item $l_1^2 = 2\kappa \| Q \|_{L^3}^3$. This case corresponds to the blue curve on the phase portrait, ending at $(+\infty,-l_1)$ as $t\rightarrow -\infty$ and to $(0,0)$ as $t\rightarrow+\infty$.
\end{itemize}
By the previous discussion on the phase portrait, each of the solution is uniquely determined by the hamiltonian, thus the uniqueness of the solution up to time translation.
\end{proof}

\begin{lemm}[Asymptote]\label{lemm:asymp_Z}
Let $0<l_1$ with $l_1^2 < 2\kappa \| Q \|_{L^3}^3$. Then the unique solution $Z$ of \eqref{defi:eq_Z_general} with the initial condition $(Z(0),\dot{Z}(0))=(Y_0,0)$, with $Y_0$ defined in \eqref{defi:Z_0_l_1}, is even and has an asymptote at $+\infty$. More precisely, there exist constants $l = l(l_1)\in \mathbb{R}$ and $c=c(l_1)>0$ such that
\begin{align*}
    \left\vert Z(t) -l_1 t -l \right\vert \lesssim e^{-ct}.
\end{align*}
\end{lemm}

\begin{proof}
If $Z(t)$ is a solution to \eqref{defi:eq_Z_general} with the initial condition $(Y_0,0)$, then $t\mapsto Z(-t)$ is also a solution, and by uniqueness we have $Z(t)=Z(-t)$. We can thus focus on the case $t\geq 0$.

Since $Z(t)>0$, for all $t >0$ (from Proposition \ref{propo:phase_portrait}), 
\begin{align*}
    \ddot{Z}(t) \geq \kappa \int_{\mathbb{R}^2} Q(x+Z,y) \partial_x (Q^2)(x,y) dx dy>0 .
\end{align*}
Hence, for a fixed $t_0>0$ and any $t>t_0$, we have $Z(t)>\dot{Z}(t_0)(t-t_0)$ with $\dot{Z}(t_0)>0$. Moreover, it follows from Lemma \ref{lemm:hamiltonian} that
\begin{align}\label{equation_Z_dot}
    \frac{l_1^2}{2} = H(Z(t),\dot{Z}(t)) = \frac{\dot{Z}(t)^2}{2} + \kappa \int_{\mathbb{R}^2} Q(x+Z(t),y) Q^2(x,y) dx dy.
\end{align}
As underlined in the proof of Proposition \ref{propo:phase_portrait}, the function $Z \mapsto \int Q(x+Z,y) Q^2(x,y)dx dy$ is strictly decreasing on $\mathbb{R}_+$. Thus for $t\geq t_0$
\begin{align*}
    l_1^2 - 2\kappa \int_{\mathbb{R}^2} Q(x+\dot{Z}(t_0)(t-t_0),y) Q^2(x,y) dx dy \leq \dot{Z}(t)^2 .
\end{align*}
From Lemma \ref{lemm:fg}, there exist two positive constants $c$ and $C$ such that, for $t$ large enough,
$ l_1^2- Ce^{-ct} \leq \dot{Z}(t)^2$, and thus $  0<l_1-\dot{Z}(t) \leq \frac{C}{l_1}e^{-ct}$.
An integration from $t$ to $+\infty$ yields
\begin{align*}
    0 \leq \lim_{s\rightarrow +\infty} \left( l_1 s -Z(s)\right) - l_1t +Z(t) \lesssim e^{-ct}. 
\end{align*}
Since $l_1 t-Z(t)$ is concave and bounded on $\mathbb{R}_+$, it has a finite limit $l$ at $+\infty$, which concludes the proof of the lemma.
\end{proof}

The following proposition gives an approximation of the right-hand-side of \eqref{defi:eq_Z_general} for $Z$ large.

\begin{propo}\label{propo:approx_e^-Z}
There exist two positive constants $Z^\star$ and $c$ such that it holds, for any $Z>Z^\star$ and $i\in \{0,1,2\}$,
\begin{align}
    & \left\vert \int_{\mathbb{R}^2} Q (x+Z,y) \partial_x^i(Q^2)(x,y) dx dy -(-1)^i c\frac{e^{-Z}}{\sqrt{Z}} \right\vert \lesssim \frac{e^{-Z}}{Z},  \label{propo:approx_e^-Z.1}\\
    & \left\vert \int_{\mathbb{R}^2} \left( \partial_x Q + Q \right)(x+Z,y) Q^2(x,y) dx dy \right\vert \lesssim \frac{e^{-Z}}{Z}. \label{propo:approx_e^-Z.2}
\end{align}
\end{propo}

\begin{rema}
Proposition \ref{propo:approx_e^-Z} echoes Remark \ref{rem:non_explicit_solutions}. Indeed, it follows from \eqref{propo:approx_e^-Z.1} that the asymptotic of \eqref{defi:eq_Z} has the form $\ddot{Z}=ce^{-Z}Z^{-\frac12}$ which does not possess explicit solutions, unlike its $1$-dimensional equivalent $\ddot{Z}=ce^{-Z}$ (see \cite{MM11}). 
\end{rema}
\begin{toexclude}
Notice that this proposition echoes Remark \ref{rem:non_explicit_solutions} with the equation $\ddot{Z}=ce^{-Z}$ found in \cite{MM11}. Indeed, the assumption to have $Z$ large makes sense in the situation that the two solitary waves remain far one from another. The differences between the $1$ and the $2$-dimensional cases are the constant coefficient, but also the term $Z^{-\frac{1}{2}}$ that comes from the dimension $2$.
\end{toexclude}

\begin{proof}
We only prove estimate \eqref{propo:approx_e^-Z.1} in the case $i=0$, since the proof for the other cases is similar. On the one hand, it follows by using \eqref{est.decay.QzQ} and taking $Z^\star$ large enough that
\begin{align*}
\MoveEqLeft
    \int_{\vert(x,y)\vert \geq Z^\frac14} Q(x+ Z,y) Q^2(x,y)dxdy \\
    & \leq C \sup_{(x,y)} \left(Q(x+Z,y) Q(x,y)\right) \int_{\vert(x,y)\vert>Z^\frac14} \vert Q(x,y) \vert dx dy \lesssim e^{- Z-Z^\frac14}.
\end{align*}
On the other hand, we infer from the asymptotic expansion of Proposition \ref{propo:Q_n_app_gen} that
\begin{align*}
    \left\vert \int_{\vert (x,y)\vert \leq Z^\frac14} Q(x+Z,y) Q^2(x,y) dx dy - \frac{e^{-Z}}{Z^\frac{1}{2}} \left( \kappa a_0 \int_{\vert (x,y)\vert \leq Z^\frac14} e^{-x} Q^2(x,y) dx dy \right) \right\vert \lesssim \frac{e^{-Z}}{Z}.
\end{align*}
Therefore, we conclude the proof of \eqref{propo:approx_e^-Z.1} in the case $i=0$ by gathering these two estimates. 

Estimate \eqref{propo:approx_e^-Z.2} follows by combining \eqref{propo:approx_e^-Z.1} in the case $i=0$ with \eqref{propo:approx_e^-Z.1} in the case $i=1$ and observing that, after integration by parts, the terms of order $e^{-Z}Z^{-\frac12}$ cancel each other out. Note that it is important that the constant $c$ appearing in \eqref{propo:approx_e^-Z.1} is independent of $i \in \{0,1\}$. 
\end{proof}

\subsection{Quantified evolution of a function \texorpdfstring{$z$}{z} close to \texorpdfstring{$Z$}{Z}}

While the previous subsection dealt with the general equation \eqref{defi:eq_Z_general} for any arbitrary positive constant $\kappa$, this subsection focuses on the ODEs \eqref{defi:eq_Z_general} which are close to \eqref{defi:eq_Z}. More precisely, for a parameter $\nu\in(-\frac12,\frac12)$, we are interested in the following equation and its associated Hamiltonian
\begin{align}\label{defi:eq_Z_nu}
    & \Ddot{Z}(t)= \frac{2(1+\nu)}{\langle \Lambda Q, Q \rangle} \int_{\mathbb{R}^2} Q(x + Z(t),y) \partial_x (Q^2) (x,y) dxdy \\
    & H_\nu (Y_0,Y_1) := \frac{Y_1^2}{2} + \frac{2(1+\nu)}{\langle \Lambda Q, Q \rangle} \int_{\mathbb{R}^2} Q(x+Y_0,y) Q^2(x,y) dx dy. \label{defi:H_nu}
\end{align}
Recall from \eqref{eq:id2_Q} that $\langle \Lambda Q, Q \rangle>0$. The case $\nu=0$ corresponds to \eqref{defi:eq_Z}, where $Z$ is a good approximation of the distance $z$ between the $2$ solitary waves during the collision. However, due to the error terms, $z$ is not an exact solution of \eqref{defi:eq_Z}. This is why we study the set of equations \eqref{defi:eq_Z_nu} with a small parameter $\nu$.

\begin{lemm}[Equivalent of $l_1$ for $Y_0$ large]\label{lemm:l_1_Z_0}
There exist $Y^\star>0$ large enough and two positive constants $c$ and $C$ such that the following holds. For any $\nu \in (-\frac12,\frac12)$ and $Y_0>Y^\star$, define $l_1>0$ by $H_\nu(Y_0,0) = \frac{l_1^2}2$. Then the following holds
\begin{align}\label{eq:comparison_l_1_e^Z_0}
    c\frac{e^{-Y_0}}{\sqrt{Y_0}} \leq l_1^2 \leq C \frac{e^{-Y_0}}{\sqrt{Y_0}}.
\end{align}
\end{lemm}

\begin{proof}
From the definition of $H_\nu$ in \eqref{defi:H_nu}, we have 
\begin{align*} 
l_1^2=\frac{4(1+\nu)}{\langle \Lambda Q, Q \rangle} \int_{\mathbb{R}^2} Q(x+Y_0,y) Q^2(x,y) dx dy ,
\end{align*}
which implies \eqref{eq:comparison_l_1_e^Z_0} by using \eqref{propo:approx_e^-Z.1} with $i=0$ and choosing $Y^\star$ large enough. Note that the constants $c$ and $C$ can be chosen uniformly in $\nu \in (-\frac12,\frac12)$.
\end{proof}

\begin{lemm}[Uniform lower bound of $\dot{Z}$ close to $0$]\label{lemm:dot_Z_T''}
There exist $Y^\star>0$, $\eta^\star>0$ small and positive constants $c$ and $C$ such that the following holds. For any $Y_0>Y^\star$, $\nu \in (-\frac12, \frac12)$ and $\eta\in (0, \eta^\star)$, consider $Z$ the unique even $\mathcal{C}^2$-function solving $H_\nu(Z,\dot{Z})=H_\nu(Y_0,0)$ and $T$ the unique positive time such that $Z(T) = Y_0+\eta^2$. Denote $l_1$ as in \eqref{defi:Z_0_l_1}. Then
\begin{align*}
    c \eta l_1 \leq \dot{Z}(T) \leq C \eta l_1.
\end{align*}
\end{lemm}

\begin{proof}
Recall that the existence and uniqueness of the even function $Z$ are given in Proposition \ref{propo:phase_portrait} and Corollary \ref{coro:initial_condition}. The conservation of the Hamiltonian in \eqref{defi:H_nu} associated to $Z$ implies that
\begin{align*}
\dot{Z}(T)^2 = -\frac{4(1+\nu)}{\langle \Lambda Q, Q \rangle} \int_{0}^{Z(T)-Y_0} \left(\int_{\mathbb{R}^2} \partial_x Q(x+Y_0+s,y)  Q^2(x,y) dx dy \right) ds.
\end{align*}
Let us find an equivalent on the inner integral for $Y_0$ large enough and a fixed value $s>0$.

Observe that $0 \le Z(T)-Y_0 = \eta^2$. Inspired from the computation of \eqref{est:DL_Q_x_direction}, we adjust $\eta^\star$ so that for $s\leq \eta^2 <(\eta^\star)^2$, it holds on $\mathbb{R}^2$
\begin{align*}
    \left\vert \partial_x Q(x+ Y_0+s,y)- \partial_x Q (x+Y_0,y) \right\vert \lesssim s \frac{e^{-\left\vert (x+Y_0,y)\right\vert}}{\langle (x+Y_0,y) \rangle^\frac{1}{2}} .
\end{align*}
Thus, it follows arguing as in \eqref{est.decay.QzQ} that
\begin{align*}
\MoveEqLeft
    \left\vert \int_{s=0}^{Z(T)-Y_0} \int_{\mathbb{R}^2} \left( \partial_x Q \left( x+ Y_0+ s,y\right) - \partial_x Q(x+ Y_0,y) \right) Q^2(x,y) dxdy ds \right\vert \\
    & \lesssim \eta^4 \sup \left( \frac{e^{-\left\vert (x+Y_0,y) \right\vert }}{\langle (x+Y_0,y) \rangle^\frac12} Q(x,y)\right) \| Q \|_{L^1} \lesssim \eta^4 \frac{e^{-Y_0}}{Y_0^\frac12}.
\end{align*}
Then, we deduce from \eqref{propo:approx_e^-Z.2} that
\begin{align*}
    \left\vert \dot{Z}(T) ^2 - \frac{4(1+\nu)}{\langle \Lambda Q, Q \rangle} \int_{s=0}^{Z(T)-Y_0} \int_{\mathbb{R}^2} Q(x+ Y_0 ,y) Q^2(x,y) dxdy ds \right\vert  \leq \int_{s=0}^{Z(T)-Y_0} C \frac{e^{-Y_0}}{Y_0}  ds + C\eta^4 \frac{e^{-Y_0}}{Y_0^\frac12},
\end{align*}
which combined with \eqref{defi:Z_0_l_1} yields 
\begin{align*}
    \left\vert \dot{Z}(T)^2 - \eta^2 l_1^2 \right\vert
    \leq \eta^2 \left( C \frac{e^{-Y_0}}{Y_0} +C \eta^2 \frac{e^{-Y_0}}{Y_0^\frac12} \right). 
\end{align*}
By using Lemma \ref{lemm:l_1_Z_0}, where the constants can be chosen uniformly in $\nu$, we obtain that for $l_1$ small enough
\begin{align*}
   \eta^2 l_1^2 \left( 1- \frac{C}{\vert \ln(l_1) \vert^\frac12 } -C \eta^2\right) \leq \dot{Z}(T)^2 \leq \eta^2 l_1^2 \left( 1+ \frac{C}{\vert \ln(l_1) \vert^\frac12 } +C \eta^2\right) .
\end{align*}
We conclude the proof of the lemma by choosing $l_1^\star$ and $\eta^\star$ smaller if necessary and by taking the square root of the previous estimate.
\end{proof}

The next proposition quantifies the defect of closeness of a function to a solution to the equation \eqref{defi:eq_Z_nu}. More precisely, let us denote by $Z$ the even solution to \eqref{defi:eq_Z_nu} with $\nu=0$. If a function $z$ defined on a time interval $[t_0,t_1]$ is close to $Z$ at time $t_0$ and the flows $H_{\pm \nu}(z,\dot{z})$ are well controlled along the time, then $z$ is close to $Z$ at any time $t\in [t_0,t_1]$.

\begin{propo}[Evolution of $(z,\dot{z})$ with $H_0(z,\dot{z})$ almost constant]\label{propo:H_positive_times}
Let $\eta\in (0,1)$. There exist some constants $C>0$, $\varepsilon_0^\star$ and $\nu_6^\star \leq \frac{1}{2}$ such that the following holds. Let $\nu\in (0, \nu_6^\star)$ and $\varepsilon_0 \in (0, \varepsilon_0^\star)$. Let us define a constant $h < H_{-\nu_6^\star}(0,0)$ small and the three hamiltonian $H:=H_0$, $H_+:= H_\nu$ and $H_-:=H_{-\nu}$. We set $Z$ the unique even $\mathcal{C}^1$-function such that $H(Z,\dot{Z})=h$ and $t_0>0$ satisfying $Z(t_0)\geq Z(0)+\eta^2$. Let a $\mathcal{C}^1$-function $z$ defined on a time interval $[t_0,t_1]$ that satisfies the initial condition
\begin{align*}
    \left\vert z(t_0) -Z(t_0) \right\vert \leq \varepsilon_0,
\end{align*}
and for any $t\in [t_0,t_1]$
\begin{align}\label{condition:H_+(z)_H_-(z)}
    H_-(z(t),\dot{z}(t)) \leq h \leq H_+(z(t),\dot{z}(t)).
\end{align}
Then for any $t$ in $[t_0,t_1]$ we have
\begin{align*}
    \left\vert z(t) -Z(t) \right\vert \leq C \nu + C\varepsilon_0.
\end{align*}
\end{propo}

\begin{figure}[h]
    \centering
    \begin{tikzpicture}[scale=0.85]
            \draw[->,,samples=\Num] (-8,0) -- (-4,0);
            \draw[->,samples=\Num] (-7.5,-1.5) -- (-7.5,1.5);
            \draw (-3.7,0) node[below] {$Z_-$};
            \draw (-7.5,1.5) node[left] {$\dot{Z}_-$};
            
            \draw [->,>=latex,domain=-1.6:1,samples=\Num] plot [variable=\t] (-6.6+\t*\t,{rad(atan(\t))/1.5});
            \draw [domain=1:1.6,samples=\Num] plot [variable=\t] (-6.6+\t*\t,{rad(atan(\t))/1.5});
            \draw (-4,0.75) node[right] {$h_0$};
            
            \draw [->,blue,>=latex,domain=-1.75:1,samples=\Num] plot [variable=\t] (-7.1+\t*\t,{rad(atan(\t))*1.2});
            \draw [blue,domain=1:1.75,samples=\Num] plot [variable=\t] (-7.1+\t*\t,{rad(atan(\t))*1.2});
            \draw (-4,1.25) node[right] {$h$};
            
            \draw [->,>=latex,domain=-1.85:1,samples=\Num] plot [variable=\t] (-7.4+\t*\t,{rad(atan(\t))*1.5});
            \draw [domain=1:1.85,samples=\Num] plot [variable=\t] (-7.4+\t*\t,{rad(atan(\t))*1.5});
            \draw (-4,1.75) node[right] {$h_1$};
            
            \draw [red,domain=1.2:1.75,samples=\Num] plot [variable=\t] (-7.25+\t*\t,{rad(atan(\t))*(1+\t *sin(\t*1200)/15});

            \draw[->,samples=\Num] (-2,0) -- (2,0);
            \draw[->,samples=\Num] (-1.5,-1.5) -- (-1.5,1.5);
            \draw (2.3,0) node[below] {$Z$};
            \draw (-1.5,1.5) node[left] {$\dot{Z}$};
            
            \draw [->,>=latex,domain=-1.5:1,samples=\Num] plot [variable=\t] (-0.3+\t*\t,{rad(atan(\t))/2});
            \draw [domain=1:1.5,samples=\Num] plot [variable=\t] (-0.3+\t*\t,{rad(atan(\t))/2});
            \draw (2,0.5) node[right] {$h_0$};
            
            \draw [->,>=latex,domain=-1.65:1,samples=\Num] plot [variable=\t] (-0.75+\t*\t,{rad(atan(\t))});
            \draw [domain=1:1.65,samples=\Num] plot [variable=\t] (-0.75+\t*\t,{rad(atan(\t))});
            \draw (2,1) node[right] {$h$};
            
            \draw [->,>=latex,domain=-1.75:1,samples=\Num] plot [variable=\t] (-1.1+\t*\t,{rad(atan(\t))*1.4});
            \draw [domain=1:1.75,samples=\Num] plot [variable=\t] (-1.1+\t*\t,{rad(atan(\t))*1.4});
            \draw (2,1.5) node[right] {$h_1$};
            
            \draw [red,domain=1.2:1.75,samples=\Num] plot [variable=\t] (-1.25+\t*\t,{rad(atan(\t))*(1+\t *sin(\t*1200)/15});
        
            \draw[->,samples=\Num] (4,0) -- (8,0);
            \draw[->,samples=\Num] (4.5,-1.5) -- (4.5,1.5);
            \draw (8.5,0) node[below] {$Z_+$};
            \draw (4.5,1.5) node[left] {$\dot{Z}_+$};
            
            \draw [->,>=latex,domain=-1.45:1,samples=\Num] plot [variable=\t] (5.9+\t*\t,{rad(atan(\t))/2.5});
            \draw [domain=1:1.45,samples=\Num] plot [variable=\t] (5.9+\t*\t,{rad(atan(\t))/2.5});
            \draw (8,0.25) node[right] {$h_0$};
            
            \draw [->,blue,>=latex,domain=-1.6:1,samples=\Num] plot [variable=\t] (5.5+\t*\t,{rad(atan(\t))*0.8});
            \draw [blue,domain=1:1.6,samples=\Num] plot [variable=\t] (5.5+\t*\t,{rad(atan(\t))*0.8});
            \draw (8,0.75) node[right] {$h$};
            
            \draw [->,>=latex,domain=-1.75:1,samples=\Num] plot [variable=\t] (5+\t*\t,{rad(atan(\t))*1.3});
            \draw [domain=1:1.75,samples=\Num] plot [variable=\t] (5+\t*\t,{rad(atan(\t))*1.3});
            \draw (8,1.25) node[right] {$h_1$};
            
            \draw [red,domain=1.2:1.75,samples=\Num] plot [variable=\t] (4.75+\t*\t,{rad(atan(\t))*(1+\t *sin(\t*1200)/15});
        
    \end{tikzpicture}
    \caption{Comparison of a trajectory $(z,\dot{z})$ in red with the phase portrait associated to the equations related to $H_-$ (left), $H$ (middle) and $H_+$ (right). On each graph, the three lines are associated to three values $h_0<h<h_1$.}
    \label{fig:z_Z_-_Z_+}
\end{figure}
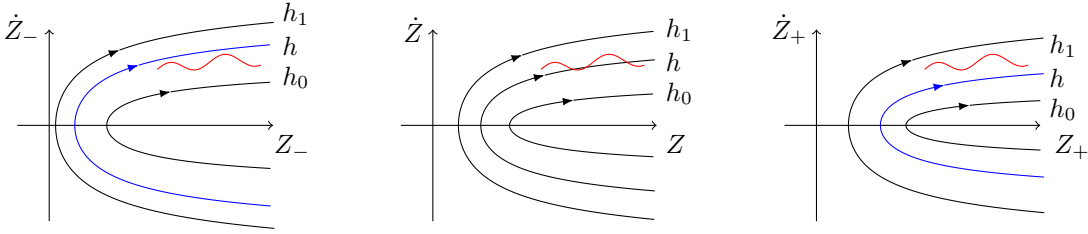

\begin{rema}
    In the proof of the collision, we will set $h=2\mu_0^2$.
\end{rema}

\begin{proof} For a constant $C_\star$ to define later, define the inequality
\begin{align}\label{bootstrap_estimate_z_Z}
    \left\vert z(t) -Z(t) \right\vert \leq C_\star \nu + C_\star \varepsilon_0.
\end{align}
The coefficients $\nu_6^\star$ and $\varepsilon_0^\star$ are adjusted so that 
\begin{align}\label{restriction_e_0_nu}
    C_\star \nu_6^\star + C_\star \varepsilon_0^\star \leq \frac{\eta^2}{4}.
\end{align}
Define the time
\begin{align*}
    T^\star := \sup \left\{ t \in [t_0;t_1]; \, \forall t \in [t_0;T^\star), \text{ estimate } \eqref{bootstrap_estimate_z_Z} \text{ holds} \right\}.
\end{align*}

$T^\star$ is well-defined by assumption and continuity of $z$ and $Z$. By a bootstrap argument, we claim that $T^\star=t_1$. We suppose that $T^\star<t_1$ and show that this statement leads to a contradiction. To do so, we compare the two functions $z$ and $Z$ along the flow of two well-chosen functions $Z_-$ and $Z_+$ associated to the two hamiltonian $H_+$ and $H_-$. We thus define $Z_-$ and $Z_+$ as the unique even functions satisfying
\begin{align}\label{defi:Z_-_Z_+}
    H_-(Z_-,\dot{Z}_-)=h, \quad \text{and} \quad H_+(Z_+,\dot{Z}_+)=h.
\end{align}
Those two functions are represented in blue in Figure \ref{fig:z_Z_-_Z_+}. Remark that the constraint $h<H_{-\nu_6^\star}(0,0)$ justifies that the function $Z_-$ is associated to an initial condition under the form $(Y_0,0)$ for $Y_0>0$, and not under the form $(0,Y_1)$ as explained in Proposition \ref{propo:phase_portrait}. $Z$ (respectively $Z_-$, $Z_+$) goes from $\mathbb{R}$ to $(Z(0);+\infty)$ (resp. $Z_-(0)$, $Z_+(0)$). By abuse of notations, we denote by $Z^{-1}$ (resp. $Z_-^{-1}$, $Z_+^{-1}$) the bijection from $(Z(0);+\infty)$ (resp. $Z_-(0)$, $Z_+(0)$) to $\mathbb{R}_+^*$.

We will make use of the following claims, whose proofs are postponed.
\begin{claim}[Comparison of the three functions]\label{claim:comparison_Z_Z_+_Z_-}
At the initial time, we have
\begin{align}\label{eq:Z_Z_+_Z_-(0)}
    Z_-(0)<Z(0)<Z_+(0).
\end{align}
In particular, the function $\nu \mapsto Z_+(0)$ for $\nu$ close to $0$ is a decreasing function and
\begin{align}\label{eq:limit_nu}
    \lim_{\nu \rightarrow 0} Z_+(0) = Z(0).
\end{align}
For any $s>Z_+(0)$ it holds
\begin{align}\label{eq:Z_Z_+_Z_-(s)}
    \dot{Z}_+\circ Z_+^{-1}(s) \leq \dot{Z} \circ Z^{-1}(s) \leq \dot{Z}_-\circ Z_-^{-1} (s).
\end{align}
\end{claim}

\begin{claim}[Comparison of the derivatives of the inverse functions]\label{claim:comparison_inverse_Z_Z_+_Z_-}
The two following inequalities hold
\begin{align*}
    & \frac{1}{\dot{Z} \circ Z^{-1}(s)} - \frac{1}{\dot{Z}_- \circ Z_-^{-1}(s)} \leq \frac{2\nu}{\langle \Lambda Q, Q \rangle} \frac{\int Q(x+s,y) Q^2(x,y)dx dy}{ \left( \dot{Z} \circ Z^{-1}(s)\right)^3}, \\
    & \frac{1}{\dot{Z} \circ Z^{-1}(s)} - \frac{1}{\dot{Z}_+\circ Z_+^{-1}(s)} \geq -\frac{2\nu}{\langle \Lambda Q,Q \rangle} \frac{\int Q(x+s,y)Q^2(x,y)dxdy}{\left(\dot{Z}_+ \circ Z_+^{-1} (s)\right)^3}.
\end{align*}
\end{claim}

To continue the proof of \eqref{bootstrap_estimate_z_Z}, we split the inequality into the lower and the upper bound.

\textit{Upper bound of \eqref{bootstrap_estimate_z_Z}}. We prove in this part an upper bound of the quantity $z(t)-Z(t)$ in terms of the initial condition $z(t_0)-Z(t_0)$ and $\nu$. Notice first that $Z_-$ is a bijection from $(0;+\infty)$ to $(Z_-(0);+\infty)$. From \eqref{eq:Z_Z_+_Z_-(0)}, the bootstrap assumption \eqref{bootstrap_estimate_z_Z} and the restriction \eqref{restriction_e_0_nu} on $\epsilon_0$ and $\nu$, we have for any $t\in [t_0,T^\star)$ and the assumption on $t_0$
\begin{align*}
    Z_-(0) < Z(0) < Z(0) +\frac34 \eta^2 \leq z(t).
\end{align*}

The functions $Z_-^{-1}\circ z$ and $Z_-^{-1}\circ Z$ are well-defined on $[t_0;T^\star)$.
We now study the derivative of the difference of those functions for $t\in[t_0,T^\star)$. The assumption \eqref{condition:H_+(z)_H_-(z)} implies the bound
\begin{align*}
    \dot{Z}_- \circ Z_-^{-1} \circ z(t) = \left( 2h - \frac{4(1-\nu)}{\langle \Lambda Q, Q \rangle} \int_{\mathbb{R}^2} Q(x+z(t),y) Q^2(x,y) dx dy \right)^{\frac12} \geq \dot{z}(t)
\end{align*}
and thus
\begin{align}\label{est:Z_-^-1}
    \frac{d}{dt} \left( Z_-^{-1} \circ z(t) - Z_-^{-1}\circ Z(t) \right) \leq \dot{Z}(t) \left( \frac{1}{\dot{Z}\circ Z^{-1}} - \frac{1}{\dot{Z}_- \circ Z_-^{-1}}\right) \circ Z(t).
\end{align}

Plugging the estimate of Claim \ref{claim:comparison_inverse_Z_Z_+_Z_-} in \eqref{est:Z_-^-1} with Proposition \ref{propo:approx_e^-Z} with the assumption of $h$ small, we obtain
\begin{align*}
    \frac{d}{dt} \left( Z_-^{-1}\circ z(t) - Z_-^{-1} \circ Z(t) \right) \leq \frac{4\nu}{\dot{Z}(t)^2} \ddot{Z}(t).
\end{align*}
An integration from $t_0$ to $t\in (t_0,T^\star)$ with the lower bound provided by Lemma \ref{lemm:dot_Z_T''} yields
\begin{align*}
    \left( Z_-^{-1} \circ z(t) - Z_-^{-1} \circ Z(t) \right) - \left( Z_-^{-1} \circ z(t_0) - Z_-^{-1} \circ Z (t_0) \right) \leq \frac{4\nu}{c\eta l_1}.
\end{align*}
We continue using the mean value theorem which states that for any $t\in [t_0,T^\star)$ there exists a coefficient $\lambda_t \in \left[ Z(t)- C_\star \left( \nu+ \varepsilon_0\right),  Z(t)+ C_\star \left( \nu+ \varepsilon_0\right)\right]$ such that
\begin{align*}
    Z_-\circ z(t) - Z_-^{-1} \circ Z(t)= \frac{z(t)-Z(t)}{\dot{Z}_- \circ Z_-^{-1}(\lambda_t)}.
\end{align*}
Notice that the previous quantity is well-defined since $\lambda_t >Z_0 +\frac34 \eta^2 >Z_-(0)$, and we have
\begin{align*}
    z(t)-Z(t) \leq \frac{\dot{Z}_-\circ Z_-^{-1}(\lambda_t)}{\dot{Z}_-\circ Z_-^{-1}(\lambda_{t_0})} \left( z(t_0) -Z(t_0) \right) + \frac{4\nu}{c\eta l_1} \dot{Z}_-\circ Z_-^{-1}(\lambda_t).
\end{align*}
In particular, Lemma \ref{lemm:dot_Z_T''} applied to $\lambda_t >Z_-(0) +\frac12 \eta^2$ yields $\dot{Z}_-\circ Z_-^{-1}(\lambda_t) \geq c2^{-\frac12}\eta l_1$. We thus obtain with $\dot{Z}_-\leq l_1$
\begin{align}\label{est:z_Z_upper_bound}
    z(t)-Z(t) \leq \frac{\dot{Z}_-\circ Z_-^{-1}(\lambda_t)}{\dot{Z}_-\circ Z_-^{-1}(\lambda_{t_0})} \left( z(t_0) -Z(t_0) \right) + \frac{4\nu}{c\eta} \quad \text{with} \quad \frac{\dot{Z}_-\circ Z_-^{-1}(\lambda_t)}{\dot{Z}_-\circ Z_-^{-1}(\lambda_{t_0})} \in \left[ \frac{c}{\sqrt{2}} , \frac{\sqrt{2}}{c\eta}\right].
\end{align}

\textit{Lower bound of \eqref{bootstrap_estimate_z_Z}} The proof of the lower bound on $z(t)-Z(t)$ follows similar ideas, let us emphasize the main differences. First, the functions $Z_+^{-1}\circ Z$ and $Z_+^{-1}\circ Z$ are well-defined on $[t_0,T^\star)$. Indeed, by the limit in \eqref{eq:limit_nu}, decreasing $\nu^\star$ if necessary and the assumption \eqref{restriction_e_0_nu}, we have
\begin{align*}
    Z(0) < Z_+(0) < Z(0)+\frac{\eta^2}{4} < Z(0)+\frac34 \eta^2 <z(t).
\end{align*}
These inequalities justify the functions $Z_+^{-1}\circ Z$ and $Z_+^{-1}\circ z$. Since Proposition \ref{propo:approx_e^-Z} provides the inequality
\begin{align*}
    \frac{2 \nu}{\langle \Lambda Q, Q \rangle} \int Q(x+Z(t),y) Q^2(x,y) dx dy \leq 2 \nu \ddot{Z}_+ \circ Z_+^{-1} \circ Z(t)
\end{align*}
we can compute, as for $Z_-$, the time derivative of the difference of the functions and using assumption \eqref{condition:H_+(z)_H_-(z)} and Claim \eqref{claim:comparison_inverse_Z_Z_+_Z_-}
\begin{align*}
    \frac{d}{dt} \left( Z_+^{-1}\circ z- Z_+^{-1} \circ Z \right) (t) \geq -2 \nu \frac{\ddot{Z}_+\circ Z_+^{-1} \circ Z(t)}{\left( \dot{Z}_+ \circ Z_+^{-1} \circ Z (t) \right)^3} \dot{Z}(t). 
\end{align*}

An integration from $t_0$ to $t$ yields, by the change of variable $s'=Z_+^{-1}\circ Z(s)$
\begin{align*}
\MoveEqLeft
    \left(Z_+^{-1} \circ z - Z_+^{-1} \circ Z \right) (t) - \left( Z_+^{-1} \circ z - Z_+^{-1} \circ Z \right) (t_0) \\
        & \qquad \geq 2 \nu \left( \frac{1}{\dot{Z}_+\circ Z_+^{-1} \circ Z(t)}- \frac{1}{\dot{Z}_+ \circ Z_+^{-1} \circ Z(t_0)} \right).
\end{align*}

By the bootstrap assumption \eqref{bootstrap_estimate_z_Z} and the mean value theorem, for any $t\in [t_0, T^\star)$, there exists $\lambda_t$ such that $\left\vert \lambda_t - Z(t) \right\vert \leq C_\star \nu + C_\star \varepsilon_0$ and
\begin{align*}
    Z_+^{-1} \circ z(t) - Z_+^{-1} \circ Z (t) = \frac{1}{\dot{Z}_+ \circ Z_+^{-1}(\lambda_t)} (z(t)-Z(t)).
\end{align*}
Since $Z_+$ increases on $[0,+\infty)$, we obtain the inequality
\begin{align*}
    z(t)-Z(t) \geq \frac{\dot{Z}_+\circ Z_+^{-1}(\lambda_t)}{\dot{Z}_+\circ Z_+^{-1}(\lambda_{t_0})} (z(t_0)-Z(T_0)) -2 \nu \frac{\dot{Z}_+\circ Z_+^{-1}(\lambda_t)}{\dot{Z}_+ \circ Z_+^{-1}\circ Z(t_0)}.
\end{align*}

Since $Z(t)>Z(0)+\frac34 \eta^2 >Z_+(0)+\frac12 \eta^2$ for any $t\in [t_0,T^\star)$, by applying Lemma \ref{lemm:dot_Z_T''}, we have $\dot{Z}_+ \circ Z_+^{-1} \circ Z(t_0) \geq c\eta l_1$. We thus obtain the lower bound
\begin{align}\label{est:z_Z_lower_bound}
    z(t)-Z(t) \geq \frac{\dot{Z}_+\circ Z_+^{-1}(\lambda_t)}{\dot{Z}_+\circ Z_+^{-1}(\lambda_{t_0})} (z(t_0)-Z(T_0)) -\frac{2 \nu}{c\eta} \quad \text{with} \quad \frac{\dot{Z}_+\circ Z_+^{-1}(\lambda_t)}{\dot{Z}_+\circ Z_+^{-1}(\lambda_{t_0})} \in \left[ \frac{c}{\sqrt{2}}, \frac{\sqrt{2}}{c\eta} \right].
\end{align}

Gathering the upper bound \eqref{est:z_Z_upper_bound} and lower bound \eqref{est:z_Z_lower_bound}, we obtain that
\begin{align*}
    \left\vert z(t)-Z(t) \right\vert \leq \frac{\sqrt{2}}{c\eta} \varepsilon_0 + \frac{4}{c\eta} \nu.
\end{align*}
By choosing $C_\star > 4(c\eta)^{-1}$, we strictly improve the bound on $z(t)-Z(t)$, which contradicts the maximality of $T^\star$. Thus $T^\star=t_1$, and the proposition is proved.
\end{proof}

\begin{proof}[Proof of Claim \ref{claim:comparison_Z_Z_+_Z_-}]
By the definition of the Hamiltonian we have
\begin{align*}
    \frac{2}{\langle \Lambda Q, Q\rangle} \int Q(\cdot +Z(0),\cdot) Q^2 & = \frac{2(1-\nu)}{\langle \Lambda Q,Q \rangle} \int Q (\cdot +Z_-(0),\cdot )Q^2= \frac{2(1+\nu)}{\langle \Lambda Q,Q \rangle} \int Q (\cdot +Z_+(0),\cdot )Q^2\\
    & =h.
\end{align*}
Since $\nu$ is positive, we obtain \ref{eq:Z_Z_+_Z_-(0)}. Under the previous form, the statement on the limit of $\nu$ is straightforward.

Concerning the second inequality, all the derivatives are positive and by the definition of the Hamiltonian
\begin{align*}
    \frac{1}{1-\nu}\left( 2h -\dot{Z}_- \circ Z_-^{-1}(s)^2 \right)= 2h -\dot{Z} \circ Z^{-1}(s)^2 = \frac{1}{1+\nu} \left(2h - \dot{Z}_+ \circ Z_+^{-1}(s)^2\right)>0 ,
\end{align*}
which proves \eqref{eq:Z_Z_+_Z_-(s)} since $\nu>0$.
\end{proof}

\begin{proof}[Proof of Claim \ref{claim:comparison_inverse_Z_Z_+_Z_-}]
Let us begin with the proof of the inequality with $Z_-$. By developing, we have
\begin{align*}
    \frac{1}{\dot{Z} \circ Z^{-1}(s)} - \frac{1}{\dot{Z}_- \circ Z_-^{-1}(s)} = \frac{4\nu}{\langle \Lambda Q, Q \rangle} \frac{\int Q(x+s,y) Q^2(x,y)dx dy}{\dot{Z}\circ Z^{-1}(s) \dot{Z}_-\circ Z_-^{-1}(s) \left( \dot{Z}_- \circ Z_-^{-1}(s) + \dot{Z}\circ Z^{-1}(s)\right)}.
\end{align*}
Then, estimate \eqref{eq:Z_Z_+_Z_-(s)} concludes the proof of the first inequality. The second inequality holds with the same computations with the other side of \eqref{eq:Z_Z_+_Z_-(s)}.
\end{proof}

In the previous proposition, the constraint on the time interval $[t_0,t_1]$ is that the time interval is far enough from $0$ so that $Z(t_0) > Z(0) + \eta^2$. The coefficient $\epsilon_0^\star$ and $\nu_6^\star$ are thus adapted in terms of $\eta$. A similar statement thus holds on a time interval $t_0<t_1<0$. 

\begin{propo}\label{propo:H_negative_times}
With the notations of Proposition \ref{propo:H_positive_times}, consider the same assumptions and definitions on $\eta$, $C$, $\nu_6^\star$, $\varepsilon_0^\star$, $h$, $H$, $H_+$, $H_-$ and $Z$. Let $t_1<0$ satisfying $Z(t_1) \geq Z(0)+\eta^2$. Let $z$ be a $\mathcal{C}^1$-function defined on a time interval $[t_0,t_1]$ that satisfies the conditions
\begin{align*}
    \left\vert z(t_0) -Z(t_0) \right\vert\leq \varepsilon_0,
\end{align*}
and, for any $t\in [t_0,t_1]$,
\begin{align*}
    H_-(z(t),\dot{z}(t)) \leq h \leq H_+(z(t),\dot{z}(t)).
\end{align*}
Then for any $t\in [t_0,t_1]$, we have
\begin{align*}
    \left\vert z(t)-Z(t) \right\vert \leq C\nu + C\varepsilon_0.
\end{align*}
\end{propo}

\begin{proof}
The proof is exactly the same as the one of Proposition \ref{propo:H_positive_times}. We emphasize that at the final time $t_1$, $z(t_1)$ and $Z(t_1)$ have to be far enough from $Z(0)$. From \eqref{restriction_e_0_nu}, we indeed have
\begin{align*}
    z(t_1) \geq Z(t_1) - C_\star \left( \nu_6^\star + \varepsilon_0^\star \right) \geq Z(0) + \frac{3}{4} \eta^2,
\end{align*}
and the rest of the arguments are similar.
\end{proof}

\section*{Acknowledgements} 

The authors were supported by a Trond Mohn Forskningsstiftelse (TMF) grant. The authors would like to warmly thank Yvan Martel for many insightful conversations and for his constant encouragements. 


\bibliographystyle{plain}
\bibliography{biblio}

\end{document}